\documentclass[11pt,a4paper]{amsart}
\usepackage[utf8]{inputenc}

\pdfoutput=1 

\usepackage{amsaddr}

\title{Accessibility, planar graphs, and quasi-isometries}
\author{Joseph Paul MacManus}
\date{First draft: 23rd October, 2023. This version: 12th May, 2026}

\address{School of Mathematics,  University of Bristol, Bristol, BS8 1UG, UK}
\email{joseph.macmanus@bristol.ac.uk}


\usepackage{subcaption}

\usepackage{amsfonts}
\usepackage{amsmath}
\usepackage{amsthm}
\usepackage{amssymb}
\usepackage{stmaryrd}
\usepackage{mathrsfs}  

\usepackage[abbrev,alphabetic]{amsrefs}

\usepackage{graphicx}
\usepackage[dvipsnames]{xcolor}
\usepackage{tikz}
\usepackage{quiver}
\usepackage[colorlinks]{hyperref} 
\hypersetup{
    linkcolor=red,
    citecolor=blue,
}
\usepackage{fancyref}

\usepackage{todonotes}

\usepackage{thmtools, thm-restate}

\DeclareMathOperator{\Stab}{Stab}

\DeclareMathOperator{\diam}{diam}

\DeclareMathOperator{\QI}{QI}

\newcommand{\dHaus}[1][]{\operatorname{Haus}_{#1}}

\DeclareMathOperator{\dist}{d}

\DeclareMathOperator{\vs}{vs}
\DeclareMathOperator{\es}{es}

\newcommand{\R}{\mathbf{R}}
\newcommand{\bbS}{\mathbf{S}}
\newcommand{\HH}{\mathbf{H}}
\newcommand{\Z}{\mathbf{Z}}

\newcommand{\Qp}{\mathfrak{Q}}
\newcommand{\Pp}{\mathfrak{P}}

\newcommand{\br}{{\mathscr{B}}}

\newcommand{\E}{{\mathcal{E}}}

\newcommand{\finfaces}{\mathcal{F}^{\mathrm{c}}}
\newcommand{\inffaces}{\mathcal{F}^\infty}
\newcommand{\facedisks}{\mathcal{D}}
\newcommand{\facepaths}{\mathcal{F}}

\newcommand{\per}{{\mathcal{H}}}

\newcommand{\cV}{\mathcal V}

\newcommand{\cC}{\mathscr C}

\newcommand{\into}{\hookrightarrow}
\newcommand{\onto}{\twoheadrightarrow}
\newcommand{\actson}{\curvearrowright}

\newcommand{\pone}{\Lambda}
\newcommand{\ptwo}{\Pi}
\newcommand{\yone}{Z}
\newcommand{\ytwo}{Y}

\newcommand{\incut}[1]{\|#1 \|_{\mathrm{in}}}
\newcommand{\outcut}[1]{\|#1 \|_{\mathrm{out}}}

\newtheorem{theorem}{Theorem}[section]

\newtheorem*{theorem*}{Theorem}

\newtheorem{proposition}[theorem]{Proposition}
\newtheorem{lemma}[theorem]{Lemma}
\newtheorem{corollary}[theorem]{Corollary}
\newtheorem{claim}[theorem]{Claim}

\theoremstyle{definition}
\newtheorem{definition}[theorem]{Definition}
\newtheorem{example}[theorem]{Example}
\newtheorem{remark}[theorem]{Remark}

\makeatletter
\newcommand{\myitem}[1]{%
\item[#1]\protected@edef\@currentlabel{#1}%
}
\makeatother

\begin{document}

\begin{abstract}
    We prove that a connected, locally finite, quasi-transitive graph which is quasi-isometric to a planar graph is necessarily accessible. 
    This leads to a complete classification of the finitely generated groups which are quasi-isometric to planar graphs. In particular, such a group is virtually a free product of free and surface groups, and thus virtually admits a planar Cayley graph. 
\end{abstract}

\maketitle




\section*{Introduction}

The activity of seeking to deduce algebraic information about the structure of a finitely generated group from the geometric properties of its Cayley graphs is a classic and common pastime of geometric group theorists. Amongst the oldest examples of this is a paper of Maschke from 1896 \cite{maschke1896representation}. This work features a complete classification of those finite groups admitting a Cayley graph which is \textit{planar}. That is, it admits some embedding into the plane. 
In particular, a finite group has a planar Cayley graph if and only if it is of the form $A \times B$, where
$$
A \in \{1, \  \Z_2\} \ \ \ \text{and} \ \ \ B \in \{\Z_n, \ D_{2n}, \ A_4, \ S_4, \   A_5\}.
$$
In other words, they are precisely the finite subgroups of the group of homeomorphisms of the 2-sphere $\bbS^2$. For more discussion of this result, see \cite[Thm.~13.1.6]{knauer2019algebraic}.
Of course, the aforementioned classification is less about these groups themselves, and more a statement about finite, connected, planar, transitive graphs. Such graphs often form the skeleta of certain types of uniform polyhedra, which are famously well studied. A complete classification of those finite, connected, planar, transitive graphs can be found in \cite{fleischner1979transitive}. 

The study of infinite planar Cayley graphs began much later. Results due to Wilkie \cite{wilkie1966non} and Zieschang--Vogt--Coldewey \cite{zieschang1970flachen} classify those groups with Cayley graphs that can be embedded in the plane with no accumulation points of vertices. In particular, these results show that a one-ended group admits a planar Cayley graph only if it is either a wallpaper group or a non–Euclidean crystallographic group. See \cite[\S3]{lyndon1977combinatorial} for a good discussion. In order to study those infinite-ended planar groups, the first barrier one needs to cross is that of \textit{accessibility}. 

A celebrated theorem of Stallings \cites{stallings1968torsion, stallings1971group} says that a finitely generated group has more than one end if and only if it splits over a finite subgroup. 
The definition of an accessible group is due to Wall \cite{wall1971pairs} and says that a group is  \textit{accessible} if it splits as a graph of groups with finite edge groups, where each vertex group has at most one end. In other words, if the process of iteratively taking a finite splitting and passing to a vertex group necessarily terminates. It is a consequence of the Grushko--Neumann theorem \cites{grushko1940bases, neumann1943number} that all finitely generated torsion-free groups are accessible. It was conjectured by Wall in 1971 that all finitely generated groups are accessible. In 1985, Dunwoody demonstrated that every finitely presented group is accessible \cite{dunwoody1985accessibility}. Eight years later, Dunwoody presented the first known example of an inaccessible finitely generated group \cite{dunwoody1993inaccessible}. 
Returning to the plane, it was noted by Levinson--Maskit \cite{levinson1975special} that a consequence of Maskit's planarity theorem \cite{maskit1965theorem} is that if a Cayley graph of a given group admits a `point-symmetric embedding' in the plane then the aforementioned group is finitely presented (and thus accessible). This was later extended to all planar Cayley graphs by Droms in \cite{droms2006infinite}. A recent account of Maskit's theorem has been given by Bowditch in \cite{bowditch2022notes}. It is shown in \cite{arzhantseva2004cayley} by Arzhantseva--Cherix that `most' Cayley graphs of finitely presented groups are non-planar, in a certain statistical sense. They also prove that the admission of a planar Cayley graph is a property preserved by free products. 
An enumeration of planar Cayley graphs is given by Georgakopoulos--Hamann in \cites{georgakopoulos2019planari, georgakopoulos2019planarii}. Several other results relating to planar Cayley graphs are given in \cite{georgakopoulos2020planar}, including the observation that any group acting properly discontinuously on a planar manifold admits a planar Cayley graph. 

The concept of accessibility was later reframed by Thomassen--Woess as a graph theoretical property. They say that a connected, locally finite graph is accessible if there exists some positive integer $k \geq 1$ such that any pair of distinct ends can be separated by the removal of at most $k$ vertices. It is shown in \cite{thomassen1993vertex} that a finitely generated group is accessible if and only if its Cayley graphs are accessible as graphs, and thus the two definitions coincide. 
When studying infinite connected graphs, it is common to assume that our graphs are \textit{quasi-transitive}. That is, that the automorphism group acts with finitely many orbits on the vertex set. Dunwoody showed in \cite{dunwoody2007planar} that any connected, locally finite, quasi-transitive, planar graph is accessible. Another proof of this fact was later given by Hamann in \cite{hamann2018planar} by studying the cycle space of such a graph (see also \cite{hamann2018accessibility}). 
Recently, this result was strengthened by Esperet--Giocanti--Legrand-Duchesne in \cite{esperet2023structure}, where it is proven that any connected, locally finite, quasi-transitive, minor-excluded graph is accessible. 

It is interesting to note that accessibility is, in some sense, a `coarse' property. More precisely, after Gromov we say that two metric spaces $X$ and $Y$ are \textit{quasi-isometric} if there exists maps $f : X \to Y$, $g : Y \to X$ such that both $f$ and $g$ preserve distances up to some additive and multiplicative error, and $f \circ g$ lies a bounded distance from the identity map under the sup-norm (see \S\ref{sec:qi}). 
Accessibility in the above sense is readily seen to be a quasi-isometry invariant amongst locally finite graphs, and so we would expect coarse assumptions to have an effect on accessibility. For example, Dunwoody's proof in \cite{dunwoody1985accessibility} actually shows that 'coarsely simply connected' quasi-transitive graphs are accessible. Also, M\"oller \cite{moller1996accessibility} gives a characterisation of inaccessible quasi-transitive graphs as those with uncountably many `thick' ends, also a coarse property. 
Little else is known about how coarse hypotheses affect accessibility. To this end, we now state our main theorem.

\begin{restatable}{alphtheorem}{acc}\label{thm:acc-intro}
Let $X$ be a connected, locally finite, quasi-transitive  graph, and suppose $X$ is quasi-isometric to a planar graph.  Then $X$ is accessible. 
\end{restatable}

By \textit{planar graph} we mean a connected, 1-dimensional cell complex which admits a topological embedding into the plane. All of our graphs are unweighted. 

In the presence of accessibility, most questions relating to the structure of a particular class of graphs can be essentially reduced to understanding those graphs with exactly one end. In our case, we have the following statement relating to one-ended graphs. 

\begin{restatable}{alphtheorem}{oneend}\label{thm:one-ended-intro}
    Let $X$ be a connected, locally finite, quasi-transitive, one-ended graph, and suppose $X$ is quasi-isometric to a planar graph. Then $X$ is also quasi-isometric to some complete Riemannian plane. 
\end{restatable}

Note that while many one-ended planar graphs arise as the 1-skeleta of complete Riemannian planes, it is not true that every such graph is quasi-isometric to such a plane. The above result relies on the presence of an appropriate quasi-action in order to tame the structure of our planar graph. 

We note that Theorems~\ref{thm:acc-intro} and \ref{thm:one-ended-intro} allow us to deduce the following structural characterisation of connected, locally finite, quasi-transitive graphs which are quasi-isometric to planar graphs.

\begin{restatable}{alphcor}{graph}\label{thm:graph-intro}
    Let $X$ be a connected, locally finite, quasi-transitive graph. Then $X$ is quasi-isometric to a planar graph if and only if $X$ admits a canonical connected tree decomposition $(T, \cV)$ with bounded adhesion and $T/G$ compact, where each part is either finite or quasi-isometric to a complete Riemannian plane. 
\end{restatable}

Canonical tree decompositions are defined and discussed in \S\ref{sec:trees}, but should be interpreted as a graph-theoretical analogue to the toolbox of Bass--Serre theory. We remark that---with a little bit of extra work and a few heavy sledgehammer results---one can upgrade the complete Riemannian planes in Theorem~\ref{thm:one-ended-intro} and Corollary~\ref{thm:graph-intro} to be either the Euclidean or hyperbolic planes; see \cite{macmanus2024note} for details on this upgrade. 

\medskip

We now discuss the group-theoretical implications of Theorems~\ref{thm:acc-intro} and \ref{thm:one-ended-intro}.  
Recall that `the' Cayley graph of a finitely generated group is only well-defined up to quasi-isometry, and as such there is a rich literature on how coarse metric assumptions affect the structure of a group. For example, an important theorem of Gromov states that a Cayley graph of a finitely generated group  has polynomial growth if and only if said group is virtually nilpotent \cite{gromov1981groups}. Another key example, which will play a central role in this paper,  is a deep theorem originating in the work of Mess \cite{mess1988seifert} which provides a strong characterisation of those groups quasi-isometric to planes. The proof of this result spans several papers, with contributions from Casson--Jungreis \cite{casson1994convergence}, Gabai \cite{gabai1992convergence}, and Tukia \cite{tukia1988homeomorphic}. 
Summarising, we state the following. 

\begin{theorem*}[\cites{mess1988seifert, tukia1988homeomorphic, gabai1992convergence, casson1994convergence}]
    Let $G$ be a finitely generated group. Then $G$ is quasi-isometric to a complete Riemannian plane if and only if $G$ is a virtual surface group. 
\end{theorem*}

In the above, a \textit{virtual surface group} is a group containing a finite index subgroup isomorphic to the fundamental group of a closed orientable surface of positive genus.
Alternative proofs and extensions of this result have been given by Bowditch \cite{bowditch2004planar} and Maillot \cite{maillot2001quasi}. In particular, Maillot shows the following. 

\begin{theorem*}[\cite{maillot2001quasi}]
    Let $G$ be a finitely generated group. Then $G$ is quasi-isometric to a complete, simply connected, planar Riemannian surface with non-empty geodesic boundary if and only if $G$ is virtually free. 
\end{theorem*}

These results illustrate the philosophy that `planarity' appears to be an incredibly rigid property amongst finitely generated groups. That is, weak assumptions relating to planarity tend to imply rather strong properties. We compare the above theorems with our next application, which extends the above and further supports this philosophy. 

\begin{restatable}{alphcor}{group}\label{thm:tfae-intro}
    Let $G$ be a finitely generated group. Then the following are equivalent. 
    \begin{enumerate}
        \item $G$ is quasi-isometric to a planar graph, 

        \item $G$ is quasi-isometric to a planar \textbf{Cayley} graph, 

        \item Some finite index subgroup of $G$ admits a planar Cayley graph, 
    
        \item $G$ is virtually a free product of free and surface groups. 
    \end{enumerate}
\end{restatable}

It is clear that (4) $\implies$ (3) $\implies$ (2) $\implies$ (1). Our contribution is showing the non-trivial implication that (1) $\implies$ (4). 
The idea is to apply Theorem~\ref{thm:acc-intro}, which allows us to restrict our view to one-ended groups only. Combining this with Theorem~\ref{thm:one-ended-intro}, the characterisation of virtual surface groups described above readily applies. 

It is interesting to note that Corollary~\ref{thm:tfae-intro} together with a result of Papasoglu--Whyte \cite{papasoglu2002quasi} imply that there are precisely eight quasi-isometry classes of finitely generated groups quasi-isometric to planar graphs, since every surface group is quasi-isometric to either the Euclidean plane $\R^2$ or the hyperbolic plane $\HH^2$. In particular, every such group is quasi-isometric to one of:
$$
1, \ \ \Z, \ F_2, \ \ \Z^2, \ \  \Sigma, \ \ \Z^2 \ast \Z^2, \ \ \Sigma \ast \Sigma, \ \ \Z^2 \ast \Sigma,
$$
where $F_2$ is the free group of rank two, and $\Sigma$ denotes the fundamental group of the closed orientable surface of genus 2. Note you cannot necessarily upgrade this to commensurability. For example, if $\Sigma_g$ denotes the fundamental group of a closed surface of genus $g > 0$, then
$$
\Sigma_2 \ast \Sigma_2 \ \ \text{and} \ \ \Sigma_3 \ast \Sigma_3
$$
are not commensurable (a straightforward exercise in the Kurosh subgroup theorem) but are quasi-isometric. 

\subsection*{Discussion of the proofs}

Before diving into the technical details of this paper, it will be helpful to motivate what is to come. 

Let us begin with Theorem~\ref{thm:one-ended-intro}. 
%
If we want to show that a (2-connected, locally finite) one-ended planar graph $\Gamma$ is quasi-isometric to a complete Riemannian plane, then our natural instinct is glue 2-cells into the faces and extend the graph metric on the graph to a Riemannian metric on the resulting plane. The following two pathologies could arise, which will halt this plan in its tracks. 
\begin{enumerate}
    \item There could be `infinite face paths', so the resulting complex is not a plane. 

    \item The finite faces of $\Gamma$ could be arbitrarily big, which will stop the inclusion of our graph into the constructed Riemannian surface from being a quasi-isometry. 
\end{enumerate}
%
Our strategy of proof for Theorem~\ref{thm:one-ended-intro} is to prove that neither of the two pathologies described above can occur in a one-ended planar graph which is quasi-isometric to a quasi-transitive graph.  We show this by studying the induced quasi-action on the planar graph, and using this to obtain some control over the local features of the graph.

We now look towards the headline result, Theorem~\ref{thm:acc-intro}. The following discussion will be mainly centered around finitely generated groups, but everything is equally applicable to quasi-transitive graphs. 

It will be instructive to first consider the Maskit planarity theorem \cite{maskit1965theorem}. 

\begin{theorem}[Maskit planarity theorem]\label{thm:maskit}
    Let $S$ be a planar surface equipped with a proper cocompact action by a group $G$, such that $\Sigma = S/G$ is a finite-type orbifold. Then $\Sigma$ contains a finite collection $\mathcal C$ of essential  simple closed curves such that
    \begin{enumerate}
        \item Given $\gamma \in \mathcal C$, every component of the lift $\tilde \gamma$ of $\gamma$ in $S$ is a simple closed curve, and

        \item The normal subgroup of $\pi_1(S)$ generated by the connected components of the lifts of curves in $\mathcal C$ is equal to $\pi_1(S)$. 
    \end{enumerate}
\end{theorem}

See \cite{bowditch2022notes} for a good discussion of the above theorem. 
An immediate observation to be made is that if we glue a disk along each of the components of the lifts of curves in $\mathcal C$ then the resulting cell complex is simply connected. In particular, we deduce the following. 

\begin{corollary}\label{cor:maskit-cor}
    Let $G$ be a finitely generated group acting properly and cocompactly on a planar surface $S$. Then $G$ is finitely presented. 
\end{corollary}

Given a proper and cocompact action on a planar graph, one can often extend this to a suitable action on a planar surface using the following.

\begin{theorem}[{Whitney planar embedding theorem}]\label{thm:whitney}
    Let $\Gamma$ be a 3-connected, locally finite, planar graph. Then $\Gamma$ embeds uniquely into $\bbS^2$ up to post-composition of homeomorphisms of $\bbS^2$. 
\end{theorem}

This result is due to Whitney \cite{whitney19332} in the case of finite graphs, and was later extended to locally finite infinite graphs by Richter--Thomassen  \cite{richter20023}. It has the following consequence, which was noticed by Dunwoody in \cite{dunwoody2007planar}.

\begin{theorem}
    Let $G$ be a finitely generated group acting properly and cocompactly on a connected, locally finite, planar graph $\Gamma$. Then $G$ is finitely presented (and thus accessible). 
\end{theorem}

\begin{proof}[Sketch]
    First, assume that $\Gamma$ is 3-connected. 
    In this case, we can glue 2-cells along all finite faces of $\Gamma$, and attach copies of $\R \times [0,1]$ along infinite faces. Since the drawing of $\Gamma$ is unique by Theorem~\ref{thm:whitney}, this construction is equivariant, and the resulting complex is a planar surface with a proper cocompact $G$-action. By Corollary~\ref{cor:maskit-cor} we have that $G$ is finitely presented. 

    If $\Gamma$ is not 3-connected then a result of Droms--Servatius--Servatius \cite{droms1995structure} implies that $G$ splits as a graph of groups with finite edge groups, such that for every vertex group $G_v$ there exists a $G_v$-invariant 3-connected subgraph $\Gamma_v \subset \Gamma$, such that $\Gamma_v / G_v$ is compact. In particular, each $G_v$ is finitely presented by the previous case, and thus so is $G$. 
\end{proof}

We can now set out a rough plan of attack, modelled on the above sketch. Let $G$ be a finitely generated group which is quasi-isometric to a planar graph $\Gamma$.

\begin{enumerate}
    \myitem{(I)} Pass to a finite splitting of $G$ where each infinite-ended vertex group $G_v$ is quasi-isometric to some `highly connected' planar graph $\Gamma_v$. 

    \myitem{(II)} Abuse this high connectivity, and show that the quasi-action of $G_v$ upon $\Gamma_v$ `coarsely preserves' the faces of $\Gamma_v$. This should been viewed as a coarse-ification of Whitney's unique planar embedding theorem. 

    \myitem{(III)} Make use of this control of the quasi-action to learn something about the cycle space of our graph. Once we know a little about the cycle space of our graph, it's feasible that we might be able to say something about accessibility, e.g. via Hamann's theorem \cite{hamann2018accessibility}. 
\end{enumerate}

Life will not quite be as simple as described above, and we will need to come face to face with some delicate and subtle technicalities. These will be discussed in due course, but we suggest that the reader keep this rough plan in mind while working through the proof of Theorem~\ref{thm:acc-intro}.

\subsection*{Organisation of this paper}

We now walk through an outline of this paper. 

\begin{itemize}
    \item \S\ref{sec:prelims} introduces the basic tools and definitions required for the rest of the paper, and standardises our terminology and notation. 

    \item Next, \S\ref{sec:cuts} studies the geometry of cuts in graphs. This includes a short introduction to the Boolean ring of cuts, and also sets up the notation and basic results relating to canonical tree decompositions. 

    \item The following section, $\S\ref{sec:cuts-qi}$, continues to study cuts in graphs, but with a particular focus on proving some small technical results which describe how cuts and quasi-isometries interact. 

    \item After this, we proceed with proving Theorem~\ref{thm:one-ended-intro} in \S\ref{sec:one-end}. The style of argument employed here serves as a good warm-up for the main event later on, studying cobounded quasi-actions on one-ended planar graphs and how they interact with the faces.

    \item Up next, we begin our march towards Theorem~\ref{thm:acc-intro} by studying the cycle space of a graph in \S\ref{sec:cycles}. Particular focus is paid to the relationships that the cycle space enjoys with accessibility, planar graphs, and quasi-isometries. 

    \item \S\ref{sec:coboundaries} introduces a new tool, which we call the \emph{coboundary diameters} of a subgraph. We study how quasi-isometries affect these diameters. 

    \item The two sections after this, \S\ref{sec:friendly-faces} and \S\ref{sec:quasi-act-planar} are together the most technical part of this paper. Here we provide a detailed study of how quasi-actions affect planar graphs and their faces. The main result of these sections together is that if we begin `cutting-up' our planar graph along small cuts then eventually the dynamics of this quasi-action become quite controlled and predictable.

    \item Finally, in \S\ref{sec:accessibility} we piece all of the above together and prove Theorem~\ref{thm:acc-intro}. Following this, we deduce Corollaries~\ref{thm:graph-intro} and \ref{thm:tfae-intro}.

\end{itemize}

\subsection*{Acknowledgements}
I am grateful to Panos Papasoglu for bringing this problem to my attention and for his continuous support during this project. I also thank Martin Dunwoody, Agelos Georgakopoulos, and Bruce Richter for helpful correspondence, as well as Davide Spriano and Dawid Kielak for feedback.  



\section{Preliminaries}\label{sec:prelims}

We begin by setting up some standard notation and terminology, which will follow us throughout.

\subsection{Graphs}

By a \textit{graph}, we mean a 1-dimensional CW-complex.  
We call a $0$-cell a \textit{vertex} and a 1-cell an \textit{edge}. If $X$ is a graph, we denote by $V(X)$ the set of vertices of $X$, and $E(X)$ its set of undirected, closed edges. The vertices which $e \in E(X)$ abuts are called the \textit{endpoints of $e$}. If the endpoints of $e$ are $u, v \in V(X)$, then we may write $e = uv = vu$. 
It will also be convenient to work with \emph{oriented edges}, which is a choice of ordering of the endpoints of a given edge. Write $\vec E(X)$ for the set of oriented edges. We may notate such an edge $e \in \vec E(X)$ as an ordered pair $e = (u,v)$, where $u$ is the \emph{initial vertex} and $v$ is the \emph{terminal vertex}.
Given $e \in \vec E(X)$, we denote by $e^{-1}$ the opposite orientation of $e$.  
Note that there is a natural two-to-one map $\vec E(X) \to E(X)$ obtained by `forgetting' orientations. That is, the map $(u,v) \mapsto uv$. 

Given $v, u \in V(X)$, a \textit{combinatorial path}, or just \textit{path} between $v$ and $u$ is a sequence $e_1, \ldots , e_n$ of $e_i \in \vec E (X)$ such that $v$ is the initial vertex of $e_1$, $u$ is the terminal vertex of $e_n$, and for every $1 \leq i < n$ we have that the terminal vertex of $e_i$ is equal to the initial vertex of $e_{i+1}$. The \textit{length} of this path is $n$. We say that a path is \textit{simple} if any given vertex is visited at most once. If $p$ is a path from $u$ to $v$, and $q$ is a path from $v$ to $w$, then the concatenation $pq$ is a path from $u$ to $w$. Denote by $p^{-1}$ the reversal of $p$. A \textit{(closed) loop} is a path as above starting and ending at the same vertex. We say that a loop is \textit{simple} if every vertex is visited at most once, except the common initial/terminal vertex, which is visited exactly twice. We may parametrise loops with continuous maps $S^1 \to X$. 
The graph $X$ is said to be \textit{connected} if there exists a path connecting any two points.
If $X$ is connected, we equip $V(X)$ with a metric by defining $\dist_X(u, v)$ as the length of shortest path connecting $u, v \in V(X)$. We may extend this metric to all of $X$ by identifying each $e \in E(X)$ with a copy of the unit interval $[0,1]$. Note that throughout this paper, a connected subgraph will almost always be considered together with its own intrinsic metric, rather than the metric induced by the ambient graph. Given $r \geq 0$ and $S \subset \Gamma$, let $B_\Gamma(S;r)$ denote the closed $r$-neighbourhood of $S$ in $\Gamma$. 

The \textit{degree} or \textit{valence} of a vertex $v \in V(X)$ is the number of $e \in E X$ which abut $v$. We call $X$ \textit{locally finite} if every vertex has finite degree. We call $X$ \textit{bounded-degree } if there exists $N \geq 0$ such that every vertex has degree at most $N$. 

Given a subset $F \subset E(X)$, we denote by $X\setminus F$ the subgraph of $X$ obtained by removing the edges in $F$, but no vertices. Given a subgraph $Y \subset X$, we denote by $X\setminus Y$ the subgraph of $X$ obtained by removing all vertices of $X$ contained in $Y$, as well as all incident edges. Given a subset $U \subset V(X)$, we denote by $X[U]$ the \textit{subgraph induced by $U$}. That is, the subgraph of $X$ with vertex set precisely $U$, where we include $e \in E(X)$ if and only if both endpoints of $e$ lie in $U$. An \textit{induced subgraph} is a subgraph which is equal to the subgraph induced by its vertex set. 

If a group $G$ acts on a set with finitely many orbits, we say the action is \emph{$G$-finite}. 
When we say a group $G$ acts on a connected graph $X$, we mean by isometries with respect to the metric $\dist_X$ defined above. We may refer to a graph equipped with a $G$-action as a $G$-graph. Such an action is called \textit{quasi-transitive} if the action on $V(X)$ is $G$-finite, in which case we may call our graph a \emph{quasi-transitive $G$-graph}. Note that if $X$ is also locally finite, then there are only finitely many orbits of edges too, and $X$ is necessarily bounded degree.

We also introduce the following terminology, which is non-standard but very convenient. 

\begin{definition}[Quasi-transitively stabilised subgraph]\label{def:cocompact}
    Let $\Gamma$ be a connected $G$-graph, and $\Lambda \subset \Gamma$ a connected subgraph. We say that $\Lambda$ is \textit{quasi-transitively stabilised} if the set-wise stabiliser $\Stab(\Lambda) \leq G$ acts quasi-transitively on $\Lambda$. 
\end{definition}

\subsection{Ends and accessibility}

In what follows, let $X$ be a connected, locally finite graph.
A \textit{ray} $r$ in $X$ is an infinite sequence of directed edges $e_1, e_2, \ldots $ such that for any $r \geq 1$ the subsequence $e_1, \ldots, e_r$ is a path. A ray is \textit{simple} if every finite sub-path is simple. An infinite subpath of a simple ray $r$ is called a \textit{tail}. We call the initial vertex of $r$ the \textit{basepoint}. We say that two simple rays $r_1$, $r_2$ are \textit{end-equivalent} if for all compact subgraphs $K \subset X$, we have that there are tails of $r_1$ and $r_2$ which lie in the same connected component of $X \setminus K$. This is clearly an equivalence relation on the set of simple rays. An equivalence class is called an \textit{end} of $X$. The set of ends of $X$ is denoted by $\Omega (X)$. 
A \textit{bi-infinite ray $r$} is the union of two rays $r_1$, $r_2$ which share a common basepoint. If $r_i$ approaches $\omega_i \in \Omega (X)$, we may simply say that $r$ is a path between $\omega_1$ and $\omega_2$. 

It is helpful to sometimes consider a graph together with its ends as a single topological object. For this purpose, we have the following definition. 

\begin{definition}[Freudenthal compactification]
    Let $\Gamma$ be a connected, locally finite graph. Then, the \textit{Freudenthal compactification} of $\Gamma$ is defined to be the compact topological space $\overline \Gamma$ with underlying set $\Gamma \sqcup \Omega (\Gamma)$, whose topology is given by the open sets
$$
V_K(x) = \{y \in \Gamma \sqcup \Omega(\Gamma) : \textrm{there exists a path from $x$ to $y$ in $X \setminus K$}\},
$$
where $K$ is any compact subgraph of $\Gamma$. 
\end{definition}

It is easy to check that $\overline \Gamma$ is compact, Hausdorff,  and locally path connected \cite[\S 8.6]{diestelgraphs} and that the natural inclusion $\Gamma \into \overline \Gamma$ is a topological embedding.
We quickly record the following important property of the Freudenthal compactification, which sets it apart from other compactifications. 

\begin{proposition}[{\cite[Thm.~1.5(f)]{raymond1960end}}]\label{prop:freud-prop}
    Let $X$ be a connected, locally finite graph. Let $U \subset \overline X$ be a connected open set. Then $U \setminus \Omega (X)$ is also connected. 
\end{proposition}

Given $\omega_1, \omega_2 \in \Omega (X)$ and a finite subset $F \subset E(X)$, we say that $F$ \textit{separates} $\omega_1$ and $\omega_2$ if any bi-infinite path between these ends must cross some $e \in F$. If $K \subset \Gamma$ is a compact subgraph and $U$ is a connected component of $\Gamma \setminus K$, we say that an end $\omega \in \Omega (\Gamma)$ \textit{lies in $U$} if every simple ray in $\omega$ has infinite intersection with $U$. 
We now have the following key definition. 

\begin{definition}[Accessibility for graphs]
    Let $X$ be a connected, locally finite graph. We say that $X$ is \textit{accessible}\footnote{Really, this defines what it means for $X$ to be \emph{edge-accessible}. One can define vertex-accessibility similarly. For bounded-degree graphs, it is clear that these notions are equivalent.} if there exists $k \geq 1$ such that for any pair of distinct ends $\omega_1, \omega_2 \in \Omega (X)$, we have that $\omega_1$, $\omega_2$ can be separated by the removal of at most $k$ edges. 
\end{definition}

Accessibility was originally a purely group theoretical idea, first considered by Wall \cite{wall1971pairs}. A finitely generated group is said to be \textit{accessible} if it is isomorphic to the fundamental group of a finite graph of groups with finite edge groups, where each vertex group has at most one end. It was shown by Thomassen--Woess that a finitely generated group is accessible if and only if its Cayley graphs are accessible \cite[Thm.~1.1]{thomassen1993vertex}.

\subsection{Menger's theorem}

We next state a classical result which will play a key role throughout this paper, known as Menger's theorem. 
We first need to introduce some notation. 

\begin{definition}
    Let $\Gamma$ be a connected graph, and let $x, y \in V(\Gamma) \sqcup \Omega (\Gamma)$ be distinct.  Define the \textit{vertex separation} and \emph{edge separation} of $x$ and $y$, denoted $\vs(x,y)$ and $\es(x,y)$ respectively, as
    \begin{align*}
        \vs(x,y) &= \inf\{|S| : S \subset V(\Gamma), \ \text{$x$, $y$ lie in distinct components of $\Gamma \setminus S$}\}, \\
        \es(x,y) &= \inf\{|F| : F \subset E(\Gamma), \ \text{$x$, $y$ lie in distinct components of $\Gamma \setminus F$}\}. 
    \end{align*}
    We now define the \textit{minimal vertex/edge end-cut size of $\Gamma$} as 
    \begin{align*}
        \vs(\Gamma) &= \min\{\vs(\omega_1, \omega_2) : \omega_1, \omega_2 \in \Omega (\Gamma), \ \omega_1 \neq \omega_2\},\\
        \es(\Gamma) &= \min\{\es(\omega_1, \omega_2) : \omega_1, \omega_2 \in \Omega (\Gamma), \ \omega_1 \neq \omega_2\}.
    \end{align*}
    If $\Gamma$ has at most one end, we adopt the convention that $\es(\Gamma) = \vs(\Gamma) = \infty$. 
\end{definition}

The following is immediate, but useful to note. 

\begin{proposition}\label{prop:vs-es-coincide}
    Let $\Gamma$ be a connected, bounded-degree graph where every vertex has degree at most $d \geq 0$. Then
    $
    \vs(\Gamma) \leq \es(\Gamma) \leq d \cdot \vs(\Gamma)
    $. 
\end{proposition}

We may now state Menger's theorem as follows:

\begin{theorem}[Menger]\label{thm:menger}
    Let $\Gamma$ be a connected, locally finite graph, and let $x,y \in V(\Gamma) \sqcup \Omega (\Gamma)$. Fix $N \geq 1$. Then $\es(x,y) 
    \geq N$ if and only if there exists $N$ pairwise edge-disjoint paths connecting $x$ to $y$. 
    
    Similarly, we also have that $\vs(x,y) \geq N$ if and only if there exists $N$ pairwise internally vertex-disjoint paths connecting $x$ to $y$. 
\end{theorem}

This result was first shown to hold in finite graphs by Menger \cite{menger1927allgemeinen}. It was later extended to locally finite, infinite graphs by Halin \cite{halin1974note}. 

\begin{remark}
    One cannot drop the locally finite hypothesis as counterexamples exist. Consider a bi-infinite path $L$ plus one extra vertex $v$, and add an edge between $v$ and every vertex in $L$. The resulting graph $\Gamma$ is two-ended and $\vs(\Gamma) = 2$, but $\Gamma$ does not contain two disjoint bi-infinite paths. 
\end{remark}

\begin{corollary}\label{cor:pw-disjoint-rays}
    Let $\Gamma$ be a connected, locally finite graph, and suppose that $\vs(\Gamma) \geq N$ for some $N \geq 1$. Then between any two ends $\omega_1, \omega_2 \in \Omega (\Gamma)$ there exists a collection of $N$ pairwise disjoint bi-infinite rays connecting $\omega_1$ to $\omega_2$. 
\end{corollary}

\subsection{Coarse notions}\label{sec:qi}

We briefly introduce some terminology relating to coarse geometry, which will follow us throughout this paper.

Let $X$ be a metric space. Let $x \in X$ and $A \subset X$. We define the distance between $x$ and $A$ as 
$$
\dist_X(x,A) := \inf\{\dist_X(x,a) : a \in A\}. 
$$
If $B \subset X$ is another subset, the \textit{Hausdorff distance between $A$ and $B$} is defined as
$$
\dHaus[X](A,B) := \max\left\{\sup_{a \in A}\dist_X(a,B) \ , \ \sup_{b \in B} \dist_X(b,A) \right\}. 
$$
Note that $\dHaus[X](A,B)$ is finite if and only if $A$ is contained in a finite neighbourhood of $B$ and vice versa.

Next, we introduce one of the central objects of this paper.

\begin{definition}[Quasi-isometry]
    Let $X$, $Y$ be metric spaces, $\lambda \geq 1$, $\varepsilon \geq 0$. Then a map $\psi : X \to Y$ is a $(\lambda, \varepsilon)$-\textit{quasi-isometric embedding} if 
    $$
    \tfrac{1}{\lambda} \dist_X(x,y) - \varepsilon \leq \dist_Y(\psi(x), \psi(y)) \leq \lambda \dist_X(x,y) + \varepsilon,
    $$
    for all $x, y \in X$. We call $\psi$ a \textit{$(\lambda, \varepsilon)$-quasi-isometry} if, in addition to the above, we have that
    $$
    \dHaus[Y](Y, \psi(X)) \leq \varepsilon. 
    $$
    A map satisfying this second condition is said to be \textit{coarsely surjective}. If there exists a quasi-isometry $\psi : X \to Y$, we say that $X$ and $Y$ are \textit{quasi-isometric}. 
    
    Finally, let $\eta \geq 0$. If $\varphi : Y \to X$ is a quasi-isometry such that $\dist_X(\varphi \circ \psi(x), x) \leq \eta$ for all $x \in X$, then we call $\varphi$ a \textit{$\eta$-quasi-inverse to $\psi$}. Note that every quasi-isometry has a quasi-inverse. 
\end{definition}

The next two results demonstrate that, among over things, we may assume without loss of generality that our quasi-isometries are continuous.

\begin{proposition}\label{prop:qi-wlog}
    Let $\Gamma$, $\Pi$ be connected graphs, and $\psi : \Gamma \to \Pi$ a quasi-isometric embedding. Then there exists a subgraph $\Lambda \subset \Pi$ and a surjective, continuous quasi-isometry $\varphi : \Gamma \onto \Lambda$. The inclusion $\Lambda \into \Pi$ is also a quasi-isometric embedding. Moreover, if $\Gamma$ is bounded-degree  then so is $\Lambda$. 
\end{proposition}

\begin{proof}
    We construct $\varphi $ as follows. For every $v \in V(\Gamma)$, define $\varphi (v) = \psi(v)$. For every edge $e \in E(\Gamma)$, with endpoints $a$, $b$, set $\varphi (e)$ to a geodesic path between $\psi(a)$ and $\psi(b)$. Let $\Lambda = \varphi (\Gamma)$. We have that $\varphi $ is thus a continuous, surjective map $ \varphi  : \Gamma \onto \Lambda$. We claim that this is a quasi-isometry. 
    Indeed, note that for all $x, y \in V(\Gamma)$ we have by construction that
    $$
    \dist_{\Lambda}(\varphi (x), \varphi(y)) = \dist_{\Pi}(\psi (x), \psi(y)).
    $$
    It follows that $\varphi  : \Gamma \onto \Lambda$ is a quasi-isometric embedding since $\psi$ is, and thus $\varphi $ is a quasi-isometry. It is also clear from this construction that the inclusion $\Lambda \into \Pi$ is also a quasi-isometric embedding. 

    Suppose now that $\Gamma$ is bounded-degree. We claim that $\Lambda$ is also bounded-degree. Let $d \geq 0$ be such that every vertex in $\Gamma$ has degree at most $d$. Fix $\lambda \geq 1$ such that $\varphi $ is a $(\lambda, \lambda)$-quasi-isometry.
    Fix $v \in V\Lambda$, and let $u \in \varphi ^-1(v)$. Let $v_1, v_2 \ldots$ be the neighbours of $v$ in $\Lambda$, and for each $v_i$ choose some point $u_i \in \varphi^{-1}(v_i)$. Note that each $u_i$ lies inside a bounded neighbourhood $N$ of $u$. Since $\Gamma$ is bounded-degree, $N$ intersects a bounded number of edges, say $M$. 
    Recall that $\varphi $ maps edges to geodesics between the images of their endpoints. This implies that if $e$ is a (closed) edge in $\Gamma$, then $\varphi(e)$ intersects at most $2\lambda+1$ distinct vertices in $\Lambda$. Thus, the image $N$ in $\Lambda$ can contain at most $k = M(2\lambda+1)$ distinct vertices. Thus, $v$ has at most $k$ neighbours in $\Lambda$. Since $v$ was arbitrary and all our bounds were uniform, it follows that $\Lambda$ is bounded-degree. 
\end{proof}

A similar construction to that given in Proposition~\ref{prop:qi-wlog} also gives the following. 

\begin{proposition}\label{prop:cts-inverse}
     Let $\Gamma$, $\Pi$ be connected graphs and $\psi : \Gamma \to \Pi$ a quasi-isometry. Then there exists a continuous quasi-inverse $\varphi : \Pi \to \Gamma$. 
\end{proposition}

Continuity is important for us, for a few reasons. For one, we will be, at times, quite interested in applying the Jordan curve theorem, and so having continuous curves will be imperative. Continuity is also important as it ensures that suitable restrictions of quasi-isometries are also quasi-isometries. 

\begin{proposition}\label{prop:restriction-cts}
    Let $\Gamma$, $\Pi$ be bounded-degree, connected graphs. Let $\psi : \Pi \to \Gamma$ be a continuous quasi-isometry. Let $\Lambda \subset \Pi$ be a connected subgraph such that the inclusion map $\iota : \Lambda \into \Pi$ is a quasi-isometric embedding. Then the restriction $\psi|_\Lambda : \Lambda \onto \psi(\Lambda)$ is a quasi-isometry. 
\end{proposition}

\begin{proof}
    Let $x, y \in \Lambda$, and let $p$ be a geodesic in $\Lambda$ connecting $x$ to $y$ of length $n$. Then $\psi(p)$ contains at most at $cn$ vertices, where $c > 0$ is some constant depending only on the quasi-isometry constants of $\psi$ and the maximum degree of $\Pi$. As $\psi$ is continuous, we deduce that $\psi(p)$ contains a path of between $\psi(x)$ and $\psi(y)$ of bounded length. This gives the upper bound.
    For the lower bound, simply note that 
    $
    \dist_{\psi(\Lambda)}(\psi(x), \psi(y)) \geq \dist_{\Gamma}(\psi(x), \psi(y))$, and apply the fact that $\iota$ and $\psi$ are quasi-isometric embeddings. 
\end{proof}

\begin{remark}
    In the statement of Proposition~\ref{prop:restriction-cts} it is important to note that we consider $\psi(\Lambda)$ with its own intrinsic path metric, and not the metric induced by the ambient graph $\Pi$. This is why we require continuity of $\psi$, lest $\psi(\Lambda)$ may not even be connected. 
\end{remark}

Finally, we introduce the notion of a quasi-action. 

\begin{definition}[Quasi-actions]\label{def:quasi-action}
    Let $\Gamma$ be a connected graph, and let $\QI(\Gamma)$ denote the set of quasi-isometries $\Gamma \to \Gamma$. Let $G$ be a group, then a \textit{$\lambda$-quasi-action} of $G$ on $\Gamma$ is a map
    $$
    G \to \QI(\Gamma), \ \ g \mapsto \varphi_g,
    $$
    such that the following hold:
    \begin{enumerate}
        \item Each $\varphi_g$ is a $(\lambda, \lambda)$-quasi-isometry, 
        \item For every $g, h \in G$, $x \in \Gamma$, we have that 
        $$
        \dist_\Gamma(\varphi_{gh}(x), \varphi_{g}\circ \varphi_h (x))  \leq \lambda, 
        $$
        \item For every $g \in G$, $\varphi_g$ and $\varphi_{g^{-1}}$ are $\lambda$-quasi-inverses. 
    \end{enumerate}
We say that a quasi-action as above is $B$-cobounded, where $B \geq 0$ is some constant, if for all $x,y \in \Gamma$ there exists some $g \in G$ such that 
    $
    \dist_\Gamma(x, \varphi_g(y)) \leq B
    $. 
\end{definition}

Let $X$ be a connected graph equipped with an action $G \actson X$, and let $\Gamma$ be another connected graph. If $\varphi : X \to \Gamma$ is a quasi-isometry with quasi-inverse $\psi : \Gamma \to X$, then this induces a quasi-action of $G$ on $\Gamma$ via
$
\varphi_g := \varphi \circ g \circ \psi
$. 
It is easy to check that if the original action is cobounded, then the resulting quasi-action is cobounded. Also, if both $\varphi$ and $\psi$ are continuous then clearly so is every $\varphi_g$.

\subsection{Planar graphs}

Recall that a \textit{topological embedding} between topological spaces $X$, $Y$ is a continuous injection $f : X \into Y$ which restricts to a homeomorphism between $X$ and $f(X)$ with the subspace topology. We adopt the following convention. 

\begin{definition}[Planar graph]
    Let $\Gamma$ be a graph. We say that $\Gamma$ is \textit{planar} if there exists a topological embedding $\Gamma \into \bbS^2$. Such an embedding is called a \textit{drawing}. 
\end{definition}

The above leaves room for strange embeddings. 
The closure of the image of a planar graph in $\bbS^2$ can be viewed as an embedding of a particular {compactification} of $\Gamma$. In this paper, we will generally restrict ourselves to the case where this closure is precisely the Freudenthal compactification of $\Gamma$. 
This can generally be ensured thanks to the following,  due to Richter--Thomassen \cite{richter20023}. 

\begin{proposition}[{\cite[Lem.~12]{richter20023}}]\label{prop:2c-well-embed}
    Let $\Gamma$ be a 2-connected, locally finite, planar graph. Then the Freudenthal compactification $\overline \Gamma$ embeds into $\bbS^2$.
\end{proposition}

This result can be pushed further, and in particular the 2-connected assumption may be dropped. The following argument was suggested to me by Bruce Richter. 

\begin{proposition}\label{prop:well-embed}
    Let $\Gamma$ be a connected, locally finite, planar graph. Then the Freudenthal compactification $\overline \Gamma$ embeds into $\bbS^2$.
\end{proposition}

\begin{proof} 
    Fix an embedding $\vartheta : \Gamma \into \bbS^2$. 
    We will find a 2-connected, locally finite planar $\Pi$ such that (some subdivision of) $\Gamma$ is a subgraph of $\Pi$ and this inclusion is a quasi-isometry. From there, the embedding of $\overline \Pi$ into $\bbS^2$ given by Proposition~\ref{prop:2c-well-embed} will induce an embedding of $\overline \Gamma \into \bbS^2$.
    Intuitively, our plan is to add a small cycle surrounding each cut vertex. 
    In order to make this sketch formal, we will construct $\Pi$ as an ascending union of finite planar graphs, which is necessarily planar. Indeed, if an ascending union of finite planar graphs contained some subdivision of $K_5$ or $K_{3,3}$ then certainly some finite graph in this chain would have to, which would contradict Kuratowski's theorem. 

    We first assume that $\Gamma$ contains no cut edges. That is, for every $e \in E(\Gamma)$ we have that $\Gamma \setminus e$ is connected. 
    Choose some root $v_0 \in V(\Gamma)$. 
    We call $u \in V(\Gamma)$ a \textit{marked vertex} if it is a cut vertex of $\Gamma$. Given $r \geq 0$ we construct $\Pi_r$ as follows. Consider the closed $r$-neighbourhood of $v_0$ in $\Gamma$, which we will denote $N_r = B_\Gamma(v_0;r)$. This is a finite subgraph of $\Gamma$. Let $N_r'$ denote the subdivision of $N_r$, where we subdivide every edge into three edges. Suppose $u \in V(N_r)$ is a marked vertex of $\Gamma$, and $\dist_\Gamma(v_0, u) < r$. Let $u'$ be the image of $u$ in $N_{r}'$. Let $u_1, \ldots, u_n$ be the neighbours of $u'$ in $V(N_r')$. Since $\Gamma$ is locally finite, $n$ is finite. Let $e_i$ be the (necessarily unique) edge of $N_r'$ with endpoints $u'$, $u_i$. The drawing $\vartheta$ restricts to a drawing of $N_r$, and thus of $N_{r}'$ since these graphs are homeomorphic. This induces a cyclic ordering to $e_1, \ldots , e_n$, particularly their clockwise ordering about $u'$. We assume this is exactly how they are ordered. We add a path of length 2 joining every $u_i$ to $u_{i+1}$, with indices taken modulo $n$. We repeat this for every such $u'$, and call the resulting graph $\Pi_r$. By construction $\Pi_r$ is a finite planar graph. 
    Repeating this construction for every $r$, it is clear that we have an ascending chain 
    $$
    \Pi_0 \subset \Pi_1 \subset \Pi_2 \subset \cdots. 
    $$
    Let $\Pi = \bigcup_r \Pi_r$. 
    Let $\Gamma'$ denote the subdivision of $\Gamma$, where each edge is divided into three edges. Clearly $\Gamma$ is homeomorphic to $\Gamma'$, and the natural map $\Gamma \to \Gamma'$ is a $(3,0)$-quasi-isometry. We have a natural inclusion $\Gamma' \into \Pi$ which is an isometry onto its image. Indeed, the new paths added to $\Pi$ create no new shortcuts. Every vertex of $\Pi$ is adjacent to a vertex of $\Gamma'$, so this map is a quasi-isometry. 
    Finally, note that $\Pi$ is certainly 2-connected. Indeed, suppose $\Pi$ contained a cut vertex $w$. This vertex certainly must lie in $\Gamma'$, as every $u \in V(\Pi) \setminus V(\Gamma')$ has degree 2 and lies on a simple cycle so cannot be a cut vertex. Since $\Gamma$ contains no cut edges this implies that $w$ is induced by a cut vertex of $\Gamma$. But then by construction if $u_1$, $u_2$ are neighbours of $w$ in $\Pi$ then there is a path connecting them which avoids $w$. Thus $\Pi$ is 2-connected. Applying Proposition~\ref{prop:2c-well-embed} we conclude that the Freudenthal compactification of any connected, locally finite, planar graph \textit{without cut edges} embeds into $\bbS^2$. 

    Finally, we deal with the case where $\Gamma$ has cut edges. We replace $\Gamma$ with $\Gamma'$ where $\Gamma'$ is obtained from $\Gamma$ by `doubling' every edge. That is, $V(\Gamma') = V(\Gamma)$, and if $u, v \in V(\Gamma)$  are connected by $k$ edges in $\Gamma$ then they are connected by $2k$ edges in $\Gamma$. 
    The inclusion $\Gamma \into \Gamma'$ is certainly a quasi-isometry, and $\Gamma'$ is clearly a locally finite. To see that $\Gamma'$ is planar, note that doubling edges clearly preserves planarity in finite graphs. We now repeat a similar construction to above, and write $\Gamma'$ as an ascending union of finite planar graphs, where the $r$-th element of this chain is obtained by doubling the edges in the $r$-ball about some root vertex $v_0$ in $\Gamma$. By the previous case, the Freudenthal compactification of $\Gamma'$ is planar, and so the same can be said about $\Gamma$ since $\overline \Gamma$ is clearly homeomorphic to a subspace of $\overline {\Gamma'}$. A cartoon of this full construction can be found in Figure~\ref{fig:2-con-sup-graph}. 
\end{proof}

\begin{figure}
    \centering
    
\tikzset{every picture/.style={line width=0.75pt}} 

\begin{tikzpicture}[x=0.75pt,y=0.75pt,yscale=-1,xscale=1]

\draw [color={rgb, 255:red, 165; green, 188; blue, 216 }  ,draw opacity=1 ]   (424.75,117.54) -- (403.75,113.85) -- (391.9,135.8) -- (422.58,151.37) -- (434.02,141.74) -- (433.36,129.35) -- cycle ;
\draw [shift={(414.25,115.69)}, rotate = 189.95] [color={rgb, 255:red, 165; green, 188; blue, 216 }  ,draw opacity=1 ][fill={rgb, 255:red, 165; green, 188; blue, 216 }  ,fill opacity=1 ][line width=0.75]      (0, 0) circle [x radius= 1.34, y radius= 1.34]   ;
\draw [shift={(397.83,124.83)}, rotate = 118.36] [color={rgb, 255:red, 165; green, 188; blue, 216 }  ,draw opacity=1 ][fill={rgb, 255:red, 165; green, 188; blue, 216 }  ,fill opacity=1 ][line width=0.75]      (0, 0) circle [x radius= 1.34, y radius= 1.34]   ;
\draw [shift={(407.24,143.59)}, rotate = 26.9] [color={rgb, 255:red, 165; green, 188; blue, 216 }  ,draw opacity=1 ][fill={rgb, 255:red, 165; green, 188; blue, 216 }  ,fill opacity=1 ][line width=0.75]      (0, 0) circle [x radius= 1.34, y radius= 1.34]   ;
\draw [shift={(428.3,146.56)}, rotate = 319.91] [color={rgb, 255:red, 165; green, 188; blue, 216 }  ,draw opacity=1 ][fill={rgb, 255:red, 165; green, 188; blue, 216 }  ,fill opacity=1 ][line width=0.75]      (0, 0) circle [x radius= 1.34, y radius= 1.34]   ;
\draw [shift={(433.69,135.55)}, rotate = 266.97] [color={rgb, 255:red, 165; green, 188; blue, 216 }  ,draw opacity=1 ][fill={rgb, 255:red, 165; green, 188; blue, 216 }  ,fill opacity=1 ][line width=0.75]      (0, 0) circle [x radius= 1.34, y radius= 1.34]   ;
\draw [shift={(429.06,123.44)}, rotate = 233.91] [color={rgb, 255:red, 165; green, 188; blue, 216 }  ,draw opacity=1 ][fill={rgb, 255:red, 165; green, 188; blue, 216 }  ,fill opacity=1 ][line width=0.75]      (0, 0) circle [x radius= 1.34, y radius= 1.34]   ;
\draw [color={rgb, 255:red, 165; green, 188; blue, 216 }  ,draw opacity=1 ]   (326.94,72.04) -- (317.69,83.64) -- (314.12,101.8) -- (321.86,115.33) -- (332.56,128.38) -- (366.78,121.8) -- (378.63,99.84) -- (378.16,80.48) -- (364.9,71.53) -- (341.44,70.35) -- cycle ;
\draw [shift={(322.31,77.84)}, rotate = 128.57] [color={rgb, 255:red, 165; green, 188; blue, 216 }  ,draw opacity=1 ][fill={rgb, 255:red, 165; green, 188; blue, 216 }  ,fill opacity=1 ][line width=0.75]      (0, 0) circle [x radius= 1.34, y radius= 1.34]   ;
\draw [shift={(315.9,92.72)}, rotate = 101.12] [color={rgb, 255:red, 165; green, 188; blue, 216 }  ,draw opacity=1 ][fill={rgb, 255:red, 165; green, 188; blue, 216 }  ,fill opacity=1 ][line width=0.75]      (0, 0) circle [x radius= 1.34, y radius= 1.34]   ;
\draw [shift={(317.99,108.56)}, rotate = 60.23] [color={rgb, 255:red, 165; green, 188; blue, 216 }  ,draw opacity=1 ][fill={rgb, 255:red, 165; green, 188; blue, 216 }  ,fill opacity=1 ][line width=0.75]      (0, 0) circle [x radius= 1.34, y radius= 1.34]   ;
\draw [shift={(327.21,121.85)}, rotate = 50.65] [color={rgb, 255:red, 165; green, 188; blue, 216 }  ,draw opacity=1 ][fill={rgb, 255:red, 165; green, 188; blue, 216 }  ,fill opacity=1 ][line width=0.75]      (0, 0) circle [x radius= 1.34, y radius= 1.34]   ;
\draw [shift={(349.67,125.09)}, rotate = 349.11] [color={rgb, 255:red, 165; green, 188; blue, 216 }  ,draw opacity=1 ][fill={rgb, 255:red, 165; green, 188; blue, 216 }  ,fill opacity=1 ][line width=0.75]      (0, 0) circle [x radius= 1.34, y radius= 1.34]   ;
\draw [shift={(372.7,110.82)}, rotate = 298.36] [color={rgb, 255:red, 165; green, 188; blue, 216 }  ,draw opacity=1 ][fill={rgb, 255:red, 165; green, 188; blue, 216 }  ,fill opacity=1 ][line width=0.75]      (0, 0) circle [x radius= 1.34, y radius= 1.34]   ;
\draw [shift={(378.39,90.16)}, rotate = 268.61] [color={rgb, 255:red, 165; green, 188; blue, 216 }  ,draw opacity=1 ][fill={rgb, 255:red, 165; green, 188; blue, 216 }  ,fill opacity=1 ][line width=0.75]      (0, 0) circle [x radius= 1.34, y radius= 1.34]   ;
\draw [shift={(371.53,76)}, rotate = 214.01] [color={rgb, 255:red, 165; green, 188; blue, 216 }  ,draw opacity=1 ][fill={rgb, 255:red, 165; green, 188; blue, 216 }  ,fill opacity=1 ][line width=0.75]      (0, 0) circle [x radius= 1.34, y radius= 1.34]   ;
\draw [shift={(353.17,70.94)}, rotate = 182.88] [color={rgb, 255:red, 165; green, 188; blue, 216 }  ,draw opacity=1 ][fill={rgb, 255:red, 165; green, 188; blue, 216 }  ,fill opacity=1 ][line width=0.75]      (0, 0) circle [x radius= 1.34, y radius= 1.34]   ;
\draw [shift={(334.19,71.19)}, rotate = 173.36] [color={rgb, 255:red, 165; green, 188; blue, 216 }  ,draw opacity=1 ][fill={rgb, 255:red, 165; green, 188; blue, 216 }  ,fill opacity=1 ][line width=0.75]      (0, 0) circle [x radius= 1.34, y radius= 1.34]   ;
\draw    (43,88.2) -- (55.8,129.8) ;
\draw [shift={(55.8,129.8)}, rotate = 72.9] [color={rgb, 255:red, 0; green, 0; blue, 0 }  ][fill={rgb, 255:red, 0; green, 0; blue, 0 }  ][line width=0.75]      (0, 0) circle [x radius= 2.34, y radius= 2.34]   ;
\draw [shift={(43,88.2)}, rotate = 72.9] [color={rgb, 255:red, 0; green, 0; blue, 0 }  ][fill={rgb, 255:red, 0; green, 0; blue, 0 }  ][line width=0.75]      (0, 0) circle [x radius= 2.34, y radius= 2.34]   ;
\draw    (43,88.2) -- (105.8,96.6) ;
\draw [shift={(105.8,96.6)}, rotate = 7.62] [color={rgb, 255:red, 0; green, 0; blue, 0 }  ][fill={rgb, 255:red, 0; green, 0; blue, 0 }  ][line width=0.75]      (0, 0) circle [x radius= 2.34, y radius= 2.34]   ;
\draw [shift={(43,88.2)}, rotate = 7.62] [color={rgb, 255:red, 0; green, 0; blue, 0 }  ][fill={rgb, 255:red, 0; green, 0; blue, 0 }  ][line width=0.75]      (0, 0) circle [x radius= 2.34, y radius= 2.34]   ;
\draw    (55.8,129.8) -- (105.8,96.6) ;
\draw [shift={(105.8,96.6)}, rotate = 326.42] [color={rgb, 255:red, 0; green, 0; blue, 0 }  ][fill={rgb, 255:red, 0; green, 0; blue, 0 }  ][line width=0.75]      (0, 0) circle [x radius= 2.34, y radius= 2.34]   ;
\draw [shift={(55.8,129.8)}, rotate = 326.42] [color={rgb, 255:red, 0; green, 0; blue, 0 }  ][fill={rgb, 255:red, 0; green, 0; blue, 0 }  ][line width=0.75]      (0, 0) circle [x radius= 2.34, y radius= 2.34]   ;
\draw [color={rgb, 255:red, 0; green, 0; blue, 0 }  ,draw opacity=1 ]   (128.8,122.4) -- (105.8,96.6) ;
\draw [shift={(105.8,96.6)}, rotate = 228.28] [color={rgb, 255:red, 0; green, 0; blue, 0 }  ,draw opacity=1 ][fill={rgb, 255:red, 0; green, 0; blue, 0 }  ,fill opacity=1 ][line width=0.75]      (0, 0) circle [x radius= 2.34, y radius= 2.34]   ;
\draw [shift={(128.8,122.4)}, rotate = 228.28] [color={rgb, 255:red, 0; green, 0; blue, 0 }  ,draw opacity=1 ][fill={rgb, 255:red, 0; green, 0; blue, 0 }  ,fill opacity=1 ][line width=0.75]      (0, 0) circle [x radius= 2.34, y radius= 2.34]   ;
\draw [color={rgb, 255:red, 0; green, 0; blue, 0 }  ,draw opacity=1 ]   (160.2,111.8) -- (128.8,122.4) ;
\draw [shift={(128.8,122.4)}, rotate = 161.35] [color={rgb, 255:red, 0; green, 0; blue, 0 }  ,draw opacity=1 ][fill={rgb, 255:red, 0; green, 0; blue, 0 }  ,fill opacity=1 ][line width=0.75]      (0, 0) circle [x radius= 2.34, y radius= 2.34]   ;
\draw [shift={(160.2,111.8)}, rotate = 161.35] [color={rgb, 255:red, 0; green, 0; blue, 0 }  ,draw opacity=1 ][fill={rgb, 255:red, 0; green, 0; blue, 0 }  ,fill opacity=1 ][line width=0.75]      (0, 0) circle [x radius= 2.34, y radius= 2.34]   ;
\draw [color={rgb, 255:red, 0; green, 0; blue, 0 }  ,draw opacity=1 ]   (159,152.6) -- (128.8,122.4) ;
\draw [shift={(128.8,122.4)}, rotate = 225] [color={rgb, 255:red, 0; green, 0; blue, 0 }  ,draw opacity=1 ][fill={rgb, 255:red, 0; green, 0; blue, 0 }  ,fill opacity=1 ][line width=0.75]      (0, 0) circle [x radius= 2.34, y radius= 2.34]   ;
\draw [shift={(159,152.6)}, rotate = 225] [color={rgb, 255:red, 0; green, 0; blue, 0 }  ,draw opacity=1 ][fill={rgb, 255:red, 0; green, 0; blue, 0 }  ,fill opacity=1 ][line width=0.75]      (0, 0) circle [x radius= 2.34, y radius= 2.34]   ;
\draw    (43,88.2) -- (27.4,81.8) ;
\draw    (55.8,129.8) -- (38.6,135.4) ;
\draw    (55.8,129.8) -- (50.2,145.8) ;
\draw    (171,104.2) -- (160.2,111.8) ;
\draw    (167.8,144.2) -- (159,152.6) ;
\draw    (105.8,96.6) .. controls (117,87.4) and (123,64.2) .. (107.4,63.8) .. controls (91.8,63.4) and (86.6,79.8) .. (105.8,96.6) -- cycle ;
\draw [color={rgb, 255:red, 155; green, 155; blue, 155 }  ,draw opacity=1 ]   (255.39,89.31) -- (232.2,80.23) ;
\draw [color={rgb, 255:red, 155; green, 155; blue, 155 }  ,draw opacity=1 ]   (274.42,148.33) -- (246.47,154.57) ;
\draw [color={rgb, 255:red, 155; green, 155; blue, 155 }  ,draw opacity=1 ]   (274.42,148.33) -- (266.09,171.03) ;
\draw [color={rgb, 255:red, 155; green, 155; blue, 155 }  ,draw opacity=1 ]   (478.52,105.77) -- (462.47,116.55) ;
\draw [color={rgb, 255:red, 155; green, 155; blue, 155 }  ,draw opacity=1 ]   (478.62,173.2) -- (460.68,174.43) ;
\draw [color={rgb, 255:red, 155; green, 155; blue, 155 }  ,draw opacity=1 ]   (255.84,89.31) -- (284.98,102.93) ;
\draw [shift={(284.98,102.93)}, rotate = 0] [color={rgb, 255:red, 155; green, 155; blue, 155 }  ,draw opacity=1 ][fill={rgb, 255:red, 155; green, 155; blue, 155 }  ,fill opacity=1 ][line width=0.75]      (0, 0) circle [x radius= 1.34, y radius= 1.34]   ;
\draw [shift={(255.84,89.31)}, rotate = 25.05] [color={rgb, 255:red, 155; green, 155; blue, 155 }  ,draw opacity=1 ][fill={rgb, 255:red, 155; green, 155; blue, 155 }  ,fill opacity=1 ][line width=0.75]      (0, 0) circle [x radius= 1.34, y radius= 1.34]   ;
\draw [color={rgb, 255:red, 155; green, 155; blue, 155 }  ,draw opacity=1 ]   (284.98,102.93) -- (314.12,101.8) ;
\draw [shift={(314.12,101.8)}, rotate = 357.77] [color={rgb, 255:red, 155; green, 155; blue, 155 }  ,draw opacity=1 ][fill={rgb, 255:red, 155; green, 155; blue, 155 }  ,fill opacity=1 ][line width=0.75]      (0, 0) circle [x radius= 1.34, y radius= 1.34]   ;
\draw [shift={(284.98,102.93)}, rotate = 357.77] [color={rgb, 255:red, 155; green, 155; blue, 155 }  ,draw opacity=1 ][fill={rgb, 255:red, 155; green, 155; blue, 155 }  ,fill opacity=1 ][line width=0.75]      (0, 0) circle [x radius= 1.34, y radius= 1.34]   ;
\draw [color={rgb, 255:red, 155; green, 155; blue, 155 }  ,draw opacity=1 ]   (314.12,101.8) -- (349.2,101.23) ;
\draw [shift={(349.2,101.23)}, rotate = 359.07] [color={rgb, 255:red, 155; green, 155; blue, 155 }  ,draw opacity=1 ][fill={rgb, 255:red, 155; green, 155; blue, 155 }  ,fill opacity=1 ][line width=0.75]      (0, 0) circle [x radius= 1.34, y radius= 1.34]   ;
\draw [shift={(314.12,101.8)}, rotate = 359.07] [color={rgb, 255:red, 155; green, 155; blue, 155 }  ,draw opacity=1 ][fill={rgb, 255:red, 155; green, 155; blue, 155 }  ,fill opacity=1 ][line width=0.75]      (0, 0) circle [x radius= 1.34, y radius= 1.34]   ;
\draw [color={rgb, 255:red, 155; green, 155; blue, 155 }  ,draw opacity=1 ]   (274.87,148.33) -- (310.36,146.42) ;
\draw [shift={(310.36,146.42)}, rotate = 0] [color={rgb, 255:red, 155; green, 155; blue, 155 }  ,draw opacity=1 ][fill={rgb, 255:red, 155; green, 155; blue, 155 }  ,fill opacity=1 ][line width=0.75]      (0, 0) circle [x radius= 1.34, y radius= 1.34]   ;
\draw [shift={(274.87,148.33)}, rotate = 356.92] [color={rgb, 255:red, 155; green, 155; blue, 155 }  ,draw opacity=1 ][fill={rgb, 255:red, 155; green, 155; blue, 155 }  ,fill opacity=1 ][line width=0.75]      (0, 0) circle [x radius= 1.34, y radius= 1.34]   ;
\draw [color={rgb, 255:red, 155; green, 155; blue, 155 }  ,draw opacity=1 ]   (310.36,146.42) -- (332.56,128.38) ;
\draw [shift={(332.56,128.38)}, rotate = 320.9] [color={rgb, 255:red, 155; green, 155; blue, 155 }  ,draw opacity=1 ][fill={rgb, 255:red, 155; green, 155; blue, 155 }  ,fill opacity=1 ][line width=0.75]      (0, 0) circle [x radius= 1.34, y radius= 1.34]   ;
\draw [shift={(310.36,146.42)}, rotate = 320.9] [color={rgb, 255:red, 155; green, 155; blue, 155 }  ,draw opacity=1 ][fill={rgb, 255:red, 155; green, 155; blue, 155 }  ,fill opacity=1 ][line width=0.75]      (0, 0) circle [x radius= 1.34, y radius= 1.34]   ;
\draw [color={rgb, 255:red, 155; green, 155; blue, 155 }  ,draw opacity=1 ]   (332.56,128.38) -- (349.2,101.23) ;
\draw [shift={(349.2,101.23)}, rotate = 301.5] [color={rgb, 255:red, 155; green, 155; blue, 155 }  ,draw opacity=1 ][fill={rgb, 255:red, 155; green, 155; blue, 155 }  ,fill opacity=1 ][line width=0.75]      (0, 0) circle [x radius= 1.34, y radius= 1.34]   ;
\draw [shift={(332.56,128.38)}, rotate = 301.5] [color={rgb, 255:red, 155; green, 155; blue, 155 }  ,draw opacity=1 ][fill={rgb, 255:red, 155; green, 155; blue, 155 }  ,fill opacity=1 ][line width=0.75]      (0, 0) circle [x radius= 1.34, y radius= 1.34]   ;
\draw [color={rgb, 255:red, 155; green, 155; blue, 155 }  ,draw opacity=1 ]   (274.87,148.33) -- (299.66,133.37) ;
\draw [shift={(299.66,133.37)}, rotate = 0] [color={rgb, 255:red, 155; green, 155; blue, 155 }  ,draw opacity=1 ][fill={rgb, 255:red, 155; green, 155; blue, 155 }  ,fill opacity=1 ][line width=0.75]      (0, 0) circle [x radius= 1.34, y radius= 1.34]   ;
\draw [shift={(274.87,148.33)}, rotate = 328.89] [color={rgb, 255:red, 155; green, 155; blue, 155 }  ,draw opacity=1 ][fill={rgb, 255:red, 155; green, 155; blue, 155 }  ,fill opacity=1 ][line width=0.75]      (0, 0) circle [x radius= 1.34, y radius= 1.34]   ;
\draw [color={rgb, 255:red, 155; green, 155; blue, 155 }  ,draw opacity=1 ]   (299.66,133.37) -- (321.86,115.33) ;
\draw [shift={(321.86,115.33)}, rotate = 320.9] [color={rgb, 255:red, 155; green, 155; blue, 155 }  ,draw opacity=1 ][fill={rgb, 255:red, 155; green, 155; blue, 155 }  ,fill opacity=1 ][line width=0.75]      (0, 0) circle [x radius= 1.34, y radius= 1.34]   ;
\draw [shift={(299.66,133.37)}, rotate = 320.9] [color={rgb, 255:red, 155; green, 155; blue, 155 }  ,draw opacity=1 ][fill={rgb, 255:red, 155; green, 155; blue, 155 }  ,fill opacity=1 ][line width=0.75]      (0, 0) circle [x radius= 1.34, y radius= 1.34]   ;
\draw [color={rgb, 255:red, 155; green, 155; blue, 155 }  ,draw opacity=1 ]   (321.86,115.33) -- (349.2,101.23) ;
\draw [shift={(349.2,101.23)}, rotate = 332.72] [color={rgb, 255:red, 155; green, 155; blue, 155 }  ,draw opacity=1 ][fill={rgb, 255:red, 155; green, 155; blue, 155 }  ,fill opacity=1 ][line width=0.75]      (0, 0) circle [x radius= 1.34, y radius= 1.34]   ;
\draw [shift={(321.86,115.33)}, rotate = 332.72] [color={rgb, 255:red, 155; green, 155; blue, 155 }  ,draw opacity=1 ][fill={rgb, 255:red, 155; green, 155; blue, 155 }  ,fill opacity=1 ][line width=0.75]      (0, 0) circle [x radius= 1.34, y radius= 1.34]   ;
\draw [color={rgb, 255:red, 155; green, 155; blue, 155 }  ,draw opacity=1 ]   (255.84,89.31) -- (288.55,84.77) ;
\draw [shift={(288.55,84.77)}, rotate = 0] [color={rgb, 255:red, 155; green, 155; blue, 155 }  ,draw opacity=1 ][fill={rgb, 255:red, 155; green, 155; blue, 155 }  ,fill opacity=1 ][line width=0.75]      (0, 0) circle [x radius= 1.34, y radius= 1.34]   ;
\draw [shift={(255.84,89.31)}, rotate = 352.1] [color={rgb, 255:red, 155; green, 155; blue, 155 }  ,draw opacity=1 ][fill={rgb, 255:red, 155; green, 155; blue, 155 }  ,fill opacity=1 ][line width=0.75]      (0, 0) circle [x radius= 1.34, y radius= 1.34]   ;
\draw [color={rgb, 255:red, 155; green, 155; blue, 155 }  ,draw opacity=1 ]   (288.55,84.77) -- (317.69,83.64) ;
\draw [shift={(317.69,83.64)}, rotate = 357.77] [color={rgb, 255:red, 155; green, 155; blue, 155 }  ,draw opacity=1 ][fill={rgb, 255:red, 155; green, 155; blue, 155 }  ,fill opacity=1 ][line width=0.75]      (0, 0) circle [x radius= 1.34, y radius= 1.34]   ;
\draw [shift={(288.55,84.77)}, rotate = 357.77] [color={rgb, 255:red, 155; green, 155; blue, 155 }  ,draw opacity=1 ][fill={rgb, 255:red, 155; green, 155; blue, 155 }  ,fill opacity=1 ][line width=0.75]      (0, 0) circle [x radius= 1.34, y radius= 1.34]   ;
\draw [color={rgb, 255:red, 155; green, 155; blue, 155 }  ,draw opacity=1 ]   (317.69,83.64) -- (349.2,101.23) ;
\draw [shift={(349.2,101.23)}, rotate = 29.17] [color={rgb, 255:red, 155; green, 155; blue, 155 }  ,draw opacity=1 ][fill={rgb, 255:red, 155; green, 155; blue, 155 }  ,fill opacity=1 ][line width=0.75]      (0, 0) circle [x radius= 1.34, y radius= 1.34]   ;
\draw [shift={(317.69,83.64)}, rotate = 29.17] [color={rgb, 255:red, 155; green, 155; blue, 155 }  ,draw opacity=1 ][fill={rgb, 255:red, 155; green, 155; blue, 155 }  ,fill opacity=1 ][line width=0.75]      (0, 0) circle [x radius= 1.34, y radius= 1.34]   ;
\draw [color={rgb, 255:red, 155; green, 155; blue, 155 }  ,draw opacity=1 ]   (255.84,89.31) -- (244.38,112.01) ;
\draw [shift={(244.38,112.01)}, rotate = 0] [color={rgb, 255:red, 155; green, 155; blue, 155 }  ,draw opacity=1 ][fill={rgb, 255:red, 155; green, 155; blue, 155 }  ,fill opacity=1 ][line width=0.75]      (0, 0) circle [x radius= 1.34, y radius= 1.34]   ;
\draw [shift={(255.84,89.31)}, rotate = 116.8] [color={rgb, 255:red, 155; green, 155; blue, 155 }  ,draw opacity=1 ][fill={rgb, 255:red, 155; green, 155; blue, 155 }  ,fill opacity=1 ][line width=0.75]      (0, 0) circle [x radius= 1.34, y radius= 1.34]   ;
\draw [color={rgb, 255:red, 155; green, 155; blue, 155 }  ,draw opacity=1 ]   (244.38,112.01) -- (259.53,135.79) ;
\draw [shift={(259.53,135.79)}, rotate = 57.5] [color={rgb, 255:red, 155; green, 155; blue, 155 }  ,draw opacity=1 ][fill={rgb, 255:red, 155; green, 155; blue, 155 }  ,fill opacity=1 ][line width=0.75]      (0, 0) circle [x radius= 1.34, y radius= 1.34]   ;
\draw [shift={(244.38,112.01)}, rotate = 57.5] [color={rgb, 255:red, 155; green, 155; blue, 155 }  ,draw opacity=1 ][fill={rgb, 255:red, 155; green, 155; blue, 155 }  ,fill opacity=1 ][line width=0.75]      (0, 0) circle [x radius= 1.34, y radius= 1.34]   ;
\draw [color={rgb, 255:red, 155; green, 155; blue, 155 }  ,draw opacity=1 ]   (259.53,135.79) -- (274.87,148.33) ;
\draw [shift={(274.87,148.33)}, rotate = 39.25] [color={rgb, 255:red, 155; green, 155; blue, 155 }  ,draw opacity=1 ][fill={rgb, 255:red, 155; green, 155; blue, 155 }  ,fill opacity=1 ][line width=0.75]      (0, 0) circle [x radius= 1.34, y radius= 1.34]   ;
\draw [shift={(259.53,135.79)}, rotate = 39.25] [color={rgb, 255:red, 155; green, 155; blue, 155 }  ,draw opacity=1 ][fill={rgb, 255:red, 155; green, 155; blue, 155 }  ,fill opacity=1 ][line width=0.75]      (0, 0) circle [x radius= 1.34, y radius= 1.34]   ;
\draw [color={rgb, 255:red, 155; green, 155; blue, 155 }  ,draw opacity=1 ]   (255.84,89.31) -- (262.98,108.51) ;
\draw [shift={(262.98,108.51)}, rotate = 0] [color={rgb, 255:red, 155; green, 155; blue, 155 }  ,draw opacity=1 ][fill={rgb, 255:red, 155; green, 155; blue, 155 }  ,fill opacity=1 ][line width=0.75]      (0, 0) circle [x radius= 1.34, y radius= 1.34]   ;
\draw [shift={(255.84,89.31)}, rotate = 69.61] [color={rgb, 255:red, 155; green, 155; blue, 155 }  ,draw opacity=1 ][fill={rgb, 255:red, 155; green, 155; blue, 155 }  ,fill opacity=1 ][line width=0.75]      (0, 0) circle [x radius= 1.34, y radius= 1.34]   ;
\draw [color={rgb, 255:red, 155; green, 155; blue, 155 }  ,draw opacity=1 ]   (262.98,108.51) -- (269.32,130.94) ;
\draw [shift={(269.32,130.94)}, rotate = 74.22] [color={rgb, 255:red, 155; green, 155; blue, 155 }  ,draw opacity=1 ][fill={rgb, 255:red, 155; green, 155; blue, 155 }  ,fill opacity=1 ][line width=0.75]      (0, 0) circle [x radius= 1.34, y radius= 1.34]   ;
\draw [shift={(262.98,108.51)}, rotate = 74.22] [color={rgb, 255:red, 155; green, 155; blue, 155 }  ,draw opacity=1 ][fill={rgb, 255:red, 155; green, 155; blue, 155 }  ,fill opacity=1 ][line width=0.75]      (0, 0) circle [x radius= 1.34, y radius= 1.34]   ;
\draw [color={rgb, 255:red, 155; green, 155; blue, 155 }  ,draw opacity=1 ]   (269.32,130.94) -- (274.87,148.33) ;
\draw [shift={(274.87,148.33)}, rotate = 72.29] [color={rgb, 255:red, 155; green, 155; blue, 155 }  ,draw opacity=1 ][fill={rgb, 255:red, 155; green, 155; blue, 155 }  ,fill opacity=1 ][line width=0.75]      (0, 0) circle [x radius= 1.34, y radius= 1.34]   ;
\draw [shift={(269.32,130.94)}, rotate = 72.29] [color={rgb, 255:red, 155; green, 155; blue, 155 }  ,draw opacity=1 ][fill={rgb, 255:red, 155; green, 155; blue, 155 }  ,fill opacity=1 ][line width=0.75]      (0, 0) circle [x radius= 1.34, y radius= 1.34]   ;
\draw [color={rgb, 255:red, 155; green, 155; blue, 155 }  ,draw opacity=1 ]   (255.84,89.31) -- (239.79,70.59) ;
\draw [color={rgb, 255:red, 155; green, 155; blue, 155 }  ,draw opacity=1 ]   (273.97,148.33) -- (274.56,175) ;
\draw [color={rgb, 255:red, 155; green, 155; blue, 155 }  ,draw opacity=1 ]   (274.42,148.33) -- (254.2,162.52) ;
\draw [color={rgb, 255:red, 155; green, 155; blue, 155 }  ,draw opacity=1 ]   (349.2,101.23) -- (378.63,99.84) ;
\draw [shift={(378.63,99.84)}, rotate = 0] [color={rgb, 255:red, 155; green, 155; blue, 155 }  ,draw opacity=1 ][fill={rgb, 255:red, 155; green, 155; blue, 155 }  ,fill opacity=1 ][line width=0.75]      (0, 0) circle [x radius= 1.34, y radius= 1.34]   ;
\draw [shift={(349.2,101.23)}, rotate = 357.3] [color={rgb, 255:red, 155; green, 155; blue, 155 }  ,draw opacity=1 ][fill={rgb, 255:red, 155; green, 155; blue, 155 }  ,fill opacity=1 ][line width=0.75]      (0, 0) circle [x radius= 1.34, y radius= 1.34]   ;
\draw [color={rgb, 255:red, 155; green, 155; blue, 155 }  ,draw opacity=1 ]   (378.63,99.84) -- (403.75,113.85) ;
\draw [shift={(403.75,113.85)}, rotate = 29.14] [color={rgb, 255:red, 155; green, 155; blue, 155 }  ,draw opacity=1 ][fill={rgb, 255:red, 155; green, 155; blue, 155 }  ,fill opacity=1 ][line width=0.75]      (0, 0) circle [x radius= 1.34, y radius= 1.34]   ;
\draw [shift={(378.63,99.84)}, rotate = 29.14] [color={rgb, 255:red, 155; green, 155; blue, 155 }  ,draw opacity=1 ][fill={rgb, 255:red, 155; green, 155; blue, 155 }  ,fill opacity=1 ][line width=0.75]      (0, 0) circle [x radius= 1.34, y radius= 1.34]   ;
\draw [color={rgb, 255:red, 155; green, 155; blue, 155 }  ,draw opacity=1 ]   (403.75,113.85) -- (416.7,131.59) ;
\draw [shift={(416.7,131.59)}, rotate = 53.88] [color={rgb, 255:red, 155; green, 155; blue, 155 }  ,draw opacity=1 ][fill={rgb, 255:red, 155; green, 155; blue, 155 }  ,fill opacity=1 ][line width=0.75]      (0, 0) circle [x radius= 1.34, y radius= 1.34]   ;
\draw [shift={(403.75,113.85)}, rotate = 53.88] [color={rgb, 255:red, 155; green, 155; blue, 155 }  ,draw opacity=1 ][fill={rgb, 255:red, 155; green, 155; blue, 155 }  ,fill opacity=1 ][line width=0.75]      (0, 0) circle [x radius= 1.34, y radius= 1.34]   ;
\draw [color={rgb, 255:red, 155; green, 155; blue, 155 }  ,draw opacity=1 ]   (349.2,101.23) -- (366.78,121.8) ;
\draw [shift={(366.78,121.8)}, rotate = 0] [color={rgb, 255:red, 155; green, 155; blue, 155 }  ,draw opacity=1 ][fill={rgb, 255:red, 155; green, 155; blue, 155 }  ,fill opacity=1 ][line width=0.75]      (0, 0) circle [x radius= 1.34, y radius= 1.34]   ;
\draw [shift={(349.2,101.23)}, rotate = 49.49] [color={rgb, 255:red, 155; green, 155; blue, 155 }  ,draw opacity=1 ][fill={rgb, 255:red, 155; green, 155; blue, 155 }  ,fill opacity=1 ][line width=0.75]      (0, 0) circle [x radius= 1.34, y radius= 1.34]   ;
\draw [color={rgb, 255:red, 155; green, 155; blue, 155 }  ,draw opacity=1 ]   (366.78,121.8) -- (391.9,135.8) ;
\draw [shift={(391.9,135.8)}, rotate = 29.14] [color={rgb, 255:red, 155; green, 155; blue, 155 }  ,draw opacity=1 ][fill={rgb, 255:red, 155; green, 155; blue, 155 }  ,fill opacity=1 ][line width=0.75]      (0, 0) circle [x radius= 1.34, y radius= 1.34]   ;
\draw [shift={(366.78,121.8)}, rotate = 29.14] [color={rgb, 255:red, 155; green, 155; blue, 155 }  ,draw opacity=1 ][fill={rgb, 255:red, 155; green, 155; blue, 155 }  ,fill opacity=1 ][line width=0.75]      (0, 0) circle [x radius= 1.34, y radius= 1.34]   ;
\draw [color={rgb, 255:red, 155; green, 155; blue, 155 }  ,draw opacity=1 ]   (391.9,135.8) -- (416.7,131.59) ;
\draw [shift={(416.7,131.59)}, rotate = 350.35] [color={rgb, 255:red, 155; green, 155; blue, 155 }  ,draw opacity=1 ][fill={rgb, 255:red, 155; green, 155; blue, 155 }  ,fill opacity=1 ][line width=0.75]      (0, 0) circle [x radius= 1.34, y radius= 1.34]   ;
\draw [shift={(391.9,135.8)}, rotate = 350.35] [color={rgb, 255:red, 155; green, 155; blue, 155 }  ,draw opacity=1 ][fill={rgb, 255:red, 155; green, 155; blue, 155 }  ,fill opacity=1 ][line width=0.75]      (0, 0) circle [x radius= 1.34, y radius= 1.34]   ;
\draw    (160.2,111.8) -- (166.6,117.8) ;
\draw [color={rgb, 255:red, 155; green, 155; blue, 155 }  ,draw opacity=1 ]   (475.17,167.13) -- (460.68,174.43) ;
\draw [color={rgb, 255:red, 155; green, 155; blue, 155 }  ,draw opacity=1 ]   (462.47,116.55) -- (467.76,103.13) ;
\draw [color={rgb, 255:red, 155; green, 155; blue, 155 }  ,draw opacity=1 ]   (480.6,122.75) -- (462.47,116.55) ;
\draw [color={rgb, 255:red, 155; green, 155; blue, 155 }  ,draw opacity=1 ]   (474.67,128.36) -- (462.47,116.55) ;
\draw [color={rgb, 255:red, 155; green, 155; blue, 155 }  ,draw opacity=1 ]   (416.7,131.59) -- (424.75,117.54) ;
\draw [shift={(424.75,117.54)}, rotate = 0] [color={rgb, 255:red, 155; green, 155; blue, 155 }  ,draw opacity=1 ][fill={rgb, 255:red, 155; green, 155; blue, 155 }  ,fill opacity=1 ][line width=0.75]      (0, 0) circle [x radius= 1.34, y radius= 1.34]   ;
\draw [shift={(416.7,131.59)}, rotate = 299.82] [color={rgb, 255:red, 155; green, 155; blue, 155 }  ,draw opacity=1 ][fill={rgb, 255:red, 155; green, 155; blue, 155 }  ,fill opacity=1 ][line width=0.75]      (0, 0) circle [x radius= 1.34, y radius= 1.34]   ;
\draw [color={rgb, 255:red, 155; green, 155; blue, 155 }  ,draw opacity=1 ]   (424.75,117.54) -- (442.96,111.8) ;
\draw [shift={(442.96,111.8)}, rotate = 342.52] [color={rgb, 255:red, 155; green, 155; blue, 155 }  ,draw opacity=1 ][fill={rgb, 255:red, 155; green, 155; blue, 155 }  ,fill opacity=1 ][line width=0.75]      (0, 0) circle [x radius= 1.34, y radius= 1.34]   ;
\draw [shift={(424.75,117.54)}, rotate = 342.52] [color={rgb, 255:red, 155; green, 155; blue, 155 }  ,draw opacity=1 ][fill={rgb, 255:red, 155; green, 155; blue, 155 }  ,fill opacity=1 ][line width=0.75]      (0, 0) circle [x radius= 1.34, y radius= 1.34]   ;
\draw [color={rgb, 255:red, 155; green, 155; blue, 155 }  ,draw opacity=1 ]   (442.96,111.8) -- (463.38,116.55) ;
\draw [shift={(463.38,116.55)}, rotate = 13.1] [color={rgb, 255:red, 155; green, 155; blue, 155 }  ,draw opacity=1 ][fill={rgb, 255:red, 155; green, 155; blue, 155 }  ,fill opacity=1 ][line width=0.75]      (0, 0) circle [x radius= 1.34, y radius= 1.34]   ;
\draw [shift={(442.96,111.8)}, rotate = 13.1] [color={rgb, 255:red, 155; green, 155; blue, 155 }  ,draw opacity=1 ][fill={rgb, 255:red, 155; green, 155; blue, 155 }  ,fill opacity=1 ][line width=0.75]      (0, 0) circle [x radius= 1.34, y radius= 1.34]   ;
\draw [color={rgb, 255:red, 155; green, 155; blue, 155 }  ,draw opacity=1 ]   (416.7,131.59) -- (433.36,129.35) ;
\draw [shift={(433.36,129.35)}, rotate = 0] [color={rgb, 255:red, 155; green, 155; blue, 155 }  ,draw opacity=1 ][fill={rgb, 255:red, 155; green, 155; blue, 155 }  ,fill opacity=1 ][line width=0.75]      (0, 0) circle [x radius= 1.34, y radius= 1.34]   ;
\draw [shift={(416.7,131.59)}, rotate = 352.35] [color={rgb, 255:red, 155; green, 155; blue, 155 }  ,draw opacity=1 ][fill={rgb, 255:red, 155; green, 155; blue, 155 }  ,fill opacity=1 ][line width=0.75]      (0, 0) circle [x radius= 1.34, y radius= 1.34]   ;
\draw [color={rgb, 255:red, 155; green, 155; blue, 155 }  ,draw opacity=1 ]   (433.36,129.35) -- (451.57,123.62) ;
\draw [shift={(451.57,123.62)}, rotate = 342.52] [color={rgb, 255:red, 155; green, 155; blue, 155 }  ,draw opacity=1 ][fill={rgb, 255:red, 155; green, 155; blue, 155 }  ,fill opacity=1 ][line width=0.75]      (0, 0) circle [x radius= 1.34, y radius= 1.34]   ;
\draw [shift={(433.36,129.35)}, rotate = 342.52] [color={rgb, 255:red, 155; green, 155; blue, 155 }  ,draw opacity=1 ][fill={rgb, 255:red, 155; green, 155; blue, 155 }  ,fill opacity=1 ][line width=0.75]      (0, 0) circle [x radius= 1.34, y radius= 1.34]   ;
\draw [color={rgb, 255:red, 155; green, 155; blue, 155 }  ,draw opacity=1 ]   (451.57,123.62) -- (463.38,116.55) ;
\draw [shift={(463.38,116.55)}, rotate = 329.1] [color={rgb, 255:red, 155; green, 155; blue, 155 }  ,draw opacity=1 ][fill={rgb, 255:red, 155; green, 155; blue, 155 }  ,fill opacity=1 ][line width=0.75]      (0, 0) circle [x radius= 1.34, y radius= 1.34]   ;
\draw [shift={(451.57,123.62)}, rotate = 329.1] [color={rgb, 255:red, 155; green, 155; blue, 155 }  ,draw opacity=1 ][fill={rgb, 255:red, 155; green, 155; blue, 155 }  ,fill opacity=1 ][line width=0.75]      (0, 0) circle [x radius= 1.34, y radius= 1.34]   ;
\draw [color={rgb, 255:red, 155; green, 155; blue, 155 }  ,draw opacity=1 ]   (461.59,174.43) -- (441.22,166.55) ;
\draw [shift={(441.22,166.55)}, rotate = 0] [color={rgb, 255:red, 155; green, 155; blue, 155 }  ,draw opacity=1 ][fill={rgb, 255:red, 155; green, 155; blue, 155 }  ,fill opacity=1 ][line width=0.75]      (0, 0) circle [x radius= 1.34, y radius= 1.34]   ;
\draw [shift={(461.59,174.43)}, rotate = 201.16] [color={rgb, 255:red, 155; green, 155; blue, 155 }  ,draw opacity=1 ][fill={rgb, 255:red, 155; green, 155; blue, 155 }  ,fill opacity=1 ][line width=0.75]      (0, 0) circle [x radius= 1.34, y radius= 1.34]   ;
\draw [color={rgb, 255:red, 155; green, 155; blue, 155 }  ,draw opacity=1 ]   (441.22,166.55) -- (422.58,151.37) ;
\draw [shift={(422.58,151.37)}, rotate = 219.16] [color={rgb, 255:red, 155; green, 155; blue, 155 }  ,draw opacity=1 ][fill={rgb, 255:red, 155; green, 155; blue, 155 }  ,fill opacity=1 ][line width=0.75]      (0, 0) circle [x radius= 1.34, y radius= 1.34]   ;
\draw [shift={(441.22,166.55)}, rotate = 219.16] [color={rgb, 255:red, 155; green, 155; blue, 155 }  ,draw opacity=1 ][fill={rgb, 255:red, 155; green, 155; blue, 155 }  ,fill opacity=1 ][line width=0.75]      (0, 0) circle [x radius= 1.34, y radius= 1.34]   ;
\draw [color={rgb, 255:red, 155; green, 155; blue, 155 }  ,draw opacity=1 ]   (422.58,151.37) -- (416.7,131.59) ;
\draw [shift={(416.7,131.59)}, rotate = 253.43] [color={rgb, 255:red, 155; green, 155; blue, 155 }  ,draw opacity=1 ][fill={rgb, 255:red, 155; green, 155; blue, 155 }  ,fill opacity=1 ][line width=0.75]      (0, 0) circle [x radius= 1.34, y radius= 1.34]   ;
\draw [shift={(422.58,151.37)}, rotate = 253.43] [color={rgb, 255:red, 155; green, 155; blue, 155 }  ,draw opacity=1 ][fill={rgb, 255:red, 155; green, 155; blue, 155 }  ,fill opacity=1 ][line width=0.75]      (0, 0) circle [x radius= 1.34, y radius= 1.34]   ;
\draw [color={rgb, 255:red, 155; green, 155; blue, 155 }  ,draw opacity=1 ]   (461.59,174.43) -- (448.19,153.83) ;
\draw [shift={(448.19,153.83)}, rotate = 0] [color={rgb, 255:red, 155; green, 155; blue, 155 }  ,draw opacity=1 ][fill={rgb, 255:red, 155; green, 155; blue, 155 }  ,fill opacity=1 ][line width=0.75]      (0, 0) circle [x radius= 1.34, y radius= 1.34]   ;
\draw [shift={(461.59,174.43)}, rotate = 236.96] [color={rgb, 255:red, 155; green, 155; blue, 155 }  ,draw opacity=1 ][fill={rgb, 255:red, 155; green, 155; blue, 155 }  ,fill opacity=1 ][line width=0.75]      (0, 0) circle [x radius= 1.34, y radius= 1.34]   ;
\draw [color={rgb, 255:red, 155; green, 155; blue, 155 }  ,draw opacity=1 ]   (448.19,153.83) -- (434.02,141.74) ;
\draw [shift={(434.02,141.74)}, rotate = 220.46] [color={rgb, 255:red, 155; green, 155; blue, 155 }  ,draw opacity=1 ][fill={rgb, 255:red, 155; green, 155; blue, 155 }  ,fill opacity=1 ][line width=0.75]      (0, 0) circle [x radius= 1.34, y radius= 1.34]   ;
\draw [shift={(448.19,153.83)}, rotate = 220.46] [color={rgb, 255:red, 155; green, 155; blue, 155 }  ,draw opacity=1 ][fill={rgb, 255:red, 155; green, 155; blue, 155 }  ,fill opacity=1 ][line width=0.75]      (0, 0) circle [x radius= 1.34, y radius= 1.34]   ;
\draw [color={rgb, 255:red, 155; green, 155; blue, 155 }  ,draw opacity=1 ]   (434.02,141.74) -- (416.7,131.59) ;
\draw [shift={(416.7,131.59)}, rotate = 210.37] [color={rgb, 255:red, 155; green, 155; blue, 155 }  ,draw opacity=1 ][fill={rgb, 255:red, 155; green, 155; blue, 155 }  ,fill opacity=1 ][line width=0.75]      (0, 0) circle [x radius= 1.34, y radius= 1.34]   ;
\draw [shift={(434.02,141.74)}, rotate = 210.37] [color={rgb, 255:red, 155; green, 155; blue, 155 }  ,draw opacity=1 ][fill={rgb, 255:red, 155; green, 155; blue, 155 }  ,fill opacity=1 ][line width=0.75]      (0, 0) circle [x radius= 1.34, y radius= 1.34]   ;
\draw [color={rgb, 255:red, 155; green, 155; blue, 155 }  ,draw opacity=1 ]   (349.2,101.23) .. controls (334.19,80.9) and (331.92,88.5) .. (326.94,72.04) ;
\draw [shift={(326.94,72.04)}, rotate = 253.14] [color={rgb, 255:red, 155; green, 155; blue, 155 }  ,draw opacity=1 ][fill={rgb, 255:red, 155; green, 155; blue, 155 }  ,fill opacity=1 ][line width=0.75]      (0, 0) circle [x radius= 1.34, y radius= 1.34]   ;
\draw [shift={(349.2,101.23)}, rotate = 233.55] [color={rgb, 255:red, 155; green, 155; blue, 155 }  ,draw opacity=1 ][fill={rgb, 255:red, 155; green, 155; blue, 155 }  ,fill opacity=1 ][line width=0.75]      (0, 0) circle [x radius= 1.34, y radius= 1.34]   ;
\draw [color={rgb, 255:red, 155; green, 155; blue, 155 }  ,draw opacity=1 ]   (326.94,72.04) .. controls (337.82,31.53) and (385.86,36.59) .. (378.16,80.48) ;
\draw [shift={(378.16,80.48)}, rotate = 99.96] [color={rgb, 255:red, 155; green, 155; blue, 155 }  ,draw opacity=1 ][fill={rgb, 255:red, 155; green, 155; blue, 155 }  ,fill opacity=1 ][line width=0.75]      (0, 0) circle [x radius= 1.34, y radius= 1.34]   ;
\draw [shift={(326.94,72.04)}, rotate = 285.03] [color={rgb, 255:red, 155; green, 155; blue, 155 }  ,draw opacity=1 ][fill={rgb, 255:red, 155; green, 155; blue, 155 }  ,fill opacity=1 ][line width=0.75]      (0, 0) circle [x radius= 1.34, y radius= 1.34]   ;
\draw [color={rgb, 255:red, 155; green, 155; blue, 155 }  ,draw opacity=1 ]   (378.16,80.48) .. controls (366.83,87.23) and (368.19,93.56) .. (349.2,101.23) ;
\draw [shift={(349.2,101.23)}, rotate = 158] [color={rgb, 255:red, 155; green, 155; blue, 155 }  ,draw opacity=1 ][fill={rgb, 255:red, 155; green, 155; blue, 155 }  ,fill opacity=1 ][line width=0.75]      (0, 0) circle [x radius= 1.34, y radius= 1.34]   ;
\draw [shift={(378.16,80.48)}, rotate = 149.21] [color={rgb, 255:red, 155; green, 155; blue, 155 }  ,draw opacity=1 ][fill={rgb, 255:red, 155; green, 155; blue, 155 }  ,fill opacity=1 ][line width=0.75]      (0, 0) circle [x radius= 1.34, y radius= 1.34]   ;
\draw [color={rgb, 255:red, 155; green, 155; blue, 155 }  ,draw opacity=1 ]   (349.84,100.81) .. controls (341.37,89.89) and (339.63,84.28) .. (341.44,70.35) ;
\draw [shift={(341.44,70.35)}, rotate = 277.42] [color={rgb, 255:red, 155; green, 155; blue, 155 }  ,draw opacity=1 ][fill={rgb, 255:red, 155; green, 155; blue, 155 }  ,fill opacity=1 ][line width=0.75]      (0, 0) circle [x radius= 1.34, y radius= 1.34]   ;
\draw [shift={(349.84,100.81)}, rotate = 232.2] [color={rgb, 255:red, 155; green, 155; blue, 155 }  ,draw opacity=1 ][fill={rgb, 255:red, 155; green, 155; blue, 155 }  ,fill opacity=1 ][line width=0.75]      (0, 0) circle [x radius= 1.34, y radius= 1.34]   ;
\draw [color={rgb, 255:red, 155; green, 155; blue, 155 }  ,draw opacity=1 ]   (341.44,70.35) .. controls (353.23,46.3) and (361.39,53.89) .. (364.9,71.53) ;
\draw [shift={(364.9,71.53)}, rotate = 78.75] [color={rgb, 255:red, 155; green, 155; blue, 155 }  ,draw opacity=1 ][fill={rgb, 255:red, 155; green, 155; blue, 155 }  ,fill opacity=1 ][line width=0.75]      (0, 0) circle [x radius= 1.34, y radius= 1.34]   ;
\draw [shift={(341.44,70.35)}, rotate = 296.1] [color={rgb, 255:red, 155; green, 155; blue, 155 }  ,draw opacity=1 ][fill={rgb, 255:red, 155; green, 155; blue, 155 }  ,fill opacity=1 ][line width=0.75]      (0, 0) circle [x radius= 1.34, y radius= 1.34]   ;
\draw [color={rgb, 255:red, 155; green, 155; blue, 155 }  ,draw opacity=1 ]   (364.9,71.53) .. controls (358.21,83.85) and (357.31,91.87) .. (349.84,100.81) ;
\draw [shift={(349.84,100.81)}, rotate = 129.88] [color={rgb, 255:red, 155; green, 155; blue, 155 }  ,draw opacity=1 ][fill={rgb, 255:red, 155; green, 155; blue, 155 }  ,fill opacity=1 ][line width=0.75]      (0, 0) circle [x radius= 1.34, y radius= 1.34]   ;
\draw [shift={(364.9,71.53)}, rotate = 118.47] [color={rgb, 255:red, 155; green, 155; blue, 155 }  ,draw opacity=1 ][fill={rgb, 255:red, 155; green, 155; blue, 155 }  ,fill opacity=1 ][line width=0.75]      (0, 0) circle [x radius= 1.34, y radius= 1.34]   ;
\draw    (255.84,89.31) ;
\draw [shift={(255.84,89.31)}, rotate = 0] [color={rgb, 255:red, 0; green, 0; blue, 0 }  ][fill={rgb, 255:red, 0; green, 0; blue, 0 }  ][line width=0.75]      (0, 0) circle [x radius= 3.35, y radius= 3.35]   ;
\draw    (274.87,148.33) ;
\draw [shift={(274.87,148.33)}, rotate = 0] [color={rgb, 255:red, 0; green, 0; blue, 0 }  ][fill={rgb, 255:red, 0; green, 0; blue, 0 }  ][line width=0.75]      (0, 0) circle [x radius= 3.35, y radius= 3.35]   ;
\draw [color={rgb, 255:red, 208; green, 2; blue, 27 }  ,draw opacity=1 ]   (416.7,131.59) ;
\draw [shift={(416.7,131.59)}, rotate = 0] [color={rgb, 255:red, 208; green, 2; blue, 27 }  ,draw opacity=1 ][fill={rgb, 255:red, 208; green, 2; blue, 27 }  ,fill opacity=1 ][line width=0.75]      (0, 0) circle [x radius= 3.35, y radius= 3.35]   ;
\draw    (463.38,116.55) ;
\draw [shift={(463.38,116.55)}, rotate = 0] [color={rgb, 255:red, 0; green, 0; blue, 0 }  ][fill={rgb, 255:red, 0; green, 0; blue, 0 }  ][line width=0.75]      (0, 0) circle [x radius= 3.35, y radius= 3.35]   ;
\draw    (460.68,174.43) ;
\draw [shift={(460.68,174.43)}, rotate = 0] [color={rgb, 255:red, 0; green, 0; blue, 0 }  ][fill={rgb, 255:red, 0; green, 0; blue, 0 }  ][line width=0.75]      (0, 0) circle [x radius= 3.35, y radius= 3.35]   ;
\draw [color={rgb, 255:red, 208; green, 2; blue, 27 }  ,draw opacity=1 ]   (349.84,100.81) ;
\draw [shift={(349.84,100.81)}, rotate = 0] [color={rgb, 255:red, 208; green, 2; blue, 27 }  ,draw opacity=1 ][fill={rgb, 255:red, 208; green, 2; blue, 27 }  ,fill opacity=1 ][line width=0.75]      (0, 0) circle [x radius= 3.35, y radius= 3.35]   ;
\draw [color={rgb, 255:red, 208; green, 2; blue, 27 }  ,draw opacity=1 ]   (128.8,122.4) ;
\draw [shift={(128.8,122.4)}, rotate = 0] [color={rgb, 255:red, 208; green, 2; blue, 27 }  ,draw opacity=1 ][fill={rgb, 255:red, 208; green, 2; blue, 27 }  ,fill opacity=1 ][line width=0.75]      (0, 0) circle [x radius= 3.35, y radius= 3.35]   ;
\draw [color={rgb, 255:red, 208; green, 2; blue, 27 }  ,draw opacity=1 ]   (105.8,96.6) ;
\draw [shift={(105.8,96.6)}, rotate = 0] [color={rgb, 255:red, 208; green, 2; blue, 27 }  ,draw opacity=1 ][fill={rgb, 255:red, 208; green, 2; blue, 27 }  ,fill opacity=1 ][line width=0.75]      (0, 0) circle [x radius= 3.35, y radius= 3.35]   ;

\draw (192,108.2) node [anchor=north west][inner sep=0.75pt]    {$\Longrightarrow $};

\end{tikzpicture}

    \caption{Constructing a 2-connected planar super-graph. The marked vertices are depicted in red.}
    \label{fig:2-con-sup-graph}
\end{figure}
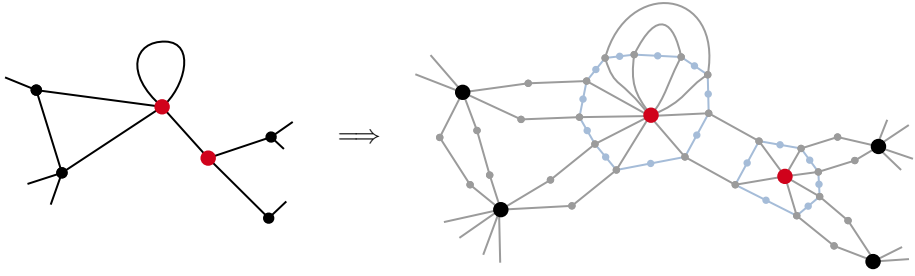

In light of Proposition~\ref{prop:well-embed}, we introduce the following terminology for the sake of brevity. 

\begin{definition}[Good drawing]\label{def:good-drawing}
    Let $\Gamma$ be a connected, locally finite, planar graph. Then a \textit{good drawing} of $\Gamma$ is an embedding $\vartheta : \overline \Gamma \into \bbS^2$. 
\end{definition}

Some planar graphs admit another nice type of drawing. We say that a planar graph $\Gamma$ admits a \textit{vertex accumulation point-free drawing}, or \textit{VAP-free drawing}, if there exists a topological embedding $\vartheta : \Gamma \into \R^2$ such that $\vartheta(V(\Gamma))$ is a discrete subset of $\R^2$. The following is a standard fact and follows easily from Proposition~\ref{prop:well-embed}. 

\begin{corollary}\label{cor:vap-free}
    Let $\Gamma$ be a connected, locally finite, one-ended, planar graph. Then $\Gamma$ admits a VAP-free drawing. 
\end{corollary}

\begin{remark}
    It is important to note that not every planar graph admits a VAP-free drawing. For example, the standard Cayley graph of $\Z^2 \ast \Z_2$ (a.k.a. the `tree of flats') is planar but admits no such drawing.
\end{remark}

An important feature of planar graphs is their \textit{faces}. Informally, it is clear what we mean by a `face'. However, this word could be referring to several different but closely related concepts, especially in the realm of infinite planar graphs. Is a `face' a component of $\bbS^2 \setminus \vartheta(\overline \Gamma)$, or actually a particular form of subgraph of $\Gamma$? For the sake of clarity, we now set up some notation to help us discuss faces without ambiguity. 

\begin{definition}[Faces and facial subgraphs]\label{def:faces}
    Let $\Gamma$ be a connected, locally finite, planar graph with a fixed good drawing $\vartheta$. The connected components of $\bbS^2 \setminus \vartheta(\overline \Gamma)$ are referred to as the \textit{faces} of $\Gamma$, with respect to the drawing $\vartheta$. Let $\facedisks(\Gamma)$ denote the set of faces. 
    
    Given $U \in \facedisks(\Gamma)$, write 
    $$
    \facepaths [U] := \vartheta^{-1}(\partial U) \setminus \Omega (\Gamma). 
    $$
    We call the subgraph $\facepaths [U]$ the \textit{facial subgraph of $\Gamma$ bordering $U$}. 
    Let 
    $
    \facepaths (\Gamma) 
    $
    denote the set of facial subgraphs of $\Gamma$. 
    $$
    \inffaces (\Gamma) = \{f \in \facepaths (\Gamma) : \text{$f$ is infinite}\}, \ \  \finfaces (\Gamma) = \{f \in \facepaths  (\Gamma) : \text{$f$ is finite}\}.
    $$
\end{definition}

\begin{remark}
    Note that $f \in \inffaces (\Gamma)$ need not be connected in general. However, the closure of $f$ in $\overline \Gamma$ will always be connected. 
\end{remark}

\begin{remark}
    Really, we should include mention of the drawing $\vartheta$ in our notation above. However, for our purposes this drawing will always be fixed in advance and so there is no risk of confusion. 
\end{remark}

There is a key benefit to working with 2-connected planar graphs, which is that their faces are incredibly sensible. More precisely, we have the following. 

\begin{proposition}[{\cite[Prop.~3]{richter20023}}]\label{prop:simple-face}
    Let $K$ be a compact, 2-connected, locally connected subset of the sphere. Then the boundary of every component of $\bbS^2 \setminus K$ is a simple closed curve.  
\end{proposition}

In particular, if $\Gamma$ is a 2-connected locally finite planar graph then the $\vartheta$-image of the closure of every $f \in \facepaths(\Gamma)$ in $\overline \Gamma$ is a simple closed curve in $\bbS^2$. More precisely, if $f \in \finfaces (\Gamma)$ then $f$ is a simple loop, and if $f \in \inffaces (\Gamma)$ then $f$ is a disjoint union of bi-infinite lines.









\section{Cutting up graphs}\label{sec:cuts}

In this section we collect definitions and results relating to cutting up graphs into simpler pieces.

\subsection{The Boolean ring of cuts}

    We open by standardising our terminology. 
    Let $X$ be a connected graph. 
    Given a subset $b \subset V(X)$, let $b^\ast$ denote the complement $V(X) \setminus b$. Let $\delta b$, called the \textit{coboundary of $b$}, denote the set of edges in $X$ with exactly one endpoint in $b$. 
    Given any $b \subset V(X)$, clearly $\delta b = \delta b^\ast$.
    Let 
    $$
    \br (X) = \{b \subset V(X) : \text{$\delta b$ is finite} \}. 
    $$
    Clearly $\br (X)$ is closed under the operations of symmetric difference, intersection, and complementation. This makes $\br (X)$ into a Boolean ring. That is, a commutative ring with unity such that every element $r$ satisfies the equation $r^2 = r$. The multiplicative operation is intersection, while the additive operation is symmetric difference. We will sometimes refer to elements of $\br(X)$ as \emph{cuts}. A cut $b \in \br(X)$ is said to be \emph{tight} if both $X[b]$ and $X[b^\ast]$ are connected. 
    
    Given a group $G$ acting on $X$, this induces an action of $G$ upon $\br (X)$. Thus, we may view $\br (X)$ as a $G$-module by `forgetting' the multiplicative operation. 
    Let $\br_n (X)$ denote the subring of $\br (X)$ generated by elements $b$ such that $|\delta b| < n$. 
    Given $b_1, b_2 \in \br (X)$, we say that $b_1$ \textit{crosses} $b_2$ if the intersections
    $$
    b_1 \cap b_2, \  \ b_1 \cap b_2^\ast, \ \ b_1^\ast \cap b_2, \ \ b_1^\ast \cap b_2^\ast
    $$
    are all non-empty. If $b_1$ and $b_2$ do not cross, we say they are \textit{nested}. We say that a subset $\E \subset \br (X)$ is \textit{nested} if any two $b_1, b_2 \in \E$ are nested. If $\E$ is closed under taking complements, then we say it is \emph{symmetric}. 
    We now state the following key theorem.
    
    \begin{theorem}[Dicks--Dunwoody, {\cite[Thm.~II.2.20]{dicks1989groups}}]\label{thm:dicks-dunwoody-genset}
        Let $X$ be a connected graph and $G$ a group acting on $X$. Then there is a sequence 
        $$
        \E_1 \subset \E_2 \subset \ldots
        $$
        of $G$-invariant nested subsets of $\br (X)$ consisting of tight elements, such that $\E_n$ generates $\br_n (X)$ as a Boolean ring. 
    \end{theorem}

    The following characterisation of accessibility is helpful, due to Thomassen--Woess \cite{thomassen1993vertex}. 

    \begin{theorem}[{\cite[Thm.~7.6]{thomassen1993vertex}}]\label{thm:access-char-br}
        Let $X$ be a connected, locally finite graph equipped with a quasi-transitive action by a group $G$.  Then $X$ is accessible if and only if there exists $n \geq 1$ such that $\br (X) = \br_n (X)$. In other words, $X$ is accessible if and only if there is a nested, $G$-invariant, $G$-finite generating set of $\br (X)$. 
    \end{theorem}

    The following proposition was first observed by Dunwoody in \cite{dunwoody1982cutting} in the case of minimal cuts, and subsequently extended by Thomassen--Woess in \cite{thomassen1993vertex} to tight cuts of bounded size.  

    \begin{proposition}[{\cite[Prop.~4.1]{thomassen1993vertex}}]\label{prop:finitelymany-cuts}
        Let $X$ be a connected graph, $e \in E(X)$, and $k \geq 1$. Then there exists only finitely many tight $b \in \br (X)$ such that $\delta b$ contains $e$ and $|\delta b| \leq k$. 
    \end{proposition}

    Note that Proposition~\ref{prop:finitelymany-cuts} immediately implies that, if $X$ is a locally finite, quasi-transitive $G$-graph, then $\E_n \subset \br(X)$ is finite for all $n \geq 1$, where $\E_n$ is as in Theorem~\ref{thm:dicks-dunwoody-genset}. 
    Moreover, the following corollary is immediate.


    \begin{corollary}\label{cor:bounded-diam}
        Let $X$ be a connected, locally finite, quasi-transitive graph. Then for every $n > 0$ there exists $m > 0$ such that for all tight $b \in \br (X)$, if $|\delta b| < n$ then $\diam(\delta b) < m$. 
    \end{corollary}

    The following fact is also useful and worthy of mention. This was first observed by M\"oller \cite{moller1992ends}, and a short proof can be found in \cite{thomassen1993vertex}. 

    \begin{proposition}[{\cite[Prop.~7.1]{thomassen1993vertex}}]\label{prop:sep-ends}
        Let $X$ be a connected, locally finite graph. Let $\E$ be a subset of $\br (X)$, and let $R$ be the subring of $\br (X)$ generated by $\E$. If $\omega_1, \omega_2 \in \Omega (X)$ are separated by some $b \in R$, then there is some $b' \in \E$ which separates them too. 
    \end{proposition}

    We conclude this subsection by recording the following cheap tricks for creating tight cuts. 

    \begin{proposition}\label{prop:tight-trick}
        Let $X$ be a connected graph and let $b_0 \in \br (X)$ such that $X[b_0]$ is connected. Let $U$ be a connected component of $X \setminus X[b_0]$, and let $b_1$ denote the set of vertices of $U$. Then $b_1$ is a tight element of $\br (X)$.
    \end{proposition}

    \begin{proof}
        Clearly $\delta b_1 \subset \delta b_0$, so in particular $\delta b_1$ is finite and thus  $b_1 \in \br (X)$. By construction, $X[b_1]$ is connected. We need only observe that $X[b_1^\ast]$ is connected, but this is clear since $X[b_0]$ is connected and every edge in $\delta b_1$ abuts $b_0$. 
    \end{proof}

    \begin{proposition}\label{prop:tight-trick-paths}
        Let $X$ be a connected graph. Let $b \in \br (X)$, $e_1, e_2 \in \delta b$. Suppose there exists paths $p \subset X[b]$, $q \subset X[b^\ast]$, both connecting an endpoint of $e_1$ to an endpoint of $e_2$. Then there exists a tight element $b' \in \br (\Gamma)$ such that $\delta b' \subset \delta b$ and $e_1, e_2 \in \delta b'$. 
    \end{proposition}

    \begin{proof}
        Let $b_0 \subset b$ be the vertex set of the connected component of $\Gamma[b]$ which contains $p$. Now, let $b'$ be the vertex set of the connected component of $\Gamma[b_0]$ which contains $q$. By Proposition~\ref{prop:tight-trick}, this is a tight element, and certainly $\delta b' \subset \delta b$ and $e_1, e_2 \in \delta b'$. 
    \end{proof}

\subsection{Peripheral systems and elliptic cuts}\label{sec:ellipticcuts}

In this section we introduce some machinery which pertains to  \emph{elliptic cuts} of a graph relative to some \emph{peripheral system}. This machinery is developed in detail in \cite{macmanusrelacc}, and so we defer to this as a reference for the results we need. 
%
%
We begin with the following definition. 

\begin{definition}[Peripheral systems]\label{def:per}
    Let $X$ be a connected graph. Then a collection $\per$ of subsets of $V(X)$ is called a \emph{peripheral system}. We say that $\per$ is \emph{thin} if every $v \in V(X)$ is contained in at most finitely many $H \in \per$. If $X$ is a $G$-graph, then we say that $\per$ is \emph{$G$-invariant} if for all $g \in G$, we have that $gH \in \per$. 

    The \emph{cone-off} of $X$ over $\per$, denoted $\widehat X_\per$, is the graph formed by adding a new vertex $v_H$ for each $H \in \per$ and connecting this with an edge to every $u \in H$. 
\end{definition}

\begin{definition}[Elliptic cuts]\label{def:elliptic}
    Let $X$ be a connected graph, and $\per$ a peripheral system. We say that $b \in \br(X)$ is \emph{$\per$-elliptic} (or just \emph{elliptic}) if for every $H \in \per$, either $b \cap H$ or $b^\ast \cap H$ is finite. The subset of $\br(X)$ consisting of elliptic cuts is denoted by $\br_\per(X)$. 
\end{definition}

It is an exercise in the algebra of sets to prove that $\br_\per(X)$ is a subring of $\br(X)$; see \cite[Prop.~3.3]{macmanusrelacc} for details. If $X$ is a $G$-graph and $\per$ is $G$-invariant, then similarly $\br_\per(X)$ inherits the structure of a $G$-submodule of $\br(X)$.

    \subsection{Tree decompositions and structure trees}\label{sec:trees}

        We now discuss \emph{tree decompositions} of graphs, how to construct them, and their relationship with accessibility. The results in this section are well known, so we will aim for brevity.

        \begin{definition}[Separation]
            Let $X$ be a connected graph. 
            A \textit{separation} of $X$ is a triple $(Y, S, Z)$ where $Y, S, Z \subset V(X)$ are pairwise disjoint subsets such that $V(X) = Y \cup S \cup Z$, and there is no edge with one endpoint in $Y$ and the other in $Z$. 
        \end{definition}
        

        \begin{definition}[Tree decomposition]
            Let $X$ be a connected graph. A \textit{tree decomposition} of $X$ is a pair $(T, \cV)$, where $T$ is a connected simplicial tree and $\cV = (V_v)_{v\in V(T)}$ is a collection of subsets $V_v \subset V(X)$ indexed by $V(T)$, called \textit{bags}, such that the following hold:
        \begin{enumerate}
            \item $V(X) = \bigcup_{v \in V(T)} V_v$,

            \item For all $uw \in E(X)$, there exists $v \in V(T)$ such that $u,w \in V_v$, 

            \item For all $u \in V(X)$, the subgraph of $T$ induced by the set 
            $$
            \{v \in V(T) : u \in V_v\}
            $$
            is connected.
        \end{enumerate}
        Given an oriented edge $e = (u,w) \in \vec E(T)$, let $T_u$, $T_w$ denote the connected components of $T-\{uw\}$ which contain $u$ and $w$, respectively. Then the \textit{edge-separation} corresponding to this edge is the separation $(Y_e, S_e, Z_e)$ where $S_e = V_u \cap V_w$, $Y_e = \bigcup_{s \in V(T_u)} V_s \setminus S_e$, and  $Z_e = \bigcup_{s \in V(T_w)} V_s \setminus S_e$. The sets $S_e$, $e \in \vec E(T)$ are called the \textit{adhesion sets} of the tree decomposition. We say that this tree decomposition has \textit{bounded adhesion} if there exists $N \geq 1$ such that every adhesion set contains at most $N$ vertices. 

        The subgraphs $X[V_v]$ induced by the bags are called the \emph{parts} of the tree decomposition. We say that the tree decomposition is \emph{connected} if all parts are connected.

        If $X$ is a $G$-graph, then we call this tree decomposition \textit{$G$-canonical} if $G$ acts on $T$ by isometries such that for all $g \in G$, $u \in V(T)$ we have that $gV_u =V_{gu}$. 
        \end{definition}


    If the group $G$ acting on $X$ is clear from context, then we may suppress mention of $G$ from our terminology and simply call a tree decomposition \textit{canonical}. 
%
%
   %
%
    The following is an easy exercise, but worth noting.  

    \begin{proposition}\label{prop:replace-parts-with-nbhd}
        Let $X$ be a connected, locally finite, quasi-transitive graph. Let $(T, \cV)$ be a canonical tree decomposition, where $\cV = (V_u)_{u\in V(T)}$ and let $r \geq 0$. For each $u \in V(T)$, let $V'_u$ denote the set of vertices which lie in the closed $r$-neighbourhood of $V_u$, and let $\cV' = (V'_u)_{u\in V(T)}$. Then $(T, \cV')$ is a canonical tree decomposition. Moreover, if $(T, \cV)$ has bounded adhesion and/or is connected, then the same holds for $(T, \cV')$. 
    \end{proposition}

    It is helpful to view tree decompositions as essentially the graph-theoretical analogue of Bass--Serre trees found in geometric group theory. Indeed, if $G \actson X$ then canonical tree decompositions of $X$ induce splittings of $G$ as a graph of groups in the usual way. 
    
    We have the following standard proposition. 

    \begin{proposition}\label{prop:tree-decomp-acc}
        Let $X$ be a connected, locally finite, quasi-transitive $G$-graph. Let $(T, \cV)$ be a $G$-canonical, connected tree decomposition with bounded adhesion and $T/G$ compact, where $\cV = (V_u)_ {u \in V(T)}$. Then the following hold:
        \begin{enumerate}
            \item\label{itm:tree-1} For each $u \in V(T)$, the part $X[V_u]$ is quasi-transitively stabilised, in the sense of Definition~\ref{def:cocompact}.


            \item\label{itm:tree-4}  If each part $X[V_u]$ is accessible then $X$ is accessible.
        \end{enumerate}
    \end{proposition}

    \begin{proof}
        Property~(\ref{itm:tree-1}) follows from the fact that $\Stab(u)$ must act with finitely many orbits on the set of edges going into $u$, since $T/G$ is compact. See \cite[Prop.~4.5]{hamann2022stallings} for details.\footnote{The cited result in \cite{hamann2022stallings} is stated for tree decompositions $(T,\cV)$ where $G$ acts on $T$ with a single orbit of edges. An almost identical argument goes through assuming just finitely many orbits of edges.}


        To see (\ref{itm:tree-4}), suppose that every part $X_u := X[V_u]$ is accessible. So each $X_u$ admits a canonical tree decomposition of bounded adhesion $(T^u, \cV^u)$ which effectively distinguishes all of its ends \cite[Thm.~6.4]{hamann2022stallings}. Since $T/G$ is compact, we may assume that this bounded on adhesion sets is uniform across all parts. A standard blow-up argument (see e.g. \cite[Prop.~7.2]{carmesin2022canonical}) yields a tree decomposition of $X$ of bounded adhesion which distinguishes all ends of $X$. In particular, $X$ is accessible. 
    \end{proof}

    We now describe how to construct tree decompositions from nested subsets of $\br(X)$. Before we can do this, we must first construct the tree itself. For this, we briefly recall some terminology.
    A \emph{pocset} is a set $\E$ equipped with a partial ordering $\leq$ and an order reversing involution $s \mapsto s^\ast$, such that  $a \not\leq a^\ast$ for all $a \in \E$. 
    We say that $S$ is \emph{nested} if for all $a, b \in \E$, one of
    $$
    a \leq b, \ \ a \leq b^\ast, \ \ a^\ast \leq b, \ \ a^\ast \leq b^\ast 
    $$
    holds. Finally, $\E$ is \emph{discrete} if for all $a, b \in \E$, the interval
    $$
    \{c \in \E: a \leq c \leq b\}
    $$
    is finite.
    The canonical example of a discrete, nested pocset is the set of oriented edges $\vec E(T)$ of a tree $T$, where $e \leq f$ if and only if $e$ and $f$ lie on some common geodesic, oriented to point in the same direction along this path, with $f$ lying `in front of' $e$.
    It is a classical fact, due to Dunwoody, that this is essentially the only example.

    \begin{theorem}[{\cite[Thm.~2.1]{dunwoody1979accessibility}}]
        Let $\E$ be a discrete, nested pocset. Then there is a tree $T = T(\E)$ such that $\E$ is canonically isomorphic to $\vec E(T)$ as a pocset. 
    \end{theorem}

    It is worth stating where the vertices of $T$ come from. Define a relation $\sim$ on $\E$ as follows. Given $a, b \in \E$, say that $a \ll b$ if $a \leq b$ and $a \leq c \leq b$ implies $a = c$ or $b = c$. We then define
    $$
    \text{$a \sim b$ if $a = b$ or $a \ll b^\ast$. }
    $$
    It is an exercise to check that $\sim$ is an equivalence relation on $\E$. The vertex set of $T$ can then be taken to be the set of $\sim$-equivalence classes. Those directed edges of $T$ which point `into' a vertex $v$ are then precisely those $A \in \E$ such that $A \in v$. We call $T$ the \textit{structure tree} of $\mathcal E$. For more details, see \cite{dunwoody1979accessibility}.

    If $X$ is a connected, locally finite quasi-transitive $G$-graph and $\E \subset \br(X)$ is a symmetric, nested, $G$-invariant, $G$-finite subset, then clearly $\E$ has the structure of a discrete, nested pocset. This allows us to build a tree decomposition which encodes the separation properties of $\E$. 
    The following construction essentially appears in \cite[\S7]{thomassen1993vertex}.

    \begin{theorem}\label{thm:vertex-subgraph}
        Let $X$ be a connected, locally finite, quasi-transitive $G$-graph. Let $\E \subset \br(X)$ be a nested, symmetric, $G$-invariant, $G$-finite subset. Let $T = T(\E)$ denote the structure tree of $\E$. Then there exists a $G$-canonical tree decomposition $(T, \cV)$ of $X$ with the following properties:
        \begin{enumerate}
            \item $(T, \cV)$ has bounded adhesion and connected parts, and each part is a quasi-transitively stabilised subgraph. 

            \item There exists $C > 0$ such that for every $b \in \E$, we have that 
            $$
            \dHaus[X](b, Y_b) \leq C,
            $$
            where $(Y_b, S_b, Z_b)$ is the edge-separation induced by $b \in \E \equiv \vec E(T)$. 
        \end{enumerate}
        Furthermore, if $\E$ is taken to be $\E_n$ as in Theorem~\ref{thm:dicks-dunwoody-genset}, then we also have that $\es(X_v) \geq n$. 
    \end{theorem}

    \begin{proof}
        Fix $v \in V(T)$. 
        Given $b \in v$ and $m \geq 1$, let $R(m, b)$ be the subgraph of $X$ induced by the set of vertices which lie a distance of at most $m$ from $b^\ast$. We choose $m$ sufficiently large so that $R(m, b)$ satisfies the following:
        \begin{enumerate}
            \item $R(m,b)$ contains all geodesics in $X[b]$ containing endpoints of edges in $\delta b$, 

            \item\label{itm:disjoint-paths} If there exists $n$ pairwise edge-disjoint paths $p_1, \ldots , p_n$ in $X[b]$ with endpoints in $\delta b$, then $n$ pairwise edge-disjoint paths $p_1', \ldots , p_n'$ in $R(m, b)$ such that $p_i'$ has the same endpoints as $p_i$. 
        \end{enumerate}
        Such an $m$ clearly exists since $G$ acts on $\E$ with finitely many orbits. 
        We then define 
        $$
        X_v = \bigg(\bigcap_{b \in v} X[b^\ast]\bigg) \cup \bigg(\bigcup_{b \in v} (\delta b \cup R(m, b))\bigg). 
        $$
        We assume that $X_v$ is an induced subgraph. If not, then add back the missing edges. 

        \begin{claim}
            $X_v$ is connected.
        \end{claim}

        \begin{proof}
            Let $u, w \in X_v$, we will show that there is a path in $X_v$ connecting $u$ and $w$. It is sufficient to consider the case where $u, w \in \bigcap_{b \in v} X[b^\ast]$, as clearly every other vertex in $X_v$ has a path to this subset. Let $P \subset X$ be a geodesic in $X$ connecting $u, w$. By construction, the pieces of $P$ which are not contained in $\bigcap_{b \in v} X[b^\ast]$ are actually contained in $R(m,b)$ for some $b \in v$. The claim follows. 
        \end{proof}

        Given $b \in \E$, it is clear by construction that $Y_b$ is contained in some finite neighbourhood of $b$, and vice versa. Since $\E$ is $G$-finite, it follows that the Hausdorff distance between $b$ and $Y_b$ is uniformly bounded.
        
        If $\E = \E_n$, then the fact that $\es(X_v) \geq n$ essentially follows directly from property (\ref{itm:disjoint-paths}) above. For full details, see \cite[Lem.~8.1]{thomassen1993vertex}. 
        To complete the construction, simply let $\cV = (V(X_v))_ {v \in V(T)}$. It is easy to verify that $(T, \cV)$ satisfies our requirements.
    \end{proof}

\section{Cuts and quasi-isometries}\label{sec:cuts-qi}

In this section we record some technical results relating to how cuts in graphs interact with quasi-isometries. 

\subsection{Pushing cuts through quasi-isometries}

First, we record the following easy facts, which describes how cuts are translated by quasi-isometries. The first deals with vertex-cuts, and the second with edge cuts.


\begin{lemma}\label{lem:cut-through-qi}
    For all $\lambda \geq 1$, there exists $R > 0$ such that the following holds:
    
    Let $\psi : X \to Y$ be a $(\lambda, \lambda)$-quasi-isometry between connected, locally finite graphs. 
    Then for all $b \in \br (X)$,  there exists $b' \in \br (Y)$ such that 
    $$
    \dHaus[Y](\psi(b), b') < R,
    $$
    and $\delta b'$ is contained in the $R$-neighbourhood of $\psi(\delta b)$. Moreover if $X[b]$ is connected then we may take $b'$ so that $Y[b']$ is connected. 
\end{lemma}

\begin{proof}
    Let $\varphi : Y \to X$ be some choice of quasi-inverse to $\psi$. 
    Fix $\lambda \geq 1$ such that $\psi$ and $\varphi$ are $(\lambda, \lambda)$-quasi-isometries, and $\varphi$ is a $\lambda$-quasi-inverse to $\psi$.

    Let $b' = B_{Y}(\psi(b);R) \cap V(Y)$ for some large $R > 0$, say $ R = 100\lambda^5$. Clearly if $X[b]$ is connected then so is $Y[b']$. 
    Suppose $u$ and $v$ are adjacent vertices in $Y$ such that $u \in b'$ and $v \not\in b'$. 
    Since $R$ is sufficiently large compared to the quasi-isometry constants, we see that $\varphi(v) \not\in b$, but certainly both $\varphi(v)$ and $\varphi(u)$ lie in a bounded neighbourhood of $b$.  Thus, $u$ and $v$ lie in a bounded neighbourhood of $\psi(\delta b)$. Since $u$ and $v$ are arbitrary and $Y$ is locally finite, we have that $\delta b'$ is finite and so $b' \in \br (Y)$.
\end{proof}

\subsection{Cuts and quasi-actions}

We now apply the above to prove a lemma  about graphs equipped with cobounded quasi-actions.




\begin{lemma}\label{lem:bounded-finite-pieces}
    Let $\Gamma$ be a connected, locally finite graph equipped with a cobounded quasi-action by a group $G$. Then there exists $C \geq 1$ such that the following holds. For all \textbf{finite} $b \in \br(X)$, we have that
    $$
    \dist_\Gamma(z, \delta b) < C \diam_\Gamma(\delta b) + C, 
    $$
    for all $z \in b$. 
\end{lemma}

\begin{proof}
    Fix $\lambda > 1$ such that the quasi-action of $G$ on $\Gamma$ is a $\lambda$-quasi-action, and $B$ such that this quasi-action is $B$-cobounded. Fix $R \geq 0$ as in Lemma~\ref{lem:cut-through-qi}. We may assume without loss of generality that $\Gamma$ is infinite, lest the lemma be vacuous. 
    Let $b \in \br(X)$ be finite. Since $b$ is finite, we may choose $z \in b$ as to maximise $D := \dist_\Gamma(z, \delta b)$. Assume without loss of generality that $D$ is much larger than $R$. 
    Choose $g \in G$ such that $\dist_\Gamma(z, \varphi_g(\delta b)) \leq B$. 
    By Lemma~\ref{lem:cut-through-qi}, there exists $b' \in \br(X)$ such that 
    $$
    \dHaus[Y](\varphi_g(b), b') < R,
    $$
    and $\delta b'$ is contained in the $R$-neighbourhood of $\varphi_g(\delta b)$.
    If $D$ is taken to be sufficiently large, then $\delta b' \subset X[b]$ and $y := \varphi_g(z)\in b'$.  In particular, we can deduce that $b' \subset b$, since $b'$ is finite. 
    We now compute 
    \begin{align*}
        D \geq \dist_\Gamma(y, \delta b) &\geq \dist_\Gamma(y, \delta b') + \inf_{x \in \delta b}\dist_\Gamma(x, \delta b') \\
        &\geq \Big(\tfrac 1 \lambda D - \lambda - R\Big) + \Big(D - B - \lambda \diam_\Gamma(\delta b) - \lambda - R\Big).
    \end{align*}
    Combining and simplifying the above, we deduce that 
    $$
    D \leq \lambda^2\diam_\Gamma(\delta b) + \lambda(2\lambda + 2R + B). 
    $$
    Thus, by setting $C = \lambda(2\lambda + 2R + B)$  we are done. 
\end{proof}

We note the following application of Lemma~\ref{lem:bounded-finite-pieces}. Firstly, we state a definition which will be helpful for brevity later on. 

\begin{definition}[Almost 2-connected graph]\label{def:nearly-2-conn}
    Let $\Gamma$ be an infinite, connected, locally finite graph. We say that $\Gamma$ is \textit{almost 2-connected} if there exists a unique maximal 2-connected infinite subgraph $\Gamma_0 \subset \Gamma$ and the inclusion map $\Gamma_0 \into \Gamma$ is a quasi-isometry. The subgraph $\Gamma_0$ is called the \textit{2-connected core} of $\Gamma$. 
\end{definition}

Note that every maximal 2-connected subgraph is necessarily isometrically embedded. This means that checking the quasi-isometry condition in the above definition amounts to just checking that the inclusion is coarsely surjective. Intuitively, an almost 2-connected graph $\Gamma$ is obtained from a 2-connected graph $\Gamma_0$ by `attaching' a selection of boundedly small finite graphs at cut vertices. 

\begin{lemma}\label{lem:2-conn-subgraph}
     Let $\Gamma$ be a infinite, connected, locally finite graph equipped with a cobounded quasi-action. Suppose $\vs(\Gamma) > 1$. Then $\Gamma$ is almost 2-connected. 
\end{lemma}

\begin{proof}
    It is a standard fact from graph theory \cite{harary2022block} that any connected graph can be expressed as a union of maximal 2-connected subgraphs (called \textit{blocks}) which each intersect in at most one vertex. The dual graph to this decomposition, where the vertices are the blocks and edges correspond to non-empty intersection, is a tree. 
    
    Since $\vs(\Gamma) > 1$, no cut vertex in $\Gamma$ separates ends.
    It follows from Lemma~\ref{lem:bounded-finite-pieces} that there exists some uniform constant $C \geq 0$ such that there is exactly one infinite block $\Lambda$ in $\Gamma$, and every other block has diameter at most $C$. The inclusion map $\Lambda \into \Gamma$ is easily seen to be an isometric embedding and coarsely surjective, and thus a quasi-isometry. 
\end{proof}

\subsection{Images of highly-connected subgraphs}\label{sec:cut-qi-reprise}

We now quickly record the following pigeonhole argument, which will allow us to push the increased cut-size created in Theorem~\ref{thm:vertex-subgraph} through a quasi-isometry.

\begin{lemma}\label{lem:qi-increase-cuts}
    Let $X$, $Y$ be bounded-degree  connected graphs, and let $\varphi : X \to Y$ be a quasi-isometry. Let $m > 0$ and fix a subgraph $\Lambda \subset X$ such that $\vs(\Lambda) \geq m$. Then 
    $$
    \vs(\varphi(\Lambda)) \geq C m,
    $$
    where $C = C(\varphi) > 0$ is some constant depending only on $\varphi$ and $X$. 
\end{lemma}

\begin{proof}
    We assume without loss of generality that $\varphi$ is continuous by Proposition~\ref{prop:qi-wlog}. 
    
    Fix $\lambda \geq 1$ such that $\varphi$ is a $(\lambda, \lambda)$-quasi-isometry. Write $\varphi(\Lambda) =: \Pi$. 
    Suppose that there are $k$ vertices $v_1, \ldots , v_k$ in $\Pi$ whose removal separates distinct ends $\omega_1, \omega_2 \in \Omega (\Pi)$. Let $\xi_i = \varphi^{-1}(\omega_i)$ for each $i = 1,2$, recalling that quasi-isometries induce well-defined bijections on the corresponding sets of ends. Since $\vs(\Lambda) \geq m$ and the $\Gamma_i$ are locally finite, we have by Menger's theorem (\ref{thm:menger}) that there exists $m$ pairwise disjoint bi-infinite paths $\alpha_1, \ldots, \alpha_m$ in $\Lambda$ between $\xi_1$ and $\xi_2$. 
    
    Thus each $\alpha_i' := \varphi (\alpha_i)$ is a bi-infinite path between the ends $\omega_1$ and $\omega_2$. Each $\alpha_i'$ must pass through some $v_j$, so by the Pigeonhole Principle there exists some $v_j$ such that at least $m/k$ of the $\alpha_i'$ pass through $v_j$. By relabelling, we can assume without loss of generality that $\alpha_1', \ldots, \alpha_{m/k}'$ pass through $v_1$. 
    Let $x \in \varphi^{-1}(v_1)$. As $\varphi$ is a quasi-isometry, we have that $\alpha_1, \ldots \alpha_{m/k}$ must intersect the closed $r$-neighbourhood of $x$, for some $r = r(\varphi) > 0$ depending only on $\varphi$. Combining this observation with the assumption that $X$ is bounded-degree  together with the fact that the $\alpha_i$ are disjoint, we deduce that
    $$
    \frac m k \leq |B_{X}(x;r)|.
    $$
    The right-hand side is bounded above by some uniform constant since $X$ is bounded-degree. The lemma follows. 
\end{proof}


\subsection{Bounding tight cuts}

Finally, we have the following lemma, which is a coarse version of Corollary~\ref{cor:bounded-diam}. 

    \begin{lemma}\label{lem:bounded-diam-coarse}
        Let $X$, $\Gamma$ be bounded-degree, connected graphs. Let $X$ be quasi-transitive, and suppose $X$ and $\Gamma$ are quasi-isometric. Then for every $n > 0$ there exists $m > 0$ such that for all tight $b \in \br (\Gamma)$, if $|\delta b| < n$ then $\diam(\delta b) < m$. 
    \end{lemma}

    \begin{proof}
        Let $\varphi : X \to \Gamma$ be a quasi-isometry with quasi-inverse $\psi : \Gamma \to X$. We assume without loss of generality that these are continuous maps. As usual, fix $\lambda \geq 1$ which is larger than all quasi-isometry constants involved. 
        Fix $n > 0$, and let $(b_i)_{i \geq 1}$ be a sequence of tight cuts in $\br(\Gamma)$ such that $|\delta b_i| < n$ for all $i$, and $\diam(\delta b_i) \to \infty$. We will find boundedly small tight cuts of arbitrarily large diameter in $X$, contradicting Corollary~\ref{cor:bounded-diam}. 

        Since each $\delta b_i$ contains at most $n$ edges, it is clear that we can choose decompositions
        $$
        \delta b_i = C_i \sqcup D_i,
        $$
        for each $i$ 
        such that the infimal distance $\dist_\Gamma(C_i, D_i) \to \infty$ as $i \to \infty$.
        Let $W_{i} = \Gamma[b^\ast_i]$, which is connected since $b_i$ is tight. 
        Let $q_i$ be a path through $W_i$ connecting $C_i$ to $D_i$.

        For each $i \geq 1$, using Lemma~\ref{lem:cut-through-qi}, choose $b_i' \in \br (X)$ such that the Hausdorff distance  $\dHaus(\psi(b_i), b_i')$ is uniformly bounded above, and $\delta b_i'$ is contained in a bounded neighbourhood of $\psi(\delta b_i)$. Since each $b_i$ was tight, we may assume that $X[b_i']$ is connected for every $i \geq 1$. As $X$ is bounded-degree, we deduce that there exists some uniform $N > 0$ such that each $b_i'$ satisfies $|\delta b_i'| < N$. 
        We are almost done, but each $b_i'$ may not be tight. We will now find some tight $b_i'' \in \br (X)$ such that $\delta b_i'' \subset \delta b_i'$, and $\diam(\delta b_i'') \to \infty$ as $i \to \infty$. Once this is achieved, we are done.

        Let $r > 0$ be such that  for every $i \geq 1$, $\delta b_i'$ is contained in the $r$-neighbourhood of $\psi(\delta b_i)$. Assume without loss of generality that for all $i \geq 1$, we have that 
        $$
        \dist_\Gamma(C_i, D_i) > 3\lambda (r + \lambda).
        $$
        Let $R = \lambda (r + \lambda) + 1$.
        Then, there must exist a subpath $q'_i$ of $q_i$ such that $q_i'$ lies outside of the $R$-neighbourhood of $\delta b_i$, but $q_i'$ begins in the $2R$-neighbourhood of $C_i$, and ends in the $2R$-neighbourhood of $D_i$. 
        Let $q_i'' = \psi(q_i')$. By our choice of $R$, $q_i''$ is disjoint from the $r$-neighbourhood of $\psi(\delta b_i)$, and thus disjoint from $\delta b_i'$. It follows that $q_i''$ is contained a single connected component of $X \setminus \delta b_i'$.
        Note that the endpoints of $q_i''$ are uniformly close to $\delta b_i'$ by construction, but they become arbitrarily far part as $i \to \infty$. 

        Let $W_i'$ denote the connected component of $X \setminus X[b_i']$ containing $q_i''$, and let $b_i''$ be the vertex set of $W_i'$. By Proposition~\ref{prop:tight-trick}, we have that $b_i''$ is tight. Clearly $\delta b_i'' \subset \delta b_i'$, and $\diam (\delta b_i'') \to \infty$ as $i \to \infty$. The lemma follows.
    \end{proof}


\section{The one-ended case}\label{sec:one-end}

In this section we prove Theorem~\ref{thm:one-ended-intro} from the introduction. 

\subsection{Basic set-up}

Throughout this section, let $X$ be a connected, locally finite, one-ended graph equipped with a quasi-transitive group action $G \actson X$. Let $\Gamma$ be a connected planar graph, and let $\varphi: X \to \Gamma$ be a quasi-isometry with quasi-inverse $\psi: \Gamma \to X$. 
We will assume that $\Gamma$ is bounded-degree, and that both $\varphi$ and $\psi$ are continuous. Propositions~\ref{prop:qi-wlog},~\ref{prop:cts-inverse} demonstrate that these assumptions are completely inconsequential. We may also assume that $\Gamma$ is 2-connected, by an application of Lemma~\ref{lem:2-conn-subgraph}. Thus, every $f \in \facepaths (\Gamma)$ is either a simple cycle or a disjoint union of bi-infinite rays.  In fact, since $\Gamma$ is one-ended, it is easy to see that every $f \in \inffaces(\Gamma)$ must be a single bi-infinite line by Proposition~\ref{prop:simple-face}.

By Corollary~\ref{cor:vap-free}, $\Gamma$ admits a VAP-free embedding $\vartheta : \Gamma \into \R^2$. 
Fix $\lambda, B \geq 1$ be such that the induced quasi-action of $G$ on $\Gamma$ is a $B$-cobounded $\lambda$-quasi-action. 

\subsection{Diverging rays}

We have the following standard construction. 

\begin{lemma}\label{lem:geodesic-rays}
    There exists three geodesic rays $\gamma_1$, $\gamma_2$, $\gamma_3$ based at a common vertex such that
    $$
    \dist_X(\gamma_i(n), \gamma_j(m)) \to \infty,
    $$
    as $n, m \to \infty$, for any distinct $i, j = 1,2,3$. 
\end{lemma}

\begin{proof}
    This is a standard application of the Arzel\'a--Ascoli theorem. First, note that $X$ contains arbitrarily long geodesic segments, since it is locally finite and infinite. We may apply quasi-transitivity to translate their midpoints to a common vertex, after possibly passing to a subsequence. Taking a limit, we find a bi-infinite geodesic $\gamma$. Since $X$ is not two-ended, there exist points in $X$ which are arbitrarily far from $\gamma$, say $(v_n) \subset V(X)$ such that $\dist_X(v_n, \gamma) > n$. For each $n > 0$, let $p_n$ be a shortest path from $v_n$ to $\gamma$. Let $x_n$ denote the point where $p_n$ meets $\gamma$. Applying quasi-transitivity again, we translate these figures so that all the $x_n$ lie on the same point. Applying the  Arzel\'a--Ascoli theorem again and taking another limit, the result is three geodesic rays which pairwise diverge from each other. 
\end{proof}

We now apply the above and push the resulting rays through the quasi-isometry $\varphi$, obtaining a similar feature in $\Gamma$. 

\begin{lemma}\label{lem:disjoint-paths-one-end}
    There exists three distinct $(\lambda, \lambda)$-quasi-geodesic rays $\alpha_1, \alpha_2, \alpha_3 : [0,\infty) \to \Gamma$ based at a common vertex, such that
    $$
    \dist_X(\alpha_i(n), \alpha_j(m)) \to \infty,
    $$
    as $n, m \to \infty$, for any distinct $i, j = 1,2,3$. 
\end{lemma}

\begin{proof}
    This follows from Lemma~\ref{lem:geodesic-rays} by letting $\alpha_i = \varphi \circ \gamma_i$, where the $\gamma_i$ are the rays constructed in the aforementioned lemma. 
\end{proof}

\subsection{Ruling out pathologies}

From now on, the name of the game is to show that $\Gamma$ cannot be `too wild'. In practice, we want to control what the faces of $\Gamma$ look like. 

\begin{lemma}
    $\Gamma$ contains at most one bi-infinite face path. 
\end{lemma}

\begin{proof}
    Suppose there exists distinct  infinite face paths $ f _1, f _2 \subset \Gamma$. Pick $x_i \in  f _i$ and let $\ell$ be a geodesic between $x_1$ and $x_2$. By an application of the Jordan curve theorem, it is easy to see that $\ell$ is a compact subset which separates infinite sets of vertices. This contradicts the fact that $\Gamma$ is one-ended.
\end{proof}

\begin{lemma}\label{lem:far-from-face}
    If $\Gamma$ contains a bi-infinite face path $f$, then $\dHaus[\Gamma](f, \Gamma) = \infty$. 
\end{lemma}

\begin{proof}
    Assume there is an infinite face $f$. 
    Apply Lemma~\ref{lem:disjoint-paths-one-end} to obtain three quasi-geodesic rays $\alpha_1, \alpha_2, \alpha_3 : [0,\infty) \to \Gamma$ based at a common vertex which pairwise diverge. Two of these rays, say $\alpha_1$, $\alpha_2$, will together form a Jordan curve (if we include the point at infinity) which separates some tail of $\alpha_3$ from the infinite face $ f $. Any path from this tail of $\alpha_3$ to $ f $ must cross through $\alpha_1$ and $\alpha_2$. Since $\alpha_3$ diverges from these two rays, it also diverges from $f$. See Figure~\ref{fig:diverge-from-face} for a cartoon. 
\end{proof}

\begin{figure}
        \centering

        \tikzset{every picture/.style={line width=0.75pt}} 

\begin{tikzpicture}[x=0.75pt,y=0.75pt,yscale=-1,xscale=1]

\draw [color={rgb, 255:red, 155; green, 155; blue, 155 }  ,draw opacity=1 ]   (166.44,146.52) .. controls (185,165.5) and (194.81,155.93) .. (207,153) .. controls (219.5,150) and (239.5,124) .. (247.5,95) .. controls (255.5,66) and (228.5,66) .. (217,65) .. controls (205.5,64) and (152.5,69) .. (145.5,90) .. controls (139.06,109.32) and (148.28,123.98) .. (152.52,130.88) ;
\draw [shift={(153.5,132.5)}, rotate = 239.74] [color={rgb, 255:red, 155; green, 155; blue, 155 }  ,draw opacity=1 ][line width=0.75]    (10.93,-3.29) .. controls (6.95,-1.4) and (3.31,-0.3) .. (0,0) .. controls (3.31,0.3) and (6.95,1.4) .. (10.93,3.29)   ;
\draw [shift={(165,145)}, rotate = 47.12] [color={rgb, 255:red, 155; green, 155; blue, 155 }  ,draw opacity=1 ][line width=0.75]    (10.93,-3.29) .. controls (6.95,-1.4) and (3.31,-0.3) .. (0,0) .. controls (3.31,0.3) and (6.95,1.4) .. (10.93,3.29)   ;
\draw [color={rgb, 255:red, 74; green, 144; blue, 226 }  ,draw opacity=1 ]   (161.12,156.42) .. controls (184.14,177.32) and (301.12,160.64) .. (304.5,123) .. controls (308,84) and (322.5,52.5) .. (218,49) .. controls (116.63,45.61) and (115.01,119.83) .. (139.18,134.79) ;
\draw [shift={(141.5,136)}, rotate = 203.46] [fill={rgb, 255:red, 74; green, 144; blue, 226 }  ,fill opacity=1 ][line width=0.08]  [draw opacity=0] (7.14,-3.43) -- (0,0) -- (7.14,3.43) -- cycle    ;
\draw [shift={(159,154)}, rotate = 56.04] [fill={rgb, 255:red, 74; green, 144; blue, 226 }  ,fill opacity=1 ][line width=0.08]  [draw opacity=0] (7.14,-3.43) -- (0,0) -- (7.14,3.43) -- cycle    ;
\draw [color={rgb, 255:red, 74; green, 144; blue, 226 }  ,draw opacity=1 ]   (304.5,123) .. controls (345.5,103) and (355,77.5) .. (360.5,57.5) .. controls (365.92,37.8) and (318.46,17.61) .. (267.82,16.54) ;
\draw [shift={(265.5,16.5)}, rotate = 0.56] [fill={rgb, 255:red, 74; green, 144; blue, 226 }  ,fill opacity=1 ][line width=0.08]  [draw opacity=0] (7.14,-3.43) -- (0,0) -- (7.14,3.43) -- cycle    ;
\draw [shift={(304.5,123)}, rotate = 334] [color={rgb, 255:red, 74; green, 144; blue, 226 }  ,draw opacity=1 ][fill={rgb, 255:red, 74; green, 144; blue, 226 }  ,fill opacity=1 ][line width=0.75]      (0, 0) circle [x radius= 2.68, y radius= 2.68]   ;
\draw    (304.5,123) ;
\draw [shift={(304.5,123)}, rotate = 0] [color={rgb, 255:red, 0; green, 0; blue, 0 }  ][fill={rgb, 255:red, 0; green, 0; blue, 0 }  ][line width=0.75]      (0, 0) circle [x radius= 3.35, y radius= 3.35]   ;
\draw [color={rgb, 255:red, 155; green, 155; blue, 155 }  ,draw opacity=0.49 ]   (152,107.5) -- (182.5,77) ;
\draw [color={rgb, 255:red, 155; green, 155; blue, 155 }  ,draw opacity=0.49 ]   (159,119.5) -- (203,75.5) ;
\draw [color={rgb, 255:red, 155; green, 155; blue, 155 }  ,draw opacity=0.49 ]   (169.5,130) -- (220.5,79) ;
\draw [color={rgb, 255:red, 155; green, 155; blue, 155 }  ,draw opacity=0.49 ]   (178,140.5) -- (224.25,94.25) -- (233.5,85) ;
\draw [color={rgb, 255:red, 155; green, 155; blue, 155 }  ,draw opacity=0.49 ]   (185,150.5) -- (231.25,104.25) -- (236.5,99) ;

\draw (252,86.9) node [anchor=north west][inner sep=0.75pt]  [color={rgb, 255:red, 155; green, 155; blue, 155 }  ,opacity=1 ]  {$f $};
\draw (152.5,133.9) node [anchor=north west][inner sep=0.75pt]  [color={rgb, 255:red, 128; green, 128; blue, 128 }  ,opacity=1 ]  {$\infty $};
\draw (283,112.4) node [anchor=north west][inner sep=0.75pt]    {$x$};
\draw (120,47.4) node [anchor=north west][inner sep=0.75pt]  [color={rgb, 255:red, 74; green, 144; blue, 226 }  ,opacity=1 ]  {$\alpha _{1}$};
\draw (259,172.4) node [anchor=north west][inner sep=0.75pt]  [color={rgb, 255:red, 74; green, 144; blue, 226 }  ,opacity=1 ]  {$\alpha _{3}$};
\draw (348.5,91.4) node [anchor=north west][inner sep=0.75pt]  [color={rgb, 255:red, 74; green, 144; blue, 226 }  ,opacity=1 ]  {$\alpha _{2}$};
\draw  [color={rgb, 255:red, 255; green, 255; blue, 255 }  ,draw opacity=0 ][fill={rgb, 255:red, 255; green, 255; blue, 255 }  ,fill opacity=0.6 ]  (172.5,94) -- (212.5,94) -- (212.5,118) -- (172.5,118) -- cycle  ;
\draw (175.5,98) node [anchor=north west][inner sep=0.75pt]  [color={rgb, 255:red, 155; green, 155; blue, 155 }  ,opacity=0.5 ] [align=left] {Face};

\end{tikzpicture}

        \caption{The ray $\alpha_2$ is forced to diverge from $ f $ due to the Jordan curve theorem. }
        \label{fig:diverge-from-face}
    \end{figure}
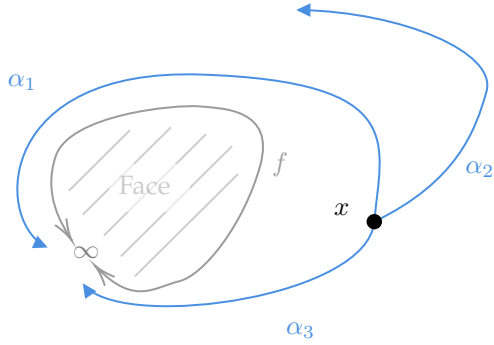

\begin{lemma}\label{lem:uniform-faces}
    There exists some uniform constant $r \geq 0$ such that every finite face cycle of $\Gamma$ has length at most $n$.  
\end{lemma}

\begin{proof}
    Apply Lemma~\ref{lem:disjoint-paths-one-end} to obtain three quasi-geodesic rays $\alpha_1, \alpha_2, \alpha_3 : [0,\infty) \to \Gamma$ based at a common vertex $x \in V(\Gamma)$ which pairwise diverge. Fix $r > 0$ such that for all $i\neq j$ we have that
    $$
    \dist_\Gamma\Big( \alpha_i \setminus B_\Gamma(x;r), \ \alpha_j \setminus B_\Gamma(x;r)   \Big) > \lambda^2
    $$
    Since we have a cobounded quasi-action, it is not hard to see that we may choose $x$ freely, up to slightly inflating the quasi-geodesic constants and the constant $r$ by some controlled amount. 
    
    If $\Gamma$ has no infinite face then choose $x \in V(\Gamma)$ arbitrarily. Otherwise, apply Lemma~\ref{lem:far-from-face} and choose $x \in V(\Gamma)$ such that $x$ lies at least $M = 1000\lambda^{1000}Br$ from this infinite face. 
    
    Since there is no infinite face near $x$, we may connect $\alpha_i$ to $\alpha_{i+1}$ by a path $\beta_i$ which stays inside the region bounded by $\alpha_i \cup \alpha_{i+1}$ not containing $\alpha_{i+2}$ (where indices are taken modulo 3). 
    We can also assume that each $\beta_i$ is disjoint from the $M$-neighbourhood of $x$.  
    The subgraph $\alpha_1 \cup \alpha_2 \cup \alpha_3 \cup \beta_1 \cup \beta_2 \cup \beta_3$ contains a Jordan curve $J$ which separates $x$ from infinity, which is disjoint from the $M$-neighbourhood of $x$. 
    See Figure~\ref{fig:enclosing-circle} for a cartoon. 

    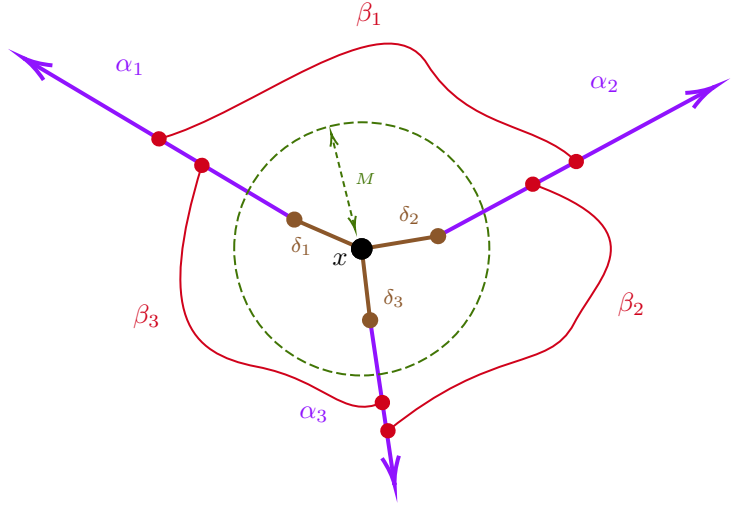
\begin{figure}
        \centering

        \tikzset{every picture/.style={line width=0.75pt}} 

\begin{tikzpicture}[x=0.75pt,y=0.75pt,yscale=-1,xscale=1]

\draw [color={rgb, 255:red, 144; green, 19; blue, 254 }  ,draw opacity=1 ][line width=1.5]    (252.15,118.02) .. controls (254.19,119.25) and (256.22,132.6) .. (269.08,132.75) .. controls (281.94,132.9) and (316.83,94.25) .. (285.8,132.76) ;
\draw [color={rgb, 255:red, 144; green, 19; blue, 254 }  ,draw opacity=1 ][line width=1.5]    (323.9,126.36) .. controls (316.08,131) and (286.47,146.35) .. (299.33,146.5) .. controls (312.19,146.65) and (283.83,107.25) .. (285.8,132.76) ;
\draw [color={rgb, 255:red, 144; green, 19; blue, 254 }  ,draw opacity=1 ][line width=1.5]    (289.97,168.63) .. controls (288.08,152.5) and (247.08,187) .. (258.08,155.25) .. controls (269.08,123.5) and (244.58,118.5) .. (285.8,132.76) ;
\draw [color={rgb, 255:red, 144; green, 19; blue, 254 }  ,draw opacity=1 ][line width=1.5]    (119.57,38.36) -- (252.15,118.02) ;
\draw [shift={(117,36.81)}, rotate = 31] [color={rgb, 255:red, 144; green, 19; blue, 254 }  ,draw opacity=1 ][line width=1.5]    (14.21,-4.28) .. controls (9.04,-1.82) and (4.3,-0.39) .. (0,0) .. controls (4.3,0.39) and (9.04,1.82) .. (14.21,4.28)   ;
\draw [color={rgb, 255:red, 144; green, 19; blue, 254 }  ,draw opacity=1 ][line width=1.5]    (459.2,52.71) -- (323.9,126.36) ;
\draw [shift={(461.83,51.28)}, rotate = 151.44] [color={rgb, 255:red, 144; green, 19; blue, 254 }  ,draw opacity=1 ][line width=1.5]    (14.21,-4.28) .. controls (9.04,-1.82) and (4.3,-0.39) .. (0,0) .. controls (4.3,0.39) and (9.04,1.82) .. (14.21,4.28)   ;
\draw [color={rgb, 255:red, 144; green, 19; blue, 254 }  ,draw opacity=1 ][line width=1.5]    (301.77,248.53) -- (289.97,168.63) ;
\draw [shift={(302.21,251.5)}, rotate = 261.6] [color={rgb, 255:red, 144; green, 19; blue, 254 }  ,draw opacity=1 ][line width=1.5]    (14.21,-4.28) .. controls (9.04,-1.82) and (4.3,-0.39) .. (0,0) .. controls (4.3,0.39) and (9.04,1.82) .. (14.21,4.28)   ;
\draw  [color={rgb, 255:red, 65; green, 117; blue, 5 }  ,draw opacity=1 ][dash pattern={on 3.75pt off 1.5pt}] (222.12,132.76) .. controls (222.12,97.58) and (250.63,69.07) .. (285.8,69.07) .. controls (320.97,69.07) and (349.48,97.58) .. (349.48,132.76) .. controls (349.48,167.93) and (320.97,196.44) .. (285.8,196.44) .. controls (250.63,196.44) and (222.12,167.93) .. (222.12,132.76) -- cycle ;
\draw [line width=1.5]    (285.8,132.76) ;
\draw [shift={(285.8,132.76)}, rotate = 0] [color={rgb, 255:red, 0; green, 0; blue, 0 }  ][fill={rgb, 255:red, 0; green, 0; blue, 0 }  ][line width=1.5]      (0, 0) circle [x radius= 4.36, y radius= 4.36]   ;
\draw [shift={(285.8,132.76)}, rotate = 0] [color={rgb, 255:red, 0; green, 0; blue, 0 }  ][fill={rgb, 255:red, 0; green, 0; blue, 0 }  ][line width=1.5]      (0, 0) circle [x radius= 4.36, y radius= 4.36]   ;
\draw [color={rgb, 255:red, 208; green, 2; blue, 27 }  ,draw opacity=1 ]   (184.58,77.42) .. controls (229.9,65.18) and (297.76,6.22) .. (318.34,39.6) .. controls (338.92,72.97) and (372.84,69.07) .. (392.87,88.82) ;
\draw [shift={(392.87,88.82)}, rotate = 44.6] [color={rgb, 255:red, 208; green, 2; blue, 27 }  ,draw opacity=1 ][fill={rgb, 255:red, 208; green, 2; blue, 27 }  ,fill opacity=1 ][line width=0.75]      (0, 0) circle [x radius= 3.35, y radius= 3.35]   ;
\draw [shift={(184.58,77.42)}, rotate = 344.89] [color={rgb, 255:red, 208; green, 2; blue, 27 }  ,draw opacity=1 ][fill={rgb, 255:red, 208; green, 2; blue, 27 }  ,fill opacity=1 ][line width=0.75]      (0, 0) circle [x radius= 3.35, y radius= 3.35]   ;
\draw [color={rgb, 255:red, 208; green, 2; blue, 27 }  ,draw opacity=1 ]   (205.99,90.76) .. controls (183.19,165.85) and (198.76,185.31) .. (233.24,191.99) .. controls (267.72,198.66) and (276.07,218.69) .. (296.09,210.06) ;
\draw [shift={(296.09,210.06)}, rotate = 336.71] [color={rgb, 255:red, 208; green, 2; blue, 27 }  ,draw opacity=1 ][fill={rgb, 255:red, 208; green, 2; blue, 27 }  ,fill opacity=1 ][line width=0.75]      (0, 0) circle [x radius= 3.35, y radius= 3.35]   ;
\draw [shift={(205.99,90.76)}, rotate = 106.89] [color={rgb, 255:red, 208; green, 2; blue, 27 }  ,draw opacity=1 ][fill={rgb, 255:red, 208; green, 2; blue, 27 }  ,fill opacity=1 ][line width=0.75]      (0, 0) circle [x radius= 3.35, y radius= 3.35]   ;
\draw [color={rgb, 255:red, 208; green, 2; blue, 27 }  ,draw opacity=1 ]   (371.17,100.22) .. controls (434.58,123.58) and (404.54,145.83) .. (391.75,170.3) .. controls (378.96,194.77) and (351.71,180.87) .. (298.87,224.25) ;
\draw [shift={(298.87,224.25)}, rotate = 140.61] [color={rgb, 255:red, 208; green, 2; blue, 27 }  ,draw opacity=1 ][fill={rgb, 255:red, 208; green, 2; blue, 27 }  ,fill opacity=1 ][line width=0.75]      (0, 0) circle [x radius= 3.35, y radius= 3.35]   ;
\draw [shift={(371.17,100.22)}, rotate = 20.22] [color={rgb, 255:red, 208; green, 2; blue, 27 }  ,draw opacity=1 ][fill={rgb, 255:red, 208; green, 2; blue, 27 }  ,fill opacity=1 ][line width=0.75]      (0, 0) circle [x radius= 3.35, y radius= 3.35]   ;
\draw [color={rgb, 255:red, 65; green, 117; blue, 5 }  ,draw opacity=1 ] [dash pattern={on 2.25pt off 1.5pt}]  (282.02,121.06) -- (270.98,76.44) ;
\draw [shift={(270.5,74.5)}, rotate = 76.1] [color={rgb, 255:red, 65; green, 117; blue, 5 }  ,draw opacity=1 ][line width=0.75]    (6.56,-1.97) .. controls (4.17,-0.84) and (1.99,-0.18) .. (0,0) .. controls (1.99,0.18) and (4.17,0.84) .. (6.56,1.97)   ;
\draw [shift={(282.5,123)}, rotate = 256.1] [color={rgb, 255:red, 65; green, 117; blue, 5 }  ,draw opacity=1 ][line width=0.75]    (6.56,-1.97) .. controls (4.17,-0.84) and (1.99,-0.18) .. (0,0) .. controls (1.99,0.18) and (4.17,0.84) .. (6.56,1.97)   ;

\draw (161.45,35.94) node [anchor=north west][inner sep=0.75pt]  [color={rgb, 255:red, 144; green, 19; blue, 254 }  ,opacity=1 ]  {$\alpha _{1}$};
\draw (398.38,44.84) node [anchor=north west][inner sep=0.75pt]  [color={rgb, 255:red, 144; green, 19; blue, 254 }  ,opacity=1 ]  {$\alpha _{2}$};
\draw (253.22,210.59) node [anchor=north west][inner sep=0.75pt]  [color={rgb, 255:red, 144; green, 19; blue, 254 }  ,opacity=1 ]  {$\alpha _{3}$};
\draw (269.63,133.58) node [anchor=north west][inner sep=0.75pt]    {$x$};
\draw (170.24,159.42) node [anchor=north west][inner sep=0.75pt]  [color={rgb, 255:red, 208; green, 2; blue, 27 }  ,opacity=1 ]  {$\beta _{3}$};
\draw (281.47,7.58) node [anchor=north west][inner sep=0.75pt]  [color={rgb, 255:red, 208; green, 2; blue, 27 }  ,opacity=1 ]  {$\beta _{1}$};
\draw (412.17,153.3) node [anchor=north west][inner sep=0.75pt]  [color={rgb, 255:red, 208; green, 2; blue, 27 }  ,opacity=1 ]  {$\beta _{2}$};
\draw (281,93.4) node [anchor=north west][inner sep=0.75pt]  [font=\tiny,color={rgb, 255:red, 65; green, 117; blue, 5 }  ,opacity=1 ]  {$M$};

\end{tikzpicture}

        \caption{Our choice of the paths $\beta_i$ traces out a Jordan curve separating $x$ from infinity.}
        \label{fig:enclosing-circle}
    \end{figure}

    Suppose now we translate this figure somewhere using our quasi-action. Fix $g \in G$. To ease notation, let us denote the quasi-isometry $\varphi_g$ with the following shorthand:  
    $$
    \varphi_g(x) = x'.
    $$ 
    We have that the $\alpha_i'$ are disjoint outside of some controlled neighbourhood of $x'$, and also that 
    \begin{equation}\label{eq:beta-disjoint}
       \beta_i' \cap \alpha_{i+2}' = \emptyset 
    \end{equation}
    for each $i = 1,2,3$. 
    Each $\alpha_{i}'$ is still a quasi-geodesic ray based at $x'$, heading towards the lone end of $\Gamma$. If we try to draw the $\beta_i'$, then the equation (\ref{eq:beta-disjoint}) forces us to once again trace a (possibly not simple, but certainly continuous) closed curve separating $x'$ from infinity in the plane. 
    We have that $J'$ is certainly disjoint from the $B$-neighbourhood of $x'$. 

    Let $U$ be the connected component of $\bbS^2 \setminus \vartheta(J')$ which contains $\vartheta(x')$, and let $b = \vartheta^{-1}(\overline U)$.
    Note that $b$ is finite, and $\delta b$ has diameter in $\Gamma$ which depends only on $J$ and the quasi-action, say at most $D \geq 1$. 
    Any face of $\Gamma$ which intersects the $B$-neighbourhood of $x'$ cannot `cross' $J'$, and so is contained in $b$. Applying Lemma~\ref{lem:bounded-finite-pieces}, we see that $b$ can only contain boundedly many vertices. This puts a uniform upper bound on the lengths of the faces which are $B$-close to $x'$.
    Since $g \in G$ was chosen arbitrarily and the quasi-action is $B$-cobounded, the lemma follows. 
\end{proof}

\begin{lemma}\label{lem:noinffaces}
    The planar graph $\Gamma$ contains no infinite facial subgraphs. 
\end{lemma}

\begin{proof}
    The method to prove this is very similar to the proof of Lemma~\ref{lem:uniform-faces}, so we will only sketch the argument. 

    Suppose there were some $ f  \in \inffaces (\Gamma)$. Then by Lemma~\ref{lem:far-from-face} we can choose some basepoint $x \in V(\Gamma)$ which is arbitrarily far from $ f $. We apply the same construction as in the proof of Lemma~\ref{lem:uniform-faces} and obtain a Jordan curve $J$ separating $x$ from infinity which is disjoint from the $M$ neighbourhood of $x$, where $M$ is taken as in the proof of Lemma~\ref{lem:uniform-faces}. 

    As before, we may translate this Jordan curve anywhere using the quasi-action of $G$. The curve
    $\varphi_g(J)$ still separates $y = \varphi_g(x)$ from infinity in the plane, and lies disjoint from the $B$-neighbourhood of $y$. But then the $B$-neighbourhood of $y$ cannot contain an infinite face. Since this quasi-action is $B$-cobounded and $g \in G$ was chosen arbitrarily, we are forced to conclude that there is no infinite face in $\Gamma$. 
\end{proof}

\subsection{Conclusion}

We may now deduce one of our main results. 

\oneend*

\begin{proof}
    By Lemmas~\ref{lem:uniform-faces},~\ref{lem:noinffaces}, we may assume that $\Gamma$ contains no infinite face paths, and that every face of $\Gamma$ has bounded length. 
    It is then clear that if we attach a 2-cell along each face path, then the resulting cell complex $K$ is homeomorphic to $\R^2$. Subdivide the 2-cells of $K$ and obtain a triangulation, so that the inclusion of $\Gamma$ into the 1-skeleton $K^1$ is an isometric embedding. Since the 2-cells of $K$ have boundedly small boundaries, we only need to subdivide each cell into a bounded number of triangles. Thus, the inclusion $\Gamma \into K^1$ is coarsely surjective and thus a quasi-isometry. We now extend the metric on $K^1$ to a metric on $K$, and obtain a piecewise-linear complete plane which is quasi-isometric to $\Gamma$. Standard smoothing arguments imply that $\Gamma$ is thus quasi-isometric to a complete Riemannian plane. 
\end{proof}

This concludes the one-ended case of our results. The rest of this paper will be focused on dealing with infinite-ended graphs, and the issue of accessibility.

\section{The cycle space}\label{sec:cycles}

We now discuss the \emph{cycle space} of a graph. As the title of this paper might suggest, there will be particular attention paid to how the cycle space interacts with accessibility, planar graphs, and quasi-isometries. 

\subsection{Generating the cycle space}

We begin with some basic definitions.

\begin{definition}[Cycle space]
    Let $X$ be a connected graph with edge set $E = E(X)$. The \emph{cycle space} of $X$, denoted $\cC(X)$, is the subgroup of $\bigoplus_E\Z_2$ consisting of all finite Eulerian subgraphs. 
\end{definition}

Morally, this is just the first simplicial homology group of $X$ with $\Z_2$-coefficients. 
We want to allow infinite sums of elements in $\cC(X)$, and will do so by considering its closure in the direct product $\prod_E \Z_2$. Note that this product is a compact topological group, by Tychonoff's theorem.  

\begin{definition}[Topological generation]
    Let $X$ be a connected graph with edge set $E = E(X)$. Let $S \subset \cC(X)$. We say that $S$ \emph{topologically generates} $\cC(X)$ if 
    $$
    \cC(X) \subset \overline {\langle S \rangle},
    $$
    where this closure is taken in $\prod_E \Z_2$. 
\end{definition}

For clarity, if $S \subset \cC(X)$ generates the cycle space as a group (i.e. not just topologically), then we refer to $S$ as an \emph{algebraic} generating set. We will often speak of a set of cycles $S \subset \cC(X)$ having \emph{bounded support}. By this, we mean that there exists $N > 0$ such that every cycle in $S$ is supported on at most $N$ edges. 

\begin{remark}
    Note that, in general, $\cC(X)$ is not a closed subset of $\prod_E \Z_2$. If $X$ is locally finite, then $\overline {\cC(X)}$ can be identified with the \emph{topological cycle space}, studied by Diestel and K\"uhn \cite{diestelkuhn2004infinite}. 
\end{remark}

Topological generation is better-suited for applications to planar graphs, as is illustrated in Figure~\ref{fig:planar-topgen}. In this example, the cycle space of this planar graph is topologically generated by cycles of length 4, but any algebraic generating set must contain cycles of arbitrarily large support.

\begin{figure}
    \centering
    \input{figs/planar-top-gen}
    \caption{The cycle space of this planar graph is topologically generated by cycles of length 4, but any algebraic generating set must contain arbitrarily long cycles.}
    \label{fig:planar-topgen}
\end{figure}




\subsection{Cycles and accessibility}


Bounded generation of the cycle space has implications for accessibility of graphs, via the following theorem of Hamann \cite{hamann2018accessibility}, which is a combinatorial variant of Dunwoody's theorem \cite{dunwoody1985accessibility}.

\begin{theorem}[Hamann]\label{thm:hamann}
    Let $X$ be a connected, locally finite, quasi-transitive graph. Suppose $\cC(X)$ is algebraically generated by cycles of bounded support. Then $X$ is accessible. 
\end{theorem}

\begin{remark}
    It is possible to weaken the hypotheses of Theorem~\ref{thm:hamann} to only require topological generation. See Appendix~A of \cite{macmanusrelacc} for details. 
\end{remark}

We make use of the following theorem, which is a relative variant of Hamann's accessibility result for graphs with finitely generated cycle spaces \cite{hamann2018accessibility}. The proof is an application of Hamann's result, taken together with more tools in \cite{macmanusrelacc}.

\begin{theorem}[{\cite[Thm.~D]{macmanusrelacc}}]\label{thm:relacc-cycles}
    Let $X$ be a connected, locally finite, quasi-transitive $G$-graph, and $\per$ be a thin, $G$-invariant peripheral system. Suppose that $\cC(\widehat X_\per)$ is topologically $G$-finitely generated. Then $\br_\per(X)$ is admits a $G$-invariant, nested, $G$-finite generating set.
\end{theorem}

\subsection{Cycles and planar graphs}

In this subsection we give some thought to cycle spaces of planar graphs. First, we have the following definition. 

\begin{definition}[Bad loops]\label{def:bad}
    Let $\Gamma$ be a connected, locally finite, planar graph with good drawing $\vartheta : \overline \Gamma \into \bbS^2$. 
    We say that a simple combinatorial loop $\ell$ in $\Gamma$ is a \textit{bad loop} if there exist two ends $\omega_1, \omega_2 \in \Omega (\Gamma)$ such that $\vartheta (\omega_1)$ and $\vartheta (\omega_2)$ lie in distinct components of $\bbS^2 \setminus \vartheta(\ell)$. 
\end{definition}

Bad loops are the main source of complexity when it comes to dealing with the cycle spaces of planar graphs. Indeed, dealing with these loops is the main technical step in \cite{hamann2018planar}. When they are not present, things are much more straightforward, as the next result demonstrates. 

\begin{proposition}\label{prop:no-bad-loops-genset}
    Let $\Gamma$ be a 2-connected, locally finite, planar graph with good drawing $\vartheta : \overline \Gamma \into \bbS^2$. 
    Suppose $\Gamma$ does not contain any bad loops. Then $\cC(\Gamma)$ is algebraically generated by $\finfaces(\Gamma)$. 
\end{proposition}

\begin{proof}
    Let $\ell$ be a closed loop in $\Gamma$. We assume without loss of generality that $\ell$ is simple, and not a bad loop. Thinking of $\ell$ as a Jordan curve in $\bbS^2$ we have that one side of $\ell$ contains only finitely many finite facial cycles. The sum of these facial cycles is exactly $\ell$. 
\end{proof}

Bad loops alone are not a huge amount of trouble, and considering topological generation of the cycle space actually remedies this problem in some cases. For example, the following is not hard to see. We leave the proof as an exercise as we will not directly use it.

\begin{proposition}
    Let $\Gamma$ be a 2-connected, locally finite, planar graph with good drawing $\vartheta : \overline \Gamma \into \bbS^2$. 
    Suppose $\inffaces(\Gamma) = \emptyset$. Then $\cC(\Gamma)$ is topologically generated by $\finfaces(\Gamma)$. 
\end{proposition}

\begin{proof}
    Left as an exercise to the reader (see Figure~\ref{fig:planar-topgen} for inspiration).
\end{proof}

Indeed, this is simply a manifestation of the fact that $\facepaths(\Gamma)$ generates the topological cycle space $\overline{ \cC(\Gamma)}$. However, `infinite generators' are of no use to us. The work-around we will employ is to consider the cone-off of $\Gamma$ over its infinite faces (see \S\ref{sec:ellipticcuts} above for the definition). After coning off these faces, the following is not hard to see. 

\begin{proposition}\label{prop:planar-rel-genset}
     Let $\Gamma$ be a 2-connected, locally finite, planar graph with a fixed good drawing. 
     Let $\mathcal F = \inffaces(\Gamma)$. Then $\cC(\widehat \Gamma_{\mathcal F})$ is topologically generated by
     $$
     \finfaces(\Gamma) \cup \{\text{cycles of length 3}\}.
     $$
\end{proposition}

Our goal will be to push this `nice' generating set through a quasi-isometry. This is the topic the next subsection.

\subsection{Cycles and quasi-isometries}

In order to apply the above result, we will need to be able to meaningfully push topological generating sets of a cycle space through quasi-isometries. To this end, we state and prove a collection of results.
First, we record the following. 

\begin{theorem}[{\cite[Lem.~1.3]{fujiwara2007note}}]\label{thm:alg-gen-qi-inv}
    Let $X$, $Y$ be connected, quasi-isometric graphs. Suppose $\cC(X)$ is algebraically generated by cycles of bounded support. Then the same is true of $\cC(Y)$. 
\end{theorem}

We will need a similar result for topological generation. When the graphs in question are locally finite, this follows from a very similar argument to that which is presented in \cite{fujiwara2007note}. However, we are interested in the cycle spaces of cone-offs, which will generally be locally infinite. We are therefore required to take some care. We will need the following two definitions. 

\begin{definition}[Triangular cycles]
    Let $X$ be a graph with peripheral system $\per$. Let $c \in \cC(\widehat X_\per)$ be a cycle in the cone-off. 
    We say that $c$ is \emph{triangular} if $c$ is supported on at most two edges which do not lie in $E(X)$. 
\end{definition}

\begin{definition}[Peripheral-preserving quasi-isometry]
    Let $X$, $Y$ be connected graphs with respective peripheral systems $\per$, $\mathcal K$. Let $\varepsilon > 0$. Then a quasi-isometry $\varphi : X \to Y$ is said to be \emph{$\varepsilon$-peripheral preserving} if there exists a bijection $f : \per \to \mathcal K$ such that
    $$
    \dHaus[Y](\varphi(H), f(H)) < \varepsilon,
    $$
    for all $H \in \per$. If $\varphi$ is $\varepsilon$-peripheral preserving for some $\varepsilon > 0$, then we may just call $\varphi$ \emph{peripheral preserving}. 
\end{definition}

It is immediate that a quasi-inverse to a peripheral preserving quasi-isometry is also peripheral preserving, witnessed by the inverse $f^{-1} : \mathcal K \to \per$. 
We now prove the following theorem, which is specialised to our purposes. 

\begin{theorem}\label{thm:cone-off-qi-chomp}
    Let $X$ be a connected, locally finite, quasi-transitive $G$-graph with a thin, $G$-invariant peripheral system $\per$. Let $\Gamma$ be a connected, bounded-degree graph with thin peripheral system $\mathcal F$, and let  $\varphi : X \to \Gamma$ be peripheral preserving quasi-isometry.

    Suppose $\cC(\widehat \Gamma_\facepaths)$ admits a topological generating set of bounded support, containing only triangular cycles. Then $\cC(\widehat X_\per)$ is topologically $G$-finitely generated. 
\end{theorem}

\begin{proof}
    Let $\psi : \Gamma \to X$ be a quasi-inverse to $\varphi$. Assume without loss of generality that $\varphi$ and $\psi$ send vertices to vertices and edges to combinatorial geodesics. Fix $\varepsilon > 0$ such that both $\varphi$ and $\psi$ are $\varepsilon$-peripheral preserving. Extend $\varphi$ as follows.
    Define $\widehat \varphi$ on $V(\widehat X_\per)$ via
    $$
    \widehat \varphi (v) = 
    \begin{cases}
        \varphi(v) &\text{if $v \in V(X)$,}\\
        v_{f(H)} &\text{if $v = v_H$ is a cone vertex.}
    \end{cases}
    $$
    We then extend $\widehat \varphi$ to the edges of $\widehat X_\per$ so that it agrees with $\varphi$ on $E(X)$, and given an edge $e$ connecting some cone vertex $v_H$ to $u \in H$ for some $H \in \per$, we choose some geodesic in $\Gamma$ of length at most $\varepsilon$ connecting $\varphi(u)$ to some vertex $u' \in f(H)$, and then extend this with a single edge to a path from $\varphi(u)$ to $v_{f(H)}$ of length at most $\varepsilon + 1$.
    We define $\widehat \psi$ similarly, using the inverse $f^{-1}$ in place of $f$. 

    \begin{claim}
        $\widehat \varphi$ and $\widehat \psi$ are proper\footnote{Recall a continuous map is \emph{proper} if preimages if compact sets are compact.} quasi-isometries, and are quasi-inverses to one another. 
    \end{claim}

    \begin{proof}
        It is immediate that these  maps are Lipschitz, and have the property that there exists $\eta > 0$ so that
        $$
        \dist\Big(x, \widehat \psi \circ \widehat \varphi(x)\Big) < \eta
        $$
        for all $x \in \widehat X_\per$. This is sufficient to deduce that $\widehat \varphi$ and $\widehat \psi$ are quasi-isometries which are quasi-inverses to one another. We now observe that these maps are proper. Let $K \subset \widehat \Gamma_\facepaths$ be some compact subset. Since our maps send vertices to vertices and edges to closed simple paths, it is sufficient to assume $K$ is a single edge $e \in E(\widehat \Gamma_\facepaths)$. 
        First, let us assume that some endpoint of $e$ is a cone-vertex, say $v_H$. Let $u$ denote the other endpoint of $e$. Let $e_1, e_2 \ldots$ be a sequence of edges in $E(\widehat X_\per)$ such that $e \subset \widehat \varphi(e_i)$ for all $i > 0$. By construction, the $e_i$ share an endpoint which is the cone-vertex $v_{f^{-1}(H)}$. Let $u_i \in V(X)$ denote the other endpoint of $e_i$. It is easy to see that the $u_i$ must map uniformly close to $u$ under $\varphi$, but since $\Gamma$ and $X$ are bounded-degree, and $\varphi$ is a quasi-isometry, this implies a uniform bound on how many such $u_i$ there are. An almost identical argument goes through in the case that neither endpoint of $e$ is a cone-vertex. In particular, we deduce that $\widehat \varphi$ is proper. A symmetrical argument shows that $\widehat \psi$ is proper. 
    \end{proof}

    Since $\widehat \psi$ is a continuous map which sends edges to combinatorial paths, each edge determines a finite collection of edges in $E(\widehat X_\per)$, namely the edges forming its image. Extending linearly, we have a natural map 
    $$
    \Psi : \bigoplus_{E(\widehat \Gamma_\facepaths)}\Z_2 \to \bigoplus _{E(\widehat X_\per)} \Z_2,
    $$
    which restricts to a homomorphism $\Psi : \cC(\widehat \Gamma_\facepaths) \to \cC(\widehat X_\per)$. Indeed, this is just the morphism which $\widehat \psi$ induces on the first homology. 
    We similarly get a morphism $\Phi :\cC(\widehat X_\per) \to \cC(\widehat \Gamma_\facepaths) $. Since our maps are proper, we actually get something more, namely that $\Psi$ (and $\Phi)$ extends to a continuous homomorphism
    $$
        \Psi : \prod_{E(\widehat \Gamma_\facepaths)}\Z_2 \to \prod _{E(\widehat X_\per)} \Z_2.
    $$
    Indeed, by properness have that any edge $e \in E(\widehat X_\per)$  is contained in the image of at most finitely many edges under $\widehat \psi$. Thus, this extension of $\Psi$ is well-defined. 
    To see that $\Psi$ is continuous, let $\sigma_i \to \sigma $ in $\prod_{E(\widehat \Gamma_\facepaths)}\Z_2$, so the $\sigma_i$ pointwise converge to $\sigma$.  Fix $e \in E(\widehat X_\per)$, and let $e_1, \ldots, e_n$ be the finitely many edges such that $e \subset \psi(e_i)$. Eventually the $\sigma_i$ will all agree with $\sigma$ on the $e_i$, and so $\Psi(\sigma_i)(e) = \Psi(\sigma)(e)$ for all sufficiently large $i > 0$. Since $e$ was arbitrary, this implies $\Psi(\sigma_i) \to \Psi(\sigma)$, and so $\Psi$ is continuous. We extend $\Phi$ similarly. 

    We now have a sequence of claims.

    \begin{claim}\label{claim:special-subgroup}
        There exists a subgroup $K_0 \leq \cC(\widehat X_\per)$ which is (algebraically) generated by triangular cycles of bounded support, such that for all $c \in \cC(\widehat X_\per)$, we have 
        $$
        \Psi \circ \Phi(c) \in c + K_0.
        $$
    \end{claim}

    \begin{proof}
        We first construct $K_0$. For every $v \in V(\widehat X_\per)$, let $P_v$ be a simple path of uniformly bounded length connecting $v$ to $v' := \widehat \psi \circ \widehat \varphi(v)$. If $v = v_H$ is a cone-vertex then $v' = v$, and so we take $P_v$ to be the empty path. Otherwise, we will take $P_v$ to be a path through $X$, which importantly does not pass through any cone-vertices. Given an edge $e \in E(\widehat X_\per)$ with endpoints $u$, $v$, we define a cycle 
        $$
        a_e = P_u + P_v + e + \Psi \circ \Phi(e).
        $$
        It is easy to see that this is indeed a cycle supported on at most a uniformly bounded number of edges, and that it is triangular. Indeed, if $e$ is contained in $X$ then $a_e$ is also contained entirely within $X$, and if $e$ adjoins a cone-vertex, then we have that, say $P_u$ is the empty path, $P_v$ is supported on exactly one cone-edge by construction, as is $\Psi \circ \Phi(e)$. 
        We then let 
        $$
        K_0 := \Big\langle a_e : e \in E(\widehat X_\per) \Big\rangle.
        $$
        Now, given $c \in \cC(\widehat X_\per)$, with support $e_1, \ldots, e_n$. Write $a_i = a_{e_i}$. A quick computation verifies that 
        $$
        c + \sum_i a_i = \sum_i e_i + a_i = \sum_i \Psi \circ \Phi(e_i) = \Psi \circ \Phi(c).
        $$
        The claim follows. 
    \end{proof}

    \begin{claim}\label{claim:triangles-preserved}
        If $c \in \cC(\widehat \Gamma_\facepaths)$ is triangular, then so is $\Psi(c)$.
    \end{claim}

    \begin{proof}
        This follows immediately from the construction of $\widehat \psi$. 
    \end{proof}

    \begin{claim}\label{claim:cofinite-triangles}
        For every $N > 0$, the group $G$ acts with finitely many orbits on the set of simple triangular cycles in $\cC(\widehat X_\per)$ of length at most $N$. 
    \end{claim}

    \begin{proof}
        A simple triangular cycle of length $N$ is determined by either a simple cycle in $X$ of length $N$, or a simple path in $X$ of length $N-2$. Since $X$ is locally finite and $G$ acts quasi-transitively, there are finitely many orbits of such objects. 
    \end{proof}

    Fix $\ell > 0$ such that $\cC(\widehat \Gamma_\facepaths)$ is topologically generated by triangular simple cycles supported on at most $\ell$. If $c \in \cC(\widehat \Gamma_\facepaths)$ is such a cycle, then $\Psi(c)$ is a cycle supported on at most some uniformly bounded number of edges, say $k$.
    Now, let $c \in \cC(\widehat X_\per)$ be arbitrary. Consider $\Phi(c)$. By assumption there exists a sequence of cycles $(c_i)$ in $\cC(\widehat \Gamma_\facepaths)$, where each $c_i$ is supported on at most $\ell$ edges, such that 
    $$
    \Phi(c) = \lim_{n \to \infty} \sum_{i=1}^n c_i.
    $$
    Since $\Psi$ is a continuous homomorphism, we deduce 
    $$
    \Psi \circ \Phi(c) = \lim_{n \to \infty} \sum_{i=1}^n \Psi(c_i),
    $$
    where each $\Psi(c_i)$ is supported on at most $k$ edges. By Claim~\ref{claim:triangles-preserved}, each $\Psi(c_i)$ is triangular, and we may decompose each of these as a finite sum of simple triangular cycles of bounded support. Combining this with Claim~\ref{claim:special-subgroup}, we deduce that $\cC(\widehat X_\per)$ is topologically generated by a collection of triangular simple cycles of bounded support. By Claim~\ref{claim:cofinite-triangles}, these cycles fall into finitely many orbits. The theorem follows. 
\end{proof}

This has the following immediate corollary which we record here as it may be of independent interest. 

\begin{corollary}
    Let $X$ be a connected, locally finite, quasi-transitive $G$-graph, and let $\Gamma$ be a connected, bounded-degree graph which is quasi-isometric to $X$.
    Suppose that $\cC(\Gamma)$ is topologically generated by cycles of bounded support. Then $\cC(X)$ is topologically $G$-finitely generated. 
\end{corollary}

\subsection{Upshot}

Let us now piece together the results of above into statements which are directly relevant to our setting. 

\begin{lemma}\label{lem:no-bad-loops-small-faces-acc}
    Let $X$ be a connected, locally finite, quasi-transitive graph, and $\Gamma$ be a bounded-degree, 2-connected, planar graph with a fixed good drawing, which is quasi-isometric to $X$. Suppose every $f \in \finfaces(\Gamma)$ is uniformly bounded in length, and $\Gamma$ contains no bad loops. Then $X$ is accessible. 
\end{lemma}

\begin{proof}
    By Proposition~\ref{prop:no-bad-loops-genset} and Theorem~\ref{thm:alg-gen-qi-inv}, we have that $\cC(X)$ is algebraically generated by cycles of bounded length. By Theorem~\ref{thm:hamann}, we conclude that $X$ is accessible. 
\end{proof}

\begin{lemma}\label{lem:peripheral-system-gives-relacc}
    Let $X$ be a connected, locally finite, quasi-transitive $G$-graph with a thin, $G$-invariant peripheral system $\per$. 
    Let $\Gamma$ be a bounded-degree, 2-connected, planar graph with a fixed good drawing, equipped with the peripheral system $\mathcal F = \inffaces(\Gamma)$. 
    Suppose that the following hold:
    \begin{enumerate}
        \item Every $f \in \finfaces(\Gamma)$ is uniformly bounded in length. 

        \item There exists a  peripheral preserving quasi-isometry $\varphi : X \to \Gamma$. 
    \end{enumerate}
    Then $\br_\per(X)$ contains a nested, $G$-invariant, $G$-finite generating set. 
\end{lemma}

\begin{proof}
    By Proposition~\ref{prop:planar-rel-genset}, we have that $\cC(\widehat \Gamma_{\mathcal F})$ is topologically generated by simple triangular cycles of bounded support. By Theorem~\ref{thm:cone-off-qi-chomp}, we deduce that $\cC(\widehat X_\per)$ is topologically $G$-finitely generated. The claim now follows from applying Theorem~\ref{thm:relacc-cycles}.
\end{proof}

The upshot of all of this is that if we can obtain good control over how the quasi-action of $G$ on $\Gamma$ interacts with the faces of $\Gamma$---so much so that we can satisfy the hypotheses of the above lemma---then we should be in a good position to say something about accessibility. Gaining this control will not be easy, and is the main technical challenge of this paper.
Roughly speaking, the reason this helps is that we will use Lemma~\ref{lem:peripheral-system-gives-relacc} to `cut along a representative selection of bad loops', and obtain a decomposition of our graph into pieces which are quasi-isometric to planar graphs that contain no bad loops at all. Following this, we will be able to apply Lemma~\ref{lem:no-bad-loops-small-faces-acc} to the individual pieces and deduce accessibility.


\section{Coboundary diameters}\label{sec:coboundaries}

We now introduce some new terminology and notation which will help us reason about the metric geometry of the boundaries of subgraphs. 

\subsection{Definition and basic properties}

For brevity, we will make the following notational concession. 
Given a graph $\Gamma$ and $\Lambda \subset \Gamma$ a subgraph, we define the \textit{coboundary of $\Lambda$}, denoted $\delta \Lambda$, as simply the coboundary of its vertex set. That is,  $\delta \Lambda := \delta V(\Lambda)$.

\begin{definition}[Coboundary diameters]\label{def:unif-cobound}
    Let $\Gamma$ be a connected graph and $\Lambda \subset \Gamma$ a connected subgraph. Define the \textit{inner coboundary diameter} of $\Lambda$, denoted $\incut{\Lambda}$, as 
    $$
    \incut{\Lambda} = \sup\{ \diam_\Lambda(\delta U \cap \Lambda) : \text{$U$ is a connected component of $\Gamma \setminus \Lambda$}  \}. 
    $$
    Similarly, we define the \textit{outer coboundary diameter} of $\Lambda$ as
    $$
    \outcut{\Lambda} = \sup\{ \incut{U} : \text{$U$ is a connected component of $\Gamma \setminus \Lambda$} \}. 
    $$
    We say that $\Lambda$ has \textit{uniform coboundary} if both the inner and outer coboundary diameters of $\Lambda$ are finite. If $\Lambda = \Gamma$, we adopt the convention that $\incut{\Lambda} = \outcut{\Lambda} = 0$.

    We may write $\incut{\Lambda}^\Gamma$ or $\outcut{\Lambda}^\Gamma$ if the intended super-graph $\Gamma$ is ever unclear. 
\end{definition}

Intuitively, we imagine the coboundary of $\Lambda$ as a selection of tight cuts, separating $\Lambda$ from the components of its complement. The inner diameter measures the size of these cuts from `inside' $\Lambda$, whereas the outer diameter measures this diameter from the other side of this cut. Without placing further restrictions on $\Lambda$, there is no reason these values need to correlate. However, we can say something in the presence of an appropriate group action, namely when $\Lambda$ is \emph{quasi-transitively stabilised}, in the sense of Definition~\ref{def:cocompact}.

\begin{lemma}\label{lem:cocompact-in-to-uniform}
    Let $G$ be a group acting on a connected, bounded-degree graph $\Gamma$. Let $\Lambda \subset \Gamma$ be a connected, quasi-transitively stabilised subgraph.  If $\incut{\Lambda}$ is finite, then $\Lambda$ has uniform coboundary. 
\end{lemma}

\begin{proof}
    Since $\Stab(\Lambda)$ acts on $\Lambda$ quasi-transitively and $\Gamma$ has bounded-degree, we see that $\Stab(\Lambda)$ acts with finite quotient on the set of connected components of $\Gamma \setminus \Lambda$. If $U$ is a connected component of $\Gamma \setminus \Lambda$, then we have that $\delta U$ is a finite set of edges, and so $\incut U$ is finite. Since there are only finitely many orbits of these components, we deduce that $\outcut{\Lambda}$ must be finite, and so $\Lambda$ has uniform coboundary. 
\end{proof}

\begin{proposition}\label{prop:end-injective}
    Let $\Gamma$ be a locally finite graph and $\Lambda \subset \Gamma$ a connected subgraph with $\incut{\Lambda} < \infty$. Then the natural map $\Omega (\Lambda) \to \Omega (\Gamma)$ is injective. 
\end{proposition}

\begin{proof}
    Let $\gamma_1, \gamma_2$ be rays in $\Lambda$. We need to show that if $\gamma_1$, $\gamma_2$ approach distinct ends in $\Lambda$, then they approach distinct ends in $\Gamma$. Let $K \subset \Lambda$ be a compact subgraph such that infinite subpaths of $\gamma_1$, $\gamma_2$ lie in distinct components of $\Lambda \setminus K$. Let $U_i$ denote the component of $\Lambda \setminus K$ containing the tail of $\gamma_i$. 
    Any path in $\Gamma$ connecting $U_1$ to $U_2$ must pass through some connected component $U$ of $\Gamma \setminus \Lambda$, where $\delta U$ intersects both $U_1$ and $U_2$. Only finitely many such $U$ exist since $\incut{\Lambda}$ is finite, and moreover $\delta U$ can only contain finitely many edges. Thus, let 
    $$
    K' = K \cup \bigcup_U \delta U, 
    $$
    where $U$ ranges over the aforementioned components of $\Gamma \setminus \Lambda$. Clearly $K'$ is compact, and the tails of $\gamma_1$, $\gamma_2$ lie in distinct components of $\Gamma \setminus K'$. 
\end{proof}

\begin{remark}\label{rmk:good-drawing-subgraph}
    With $\Gamma$, $\Lambda$ as above, i.e. $\incut{\Lambda}< \infty$, it follows from Proposition~\ref{prop:end-injective} that the closure of $\Lambda$ in the Freudenthal compactification $\overline \Gamma$ of $\Gamma$ is naturally homeomorphic to $\overline \Lambda$. In particular, if $\Gamma$ is planar and $\vartheta : \overline \Gamma \into \bbS^2$ is a good drawing of $\Gamma$, then the restriction of $\vartheta$ to (the closure of) $\Lambda$ is a good drawing of $\Lambda$. 
\end{remark}

\begin{proposition}\label{prop:uniform-boundary-qie}
    Let $\Gamma$ be a connected, locally finite graph, and let $\Lambda \subset \Gamma$ be a connected subgraph with uniform coboundary. Then the inclusion $\Lambda \into \Gamma$ is a quasi-isometric embedding. 
\end{proposition}

\begin{proof}
    Take a geodesic in $\Gamma$ connecting $x, y \in V (\Lambda)$. Any segment of this geodesic lying outside of $\Lambda$ has length at most $2 + \outcut{\Lambda}$, and we also know that there exists a path inside of $\Lambda$ of length at most $\incut{\Lambda}$, with precisely the same endpoints. Thus, such a geodesic can be transformed into a path contained in $\Lambda$ whose length is proportional to the length of the original geodesic. It follows immediately that the inclusion $\Lambda \into \Gamma$ is a quasi-isometric embedding. 
\end{proof}

The following is our motivating example. 

 \begin{example}
    Let $X$ be a connected, locally finite, quasi-transitive $G$-graph, and let $(T, \cV)$ be a $G$-canonical connected tree decomposition of bounded adhesion, with $T/G$ compact. By Proposition~\ref{prop:tree-decomp-acc}(\ref{itm:tree-1}), each part is quasi-transitively stabilised. It is easy to see that the inner-coboundary diameter of each part is finite, and so in particular we have by Lemma~\ref{lem:cocompact-in-to-uniform} that the parts have uniform coboundary. 
\end{example}

The next section studies how the `niceness' of the above example can be passed through quasi-isometries.

\subsection{Coboundary diameters and quasi-isometries}

We now study how coboundary diameters interact with quasi-isometries. 
In particular, we will show the following. 

\begin{restatable}[Neighbourhoods with uniform coboundary]{theorem}{neighbourhoodunifcob}\label{thm:uniform-boundary-qi}
    Let $X$ be a connected, locally finite, quasi-transitive graph and $Y \subset X$ a connected, quasi-transitively stabilised subgraph with uniform coboundary. Let $\Gamma$ be a connected, bounded-degree graph and $\varphi : X \onto \Gamma$ a continuous, surjective quasi-isometry. 
    Then there exists $\delta > 0$ such that 
    $$
    \varphi(B_X(Y;\delta))
    $$
    has uniform coboundary in $\Gamma$.
\end{restatable}

This result will have important applications in the next section.
We will prove this through a series of lemmas.

\begin{lemma}\label{lem:small-increase-still-bounded}
    Let $\Gamma$ be a connected, bounded-degree graph and $\Lambda \subset \Pi \subset \Gamma$ be connected subgraphs. If both $\incut{\Lambda}^\Gamma$ and $\dHaus[\Gamma](\Lambda, \Pi)$ are finite then $\incut{\Pi}$ is finite too. 
\end{lemma}

\begin{proof}
    Write $k = \incut{\Lambda}^\Gamma$, $r = \dHaus[\Gamma](\Lambda, \Pi)$. 
    Fix a connected component $U$ of $\Gamma \setminus \Pi$. Then there exists a unique component $V$ of $\Gamma \setminus \Lambda$ such that $U \subset V$. Let $x,y \in \delta U  \cap \Pi$. Choose $x', y' \in \Lambda$ such that 
    $$
    \dist_{\Gamma}(x, x') = \dist_{\Gamma}(x, \Lambda), \ \  \dist_{\Gamma}(y, y') = \dist_{\Gamma}(y, \Lambda). 
    $$
    In particular, $x'$, $y'$ lie in $\delta V$ and a geodesic  connecting $x$ to $x'$ or $y$ to $y'$ is internally disjoint from $\Lambda$. Such a geodesic is contained in $B_{\Gamma}(\delta V; r)$ and so has bounded length, say length at most $L \geq 0$, since $\delta V$ contains boundedly many vertices and $\Gamma$ is bounded-degree. Note that $L$ depends only on $r$, $\Gamma$ and $\Lambda$. 
 
 There is a path through $\Lambda$ of length at most $k$ connecting $x'$ to $y'$. It follows that there is a path through $\Pi$ of length at most $k + 2L$ connecting $x$ to $y$. Since $x$, $y$ and $U$ were arbitrary, it follows that $\incut{\Pi} \leq k + 2L$. 
\end{proof}

\begin{lemma}\label{lem:onto-qi}
    Let $X$, $\Gamma$ be connected bounded-degree graphs, and let $\varphi : X \onto \Gamma$ be a  surjective quasi-isometry. Let $Y \subset X$ be a subgraph. Then for all $\varepsilon > 0$ there exists $\delta > 0$ such that
    $$
    \varphi(B_X(Y;\delta)) \supset B_\Gamma(\varphi(Y);\varepsilon).
    $$
\end{lemma}

\begin{proof}
    Choose $x \in B_\Gamma(\varphi(Y);\varepsilon)$. Since $\varphi$ is surjective, let $z \in \varphi^{-1}(x)$. We have that $\dist_X(z, Y)$ is uniformly bounded, and so the lemma follows. 
\end{proof}

\begin{lemma}\label{lem:nbhd-incut-bounded}
    Let $X$, $\Gamma$ be connected bounded-degree graphs, and let $\varphi : X \to \Gamma$ be a continuous quasi-isometry. Let $Y \subset X$ a subgraph such that $\incut{Y}$ is finite. Then there exists $\varepsilon_0 > 0$ such that for all $\varepsilon > \varepsilon_0$ we have that $\incut{B_\Gamma(\varphi(Y);\varepsilon)}$ is finite. 
\end{lemma}

\begin{proof}
    Let $\psi : \Gamma \to X$ be a choice of continuous quasi-inverse to $\varphi$. Fix $\lambda \geq 1$ so that $\varphi$ and $\psi$ are $(\lambda, \lambda)$-quasi-isometries and $\lambda$-quasi-inverses to each other. 
    Let $\varepsilon_0 = \lambda^2 + 2\lambda$ and fix $\varepsilon > \varepsilon_0$.
    To ease notation, write $N = B_\Gamma(\varphi(Y);\varepsilon)$. 
    
    Let $U$ be some connected component of $\Gamma \setminus N$. 
    Since $\varepsilon > \lambda^2$, we have that $\psi(U)$ will be contained entirely in some connected component of $X \setminus Y$. Call this connected component $V$. 
    Let $u, v \in \delta U \cap N$, and let $\gamma$ be a path connecting $u$ to $v$, such that $\gamma \cap N = \{u, v\}$.
    We have that $\psi(\gamma)$ is a path connecting $\psi(u)$ to $\psi(v)$, contained entirely within $V$. 

    Let $p_1$ be a geodesic of length $\varepsilon$ from $u$ to $\varphi(Y)$, say terminating at $a \in \varphi(Y)$. Similarly let $p_2$ be a geodesic of length $\varepsilon$ joining $v$ to some $b \in \varphi(Y)$. We have that $\psi(a)$, $\psi(b)$ lie at most a distance $\lambda$ from $Y$, so let $a', b' \in Y$ lie a minimal distance from $\psi(a)$, $\psi(b)$ respectively. Let $q_1$ be a geodesic joining $\psi(a)$ to $a'$ and $q_2$ a geodesic joining $b'$ to $\psi(b)$. 
    Let 
    $$
    q = q_1 \psi(p_1)\psi(\gamma)\psi(p_2)q_2
    $$
    be the concatenation of these paths. This is a path which starts and ends in $Y$. 
    Since $\varepsilon > \lambda^2$, we have that $\psi(\gamma)$ is disjoint from $Y$, and so $\psi(u)$ and $\psi(v)$ lie in the same connected component of $X \setminus Y$. 
    There will exist some subpath $l_1$ of $q_1\psi(p_1)$ which connects $\psi(u)$ to some $c \in Y$, which is otherwise disjoint from $Y$. Similarly, there will exist some subpath $l_2$ of $q_2\psi(p_2)$ connecting $\psi(v)$ to some $d \in Y$ which is otherwise disjoint from $Y$. In particular, since $\incut{Y}$ is finite, say $\incut{Y} = k$, there is some path $l_3$ contained in $Y$ of length at most $k$ connecting $c$ to $d$. Note that the endpoints of $l_3$ lie at most a distance $\lambda$ from $\psi(p_1)$, $\psi(p_2)$. 

    Now, let $l = \varphi(l_3)$. This is a path of length at most $\lambda k + \lambda$, whose endpoints lie at a distance of at most $\lambda^2 + 2\lambda < \varepsilon$ from $p_1$ and $p_2$. Join $l$ to $p_1$ and $p_2$ with geodesics $f_1$, $f_2$. Note that $f_1$ and $f_2$ are contained within $N$. Thus, the union $p_1 \cup p_2 \cup f_1 \cup f_2 \cup l$ contains a path $q$ connecting $u$ to $v$, and is contained entirely within $N$. Finally, note that $q$ contains at most 
    $
    4 \varepsilon + \lambda k + \lambda
    $
    edges, and so since $U$, $u$, $v$ were arbitrary it follows that $\incut{N}$ is finite. 
\end{proof}

\begin{lemma}\label{lem:incut-bdd-to-unif}
    Let $X$, $\Gamma$ be connected bounded-degree graphs, and let $\varphi : X \to \Gamma$ be a continuous quasi-isometry. Let $G$ act upon $X$, and let $Y \subset X$ be a connected, quasi-transitively stabilised subgraph. Suppose that $\incut{Y}$ and $\incut{\varphi(Y)}$ are finite. Then $\varphi(Y)$ has uniform coboundary. 
\end{lemma}

\begin{proof}
    Let $\psi : \Gamma \to X$ be a choice of continuous quasi-inverse to $\varphi$. Fix $\lambda \geq 1$ so that $\varphi$ and $\psi$ are $(\lambda, \lambda)$-quasi-isometries and $\lambda$-quasi-inverses to each other. 
    Fix a connected component $U$ of $\Gamma \setminus \varphi(Y)$. Let $v, u \in \delta U \cap U$. We need to find a path of bounded length through $U$ connecting $v$ to $u$. Clearly there is some simple path in $U$ connecting them together, so denote this path $p$. We will replace $p$ with a path $p'$ of bounded length. 

    Fix $\varepsilon > 2\lambda^3 + \lambda^2$. Decompose $p$ into a composition of paths 
    $$
    p = p_0 q_1 p_1 \ldots q_n p_n,
    $$
    where each $p_i$ is contained in the closed $\varepsilon$-neighbourhood of $\varphi(Y)$, and each $q_i$ is contained in the complement of the open $\varepsilon$-neighbourhood of $\varphi(Y)$. See Figure~\ref{fig:decomposing-path} for a cartoon of this decomposition. 
    Let $x_i$, $y_i$ denote the endpoints of $q_i$. We will first find a short path in $U$ connecting $x_i$ to $y_i$. 
    
    \begin{figure}
        \centering

                   

\tikzset{every picture/.style={line width=0.75pt}} 

\begin{tikzpicture}[x=0.75pt,y=0.75pt,yscale=-1,xscale=1]

\draw  [draw opacity=0][fill={rgb, 255:red, 155; green, 155; blue, 155 }  ,fill opacity=0.53 ] (60.5,183.4) .. controls (60.5,186.77) and (63.23,189.5) .. (66.6,189.5) -- (436.9,189.5) .. controls (440.27,189.5) and (443,186.77) .. (443,183.4) -- (443,158.99) .. controls (443,158.99) and (443,158.99) .. (443,158.99) -- (60.5,158.99) .. controls (60.5,158.99) and (60.5,158.99) .. (60.5,158.99) -- cycle ;
\draw    (60.5,158.99) -- (443,158.99) ;
\draw  [draw opacity=0][fill={rgb, 255:red, 74; green, 144; blue, 226 }  ,fill opacity=0.29 ] (60.5,105) -- (443,105) -- (443,158.99) -- (60.5,158.99) -- cycle ;
\draw [color={rgb, 255:red, 74; green, 144; blue, 226 }  ,draw opacity=1 ]   (60.5,105) -- (443,105) ;
\draw [line width=1.5]    (178.5,141.5) -- (178.5,159) ;
\draw [shift={(178.5,159)}, rotate = 90] [color={rgb, 255:red, 0; green, 0; blue, 0 }  ][fill={rgb, 255:red, 0; green, 0; blue, 0 }  ][line width=1.5]      (0, 0) circle [x radius= 3.05, y radius= 3.05]   ;
\draw [line width=1.5]    (409,141) -- (409,158.5) ;
\draw [shift={(409,158.5)}, rotate = 90] [color={rgb, 255:red, 0; green, 0; blue, 0 }  ][fill={rgb, 255:red, 0; green, 0; blue, 0 }  ][line width=1.5]      (0, 0) circle [x radius= 3.05, y radius= 3.05]   ;
\draw    (171,105) .. controls (179,119) and (174.5,133) .. (178.5,141.5) ;
\draw    (401.5,104.5) .. controls (401,122.5) and (405,132.5) .. (409,141) ;
\draw [shift={(401.5,104.5)}, rotate = 91.59] [color={rgb, 255:red, 0; green, 0; blue, 0 }  ][fill={rgb, 255:red, 0; green, 0; blue, 0 }  ][line width=0.75]      (0, 0) circle [x radius= 3.35, y radius= 3.35]   ;
\draw    (340.5,103.5) .. controls (334.5,74.5) and (342.5,54) .. (370,54.5) .. controls (397.5,55) and (399,56.5) .. (401.5,104.5) ;
\draw [shift={(340.5,103.5)}, rotate = 258.31] [color={rgb, 255:red, 0; green, 0; blue, 0 }  ][fill={rgb, 255:red, 0; green, 0; blue, 0 }  ][line width=0.75]      (0, 0) circle [x radius= 3.35, y radius= 3.35]   ;
\draw    (171,105) .. controls (171.5,76) and (168.5,55) .. (196,55.5) .. controls (223.5,56) and (222.5,84.5) .. (232,106) ;
\draw [shift={(171,105)}, rotate = 270.99] [color={rgb, 255:red, 0; green, 0; blue, 0 }  ][fill={rgb, 255:red, 0; green, 0; blue, 0 }  ][line width=0.75]      (0, 0) circle [x radius= 3.35, y radius= 3.35]   ;
\draw    (269.5,104.5) .. controls (271,81.5) and (269.5,58) .. (285.5,54) .. controls (301.5,50) and (297.5,81.5) .. (300.5,105.5) ;
\draw [shift={(269.5,104.5)}, rotate = 273.73] [color={rgb, 255:red, 0; green, 0; blue, 0 }  ][fill={rgb, 255:red, 0; green, 0; blue, 0 }  ][line width=0.75]      (0, 0) circle [x radius= 3.35, y radius= 3.35]   ;
\draw    (232,106) .. controls (234.5,130.5) and (233,141) .. (249,137) .. controls (265,133) and (263.5,116) .. (269.5,104.5) ;
\draw [shift={(232,106)}, rotate = 84.17] [color={rgb, 255:red, 0; green, 0; blue, 0 }  ][fill={rgb, 255:red, 0; green, 0; blue, 0 }  ][line width=0.75]      (0, 0) circle [x radius= 3.35, y radius= 3.35]   ;
\draw    (300.5,105.5) .. controls (301.5,120.5) and (312,135) .. (322,136) .. controls (332,137) and (340.5,116) .. (340.5,103.5) ;
\draw [shift={(300.5,105.5)}, rotate = 86.19] [color={rgb, 255:red, 0; green, 0; blue, 0 }  ][fill={rgb, 255:red, 0; green, 0; blue, 0 }  ][line width=0.75]      (0, 0) circle [x radius= 3.35, y radius= 3.35]   ;
\draw [color={rgb, 255:red, 208; green, 2; blue, 27 }  ,draw opacity=1 ][line width=1.5]    (409,141) ;
\draw [shift={(409,141)}, rotate = 0] [color={rgb, 255:red, 208; green, 2; blue, 27 }  ,draw opacity=1 ][fill={rgb, 255:red, 208; green, 2; blue, 27 }  ,fill opacity=1 ][line width=1.5]      (0, 0) circle [x radius= 3.05, y radius= 3.05]   ;
\draw [color={rgb, 255:red, 208; green, 2; blue, 27 }  ,draw opacity=1 ][line width=1.5]    (178.5,141.5) ;
\draw [shift={(178.5,141.5)}, rotate = 0] [color={rgb, 255:red, 208; green, 2; blue, 27 }  ,draw opacity=1 ][fill={rgb, 255:red, 208; green, 2; blue, 27 }  ,fill opacity=1 ][line width=1.5]      (0, 0) circle [x radius= 3.05, y radius= 3.05]   ;

\draw (80.31,165.97) node [anchor=north west][inner sep=0.75pt]    {$\varphi ( Y )$};
\draw (81.11,121.55) node [anchor=north west][inner sep=0.75pt]  [color={rgb, 255:red, 74; green, 144; blue, 226 }  ,opacity=1 ]  {$B_\Gamma(\varphi(Y);\varepsilon)$};
\draw (179.5,111.4) node [anchor=north west][inner sep=0.75pt]  [font=\footnotesize]  {$p_{0}$};
\draw (185,131.4) node [anchor=north west][inner sep=0.75pt]  [color={rgb, 255:red, 208; green, 2; blue, 27 }  ,opacity=1 ]  {$u$};
\draw (415,131.4) node [anchor=north west][inner sep=0.75pt]  [color={rgb, 255:red, 208; green, 2; blue, 27 }  ,opacity=1 ]  {$v$};
\draw (239.5,112.9) node [anchor=north west][inner sep=0.75pt]  [font=\footnotesize]  {$p_{1}$};
\draw (312.5,107.9) node [anchor=north west][inner sep=0.75pt]  [font=\footnotesize]  {$p_{2}$};
\draw (383.5,115.4) node [anchor=north west][inner sep=0.75pt]  [font=\footnotesize]  {$p_{3}$};
\draw (155,63.9) node [anchor=north west][inner sep=0.75pt]  [font=\footnotesize]  {$q_{1}$};
\draw (253.5,61.4) node [anchor=north west][inner sep=0.75pt]  [font=\footnotesize]  {$q_{2}$};
\draw (351.5,61.4) node [anchor=north west][inner sep=0.75pt]  [font=\footnotesize]  {$q_{3}$};

\end{tikzpicture}

        \caption{Decomposing the path $p$ into the $p_i$ and $q_i$.}
        \label{fig:decomposing-path}
    \end{figure}

    Fix $i$, and to ease notation let $x = x_i$, $y = y_i$, $q = q_i$.  
    We have that 
    $$
    \dist_{\Gamma}(x, \varphi(Y)) = \dist_{\Gamma}(y, \varphi(Y)) = \varepsilon. 
    $$
    Let $q' = \psi(q)$, which is a path connecting $x' := \psi(x)$ to $y' := \psi(y)$. We have that $q'$ is a contained in the complement of the $\varepsilon'$-neighbourhood of $Y$, where 
    $$
    \varepsilon' = \tfrac 1 \lambda \varepsilon - \lambda >  2\lambda ^2 > 0. 
    $$ 
    We also have that $x',y' \in B_X(Y;r)$, where $r = \lambda \varepsilon + \lambda$. 
    Choose $a,b \in B_X(Y;\varepsilon' + 1)$ such that $\dist_{X}(a,x')$, $\dist_{X}(b,y')$ are minimal. In particular, these distances are at most $r$. 

    Since $Y$ is quasi-transitively stabilised, so is $N := B_X(Y;\varepsilon')$. Thus, $N$ has uniform coboundary. The vertices $a, b \in \delta N$ lie in the same connected component of $X \setminus N$, and so there is some path in $X \setminus N$ connecting them of bounded length, say $k$, where this bound depends only on $\varepsilon'$ and $Y$. Let $l_2$ be such a path of minimal length. Let $l_1$ be a geodesic connecting $x'$ to $a$ and $l_3$ a geodesic connecting $b$ to $y'$. The concatenation $l = l_1l_2l_3$ is a path of length at most $2r + k$ connecting $x'$ to $y'$, and is disjoint from $N$. Let $l' = \varphi(l)$. 
    This is a path of length at most $\lambda(2r+k) + \lambda$. Moreover, we have that 
    $$
    \inf_{z \in l}\dist_{\Gamma}(z, \varphi(Y)) \geq \tfrac 1 \lambda \varepsilon' - \lambda > \tfrac 1 \lambda (2\lambda^2) - \lambda > 0. 
    $$
    The endpoints of $l'$ are $x'' := \varphi \circ \psi(x)$ and $y'' := \varphi \circ \psi(y)$, which lie a distance at most $\lambda$ from $x$ and $y$ respectively. Join $x$ to $x''$ by a geodesic of length at most $\lambda$, and similarly join $y$ to $y''$. These geodesics are disjoint from $\varphi(Y)$ since $\varepsilon > 
    \lambda$. Adjoin these geodesics to the start and end of $l'$ and form a new path $l''$ connecting $x$ to $y$ of length at most 
    $
    \lambda (2r + k) + 2\lambda$. 
    Importantly, $l''$ has a uniformly bounded length depending only on $Y$ and the quasi-isometries.  and is disjoint from $\varphi(Y)$. 
    Thus, we can replace $q$ with this path $l''$. 

    Returning to our original set-up, thanks to the above we may now assume that each path $q_i$ is a path of uniformly bounded length. Since $\Gamma$ is bounded-degree  and $|\delta U|$ is bounded, $B_\Gamma(\delta U;\varepsilon)$ contains at most boundedly many vertices, say $M$. Since each $q_i$ starts at a unique vertex in $B_\Gamma(\delta U;\varepsilon)$, we see that $n \leq M$. Also, the union of the $p_i$ contains at most $M$ vertices. It follows that 
    $$
    \dist_U(u, v) \leq M(\lambda (2r + k) + 2\lambda) + M,
    $$
    This bound depends only on $X$, $Y$, $\Gamma$, and the quasi-isometries $\varphi$, $\psi$. Since $U$, $u$, and $v$ were arbitrary it follows that $\outcut{\varphi(Y)}$ is finite, and thus $\varphi(Y)$ has uniform coboundary. 
\end{proof}

Piecing the above together we can deduce the main result of this subsection, which we restate for the convenience of the reader. 

\neighbourhoodunifcob*

\begin{proof}
    By Lemma~\ref{lem:nbhd-incut-bounded} we may choose $\varepsilon > 0$ so that $\incut{B_\Gamma(\varphi(Y);\varepsilon)}$ is finite. Apply Lemma~\ref{lem:onto-qi} and choose $\delta > 0$ such that 
    $$
    Y' := \varphi(B_X(Y;\delta)) \supset B_\Gamma(\varphi(Y);\varepsilon).
    $$
    By Lemma~\ref{lem:small-increase-still-bounded} we have that $\incut{Y'}$ is finite. By Lemma~\ref{lem:incut-bdd-to-unif}, we conclude that $Y'$ has uniform coboundary.
\end{proof}

\section{Friendly-faced subgraphs of planar graphs}\label{sec:friendly-faces}

We now give some consideration to the relationship between the faces of a planar graph and the faces of its subgraphs. 
Now would be a good time for the reader to review our notation and terminology pertaining to the faces of a planar graph (see Definition~\ref{def:faces}). 

\subsection{Initial discussion}
The faces of a subgraph of a planar graph can be wildly different from the faces of the super-graph. This has the potential to cause issues for us further down the line, so the goal of this section is to try and claw back some control here. With this in mind, we have the following definition.

\begin{definition}[Friendly-faced subgraphs]\label{def:friendly-faced}
    Let $\Gamma$ be a connected, locally finite planar graph with fixed good drawing $\vartheta : \overline \Gamma \into \bbS^2$. We say that a connected subgraph $\Lambda \subset \Gamma$ is \textit{friendly-faced} if the following are satisfied: 
    \begin{enumerate}
        \item\label{itm:good-drawing-restrict} The good drawing $\vartheta$ of $\Gamma$ restricts to a good drawing of $\Lambda$.

        \item There exists $r > 0$ such that for every $U_1 \in \facedisks(\Lambda)$ (with the aforementioned induced good drawing) there exists $U_2 \in \facedisks(\Gamma)$ such that $U_2 \subset U_1$ and
        $$f_1 \subset B_\Gamma(f_2; r), $$
        where $f_i = \facepaths[U_i]$ for $i = 1,2$. 
    \end{enumerate} 
\end{definition}

Intuitively, a subgraph is friendly-faced if we can approximate its own facial subgraphs with those of the super-graph. Condition (\ref{itm:good-drawing-restrict}) is just to ensure that the induced drawing of the subgraph is sensible enough for us to work with, and is equivalent to asking that the inclusion $\Lambda \into \Gamma$ induces an inclusion $\Omega (\Lambda) \into \Omega (\Gamma)$ on the sets of ends, and thus an embedding $\overline \Lambda \into \overline \Gamma$ of Freudenthal compactifications. 

\begin{example}
    We illustrate this definition with a couple of examples, depicted in Figure~\ref{fig:friendly-faced}. Here, a fixed planar graph $\Gamma$ is shown twice, with two subgraphs $\Lambda_1$ and $\Lambda_2$ highlighted in red and blue, respectively. 

    The first example, $\Lambda_1$, is not friendly-faced. Notice that $\Lambda_1$ has one infinite face $f \in \inffaces (\Lambda_1)$, but there is no face of $\Gamma$ whose boundary facial subgraph lies close to all of $f$.

    In the second example, $\Lambda_2$ is friendly-faced. Indeed, the two faces of $\Lambda_2$ which are not already faces of $\Gamma$ are contained in some finite neighbourhood of the two infinite faces of $\Gamma$. 
\end{example}

\begin{figure}
    \centering
    \input{figs/friendly-faced}
    \caption{Non-example and example of friendly-faced subgraphs of a given graph $\Gamma$.}
    \label{fig:friendly-faced}
\end{figure}

An easy but important example of a friendly-faced subgraph is the \textit{2-connected core} of an \textit{almost 2-connected graph} (recall Definition~\ref{def:nearly-2-conn}).

\begin{proposition}
    Let $\Gamma$ be a locally finite, almost 2-connected planar graph with fixed good drawing. Let $\Gamma_0 \subset \Gamma$ be its 2-connected core. Then $\Gamma_0$ is a friendly-faced subgraph of $\Gamma$. 
\end{proposition}

\begin{proof}
    Left as an exercise to the reader. 
\end{proof}

\begin{remark}\label{rmk:simple-body}
    It is possible to give quite a concrete description of the facial subgraphs of an almost 2-connected, locally finite planar graph $\Gamma$. 

    Let $\Gamma_0$ be its 2-connected core. The closure of every facial subgraph of $\Gamma_0$ in the Freudenthal compactification is a simple closed curve by Proposition~\ref{prop:simple-face}. Let $U \in \facedisks(\Gamma)$, and let $U_0 \in \facedisks (\Gamma_0)$ be the unique element such that $U_0 \supset U$. Assume that $\diam_\Gamma(f)$ is sufficiently large (possibly infinite). Then $f = \facepaths [U]$ is obtained from $f_0 = \facepaths[U_0]$ by attaching boundedly small cacti graphs to $f_0$ at cut vertices. We may refer to $f_0$ as the \textit{simple body} of $f$, and the components of $f \setminus f_0$ as the \textit{adjoined cacti} of $f$. See Figure~\ref{fig:almost-2conn-face} for cartoon. 
\end{remark}

\begin{figure}
    \centering

\tikzset{every picture/.style={line width=0.75pt}} 

\begin{tikzpicture}[x=0.75pt,y=0.75pt,yscale=-1,xscale=1]

\draw  [draw opacity=0][fill={rgb, 255:red, 155; green, 155; blue, 155 }  ,fill opacity=0.46 ] (190.5,102.25) .. controls (190.5,54.89) and (228.89,16.5) .. (276.25,16.5) .. controls (323.61,16.5) and (362,54.89) .. (362,102.25) .. controls (362,149.61) and (323.61,188) .. (276.25,188) .. controls (228.89,188) and (190.5,149.61) .. (190.5,102.25) -- cycle ;
\draw  [color={rgb, 255:red, 0; green, 0; blue, 0 }  ,draw opacity=1 ][fill={rgb, 255:red, 255; green, 255; blue, 255 }  ,fill opacity=1 ] (208.09,102.25) .. controls (208.09,64.6) and (238.6,34.09) .. (276.25,34.09) .. controls (313.9,34.09) and (344.41,64.6) .. (344.41,102.25) .. controls (344.41,139.9) and (313.9,170.41) .. (276.25,170.41) .. controls (238.6,170.41) and (208.09,139.9) .. (208.09,102.25) -- cycle ;
\draw    (243.89,162.15) -- (246.27,157.02) ;
\draw  [fill={rgb, 255:red, 155; green, 155; blue, 155 }  ,fill opacity=0.53 ] (244.87,153.61) .. controls (245.42,152.28) and (246.95,151.65) .. (248.28,152.21) .. controls (249.61,152.76) and (250.23,154.29) .. (249.68,155.62) .. controls (249.12,156.95) and (247.6,157.57) .. (246.27,157.02) .. controls (244.94,156.46) and (244.31,154.94) .. (244.87,153.61) -- cycle ;
\draw  [fill={rgb, 255:red, 155; green, 155; blue, 155 }  ,fill opacity=0.53 ] (247.04,146.92) .. controls (247.34,146.2) and (248.17,145.86) .. (248.88,146.16) .. controls (249.6,146.46) and (249.94,147.28) .. (249.64,148) .. controls (249.34,148.72) and (248.52,149.06) .. (247.8,148.76) .. controls (247.08,148.46) and (246.74,147.63) .. (247.04,146.92) -- cycle ;
\draw    (248.28,152.21) -- (247.8,148.76) ;
\draw  [fill={rgb, 255:red, 155; green, 155; blue, 155 }  ,fill opacity=0.53 ] (254.59,151.25) .. controls (255.36,151.33) and (255.93,152.02) .. (255.85,152.8) .. controls (255.77,153.57) and (255.08,154.13) .. (254.3,154.05) .. controls (253.53,153.97) and (252.97,153.28) .. (253.05,152.51) .. controls (253.13,151.74) and (253.82,151.17) .. (254.59,151.25) -- cycle ;
\draw    (249.88,153.97) -- (253.05,152.51) ;

\draw    (234.07,156.08) -- (236.43,153.46) ;
\draw  [fill={rgb, 255:red, 155; green, 155; blue, 155 }  ,fill opacity=0.53 ] (236.21,151.17) .. controls (236.79,150.48) and (237.81,150.38) .. (238.5,150.96) .. controls (239.19,151.53) and (239.29,152.55) .. (238.71,153.24) .. controls (238.14,153.93) and (237.12,154.03) .. (236.43,153.46) .. controls (235.74,152.88) and (235.64,151.86) .. (236.21,151.17) -- cycle ;
\draw  [fill={rgb, 255:red, 155; green, 155; blue, 155 }  ,fill opacity=0.53 ] (238.73,147.58) .. controls (239.04,147.21) and (239.59,147.15) .. (239.96,147.46) .. controls (240.34,147.77) and (240.39,148.33) .. (240.08,148.7) .. controls (239.77,149.07) and (239.22,149.12) .. (238.84,148.81) .. controls (238.47,148.5) and (238.42,147.95) .. (238.73,147.58) -- cycle ;
\draw    (238.5,150.96) -- (238.84,148.81) ;
\draw  [fill={rgb, 255:red, 155; green, 155; blue, 155 }  ,fill opacity=0.53 ] (242.44,151.54) .. controls (242.88,151.73) and (243.09,152.24) .. (242.91,152.69) .. controls (242.72,153.13) and (242.2,153.34) .. (241.76,153.15) .. controls (241.31,152.97) and (241.1,152.45) .. (241.29,152.01) .. controls (241.48,151.56) and (241.99,151.35) .. (242.44,151.54) -- cycle ;
\draw    (239.14,152.3) -- (241.29,152.01) ;
\draw    (227.62,150.32) -- (229.16,148.6) ;
\draw  [fill={rgb, 255:red, 155; green, 155; blue, 155 }  ,fill opacity=0.53 ] (229.02,147.1) .. controls (229.4,146.65) and (230.07,146.58) .. (230.52,146.96) .. controls (230.98,147.34) and (231.04,148.01) .. (230.66,148.46) .. controls (230.29,148.91) and (229.61,148.98) .. (229.16,148.6) .. controls (228.71,148.22) and (228.65,147.55) .. (229.02,147.1) -- cycle ;
\draw  [fill={rgb, 255:red, 155; green, 155; blue, 155 }  ,fill opacity=0.53 ] (230.67,144.74) .. controls (230.88,144.5) and (231.24,144.46) .. (231.48,144.67) .. controls (231.73,144.87) and (231.76,145.23) .. (231.56,145.48) .. controls (231.36,145.72) and (230.99,145.75) .. (230.75,145.55) .. controls (230.5,145.35) and (230.47,144.99) .. (230.67,144.74) -- cycle ;
\draw    (230.52,146.96) -- (230.75,145.55) ;
\draw  [fill={rgb, 255:red, 155; green, 155; blue, 155 }  ,fill opacity=0.53 ] (233.11,147.34) .. controls (233.4,147.46) and (233.54,147.8) .. (233.42,148.1) .. controls (233.29,148.39) and (232.95,148.53) .. (232.66,148.4) .. controls (232.37,148.28) and (232.23,147.94) .. (232.35,147.65) .. controls (232.48,147.36) and (232.82,147.22) .. (233.11,147.34) -- cycle ;
\draw    (230.94,147.84) -- (232.35,147.65) ;
\draw    (222.86,144.83) -- (224.54,143.82) ;
\draw  [fill={rgb, 255:red, 155; green, 155; blue, 155 }  ,fill opacity=0.53 ] (224.8,142.57) .. controls (225.22,142.3) and (225.78,142.41) .. (226.05,142.83) .. controls (226.32,143.25) and (226.2,143.81) .. (225.79,144.08) .. controls (225.37,144.35) and (224.81,144.24) .. (224.54,143.82) .. controls (224.27,143.4) and (224.38,142.84) .. (224.8,142.57) -- cycle ;
\draw  [fill={rgb, 255:red, 155; green, 155; blue, 155 }  ,fill opacity=0.53 ] (226.72,141.08) .. controls (226.95,140.93) and (227.25,140.99) .. (227.4,141.22) .. controls (227.55,141.44) and (227.48,141.74) .. (227.26,141.89) .. controls (227.03,142.04) and (226.73,141.98) .. (226.58,141.75) .. controls (226.44,141.52) and (226.5,141.22) .. (226.72,141.08) -- cycle ;
\draw    (226.05,142.83) -- (226.58,141.75) ;
\draw  [fill={rgb, 255:red, 155; green, 155; blue, 155 }  ,fill opacity=0.53 ] (228.04,143.79) .. controls (228.25,143.96) and (228.28,144.27) .. (228.1,144.47) .. controls (227.93,144.68) and (227.62,144.7) .. (227.42,144.53) .. controls (227.21,144.36) and (227.18,144.05) .. (227.36,143.85) .. controls (227.53,143.64) and (227.84,143.61) .. (228.04,143.79) -- cycle ;
\draw    (226.17,143.65) -- (227.36,143.85) ;
\draw    (283.43,33.82) -- (285.1,47) ;
\draw  [fill={rgb, 255:red, 155; green, 155; blue, 155 }  ,fill opacity=0.53 ] (292.14,52.03) .. controls (292.7,55.37) and (290.44,58.52) .. (287.11,59.07) .. controls (283.78,59.63) and (280.62,57.37) .. (280.07,54.04) .. controls (279.52,50.7) and (281.77,47.55) .. (285.1,47) .. controls (288.44,46.44) and (291.59,48.7) .. (292.14,52.03) -- cycle ;
\draw  [fill={rgb, 255:red, 155; green, 155; blue, 155 }  ,fill opacity=0.53 ] (296.17,68.06) .. controls (296.47,69.86) and (295.25,71.56) .. (293.46,71.86) .. controls (291.66,72.16) and (289.95,70.94) .. (289.65,69.14) .. controls (289.36,67.34) and (290.57,65.64) .. (292.37,65.34) .. controls (294.17,65.04) and (295.87,66.26) .. (296.17,68.06) -- cycle ;
\draw    (287.11,59.07) -- (292.37,65.34) ;
\draw  [fill={rgb, 255:red, 155; green, 155; blue, 155 }  ,fill opacity=0.53 ] (275.74,68.85) .. controls (274.11,69.66) and (272.13,68.99) .. (271.32,67.35) .. controls (270.51,65.72) and (271.18,63.74) .. (272.82,62.93) .. controls (274.46,62.12) and (276.44,62.79) .. (277.24,64.43) .. controls (278.05,66.06) and (277.38,68.05) .. (275.74,68.85) -- cycle ;
\draw    (281.72,57.58) -- (277.24,64.43) ;

\draw    (280.02,51.62) -- (268.48,58.21) ;
\draw  [fill={rgb, 255:red, 155; green, 155; blue, 155 }  ,fill opacity=0.53 ] (266.53,66.64) .. controls (263.66,68.43) and (259.88,67.56) .. (258.1,64.69) .. controls (256.31,61.82) and (257.18,58.05) .. (260.05,56.26) .. controls (262.92,54.47) and (266.69,55.34) .. (268.48,58.21) .. controls (270.27,61.08) and (269.39,64.85) .. (266.53,66.64) -- cycle ;
\draw  [fill={rgb, 255:red, 155; green, 155; blue, 155 }  ,fill opacity=0.53 ] (253.26,76.5) .. controls (251.72,77.46) and (249.68,76.99) .. (248.71,75.44) .. controls (247.75,73.9) and (248.22,71.86) .. (249.77,70.89) .. controls (251.32,69.93) and (253.35,70.4) .. (254.32,71.95) .. controls (255.28,73.5) and (254.81,75.53) .. (253.26,76.5) -- cycle ;
\draw    (258.1,64.69) -- (254.32,71.95) ;
\draw  [fill={rgb, 255:red, 155; green, 155; blue, 155 }  ,fill opacity=0.53 ] (244.71,57.93) .. controls (243.34,56.73) and (243.2,54.64) .. (244.4,53.27) .. controls (245.6,51.9) and (247.69,51.76) .. (249.06,52.96) .. controls (250.43,54.17) and (250.57,56.25) .. (249.37,57.62) .. controls (248.17,59) and (246.08,59.13) .. (244.71,57.93) -- cycle ;
\draw    (257.41,59.14) -- (249.37,57.62) ;

\draw    (327.66,146.31) -- (320.52,142.3) ;
\draw  [fill={rgb, 255:red, 155; green, 155; blue, 155 }  ,fill opacity=0.53 ] (315.44,143.93) .. controls (313.59,142.97) and (312.86,140.7) .. (313.81,138.85) .. controls (314.76,136.99) and (317.04,136.26) .. (318.89,137.22) .. controls (320.74,138.17) and (321.47,140.44) .. (320.52,142.3) .. controls (319.57,144.15) and (317.3,144.88) .. (315.44,143.93) -- cycle ;
\draw  [fill={rgb, 255:red, 155; green, 155; blue, 155 }  ,fill opacity=0.53 ] (306.04,140.02) .. controls (305.04,139.51) and (304.64,138.28) .. (305.16,137.28) .. controls (305.67,136.28) and (306.9,135.89) .. (307.9,136.4) .. controls (308.9,136.92) and (309.29,138.14) .. (308.78,139.14) .. controls (308.26,140.14) and (307.04,140.54) .. (306.04,140.02) -- cycle ;
\draw    (313.81,138.85) -- (308.78,139.14) ;
\draw  [fill={rgb, 255:red, 155; green, 155; blue, 155 }  ,fill opacity=0.53 ] (313.16,129.63) .. controls (313.36,128.52) and (314.42,127.79) .. (315.52,127.99) .. controls (316.63,128.19) and (317.36,129.26) .. (317.16,130.36) .. controls (316.96,131.47) and (315.9,132.2) .. (314.79,132) .. controls (313.69,131.8) and (312.95,130.73) .. (313.16,129.63) -- cycle ;
\draw    (316.53,136.73) -- (314.79,132) ;

\draw (213.64,132.78) node [anchor=north west][inner sep=0.75pt]  [font=\tiny,rotate=-327.68]  {$\vdots $};
\draw (281.2,101) node [anchor=north west][inner sep=0.75pt]  [color={rgb, 255:red, 155; green, 155; blue, 155 }  ,opacity=0.9 ]  {$U$};
\draw (213.66,87.9) node [anchor=north west][inner sep=0.75pt]  [color={rgb, 255:red, 0; green, 0; blue, 0 }  ,opacity=1 ]  {$\facepaths[U]$};

\end{tikzpicture}

    \caption{The simple body and adjoined cacti of a facial subgraph of an almost 2-connected planar graph. Each adjoined cactus has bounded diameter.}
    \label{fig:almost-2conn-face}
\end{figure}

We also note the following lemma, which sheds some light on how the friendly-faced property interacts with chains on subgraphs. 

\begin{lemma}\label{lem:friendly-faces-chain}
    Let $\Gamma$ be a locally finite, connected planar graph. Let $\Lambda \subset \Pi \subset \Gamma$ be connected subgraphs, such that 
    \begin{enumerate}
        \item $\Pi$ has uniform coboundary in $\Gamma$,
        \item $\Lambda$ is a friendly-faced subgraph of $\Gamma$. 
    \end{enumerate}
    Then $\Lambda$ is also a friendly-faced subgraph of $\Pi$. 
\end{lemma}

\begin{proof}
    Since $\Pi$ has uniform coboundary in $\Gamma$, it is clear that $\vartheta$ restricts to a good drawing of $\Pi$. Note also that $\Pi$ is quasi-isometrically embedded into $\Gamma$, by Proposition~\ref{prop:uniform-boundary-qie}. 

    Let $U \in \facedisks(\Lambda)$. Using the fact that $\Lambda$ is a friendly-faced subgraph of $\Gamma$, let $W \in \facedisks(\Gamma)$ be such that $W \subset U$ and 
    $$
    \facepaths[U] \subset B_\Gamma(\facepaths[W]; \varepsilon), 
    $$
    for some uniform $\varepsilon > 0$. Let $V \in \facedisks(\Pi)$ be the unique face such that $U \supset V \supset W$. 

    Let $x \in \facepaths[U]$. Then there exists a path of length at most $\varepsilon$ in $\Gamma$ connecting $x$ to some point in $\facepaths[W]$. But clearly this path must intersect $\facepaths[V]$. Since $\Pi$ is quasi-isometrically embedded in $\Gamma$, there exists a path of bounded length in $\Pi$ connecting $x$ to $\facepaths[V]$. Since $x$ was arbitrary, it follows immediately that $\Lambda$ is a friendly-faced subgraph of $\Pi$. 
\end{proof}

In the presence of an appropriate group action, it is often possible to `thicken up' a subgraph into a friendly-faced subgraph. In particular, it is possible to prove the following.

\begin{proposition}\label{prop:ff-cocompact}
    Let $\Gamma$ be a connected, locally finite, quasi-transitive, planar graph. Let $\Lambda \subset \Gamma$ be a 2-connected, quasi-transitively stabilised subgraph with uniform coboundary. Then there exists some $\delta > 0$ such that $B_\Gamma(\Lambda;\delta)$ is a friendly-faced subgraph of $\Gamma$. 
\end{proposition}

We will need to prove a stronger, more coarse version of this. 
To this end, the rest of this section is devoted to proving the following.

\begin{restatable}[Neighbourhoods with friendly faces]{theorem}{friendlyfaces}\label{thm:friendly-faces}
    Let $X$ be a connected, locally finite, quasi-transitive graph and $Y \subset X$ a connected, quasi-transitively stabilised subgraph with uniform coboundary. Let $\Gamma$ be a connected, bounded-degree, planar graph, and $\varphi : X \onto \Gamma$ a continuous, surjective quasi-isometry. Suppose further that $\varphi(Y)$ is almost 2-connected. 
    Then there exists $\delta > 0$ such that 
    $$
    \varphi(B_X(Y;\delta))
    $$
    is a friendly-faced subgraph of $\Gamma$. 
\end{restatable}

Note that if we set $X = \Gamma$ and take $\varphi$ to be the identity map, then we recover Proposition~\ref{prop:ff-cocompact}. 
The reader short on time may  proceed to the next section, if they so wish. The rest of \S\ref{sec:friendly-faces} is devoted to proving Theorem~\ref{thm:friendly-faces}, and nothing beyond this point in this section will make a second appearance.

\subsection{Neighbourhoods with friendly faces}

We now work towards a proof of Theorem~\ref{thm:friendly-faces}. Most of the heavy lifting will be done by the next lemma. 

\begin{lemma}\label{lem:nice-faces}
    Let $\Gamma$ be a bounded-degree, connected, planar graph with fixed good drawing $\vartheta$. Let $\Lambda \subset \Gamma$ be an almost 2-connected subgraph with uniform coboundary.  
    Then there exists $r > 0$ such that if $\Pi \subset \Gamma$ is a connected subgraph satisfying
    $$
    B_\Gamma(\Lambda;r) \subset \Pi \ \ \ \text{and} \ \ \  \dHaus[\Gamma](\Lambda, \Pi) < \infty, 
    $$
    then $\Pi$ is a friendly-faced subgraph of $\Gamma$. 
\end{lemma}

\begin{proof}
    Since $\Lambda$ has uniform coboundary, we know by Proposition~\ref{prop:end-injective} that the inclusions $\Lambda \into \Gamma$, $\Pi \into \Gamma$ induce injections on the sets of ends, and so $\vartheta$ restricts to a good drawing of both $\Lambda$ and $\Pi$. Thus, it makes complete sense for us to speak of the faces of these subgraphs. 

    Let $r = \outcut{\Lambda}^\Gamma$, and let $\Pi \subset \Gamma$ be such that 
    $$
    B_\Gamma(\Lambda;r) \subset \Pi \ \ \ \text{and} \ \ \  \dHaus[\Gamma](\Lambda, \Pi) < \infty. 
    $$
    Let $U_1 \in \facedisks(\Pi)$, and let $U_0$ be the unique element in $\facedisks(\Lambda)$ such that $U_1 \subset U_0$. 
    
    Since $\Lambda$ is almost 2-connected, we have that $\facepaths[U_0]$ is formed of a \textit{simple body} and \textit{adjoined cacti} (see Remark~\ref{rmk:simple-body} for a description of the faces of an almost 2-connected planar graph).

    \begin{claim}\label{claim:common-faces}
         Let $x,y \in \facepaths[U_1] \cap \Lambda$. Then there exists $U_2 \in \facedisks(\Gamma)$ such that $U_2 \subset U_1$ and $x,y \in \facepaths[U_2]$. 
    \end{claim}
    
    \begin{proof}
    Clearly $x,y \in \facepaths [U_0]$. 
    Suppose for the sake of a contradiction that such a $U_2 \in \facedisks(\Gamma)$ does not exist. Then there must exist a simple path $p$ in $\Gamma$ such that the initial and terminal vertices of $p$ lie on $f_0$, and $p$ is otherwise drawn entirely within $U_0$, and also such that $\vartheta(x)$, $\vartheta(y)$ do not lie in the closure of a common component of $U_0 \setminus \vartheta(p)$. See Figure~\ref{fig:separating-path} for a cartoon. 

    Let $u$ be the second vertex to appear along $p$, and $v$ be the penultimate. In particular, $u$ and $v$ are adjacent to $\Lambda$, and are contained in the same connected component $C$ of $\Gamma \setminus \Lambda$. Note that $\vartheta(C) \subset U$. Since $\outcut{\Lambda}^\Gamma = r$, we have that $\dist_C(x,y) \leq r$. Thus, there exists a path of bounded length through $C$ connecting $x$ to $y$. In particular, we may assume that the path $p$ we constructed above has length at most $r+2$. This path will be contained in $\Pi$. But this contradicts the assumption that $x$ and $y$ both lie on the facial subgraph bordering $U_1$, where $U_1 \in \facedisks(\Pi)$ was such that $U_1 \subset U_0$. Thus, the claim follows. 
    \end{proof}

    We can extend the above claim from two vertices to arbitrarily many, via the following technical statement. 

    \begin{claim}\label{claim:tripod}
        There exists some uniform $N > 0$ depending only on $\Lambda$ such that the following holds. 
        Let $S$ be a set of vertices in $\facepaths[U_0]$ with $|S| > N$. Suppose $S$ satisfies the following: 
        \begin{enumerate}
            \item[$(\dagger)$] For any two $x,y \in S$ there exists $U \in \facedisks(\Gamma)$ such that $U \subset U_0$ and $x, y \in \facepaths [U]$.
        \end{enumerate}
        Then there exists $U_2 \in \facedisks(\Gamma)$ such that $U_2 \subset U_0$ and $S \subset \facepaths[U_2]$. 
    \end{claim}

    \begin{proof}
        Assume $N > 3$. 
        Suppose to the contrary that this were not true. Then it is easy to see that there must exist a tripod $T$ embedded in $\overline \Gamma$ with the following properties: 
        \begin{enumerate}
            \item The leaves of $T$ lie in $\facepaths[U_0]$, and $T$ otherwise maps into $U_0$ under $\vartheta$,

            \item $U_0 \setminus \vartheta(T)$ contains exactly three connected components $V_1$, $V_2$, $V_3$,

            \item There exists three distinct $x_1, x_2, x_3 \in S$ such that 
            $$
            \vartheta(x_i) \in (\overline {V_i} \cap \overline {V_{i+1}}) \setminus \overline {V_{i+2}},
            $$
            where indices are taken modulo 3. 
        \end{enumerate}
        See Figure~\ref{fig:face-tripod} for a cartoon.

        Now, let us proceed to add more elements from $S$ to this picture while preserving ($\dagger$). In particular, it is not hard to see that any $y \in S$ must satisfy one of the following properties:
        \begin{enumerate}
            \item Either $y$ lies on the simple body of $\facepaths[U_0]$, and coincides precisely with a leaf of $T$, or

            \item The vertex $y$ lies on an adjoined cactus $C$ of $\facepaths[U_0]$, and $C$ also contains a leaf of $T$.
        \end{enumerate}
        If neither of these happens, then there must exist some $x_i$ such that $y$ and $x_i$ do not lie on a common subface of $U_0$, and thus we contradict ($\dagger$).

        Finally, note that an adjoined cactus can only contain boundedly many vertices since $\Gamma$ is bounded-degree, say at most $m$ vertices. Setting $N = 3m$, the claim follows. 
    \end{proof}

    We now show that $\Pi$ is friendly-faced. 
    Note that since $\Lambda$ has uniform coboundary, $\Gamma$ is bounded-degree, and $\dHaus[\Gamma](\Pi, \Lambda)$ is finite, we see that each connected component of $\Pi \setminus \Lambda$ has at most boundedly many vertices. 
    Thus, if $\diam_{\Pi}(\facepaths[U_1])$ is sufficiently large then we must have that 
    $$
    \facepaths[U_1] \subset B_{\Pi}(\facepaths[U_1] \cap \Lambda;m), 
    $$
    where $m > 0$ is some uniform constant. We also assume that $\facepaths[U_1]$ is large enough so that $\facepaths[U_1] \cap \Lambda$ contains more than $N$ vertices, where $N > 0$ is as in Claim~\ref{claim:tripod}. Now, by Claims~\ref{claim:common-faces} and \ref{claim:tripod}, there exists some $U_2 \in \facedisks(\Gamma)$ such that $U_2 \subset U_1$, and 
    $$
    \facepaths[U_1] \cap \Lambda \subset \facepaths[U_2]. 
    $$
    We thus deduce that $\facepaths[U_1]$ lies inside of a bounded neighbourhood of $\facepaths[U_2]$ in $\Gamma$. Since $U_1$ was an arbitrary face of $\Lambda$ with $\facepaths[U_1]$ sufficiently large, it follows that $\Pi$ is friendly-faced. 
\end{proof}

We are now able to prove Theorem~\ref{thm:friendly-faces}, which we restate for the reader's convenience. 

\friendlyfaces*

\begin{proof}[Proof]
    By Lemmas~\ref{lem:cocompact-in-to-uniform} and \ref{lem:small-increase-still-bounded} we have that $B_X(Y;\varepsilon)$ has uniform coboundary for every $\varepsilon > 0$. By Theorem~\ref{thm:uniform-boundary-qi}, there exists some $\delta' > 0$ such that $\varphi(B_X(Y;\delta'))$ has uniform coboundary in $\Gamma$. Now, combining Lemma~\ref{lem:onto-qi} with Lemmas~\ref{lem:nice-faces},  we deduce that there is some $\delta > \delta'$ such that $\varphi(B_X(Y;\delta))$ is friendly-faced.
\end{proof}

\begin{figure}
\begin{minipage}{.5\textwidth}
  \centering

\tikzset{every picture/.style={line width=0.75pt}} 

\begin{tikzpicture}[x=0.75pt,y=0.75pt,yscale=-1,xscale=1]

\draw  [draw opacity=0][fill={rgb, 255:red, 155; green, 155; blue, 155 }  ,fill opacity=0.46 ] (402.1,116.25) .. controls (402.1,68.89) and (440.49,30.5) .. (487.85,30.5) .. controls (535.21,30.5) and (573.6,68.89) .. (573.6,116.25) .. controls (573.6,163.61) and (535.21,202) .. (487.85,202) .. controls (440.49,202) and (402.1,163.61) .. (402.1,116.25) -- cycle ;
\draw  [color={rgb, 255:red, 0; green, 0; blue, 0 }  ,draw opacity=1 ][fill={rgb, 255:red, 255; green, 255; blue, 255 }  ,fill opacity=1 ] (419.69,116.25) .. controls (419.69,78.6) and (450.2,48.09) .. (487.85,48.09) .. controls (525.5,48.09) and (556.01,78.6) .. (556.01,116.25) .. controls (556.01,153.9) and (525.5,184.41) .. (487.85,184.41) .. controls (450.2,184.41) and (419.69,153.9) .. (419.69,116.25) -- cycle ;
\draw    (455.49,176.15) -- (457.87,171.02) ;
\draw  [fill={rgb, 255:red, 155; green, 155; blue, 155 }  ,fill opacity=0.53 ] (456.47,167.61) .. controls (457.02,166.28) and (458.55,165.65) .. (459.88,166.21) .. controls (461.21,166.76) and (461.83,168.29) .. (461.28,169.62) .. controls (460.72,170.95) and (459.2,171.57) .. (457.87,171.02) .. controls (456.54,170.46) and (455.91,168.94) .. (456.47,167.61) -- cycle ;
\draw  [fill={rgb, 255:red, 155; green, 155; blue, 155 }  ,fill opacity=0.53 ] (458.64,160.92) .. controls (458.94,160.2) and (459.77,159.86) .. (460.48,160.16) .. controls (461.2,160.46) and (461.54,161.28) .. (461.24,162) .. controls (460.94,162.72) and (460.12,163.06) .. (459.4,162.76) .. controls (458.68,162.46) and (458.34,161.63) .. (458.64,160.92) -- cycle ;
\draw    (459.88,166.21) -- (459.4,162.76) ;
\draw  [fill={rgb, 255:red, 155; green, 155; blue, 155 }  ,fill opacity=0.53 ] (466.19,165.25) .. controls (466.96,165.33) and (467.53,166.02) .. (467.45,166.8) .. controls (467.37,167.57) and (466.68,168.13) .. (465.9,168.05) .. controls (465.13,167.97) and (464.57,167.28) .. (464.65,166.51) .. controls (464.73,165.74) and (465.42,165.17) .. (466.19,165.25) -- cycle ;
\draw    (461.48,167.97) -- (464.65,166.51) ;

\draw    (445.67,170.08) -- (448.03,167.46) ;
\draw  [fill={rgb, 255:red, 155; green, 155; blue, 155 }  ,fill opacity=0.53 ] (447.81,165.17) .. controls (448.39,164.48) and (449.41,164.38) .. (450.1,164.96) .. controls (450.79,165.53) and (450.89,166.55) .. (450.31,167.24) .. controls (449.74,167.93) and (448.72,168.03) .. (448.03,167.46) .. controls (447.34,166.88) and (447.24,165.86) .. (447.81,165.17) -- cycle ;
\draw  [fill={rgb, 255:red, 155; green, 155; blue, 155 }  ,fill opacity=0.53 ] (450.33,161.58) .. controls (450.64,161.21) and (451.19,161.15) .. (451.56,161.46) .. controls (451.94,161.77) and (451.99,162.33) .. (451.68,162.7) .. controls (451.37,163.07) and (450.82,163.12) .. (450.44,162.81) .. controls (450.07,162.5) and (450.02,161.95) .. (450.33,161.58) -- cycle ;
\draw    (450.1,164.96) -- (450.44,162.81) ;
\draw  [fill={rgb, 255:red, 155; green, 155; blue, 155 }  ,fill opacity=0.53 ] (454.04,165.54) .. controls (454.48,165.73) and (454.69,166.24) .. (454.51,166.69) .. controls (454.32,167.13) and (453.8,167.34) .. (453.36,167.15) .. controls (452.91,166.97) and (452.7,166.45) .. (452.89,166.01) .. controls (453.08,165.56) and (453.59,165.35) .. (454.04,165.54) -- cycle ;
\draw    (450.74,166.3) -- (452.89,166.01) ;
\draw    (439.22,164.32) -- (440.76,162.6) ;
\draw  [fill={rgb, 255:red, 155; green, 155; blue, 155 }  ,fill opacity=0.53 ] (440.62,161.1) .. controls (441,160.65) and (441.67,160.58) .. (442.12,160.96) .. controls (442.58,161.34) and (442.64,162.01) .. (442.26,162.46) .. controls (441.89,162.91) and (441.21,162.98) .. (440.76,162.6) .. controls (440.31,162.22) and (440.25,161.55) .. (440.62,161.1) -- cycle ;
\draw  [fill={rgb, 255:red, 155; green, 155; blue, 155 }  ,fill opacity=0.53 ] (442.27,158.74) .. controls (442.48,158.5) and (442.84,158.46) .. (443.08,158.67) .. controls (443.33,158.87) and (443.36,159.23) .. (443.16,159.48) .. controls (442.96,159.72) and (442.59,159.75) .. (442.35,159.55) .. controls (442.1,159.35) and (442.07,158.99) .. (442.27,158.74) -- cycle ;
\draw    (442.12,160.96) -- (442.35,159.55) ;
\draw  [fill={rgb, 255:red, 155; green, 155; blue, 155 }  ,fill opacity=0.53 ] (444.71,161.34) .. controls (445,161.46) and (445.14,161.8) .. (445.02,162.1) .. controls (444.89,162.39) and (444.55,162.53) .. (444.26,162.4) .. controls (443.97,162.28) and (443.83,161.94) .. (443.95,161.65) .. controls (444.08,161.36) and (444.42,161.22) .. (444.71,161.34) -- cycle ;
\draw    (442.54,161.84) -- (443.95,161.65) ;
\draw    (434.46,158.83) -- (436.14,157.82) ;
\draw  [fill={rgb, 255:red, 155; green, 155; blue, 155 }  ,fill opacity=0.53 ] (436.4,156.57) .. controls (436.82,156.3) and (437.38,156.41) .. (437.65,156.83) .. controls (437.92,157.25) and (437.8,157.81) .. (437.39,158.08) .. controls (436.97,158.35) and (436.41,158.24) .. (436.14,157.82) .. controls (435.87,157.4) and (435.98,156.84) .. (436.4,156.57) -- cycle ;
\draw  [fill={rgb, 255:red, 155; green, 155; blue, 155 }  ,fill opacity=0.53 ] (438.32,155.08) .. controls (438.55,154.93) and (438.85,154.99) .. (439,155.22) .. controls (439.15,155.44) and (439.08,155.74) .. (438.86,155.89) .. controls (438.63,156.04) and (438.33,155.98) .. (438.18,155.75) .. controls (438.04,155.52) and (438.1,155.22) .. (438.32,155.08) -- cycle ;
\draw    (437.65,156.83) -- (438.18,155.75) ;
\draw  [fill={rgb, 255:red, 155; green, 155; blue, 155 }  ,fill opacity=0.53 ] (439.64,157.79) .. controls (439.85,157.96) and (439.88,158.27) .. (439.7,158.47) .. controls (439.53,158.68) and (439.22,158.7) .. (439.02,158.53) .. controls (438.81,158.36) and (438.78,158.05) .. (438.96,157.85) .. controls (439.13,157.64) and (439.44,157.61) .. (439.64,157.79) -- cycle ;
\draw    (437.77,157.65) -- (438.96,157.85) ;
\draw    (495.03,47.82) -- (496.7,61) ;
\draw  [fill={rgb, 255:red, 155; green, 155; blue, 155 }  ,fill opacity=0.53 ] (503.74,66.03) .. controls (504.3,69.37) and (502.04,72.52) .. (498.71,73.07) .. controls (495.38,73.63) and (492.22,71.37) .. (491.67,68.04) .. controls (491.12,64.7) and (493.37,61.55) .. (496.7,61) .. controls (500.04,60.44) and (503.19,62.7) .. (503.74,66.03) -- cycle ;
\draw  [fill={rgb, 255:red, 155; green, 155; blue, 155 }  ,fill opacity=0.53 ] (507.77,82.06) .. controls (508.07,83.86) and (506.85,85.56) .. (505.06,85.86) .. controls (503.26,86.16) and (501.55,84.94) .. (501.25,83.14) .. controls (500.96,81.34) and (502.17,79.64) .. (503.97,79.34) .. controls (505.77,79.04) and (507.47,80.26) .. (507.77,82.06) -- cycle ;
\draw    (498.71,73.07) -- (503.97,79.34) ;
\draw  [fill={rgb, 255:red, 155; green, 155; blue, 155 }  ,fill opacity=0.53 ] (487.34,82.85) .. controls (485.71,83.66) and (483.73,82.99) .. (482.92,81.35) .. controls (482.11,79.72) and (482.78,77.74) .. (484.42,76.93) .. controls (486.06,76.12) and (488.04,76.79) .. (488.84,78.43) .. controls (489.65,80.06) and (488.98,82.05) .. (487.34,82.85) -- cycle ;
\draw    (493.32,71.58) -- (488.84,78.43) ;

\draw    (491.62,65.62) -- (480.08,72.21) ;
\draw  [fill={rgb, 255:red, 155; green, 155; blue, 155 }  ,fill opacity=0.53 ] (478.13,80.64) .. controls (475.26,82.43) and (471.48,81.56) .. (469.7,78.69) .. controls (467.91,75.82) and (468.78,72.05) .. (471.65,70.26) .. controls (474.52,68.47) and (478.29,69.34) .. (480.08,72.21) .. controls (481.87,75.08) and (480.99,78.85) .. (478.13,80.64) -- cycle ;
\draw  [fill={rgb, 255:red, 155; green, 155; blue, 155 }  ,fill opacity=0.53 ] (464.86,90.5) .. controls (463.32,91.46) and (461.28,90.99) .. (460.31,89.44) .. controls (459.35,87.9) and (459.82,85.86) .. (461.37,84.89) .. controls (462.92,83.93) and (464.95,84.4) .. (465.92,85.95) .. controls (466.88,87.5) and (466.41,89.53) .. (464.86,90.5) -- cycle ;
\draw    (469.7,78.69) -- (465.92,85.95) ;
\draw  [fill={rgb, 255:red, 155; green, 155; blue, 155 }  ,fill opacity=0.53 ] (456.31,71.93) .. controls (454.94,70.73) and (454.8,68.64) .. (456,67.27) .. controls (457.2,65.9) and (459.29,65.76) .. (460.66,66.96) .. controls (462.03,68.17) and (462.17,70.25) .. (460.97,71.62) .. controls (459.77,73) and (457.68,73.13) .. (456.31,71.93) -- cycle ;
\draw    (469.01,73.14) -- (460.97,71.62) ;

\draw    (539.26,160.31) -- (532.12,156.3) ;
\draw  [fill={rgb, 255:red, 155; green, 155; blue, 155 }  ,fill opacity=0.53 ] (527.04,157.93) .. controls (525.19,156.97) and (524.46,154.7) .. (525.41,152.85) .. controls (526.36,150.99) and (528.64,150.26) .. (530.49,151.22) .. controls (532.34,152.17) and (533.07,154.44) .. (532.12,156.3) .. controls (531.17,158.15) and (528.9,158.88) .. (527.04,157.93) -- cycle ;
\draw  [fill={rgb, 255:red, 155; green, 155; blue, 155 }  ,fill opacity=0.53 ] (517.64,154.02) .. controls (516.64,153.51) and (516.24,152.28) .. (516.76,151.28) .. controls (517.27,150.28) and (518.5,149.89) .. (519.5,150.4) .. controls (520.5,150.92) and (520.89,152.14) .. (520.38,153.14) .. controls (519.86,154.14) and (518.64,154.54) .. (517.64,154.02) -- cycle ;
\draw    (525.41,152.85) -- (520.38,153.14) ;
\draw  [fill={rgb, 255:red, 155; green, 155; blue, 155 }  ,fill opacity=0.53 ] (524.76,143.63) .. controls (524.96,142.52) and (526.02,141.79) .. (527.12,141.99) .. controls (528.23,142.19) and (528.96,143.26) .. (528.76,144.36) .. controls (528.56,145.47) and (527.5,146.2) .. (526.39,146) .. controls (525.29,145.8) and (524.55,144.73) .. (524.76,143.63) -- cycle ;
\draw    (528.13,150.73) -- (526.39,146) ;

\draw [color={rgb, 255:red, 65; green, 117; blue, 5 }  ,draw opacity=1 ] [dash pattern={on 3.75pt off 1.5pt}]  (420.19,121.73) -- (433.64,114.5) ;
\draw [shift={(433.64,114.5)}, rotate = 331.76] [color={rgb, 255:red, 65; green, 117; blue, 5 }  ,draw opacity=1 ][fill={rgb, 255:red, 65; green, 117; blue, 5 }  ,fill opacity=1 ][line width=0.75]      (0, 0) circle [x radius= 3.35, y radius= 3.35]   ;
\draw [color={rgb, 255:red, 65; green, 117; blue, 5 }  ,draw opacity=1 ]   (433.64,114.5) -- (444.85,119.73) ;
\draw [shift={(433.64,114.5)}, rotate = 25.02] [color={rgb, 255:red, 65; green, 117; blue, 5 }  ,draw opacity=1 ][fill={rgb, 255:red, 65; green, 117; blue, 5 }  ,fill opacity=1 ][line width=0.75]      (0, 0) circle [x radius= 2.01, y radius= 2.01]   ;
\draw [color={rgb, 255:red, 65; green, 117; blue, 5 }  ,draw opacity=1 ]   (444.85,119.73) -- (458.55,115.75) ;
\draw [shift={(444.85,119.73)}, rotate = 343.78] [color={rgb, 255:red, 65; green, 117; blue, 5 }  ,draw opacity=1 ][fill={rgb, 255:red, 65; green, 117; blue, 5 }  ,fill opacity=1 ][line width=0.75]      (0, 0) circle [x radius= 2.01, y radius= 2.01]   ;
\draw [color={rgb, 255:red, 65; green, 117; blue, 5 }  ,draw opacity=1 ]   (458.55,115.75) -- (472.75,120.48) ;
\draw [shift={(458.55,115.75)}, rotate = 18.43] [color={rgb, 255:red, 65; green, 117; blue, 5 }  ,draw opacity=1 ][fill={rgb, 255:red, 65; green, 117; blue, 5 }  ,fill opacity=1 ][line width=0.75]      (0, 0) circle [x radius= 2.01, y radius= 2.01]   ;
\draw [color={rgb, 255:red, 65; green, 117; blue, 5 }  ,draw opacity=1 ]   (472.75,120.48) -- (484.95,111.76) ;
\draw [shift={(484.95,111.76)}, rotate = 324.46] [color={rgb, 255:red, 65; green, 117; blue, 5 }  ,draw opacity=1 ][fill={rgb, 255:red, 65; green, 117; blue, 5 }  ,fill opacity=1 ][line width=0.75]      (0, 0) circle [x radius= 2.01, y radius= 2.01]   ;
\draw [shift={(472.75,120.48)}, rotate = 324.46] [color={rgb, 255:red, 65; green, 117; blue, 5 }  ,draw opacity=1 ][fill={rgb, 255:red, 65; green, 117; blue, 5 }  ,fill opacity=1 ][line width=0.75]      (0, 0) circle [x radius= 2.01, y radius= 2.01]   ;
\draw [color={rgb, 255:red, 65; green, 117; blue, 5 }  ,draw opacity=1 ]   (497.16,121.73) -- (484.95,111.76) ;
\draw [shift={(497.16,121.73)}, rotate = 219.23] [color={rgb, 255:red, 65; green, 117; blue, 5 }  ,draw opacity=1 ][fill={rgb, 255:red, 65; green, 117; blue, 5 }  ,fill opacity=1 ][line width=0.75]      (0, 0) circle [x radius= 2.01, y radius= 2.01]   ;
\draw [color={rgb, 255:red, 65; green, 117; blue, 5 }  ,draw opacity=1 ]   (515.59,115) -- (497.16,121.73) ;
\draw [shift={(515.59,115)}, rotate = 159.95] [color={rgb, 255:red, 65; green, 117; blue, 5 }  ,draw opacity=1 ][fill={rgb, 255:red, 65; green, 117; blue, 5 }  ,fill opacity=1 ][line width=0.75]      (0, 0) circle [x radius= 2.01, y radius= 2.01]   ;
\draw [color={rgb, 255:red, 65; green, 117; blue, 5 }  ,draw opacity=1 ]   (530.04,122.22) -- (515.59,115) ;
\draw [shift={(530.04,122.22)}, rotate = 206.57] [color={rgb, 255:red, 65; green, 117; blue, 5 }  ,draw opacity=1 ][fill={rgb, 255:red, 65; green, 117; blue, 5 }  ,fill opacity=1 ][line width=0.75]      (0, 0) circle [x radius= 2.01, y radius= 2.01]   ;
\draw [color={rgb, 255:red, 65; green, 117; blue, 5 }  ,draw opacity=1 ]   (537.51,115.25) -- (530.04,122.22) ;
\draw [shift={(537.51,115.25)}, rotate = 136.97] [color={rgb, 255:red, 65; green, 117; blue, 5 }  ,draw opacity=1 ][fill={rgb, 255:red, 65; green, 117; blue, 5 }  ,fill opacity=1 ][line width=0.75]      (0, 0) circle [x radius= 2.01, y radius= 2.01]   ;
\draw [color={rgb, 255:red, 65; green, 117; blue, 5 }  ,draw opacity=1 ]   (546.98,119.23) -- (537.51,115.25) ;
\draw [shift={(546.98,119.23)}, rotate = 202.83] [color={rgb, 255:red, 65; green, 117; blue, 5 }  ,draw opacity=1 ][fill={rgb, 255:red, 65; green, 117; blue, 5 }  ,fill opacity=1 ][line width=0.75]      (0, 0) circle [x radius= 2.01, y radius= 2.01]   ;
\draw [color={rgb, 255:red, 65; green, 117; blue, 5 }  ,draw opacity=1 ] [dash pattern={on 3.75pt off 1.5pt}]  (555.2,110.02) -- (546.98,119.23) ;
\draw [shift={(546.98,119.23)}, rotate = 131.73] [color={rgb, 255:red, 65; green, 117; blue, 5 }  ,draw opacity=1 ][fill={rgb, 255:red, 65; green, 117; blue, 5 }  ,fill opacity=1 ][line width=0.75]      (0, 0) circle [x radius= 3.35, y radius= 3.35]   ;
\draw [color={rgb, 255:red, 144; green, 19; blue, 254 }  ,draw opacity=1 ]   (435.48,73.37) ;
\draw [shift={(435.48,73.37)}, rotate = 0] [color={rgb, 255:red, 144; green, 19; blue, 254 }  ,draw opacity=1 ][fill={rgb, 255:red, 144; green, 19; blue, 254 }  ,fill opacity=1 ][line width=0.75]      (0, 0) circle [x radius= 3.35, y radius= 3.35]   ;
\draw [color={rgb, 255:red, 144; green, 19; blue, 254 }  ,draw opacity=1 ]   (531.43,151.1) ;
\draw [shift={(531.43,151.1)}, rotate = 0] [color={rgb, 255:red, 144; green, 19; blue, 254 }  ,draw opacity=1 ][fill={rgb, 255:red, 144; green, 19; blue, 254 }  ,fill opacity=1 ][line width=0.75]      (0, 0) circle [x radius= 3.35, y radius= 3.35]   ;

\draw (533.91,132.8) node [anchor=north west][inner sep=0.75pt]  [color={rgb, 255:red, 144; green, 19; blue, 254 }  ,opacity=1 ]  {$x$};
\draw (442.7,72.98) node [anchor=north west][inner sep=0.75pt]  [color={rgb, 255:red, 144; green, 19; blue, 254 }  ,opacity=1 ]  {$y$};
\draw (427.52,117.86) node [anchor=north west][inner sep=0.75pt]  [font=\footnotesize,color={rgb, 255:red, 65; green, 117; blue, 5 }  ,opacity=1 ]  {$u$};
\draw (538.73,121.37) node [anchor=north west][inner sep=0.75pt]  [font=\footnotesize,color={rgb, 255:red, 65; green, 117; blue, 5 }  ,opacity=1 ]  {$v$};
\draw (470.11,94.95) node [anchor=north west][inner sep=0.75pt]  [color={rgb, 255:red, 65; green, 117; blue, 5 }  ,opacity=1 ]  {$p$};
\draw (425.24,146.78) node [anchor=north west][inner sep=0.75pt]  [font=\tiny,rotate=-327.68]  {$\vdots $};
\draw (484.4,154.6) node [anchor=north west][inner sep=0.75pt]  [color={rgb, 255:red, 155; green, 155; blue, 155 }  ,opacity=0.9 ]  {$U_{0}$};

\end{tikzpicture}

  \captionof{figure}{There exists a path $p$ in $\Gamma$ passing through $U_0$ which separates $x$ from $y$ in this face.}
  \label{fig:separating-path}
\end{minipage}%
\begin{minipage}{.5\textwidth}
  \centering

\tikzset{every picture/.style={line width=0.75pt}} 

\begin{tikzpicture}[x=0.75pt,y=0.75pt,yscale=-1,xscale=1]

\draw  [draw opacity=0][fill={rgb, 255:red, 155; green, 155; blue, 155 }  ,fill opacity=0.46 ] (422.1,136.25) .. controls (422.1,88.89) and (460.49,50.5) .. (507.85,50.5) .. controls (555.21,50.5) and (593.6,88.89) .. (593.6,136.25) .. controls (593.6,183.61) and (555.21,222) .. (507.85,222) .. controls (460.49,222) and (422.1,183.61) .. (422.1,136.25) -- cycle ;
\draw  [color={rgb, 255:red, 0; green, 0; blue, 0 }  ,draw opacity=1 ][fill={rgb, 255:red, 255; green, 255; blue, 255 }  ,fill opacity=1 ] (439.69,136.25) .. controls (439.69,98.6) and (470.2,68.09) .. (507.85,68.09) .. controls (545.5,68.09) and (576.01,98.6) .. (576.01,136.25) .. controls (576.01,173.9) and (545.5,204.41) .. (507.85,204.41) .. controls (470.2,204.41) and (439.69,173.9) .. (439.69,136.25) -- cycle ;
\draw    (475.49,196.15) -- (477.87,191.02) ;
\draw  [fill={rgb, 255:red, 155; green, 155; blue, 155 }  ,fill opacity=0.53 ] (476.47,187.61) .. controls (477.02,186.28) and (478.55,185.65) .. (479.88,186.21) .. controls (481.21,186.76) and (481.83,188.29) .. (481.28,189.62) .. controls (480.72,190.95) and (479.2,191.57) .. (477.87,191.02) .. controls (476.54,190.46) and (475.91,188.94) .. (476.47,187.61) -- cycle ;
\draw  [fill={rgb, 255:red, 155; green, 155; blue, 155 }  ,fill opacity=0.53 ] (478.64,180.92) .. controls (478.94,180.2) and (479.77,179.86) .. (480.48,180.16) .. controls (481.2,180.46) and (481.54,181.28) .. (481.24,182) .. controls (480.94,182.72) and (480.12,183.06) .. (479.4,182.76) .. controls (478.68,182.46) and (478.34,181.63) .. (478.64,180.92) -- cycle ;
\draw    (479.88,186.21) -- (479.4,182.76) ;
\draw  [fill={rgb, 255:red, 155; green, 155; blue, 155 }  ,fill opacity=0.53 ] (486.19,185.25) .. controls (486.96,185.33) and (487.53,186.02) .. (487.45,186.8) .. controls (487.37,187.57) and (486.68,188.13) .. (485.9,188.05) .. controls (485.13,187.97) and (484.57,187.28) .. (484.65,186.51) .. controls (484.73,185.74) and (485.42,185.17) .. (486.19,185.25) -- cycle ;
\draw    (481.48,187.97) -- (484.65,186.51) ;

\draw    (465.67,190.08) -- (468.03,187.46) ;
\draw  [fill={rgb, 255:red, 155; green, 155; blue, 155 }  ,fill opacity=0.53 ] (467.81,185.17) .. controls (468.39,184.48) and (469.41,184.38) .. (470.1,184.96) .. controls (470.79,185.53) and (470.89,186.55) .. (470.31,187.24) .. controls (469.74,187.93) and (468.72,188.03) .. (468.03,187.46) .. controls (467.34,186.88) and (467.24,185.86) .. (467.81,185.17) -- cycle ;
\draw  [fill={rgb, 255:red, 155; green, 155; blue, 155 }  ,fill opacity=0.53 ] (470.33,181.58) .. controls (470.64,181.21) and (471.19,181.15) .. (471.56,181.46) .. controls (471.94,181.77) and (471.99,182.33) .. (471.68,182.7) .. controls (471.37,183.07) and (470.82,183.12) .. (470.44,182.81) .. controls (470.07,182.5) and (470.02,181.95) .. (470.33,181.58) -- cycle ;
\draw    (470.1,184.96) -- (470.44,182.81) ;
\draw  [fill={rgb, 255:red, 155; green, 155; blue, 155 }  ,fill opacity=0.53 ] (474.04,185.54) .. controls (474.48,185.73) and (474.69,186.24) .. (474.51,186.69) .. controls (474.32,187.13) and (473.8,187.34) .. (473.36,187.15) .. controls (472.91,186.97) and (472.7,186.45) .. (472.89,186.01) .. controls (473.08,185.56) and (473.59,185.35) .. (474.04,185.54) -- cycle ;
\draw    (470.74,186.3) -- (472.89,186.01) ;
\draw    (459.22,184.32) -- (460.76,182.6) ;
\draw  [fill={rgb, 255:red, 155; green, 155; blue, 155 }  ,fill opacity=0.53 ] (460.62,181.1) .. controls (461,180.65) and (461.67,180.58) .. (462.12,180.96) .. controls (462.58,181.34) and (462.64,182.01) .. (462.26,182.46) .. controls (461.89,182.91) and (461.21,182.98) .. (460.76,182.6) .. controls (460.31,182.22) and (460.25,181.55) .. (460.62,181.1) -- cycle ;
\draw  [fill={rgb, 255:red, 155; green, 155; blue, 155 }  ,fill opacity=0.53 ] (462.27,178.74) .. controls (462.48,178.5) and (462.84,178.46) .. (463.08,178.67) .. controls (463.33,178.87) and (463.36,179.23) .. (463.16,179.48) .. controls (462.96,179.72) and (462.59,179.75) .. (462.35,179.55) .. controls (462.1,179.35) and (462.07,178.99) .. (462.27,178.74) -- cycle ;
\draw    (462.12,180.96) -- (462.35,179.55) ;
\draw  [fill={rgb, 255:red, 155; green, 155; blue, 155 }  ,fill opacity=0.53 ] (464.71,181.34) .. controls (465,181.46) and (465.14,181.8) .. (465.02,182.1) .. controls (464.89,182.39) and (464.55,182.53) .. (464.26,182.4) .. controls (463.97,182.28) and (463.83,181.94) .. (463.95,181.65) .. controls (464.08,181.36) and (464.42,181.22) .. (464.71,181.34) -- cycle ;
\draw    (462.54,181.84) -- (463.95,181.65) ;
\draw    (454.46,178.83) -- (456.14,177.82) ;
\draw  [fill={rgb, 255:red, 155; green, 155; blue, 155 }  ,fill opacity=0.53 ] (456.4,176.57) .. controls (456.82,176.3) and (457.38,176.41) .. (457.65,176.83) .. controls (457.92,177.25) and (457.8,177.81) .. (457.39,178.08) .. controls (456.97,178.35) and (456.41,178.24) .. (456.14,177.82) .. controls (455.87,177.4) and (455.98,176.84) .. (456.4,176.57) -- cycle ;
\draw  [fill={rgb, 255:red, 155; green, 155; blue, 155 }  ,fill opacity=0.53 ] (458.32,175.08) .. controls (458.55,174.93) and (458.85,174.99) .. (459,175.22) .. controls (459.15,175.44) and (459.08,175.74) .. (458.86,175.89) .. controls (458.63,176.04) and (458.33,175.98) .. (458.18,175.75) .. controls (458.04,175.52) and (458.1,175.22) .. (458.32,175.08) -- cycle ;
\draw    (457.65,176.83) -- (458.18,175.75) ;
\draw  [fill={rgb, 255:red, 155; green, 155; blue, 155 }  ,fill opacity=0.53 ] (459.64,177.79) .. controls (459.85,177.96) and (459.88,178.27) .. (459.7,178.47) .. controls (459.53,178.68) and (459.22,178.7) .. (459.02,178.53) .. controls (458.81,178.36) and (458.78,178.05) .. (458.96,177.85) .. controls (459.13,177.64) and (459.44,177.61) .. (459.64,177.79) -- cycle ;
\draw    (457.77,177.65) -- (458.96,177.85) ;
\draw    (515.03,67.82) -- (516.7,81) ;
\draw  [fill={rgb, 255:red, 155; green, 155; blue, 155 }  ,fill opacity=0.53 ] (523.74,86.03) .. controls (524.3,89.37) and (522.04,92.52) .. (518.71,93.07) .. controls (515.38,93.63) and (512.22,91.37) .. (511.67,88.04) .. controls (511.12,84.7) and (513.37,81.55) .. (516.7,81) .. controls (520.04,80.44) and (523.19,82.7) .. (523.74,86.03) -- cycle ;
\draw  [fill={rgb, 255:red, 155; green, 155; blue, 155 }  ,fill opacity=0.53 ] (527.77,102.06) .. controls (528.07,103.86) and (526.85,105.56) .. (525.06,105.86) .. controls (523.26,106.16) and (521.55,104.94) .. (521.25,103.14) .. controls (520.96,101.34) and (522.17,99.64) .. (523.97,99.34) .. controls (525.77,99.04) and (527.47,100.26) .. (527.77,102.06) -- cycle ;
\draw    (518.71,93.07) -- (523.97,99.34) ;
\draw  [fill={rgb, 255:red, 155; green, 155; blue, 155 }  ,fill opacity=0.53 ] (507.34,102.85) .. controls (505.71,103.66) and (503.73,102.99) .. (502.92,101.35) .. controls (502.11,99.72) and (502.78,97.74) .. (504.42,96.93) .. controls (506.06,96.12) and (508.04,96.79) .. (508.84,98.43) .. controls (509.65,100.06) and (508.98,102.05) .. (507.34,102.85) -- cycle ;
\draw    (513.32,91.58) -- (508.84,98.43) ;

\draw    (511.62,85.62) -- (500.08,92.21) ;
\draw  [fill={rgb, 255:red, 155; green, 155; blue, 155 }  ,fill opacity=0.53 ] (498.13,100.64) .. controls (495.26,102.43) and (491.48,101.56) .. (489.7,98.69) .. controls (487.91,95.82) and (488.78,92.05) .. (491.65,90.26) .. controls (494.52,88.47) and (498.29,89.34) .. (500.08,92.21) .. controls (501.87,95.08) and (500.99,98.85) .. (498.13,100.64) -- cycle ;
\draw  [fill={rgb, 255:red, 155; green, 155; blue, 155 }  ,fill opacity=0.53 ] (484.86,110.5) .. controls (483.32,111.46) and (481.28,110.99) .. (480.31,109.44) .. controls (479.35,107.9) and (479.82,105.86) .. (481.37,104.89) .. controls (482.92,103.93) and (484.95,104.4) .. (485.92,105.95) .. controls (486.88,107.5) and (486.41,109.53) .. (484.86,110.5) -- cycle ;
\draw    (489.7,98.69) -- (485.92,105.95) ;
\draw  [fill={rgb, 255:red, 155; green, 155; blue, 155 }  ,fill opacity=0.53 ] (476.31,91.93) .. controls (474.94,90.73) and (474.8,88.64) .. (476,87.27) .. controls (477.2,85.9) and (479.29,85.76) .. (480.66,86.96) .. controls (482.03,88.17) and (482.17,90.25) .. (480.97,91.62) .. controls (479.77,93) and (477.68,93.13) .. (476.31,91.93) -- cycle ;
\draw    (489.01,93.14) -- (480.97,91.62) ;

\draw    (559.26,180.31) -- (552.12,176.3) ;
\draw  [fill={rgb, 255:red, 155; green, 155; blue, 155 }  ,fill opacity=0.53 ] (547.04,177.93) .. controls (545.19,176.97) and (544.46,174.7) .. (545.41,172.85) .. controls (546.36,170.99) and (548.64,170.26) .. (550.49,171.22) .. controls (552.34,172.17) and (553.07,174.44) .. (552.12,176.3) .. controls (551.17,178.15) and (548.9,178.88) .. (547.04,177.93) -- cycle ;
\draw  [fill={rgb, 255:red, 155; green, 155; blue, 155 }  ,fill opacity=0.53 ] (537.64,174.02) .. controls (536.64,173.51) and (536.24,172.28) .. (536.76,171.28) .. controls (537.27,170.28) and (538.5,169.89) .. (539.5,170.4) .. controls (540.5,170.92) and (540.89,172.14) .. (540.38,173.14) .. controls (539.86,174.14) and (538.64,174.54) .. (537.64,174.02) -- cycle ;
\draw    (545.41,172.85) -- (540.38,173.14) ;
\draw  [fill={rgb, 255:red, 155; green, 155; blue, 155 }  ,fill opacity=0.53 ] (544.76,163.63) .. controls (544.96,162.52) and (546.02,161.79) .. (547.12,161.99) .. controls (548.23,162.19) and (548.96,163.26) .. (548.76,164.36) .. controls (548.56,165.47) and (547.5,166.2) .. (546.39,166) .. controls (545.29,165.8) and (544.55,164.73) .. (544.76,163.63) -- cycle ;
\draw    (548.13,170.73) -- (546.39,166) ;

\draw [color={rgb, 255:red, 65; green, 117; blue, 5 }  ,draw opacity=1 ]   (443.41,155.1) .. controls (482.33,156) and (478,138.67) .. (511,130) ;
\draw [shift={(511,130)}, rotate = 345.28] [color={rgb, 255:red, 65; green, 117; blue, 5 }  ,draw opacity=1 ][fill={rgb, 255:red, 65; green, 117; blue, 5 }  ,fill opacity=1 ][line width=0.75]      (0, 0) circle [x radius= 1.34, y radius= 1.34]   ;
\draw [shift={(443.41,155.1)}, rotate = 1.32] [color={rgb, 255:red, 65; green, 117; blue, 5 }  ,draw opacity=1 ][fill={rgb, 255:red, 65; green, 117; blue, 5 }  ,fill opacity=1 ][line width=0.75]      (0, 0) circle [x radius= 1.34, y radius= 1.34]   ;
\draw [color={rgb, 255:red, 65; green, 117; blue, 5 }  ,draw opacity=1 ]   (511,130) .. controls (515.8,111.2) and (504.33,114.33) .. (516.48,92.5) ;
\draw [shift={(516.48,92.5)}, rotate = 299.08] [color={rgb, 255:red, 65; green, 117; blue, 5 }  ,draw opacity=1 ][fill={rgb, 255:red, 65; green, 117; blue, 5 }  ,fill opacity=1 ][line width=0.75]      (0, 0) circle [x radius= 1.34, y radius= 1.34]   ;
\draw [shift={(511,130)}, rotate = 284.32] [color={rgb, 255:red, 65; green, 117; blue, 5 }  ,draw opacity=1 ][fill={rgb, 255:red, 65; green, 117; blue, 5 }  ,fill opacity=1 ][line width=0.75]      (0, 0) circle [x radius= 1.34, y radius= 1.34]   ;
\draw [color={rgb, 255:red, 65; green, 117; blue, 5 }  ,draw opacity=1 ]   (536.76,171.28) .. controls (516.76,162.28) and (532.33,144.67) .. (511,130) ;
\draw [shift={(511,130)}, rotate = 214.51] [color={rgb, 255:red, 65; green, 117; blue, 5 }  ,draw opacity=1 ][fill={rgb, 255:red, 65; green, 117; blue, 5 }  ,fill opacity=1 ][line width=0.75]      (0, 0) circle [x radius= 1.34, y radius= 1.34]   ;
\draw [shift={(536.76,171.28)}, rotate = 204.23] [color={rgb, 255:red, 65; green, 117; blue, 5 }  ,draw opacity=1 ][fill={rgb, 255:red, 65; green, 117; blue, 5 }  ,fill opacity=1 ][line width=0.75]      (0, 0) circle [x radius= 1.34, y radius= 1.34]   ;
\draw [color={rgb, 255:red, 144; green, 19; blue, 254 }  ,draw opacity=1 ]   (552.12,176.3) ;
\draw [shift={(552.12,176.3)}, rotate = 0] [color={rgb, 255:red, 144; green, 19; blue, 254 }  ,draw opacity=1 ][fill={rgb, 255:red, 144; green, 19; blue, 254 }  ,fill opacity=1 ][line width=0.75]      (0, 0) circle [x radius= 3.35, y radius= 3.35]   ;
\draw [color={rgb, 255:red, 144; green, 19; blue, 254 }  ,draw opacity=1 ]   (443.41,155.1) ;
\draw [shift={(443.41,155.1)}, rotate = 0] [color={rgb, 255:red, 144; green, 19; blue, 254 }  ,draw opacity=1 ][fill={rgb, 255:red, 144; green, 19; blue, 254 }  ,fill opacity=1 ][line width=0.75]      (0, 0) circle [x radius= 3.35, y radius= 3.35]   ;
\draw [color={rgb, 255:red, 144; green, 19; blue, 254 }  ,draw opacity=1 ]   (516.7,81) ;
\draw [shift={(516.7,81)}, rotate = 0] [color={rgb, 255:red, 144; green, 19; blue, 254 }  ,draw opacity=1 ][fill={rgb, 255:red, 144; green, 19; blue, 254 }  ,fill opacity=1 ][line width=0.75]      (0, 0) circle [x radius= 3.35, y radius= 3.35]   ;

\draw (492.04,68.2) node [anchor=north west][inner sep=0.75pt]  [color={rgb, 255:red, 144; green, 19; blue, 254 }  ,opacity=1 ]  {$x_{2}$};
\draw (444.83,133.31) node [anchor=north west][inner sep=0.75pt]  [color={rgb, 255:red, 144; green, 19; blue, 254 }  ,opacity=1 ]  {$x_{1}$};
\draw (529.23,173.31) node [anchor=north west][inner sep=0.75pt]  [color={rgb, 255:red, 144; green, 19; blue, 254 }  ,opacity=1 ]  {$x_{3}$};
\draw (445.24,166.78) node [anchor=north west][inner sep=0.75pt]  [font=\tiny,rotate=-327.68]  {$\vdots $};
\draw (491.73,177.93) node [anchor=north west][inner sep=0.75pt]  [color={rgb, 255:red, 155; green, 155; blue, 155 }  ,opacity=0.9 ]  {$U_{0}$};

\end{tikzpicture}

  \captionof{figure}{If $x_1$, $x_2$, $x_3$ pairwise lie on subfaces of $U_0$, but share no common face, then a tripod as above must exist in $\overline \Gamma$.}
  \label{fig:face-tripod}
\end{minipage}
\end{figure}

\section{Quasi-actions on planar graphs}\label{sec:quasi-act-planar}

The next technical challenge is to study quasi-actions on planar graphs and their subgraphs. We open with a discussion of what is to come. 

\subsection{Initial discussion}\label{sec:quasi-discuss}

The following definition will be useful for brevity. 

\begin{definition}[Quasi-planar tuple]\label{def:qp}
    A \textit{quasi-planar tuple} is a sextuple 
    $$
    \Qp = (G, X, \Gamma, \vartheta,  \varphi, \psi),
    $$
    satisfying the following:
    \begin{enumerate}
        \item $X$ is a connected, locally finite, multi-ended graph,
        \item $G$ is a group acting quasi-transitively on $X$, 
        \item $\Gamma$ is a connected, bounded-degree  planar graph with good drawing $\vartheta : \overline \Gamma \into \bbS^2$,
        \item The map $\varphi : X \onto \Gamma$ is a continuous surjective quasi-isometry with continuous quasi-inverse $\psi : \Gamma \to X$. 
    \end{enumerate}
\end{definition}

By Propositions~\ref{prop:qi-wlog}, \ref{prop:cts-inverse}, we lose nothing by assuming that $\Gamma$ is bounded-degree, or that the quasi-isometries are continuous. The reason we want $\varphi$ to be surjective is so we can apply Lemma~\ref{lem:onto-qi} later in this section. 

Note that there is an induced quasi-action of $G$ upon $\Gamma$ (see Definition~\ref{def:quasi-action}). The following definition sets out criteria for this quasi-action to be `nice'.

\begin{definition}[Well-behaved quasi-planar tuple]\label{def:good-behavior}
    Let $\Qp = (G, X, \Gamma, \vartheta,  \varphi, \psi)$ be a quasi-planar tuple. We then define two \textit{good behaviour} properties as follows:

    \begin{enumerate}
        \myitem{(GB1)}\label{itm:gb1} There exists $m, n > 0$ such that the following holds. Let $f \in \facepaths(\Gamma)$ satisfy $\diam_\Gamma(f) > n$, then for all $g \in G$ there exists $f' \in \facepaths(\Gamma)$ such that
        $$
        \dHaus[\Gamma](f', \varphi_g(f)) < m. 
        $$

        \myitem{(GB2)}\label{itm:gb2} $\Gamma$ is 2-connected. 
    \end{enumerate}
    We say that $\Qp$ is \textit{well-behaved} if both of the above are satisfied. 
\end{definition}

The property \ref{itm:gb1} should be interpreted as a coarse analogue to Whitney's unique planar embedding theorem (see Theorem~\ref{thm:whitney}) and essentially says that the drawing of $\Gamma$ is `coarsely preserved' by the induced quasi-action. The reason we ask for \ref{itm:gb2} is essentially for hygiene purposes, as this ensures that faces of our planar graph are always bound by simple closed curves. 

Recall that $\finfaces (\Gamma)$ denotes the set of compact facial subgraphs of $\Gamma$. Note that \ref{itm:gb1} has the following consequence. 

\begin{proposition}\label{prop:bounded-faces}
    Let $\Qp = (G, X, \Gamma, \vartheta,  \varphi, \psi)$ be a quasi-planar tuple, and suppose $\Qp$ satisfies \ref{itm:gb1}. Then there exists $R > 0$ such that 
    $
    \diam_\Gamma(f) < R 
    $
    for all $f \in \finfaces (\Gamma)$. 
\end{proposition}

\begin{proof}
    Suppose to the contrary that $\Gamma$ contains a sequence $(f_n)_{n \geq 1}$, $f_n \in \finfaces (\Gamma)$ of compact facial subgraphs such that $\diam(f_n) \geq n$ for all $n \geq 1$. Let $v_0 \in \Gamma$ be arbitrary. Pick a vertex $v_n$ on each $f_n$ and let $g_n \in G$ be such that $\varphi_{g_n}(v_n)$ lies within a bounded distance of $v_0$ (recall the quasi-action of $G$ on $\Gamma$ is cobounded). By \ref{itm:gb1} we deduce that a bounded neighbourhood of $v_0$ must intersect compact facial subgraphs of arbitrarily large diameter. This contradicts the fact that $\Gamma$ is locally finite. 
\end{proof}

In \S\ref{sec:accessibility}, we will see that good behavior is enough to deduce accessibility. The sole purpose of \S\ref{sec:quasi-act-planar} is to reduce the general problem to the well-behaved case. In particular, we will prove the following statement, which we state now for the convenience of the reader. 

\begin{restatable}[Tree decomposition with well-behaved parts]{theorem}{treedecomp}\label{thm:2-conn-to-wb}
    Let $X$ be a connected, locally finite, quasi-transitive graph, and suppose that $X$ is quasi-isometric to some connected, locally finite planar graph $\Gamma$ with good drawing $\vartheta$. Then there exists a $G$-canonical tree decomposition $(T, \cV)$ of $X$ with bounded adhesion and $T/G$ compact, where $\cV = (V_u)_{u \in V(T)}$, such that each part $X_u := X[V_u]$ satisfies the following:
    \begin{enumerate}
        \item $X_u$ is connected and quasi-transitive, 

        \item One of the following holds: 

        \begin{enumerate}
            \item $X_u$ has at most one end, or

        \item There is a well-behaved quasi-planar tuple 
        $$
        \Qp_u = (G_u, X_u, \Gamma_u, \vartheta_u, \varphi_u, \psi_u),
        $$
        where $\Gamma_u \subset \Gamma$ and $\vartheta_u$ is the restriction of $\vartheta$ to the closure of $\Gamma_u$ in $\overline \Gamma$.

        \end{enumerate}
    \end{enumerate}
\end{restatable}

Apart from this main theorem, no other results of \S\ref{sec:quasi-act-planar} will make a second appearance outside of this section. Thus, the reader who is short on time may now skip to \S\ref{sec:accessibility}, if they so wish.

\subsection{Intersecting neighbourhoods of facial subgraphs}
We need to prove technical lemmas, which tell us that if $\Gamma$ is a planar graph which is quasi-isometric to some quasi-transitive graph $X$, then intersections of tubular neighbourhoods of facial-subgraphs are boundedly small.

\begin{lemma}\label{lem:two-faces-small-intersection}
    Let $X$ be a connected, quasi-transitive, locally finite graph. Let $\Gamma$ be a connected, locally finite, planar graph. Suppose further that $\Gamma$ is almost 2-connected. Let $\varphi : X \to \Gamma$ be a quasi-isometry.
    Then for every $r > 0$ there exists $m= m(r) > 0$ such that the following holds. Let $f_1, f_2 \in \facepaths(\Gamma)$ be distinct. Then
    $$
    \diam_\Gamma\big( B_\Gamma(f_1; r) \cap B_\Gamma(f_2; r) \big) \leq m.
    $$
\end{lemma}

\begin{proof}
    We will prove this by contradicting Lemma~\ref{lem:bounded-diam-coarse}. We first prove this under the stronger assumption that $\Gamma$ is 2-connected. 
    
    Thus, assume $\Gamma$ is 2-connected. Fix $r > 0$, and suppose that for every $m > 0$ there exists a pair of distinct $f_1, f_2 \in \facepaths(\Gamma)$, and a pair of points 
    $$
    x, y \in B_\Gamma(f_1;r) \cap B_\Gamma(f_2;r)
    $$ 
    such that $\dist_\Pi(x,y) > m$. 
    Note that since $\Gamma$ is 2-connected, the closure of every facial subgraph of $\Gamma$ in $\overline \Gamma$ is a simple closed curve Proposition~\ref{prop:simple-face}. Let $\overline {f_i}$ denote the closure of $f_i$ in $\overline \Gamma$.  
    
    We assume without loss of generality that $m$ is much larger than $r$. It is clear that there exist (possibly degenerate) simple paths $p$, $q$ in $\Gamma$ such that \begin{enumerate}
        \item The path $p$ is contained within a bounded neighbourhood of $x$, and similarly $q$ lies in a bounded neighbourhood of $y$.

        \item Both $p$ and $q$ begin on $f_1$ and end on $f_2$, and are otherwise disjoint from $f_1 \cup f_2$. 

        \item The paths $p$ and $q$ are disjoint from each other.
    \end{enumerate}
    It is easy to see pictorially that $\overline \Gamma \setminus (p \cup q)$ is disconnected, see Figure~\ref{fig:two-faces}. 
    
    Consider $f_1$.
    Since 
    $\overline {f_1} \cong \bbS^1$ and $f_1 \cap (p \cup q)$ consists of exactly two vertices, say $u$ and $v$, we have that $\overline {f_1} \setminus (p \cup q) = \overline {f_1} \setminus \{u,v\}$ is a disjoint union of two open intervals, each lying in a distinct component of $\overline \Gamma \setminus (p \cup q)$. Let $I_1$, $I_2$ denote these components. Let $e_1$, $e_2$ be the unique two edges of $\Gamma$ such that
    \begin{enumerate}
        \item Both $e_1$, $e_2$ lie inside $f_1$, and in particular in (the closure of) $I_1$.

        \item We have that $e_1$ abuts $p$, and $e_2$ abuts $q$. 
    \end{enumerate}
    Let $U_1$ denote the connected component of $\overline \Gamma \setminus (p \cup q)$ containing $I_1$. This is a connected open subset of $\overline \Gamma$. By Proposition~\ref{prop:freud-prop}, we have that $C_1 = U_1 \setminus \Omega (\Gamma)$ is connected. In particular, there is a combinatorial path through $C_1$ connects endpoints of $e_1$ and $e_2$. 
    Similarly, we also see that there is a combinatorial path through $\Gamma \setminus C_1$ connecting endpoints of $e_1$ to $e_2$. 
    Finally, let $b$ denote the vertex set of $C_1$. We have that every edge in $\delta b$ abuts $p$ or $q$, so $\delta b$ is finite and in particular uniformly bounded in size. Thus, $b \in \br(\Gamma)$. It is also clear that $e_1, e_2 \in \delta b$. By Proposition~\ref{prop:tight-trick-paths}, there exists a tight $b' \in \br(\Gamma)$ such that $\delta b' \subset \delta b$, and $e_1, e_2 \in \delta b'$.
    In particular, $\diam_\Gamma(\delta b')$ is approximately at least $m$ but $|\delta b'|$ is uniformly bounded.
    By taking $m \to \infty$, we eventually contradict Lemma~\ref{lem:bounded-diam-coarse}. 

    Now suppose that $\Gamma$ is just almost 2-connected, but not necessarily 2-connected. Let $\Gamma_0 \subset \Gamma$ be the 2-connected core. Once again, fix $r > 0$, and suppose that for every $m > 0$ there exists a pair of distinct $f_1, f_2 \in \facepaths(\Gamma)$ such that
    $$
    \diam_\Gamma\big( B_\Gamma(f_1; r) \cap B_\Gamma(f_2; r) \big) > m.
    $$
    Recall (Remark~\ref{rmk:simple-body}) that every facial subgraph of $\Gamma$ decomposes into its simple body (which is a facial subgraph of $\Gamma_0$) plus boundedly small adjoined cacti. It follows quickly from this that for some $r' > 0$ and for every $m' > 0$ there exists $f_1', f_2' \in \facepaths(\Gamma_0)$ such that 
    $$
    \diam_{\Gamma_0}\big( B_{\Gamma_0}(f_1'; r) \cap B_{\Gamma_0}(f_2'; r) \big) > m'.
    $$
    But we know that this cannot happen, since $\Gamma_0$ is also quasi-isometric to $X$ and 2-connected. Thus, the lemma follows. 
\end{proof}

\begin{figure}
\begin{minipage}{.5\textwidth}
  \centering

\tikzset{every picture/.style={line width=0.75pt}} 

\begin{tikzpicture}[x=0.75pt,y=0.75pt,yscale=-1,xscale=1]

\draw  [fill={rgb, 255:red, 250; green, 250; blue, 250 }  ,fill opacity=1 ] (180.99,40.37) .. controls (199.05,28.15) and (204.9,45.68) .. (208.09,77.02) .. controls (211.27,108.37) and (224.02,95.62) .. (221.9,114.21) .. controls (219.77,132.8) and (211.8,128.55) .. (201.71,148.21) .. controls (191.62,167.87) and (183.65,156.18) .. (157.09,116.34) .. controls (130.52,76.49) and (162.93,52.59) .. (180.99,40.37) -- cycle ;
\draw  [fill={rgb, 255:red, 250; green, 250; blue, 250 }  ,fill opacity=1 ] (243.68,66.4) .. controls (251.65,24.43) and (292.56,41.96) .. (295.74,73.3) .. controls (298.93,104.65) and (292.02,90.3) .. (289.9,108.9) .. controls (287.77,127.49) and (307.43,158.84) .. (272.37,157.77) .. controls (237.3,156.71) and (258.02,132.27) .. (258.55,115.8) .. controls (259.09,99.34) and (235.71,108.37) .. (243.68,66.4) -- cycle ;
\draw [color={rgb, 255:red, 208; green, 2; blue, 27 }  ,draw opacity=1 ][line width=1.5]    (210.21,134.93) .. controls (227.74,142.37) and (241.55,125.9) .. (255.37,131.21) ;
\draw [shift={(255.37,131.21)}, rotate = 21.04] [color={rgb, 255:red, 208; green, 2; blue, 27 }  ,draw opacity=1 ][fill={rgb, 255:red, 208; green, 2; blue, 27 }  ,fill opacity=1 ][line width=1.5]      (0, 0) circle [x radius= 1.74, y radius= 1.74]   ;
\draw [shift={(210.21,134.93)}, rotate = 22.99] [color={rgb, 255:red, 208; green, 2; blue, 27 }  ,draw opacity=1 ][fill={rgb, 255:red, 208; green, 2; blue, 27 }  ,fill opacity=1 ][line width=1.5]      (0, 0) circle [x radius= 1.74, y radius= 1.74]   ;
\draw [color={rgb, 255:red, 208; green, 2; blue, 27 }  ,draw opacity=1 ][line width=1.5]    (206.23,65.34) .. controls (221.37,73.84) and (231.99,57.37) .. (243.68,66.4) ;
\draw [shift={(243.68,66.4)}, rotate = 37.69] [color={rgb, 255:red, 208; green, 2; blue, 27 }  ,draw opacity=1 ][fill={rgb, 255:red, 208; green, 2; blue, 27 }  ,fill opacity=1 ][line width=1.5]      (0, 0) circle [x radius= 1.74, y radius= 1.74]   ;
\draw [shift={(206.23,65.34)}, rotate = 29.31] [color={rgb, 255:red, 208; green, 2; blue, 27 }  ,draw opacity=1 ][fill={rgb, 255:red, 208; green, 2; blue, 27 }  ,fill opacity=1 ][line width=1.5]      (0, 0) circle [x radius= 1.74, y radius= 1.74]   ;
\draw [color={rgb, 255:red, 65; green, 117; blue, 5 }  ,draw opacity=1 ][line width=1.5]    (203.3,48.34) -- (206.23,65.34) ;
\draw [shift={(206.23,65.34)}, rotate = 80.25] [color={rgb, 255:red, 65; green, 117; blue, 5 }  ,draw opacity=1 ][fill={rgb, 255:red, 65; green, 117; blue, 5 }  ,fill opacity=1 ][line width=1.5]      (0, 0) circle [x radius= 1.74, y radius= 1.74]   ;
\draw [shift={(205.31,59.99)}, rotate = 260.25] [fill={rgb, 255:red, 65; green, 117; blue, 5 }  ,fill opacity=1 ][line width=0.08]  [draw opacity=0] (4.64,-2.23) -- (0,0) -- (4.64,2.23) -- cycle    ;
\draw [shift={(203.3,48.34)}, rotate = 80.25] [color={rgb, 255:red, 65; green, 117; blue, 5 }  ,draw opacity=1 ][fill={rgb, 255:red, 65; green, 117; blue, 5 }  ,fill opacity=1 ][line width=1.5]      (0, 0) circle [x radius= 1.74, y radius= 1.74]   ;
\draw [color={rgb, 255:red, 65; green, 117; blue, 5 }  ,draw opacity=1 ][line width=1.5]    (201.71,148.21) -- (210.21,134.93) ;
\draw [shift={(210.21,134.93)}, rotate = 302.62] [color={rgb, 255:red, 65; green, 117; blue, 5 }  ,draw opacity=1 ][fill={rgb, 255:red, 65; green, 117; blue, 5 }  ,fill opacity=1 ][line width=1.5]      (0, 0) circle [x radius= 1.74, y radius= 1.74]   ;
\draw [shift={(207.69,138.88)}, rotate = 122.62] [fill={rgb, 255:red, 65; green, 117; blue, 5 }  ,fill opacity=1 ][line width=0.08]  [draw opacity=0] (4.64,-2.23) -- (0,0) -- (4.64,2.23) -- cycle    ;
\draw [shift={(201.71,148.21)}, rotate = 302.62] [color={rgb, 255:red, 65; green, 117; blue, 5 }  ,draw opacity=1 ][fill={rgb, 255:red, 65; green, 117; blue, 5 }  ,fill opacity=1 ][line width=1.5]      (0, 0) circle [x radius= 1.74, y radius= 1.74]   ;

\draw (254.21,84.44) node [anchor=north west][inner sep=0.75pt]  [font=\scriptsize]  {$f_{2}$};
\draw (163.9,90.28) node [anchor=north west][inner sep=0.75pt]  [font=\scriptsize]  {$f_{1}$};
\draw (222.01,46.13) node [anchor=north west][inner sep=0.75pt]  [font=\scriptsize,color={rgb, 255:red, 208; green, 2; blue, 27 }  ,opacity=1 ]  {$p$};
\draw (224.66,141.75) node [anchor=north west][inner sep=0.75pt]  [font=\scriptsize,color={rgb, 255:red, 208; green, 2; blue, 27 }  ,opacity=1 ]  {$q$};
\draw (187.48,51.77) node [anchor=north west][inner sep=0.75pt]  [font=\scriptsize,color={rgb, 255:red, 65; green, 117; blue, 5 }  ,opacity=1 ]  {$e_{1}$};
\draw (188.01,127.2) node [anchor=north west][inner sep=0.75pt]  [font=\scriptsize,color={rgb, 255:red, 65; green, 117; blue, 5 }  ,opacity=1 ]  {$e_{2}$};

\end{tikzpicture}

  \captionof{figure}{}
  \label{fig:two-faces}
\end{minipage}%
\begin{minipage}{.5\textwidth}
  \centering

\tikzset{every picture/.style={line width=0.75pt}} 

\begin{tikzpicture}[x=0.75pt,y=0.75pt,yscale=-1,xscale=1]

\draw  [fill={rgb, 255:red, 250; green, 250; blue, 250 }  ,fill opacity=1 ] (172.2,42.6) .. controls (190.26,30.38) and (181.8,41.4) .. (198.6,51.8) .. controls (215.4,62.2) and (191.8,69.8) .. (213,92.6) .. controls (234.2,115.4) and (196.29,93.74) .. (186.2,113.4) .. controls (176.11,133.06) and (161.57,129.27) .. (152.29,105.14) .. controls (143,81) and (154.14,54.82) .. (172.2,42.6) -- cycle ;
\draw  [fill={rgb, 255:red, 250; green, 250; blue, 250 }  ,fill opacity=1 ] (236.48,34.4) .. controls (240.76,22.2) and (283.8,25.4) .. (288.54,41.3) .. controls (293.29,57.21) and (284.82,58.3) .. (282.7,76.9) .. controls (280.57,95.49) and (271.79,97.87) .. (250.19,94.27) .. controls (228.59,90.67) and (246.07,76.27) .. (246.6,59.8) .. controls (247.13,43.33) and (232.2,46.6) .. (236.48,34.4) -- cycle ;
\draw [color={rgb, 255:red, 208; green, 2; blue, 27 }  ,draw opacity=1 ][line width=1.5]    (198.6,51.8) .. controls (213.74,60.3) and (223.71,29.17) .. (235.4,38.2) ;
\draw [shift={(235.4,38.2)}, rotate = 37.69] [color={rgb, 255:red, 208; green, 2; blue, 27 }  ,draw opacity=1 ][fill={rgb, 255:red, 208; green, 2; blue, 27 }  ,fill opacity=1 ][line width=1.5]      (0, 0) circle [x radius= 1.74, y radius= 1.74]   ;
\draw [shift={(198.6,51.8)}, rotate = 29.31] [color={rgb, 255:red, 208; green, 2; blue, 27 }  ,draw opacity=1 ][fill={rgb, 255:red, 208; green, 2; blue, 27 }  ,fill opacity=1 ][line width=1.5]      (0, 0) circle [x radius= 1.74, y radius= 1.74]   ;
\draw  [fill={rgb, 255:red, 250; green, 250; blue, 250 }  ,fill opacity=1 ] (243,112.6) .. controls (266.2,114.6) and (274.6,116.6) .. (276.6,132.2) .. controls (278.6,147.8) and (268.3,140.5) .. (257,144.2) .. controls (245.7,147.9) and (256.2,154.2) .. (234.6,150.6) .. controls (213,147) and (193,144.2) .. (204.2,133.8) .. controls (215.4,123.4) and (219.8,110.6) .. (243,112.6) -- cycle ;
\draw [color={rgb, 255:red, 208; green, 2; blue, 27 }  ,draw opacity=1 ][line width=1.5]    (243,112.6) .. controls (249.39,103.47) and (239.39,103.07) .. (250.19,94.27) ;
\draw [shift={(250.19,94.27)}, rotate = 320.83] [color={rgb, 255:red, 208; green, 2; blue, 27 }  ,draw opacity=1 ][fill={rgb, 255:red, 208; green, 2; blue, 27 }  ,fill opacity=1 ][line width=1.5]      (0, 0) circle [x radius= 1.74, y radius= 1.74]   ;
\draw [shift={(243,112.6)}, rotate = 304.98] [color={rgb, 255:red, 208; green, 2; blue, 27 }  ,draw opacity=1 ][fill={rgb, 255:red, 208; green, 2; blue, 27 }  ,fill opacity=1 ][line width=1.5]      (0, 0) circle [x radius= 1.74, y radius= 1.74]   ;
\draw [color={rgb, 255:red, 208; green, 2; blue, 27 }  ,draw opacity=1 ][line width=1.5]    (231.4,112.6) .. controls (225,106.6) and (221.4,111.4) .. (219.8,103.4) ;
\draw [shift={(219.8,103.4)}, rotate = 258.69] [color={rgb, 255:red, 208; green, 2; blue, 27 }  ,draw opacity=1 ][fill={rgb, 255:red, 208; green, 2; blue, 27 }  ,fill opacity=1 ][line width=1.5]      (0, 0) circle [x radius= 1.74, y radius= 1.74]   ;
\draw [shift={(231.4,112.6)}, rotate = 223.15] [color={rgb, 255:red, 208; green, 2; blue, 27 }  ,draw opacity=1 ][fill={rgb, 255:red, 208; green, 2; blue, 27 }  ,fill opacity=1 ][line width=1.5]      (0, 0) circle [x radius= 1.74, y radius= 1.74]   ;

\draw (256.61,54.04) node [anchor=north west][inner sep=0.75pt]  [font=\scriptsize]  {$f_{2}$};
\draw (161.1,71.88) node [anchor=north west][inner sep=0.75pt]  [font=\scriptsize]  {$f_{1}$};
\draw (205.21,30.53) node [anchor=north west][inner sep=0.75pt]  [font=\scriptsize,color={rgb, 255:red, 208; green, 2; blue, 27 }  ,opacity=1 ]  {$p_{1}$};
\draw (229.1,125.88) node [anchor=north west][inner sep=0.75pt]  [font=\scriptsize]  {$f_{3}$};
\draw (252.19,97.67) node [anchor=north west][inner sep=0.75pt]  [font=\scriptsize,color={rgb, 255:red, 208; green, 2; blue, 27 }  ,opacity=1 ]  {$p_{2}$};
\draw (202.59,106.47) node [anchor=north west][inner sep=0.75pt]  [font=\scriptsize,color={rgb, 255:red, 208; green, 2; blue, 27 }  ,opacity=1 ]  {$p_{3}$};

\end{tikzpicture}

  \captionof{figure}{}
  \label{fig:three-faces}
\end{minipage}
\end{figure}

We can also prove a similar result for three faces. 

\begin{lemma}\label{lem:three-faces-small-intersection}
     Let $X$ be a connected, quasi-transitive, locally finite graph. Let $\Gamma$ be a connected, locally finite, planar graph. Let $\varphi : X \to \Gamma$ be a quasi-isometry. Suppose further that $\Gamma$ is almost 2-connected. 
    Then for every $r > 0$ there exists $n = n(r)> 0$ such that the following holds. Let $f_1, f_2, f_3 \in \facepaths(\Gamma)$ be pairwise distinct. 
    Suppose there exists 
    \begin{align*}\tag{$\ddagger$}\label{eq:triple}
        x_1 &\in B_\Gamma(f_1;r) \cap B_\Gamma(f_2;r),\\ x_2 &\in B_\Gamma(f_2;r) \cap B_\Gamma(f_3;r),\\ x_3 &\in B_\Gamma(f_3;r) \cap B_\Gamma(f_1;r).
    \end{align*}
    Then $\diam_\Gamma(\{x_1,x_2,x_3\}) \leq n$. 
\end{lemma}

\begin{proof}
    This argument is very similar to that of Lemma~\ref{lem:two-faces-small-intersection}, so we shall only give a sketch and leave the details to the reader. 

    As in the proof of Lemma~\ref{lem:two-faces-small-intersection}, we can assume that $\Gamma$ is 2-connected without any real loss of generality. Fix $r > 0$ and suppose that for every $n > 0$ there exists distinct facial subgraphs $f_1, f_2, f_3 \in \facepaths(\Gamma)$ and points $x_1, x_2, x_3$ such that (\ref{eq:triple}) holds, and $\diam_\Gamma(\{x_1,x_2,x_3\}) > n$.
    Since we assume $\Gamma$ is 2-connected, the closure of each $f_i$ in $\overline \Gamma$ is a simple closed curve by Proposition~\ref{prop:simple-face}. 

    For each $i = 1,2,3$, let $p_i$ be a simple path connecting $f_i$ to $f_{i+1}$ which is contained a bounded neighbourhood of $x_i$, and intersects each of $f_i$ and $f_{i+1}$ in exactly one vertex. Here, indices are to be taken modulo 3. It is clear pictorially that $\overline \Gamma \setminus (p_1 \cup p_2 \cup p_3)$ is disconnected (see Figure~\ref{fig:three-faces}). Without loss of generality, $p_1$ and $p_2$ are at distance approximately $n$ from each other. In fact, we may assume that all three $p_i$ are pairwise far apart, as if $p_2$ is near $p_3$ then we deduce that the intersection of bounded neighbourhoods of $f_1$ and $f_2$ has large diameter, and we can apply Lemma~\ref{lem:two-faces-small-intersection}. Thus, without loss of generality the $p_i$ are pairwise disjoint.
    
    As in the proof of Lemma~\ref{lem:two-faces-small-intersection} we now construct a tight element $b \in \br(\Gamma)$ such that $|\delta b|$ is uniformly bounded but $\diam_\Gamma(\delta b)$ is approximately at least $n$. This leads to a contradiction of Lemma~\ref{lem:bounded-diam-coarse}, and so we are done.
\end{proof}


\subsection{Combinatorial Jordan curves}

We now prove the following two results, which will be helpful for showing that two points lie near a common face later on.

\begin{lemma}\label{lem:scc-bdry}
    Let $\Gamma$ be a connected, locally finite, planar graph with fixed good drawing $\vartheta$. Given a finite collection $U_1, \ldots , U_n \in \facedisks(\Gamma)$,  let $V$ be a connected component of 
    $
    \bbS^2 \setminus \bigcup_{i} \overline U_i. 
    $
    Then $\overline V$ is locally (path) connected, and every connected component of $\partial V$ is a simple closed curve.
\end{lemma}

\begin{proof}
    In what follows, we will suppress $\vartheta$ from our notation for the sake of readability. That is, we will identify $\overline \Gamma$ with its $\vartheta$-image in $\bbS^2$.

    We first show that $\overline V$ is locally connected, i.e. that for every $x \in \overline V$ and every open subset $W \subset \bbS^2$ containing $x$, there exists an open subset $O \subset W$ with $x \in O$ and $O \cap \overline V$ connected. The only non-trivial case to check is when $x \in \Omega (\Gamma)$, i.e. $x$ is an end of $\Gamma$. 
    By Proposition~\ref{prop:freud-prop}, we have that $x$ is not a local cut-point of $\overline \Gamma$, so in particular we have that $x$ lies inside the closure of at most one of the $U_i$. By replacing $W$ with a smaller open subset, we may therefore reduce to the case where $n=1$. Let us then simplify notation and write $U:= U_1$.

    Since $\overline \Gamma$  and $\bbS^2$ are locally path connected, let $O \subset W$ be an open sub-neighbourhood of $x$ such that both $O$ and 
    $\overline \Gamma \cap O$ are (path) connected. 
    It is a standard fact from plane topology that any open connected subset of the plane cannot be disconnected by the removal of any totally disconnected set. In particular, since $\partial V \subset \overline \Gamma$, we have that $\partial V \cap O$ is either empty or contains some $v_0 \in \Gamma$. That is, $v_0$ is not an end of $\Gamma$. After possibly subdividing an edge of $\Gamma$, we may assume without loss of generality that $v_0$ is a vertex of $\Gamma$. 

    Now, let $z \in \overline \Gamma \cap \overline V \cap O$ be arbitrary. Since $O \cap \overline \Gamma$ is path connected, there exists a path $p$ connecting $z$ to $v_0$ contained in $O$. By an application of Proposition~\ref{prop:freud-prop}, we may assume that $p$ is a combinatorial (possibly one-way infinite) path. 
    We claim that $p$ is contained entirely in $\overline V$. Suppose not, then there exists a `first' vertex $a$ along $p$ which is not contained $\overline V$. 
    Let $b$ be the vertex immediately preceding $a$ along $p$, so $b \in \overline V$. Let $e \in E(\Gamma)$ be the edge in $p$ connecting $a$ to $b$, so $e \cap \overline V = \{b\}$. 
    Consider now the clockwise cyclic sequence of edges $e_1, \ldots, e_m$ incoming to $b$, say $e = e_1$. Inspections reveals that there must exist $1 \leq i \leq j \leq m$ such that
    \begin{itemize}
        \item we have $e_i, e_j \subset \overline U$, and

        \item for all $\ell$ such that either $1 \leq \ell \leq i$ or $j \leq \ell \leq m$,  the interior of $e_\ell$ is disjoint from $\overline V$. 
    \end{itemize}
    Indeed, if this were not the case then it would necessarily follow that $a \in \overline V$. Note, it could be that $i = j$. 
    
    Using this observation, along with the fact that $U$ is path connected (it is open and connected) we may draw a Jordan curve $J$ which is disjoint from $\overline V$ and also separates $v_0$ from $z$, as depicted in Figure~\ref{fig:sep-curve}. This contradicts the fact that $\overline V$ is connected. 
    It follows that $\overline \Gamma \cap \overline V \cap O$ is connected. Since $O$ is connected and $\partial V \subset \overline \Gamma$ we have that $\overline V \cap O$ is connected and so $\overline V$ is locally connected. 

    Finally, as noted above we have that any connected open set of $\bbS^2$ cannot be disconnected by the removal of any totally disconnected subset. Thus $V$ is 2-connected and so is $\overline V$. By Proposition~\ref{prop:simple-face}, it follows that every connected component of  $\partial V$ is a simple closed curve. 
\end{proof}

\begin{figure}
    \centering

\tikzset{every picture/.style={line width=0.75pt}} 

\begin{tikzpicture}[x=0.75pt,y=0.75pt,yscale=-1,xscale=1]

\draw  [draw opacity=0][fill={rgb, 255:red, 241; green, 186; blue, 255 }  ,fill opacity=1 ] (249.5,170.03) .. controls (231.61,168.99) and (217.42,154.15) .. (217.42,136) .. controls (217.42,118.75) and (230.23,104.5) .. (246.86,102.23) -- (249.71,123.02) .. controls (243.32,123.89) and (238.4,129.37) .. (238.4,136) .. controls (238.4,142.98) and (243.85,148.68) .. (250.73,149.08) -- cycle ;
\draw  [draw opacity=0][fill={rgb, 255:red, 228; green, 228; blue, 228 }  ,fill opacity=1 ] (221.33,179.33) -- (260.33,193.33) -- (251.5,136) -- cycle ;
\draw  [draw opacity=0][fill={rgb, 255:red, 228; green, 228; blue, 228 }  ,fill opacity=1 ] (251.5,136) -- (277.5,98.5) -- (233.5,84) -- cycle ;
\draw [line width=2.25]    (175,136) -- (251.5,136) ;
\draw [shift={(251.5,136)}, rotate = 0] [color={rgb, 255:red, 0; green, 0; blue, 0 }  ][fill={rgb, 255:red, 0; green, 0; blue, 0 }  ][line width=2.25]      (0, 0) circle [x radius= 3.75, y radius= 3.75]   ;
\draw [shift={(175,136)}, rotate = 0] [color={rgb, 255:red, 0; green, 0; blue, 0 }  ][fill={rgb, 255:red, 0; green, 0; blue, 0 }  ][line width=2.25]      (0, 0) circle [x radius= 3.75, y radius= 3.75]   ;
\draw [line width=1.5]    (251.5,136) -- (295,145.5) ;
\draw [shift={(295,145.5)}, rotate = 12.32] [color={rgb, 255:red, 0; green, 0; blue, 0 }  ][fill={rgb, 255:red, 0; green, 0; blue, 0 }  ][line width=1.5]      (0, 0) circle [x radius= 3.05, y radius= 3.05]   ;
\draw [shift={(251.5,136)}, rotate = 12.32] [color={rgb, 255:red, 0; green, 0; blue, 0 }  ][fill={rgb, 255:red, 0; green, 0; blue, 0 }  ][line width=1.5]      (0, 0) circle [x radius= 3.05, y radius= 3.05]   ;
\draw [line width=1.5]    (132.5,129) -- (175,136) ;
\draw [shift={(175,136)}, rotate = 9.35] [color={rgb, 255:red, 0; green, 0; blue, 0 }  ][fill={rgb, 255:red, 0; green, 0; blue, 0 }  ][line width=1.5]      (0, 0) circle [x radius= 3.05, y radius= 3.05]   ;
\draw [shift={(132.5,129)}, rotate = 9.35] [color={rgb, 255:red, 0; green, 0; blue, 0 }  ][fill={rgb, 255:red, 0; green, 0; blue, 0 }  ][line width=1.5]      (0, 0) circle [x radius= 3.05, y radius= 3.05]   ;
\draw    (163.5,114) -- (175,136) ;
\draw    (178,111) -- (175,136) ;
\draw    (175,136) -- (167.67,156.67) ;
\draw    (175,136) -- (179,151.33) ;
\draw    (230,73) -- (251.5,136) ;
\draw    (285,88.5) -- (251.5,136) ;
\draw    (251.5,136) -- (214.67,188) ;
\draw [line width=1.5]    (328,146.5) -- (357,140.5) ;
\draw [shift={(357,140.5)}, rotate = 348.31] [color={rgb, 255:red, 0; green, 0; blue, 0 }  ][fill={rgb, 255:red, 0; green, 0; blue, 0 }  ][line width=1.5]      (0, 0) circle [x radius= 3.05, y radius= 3.05]   ;
\draw [shift={(328,146.5)}, rotate = 348.31] [color={rgb, 255:red, 0; green, 0; blue, 0 }  ][fill={rgb, 255:red, 0; green, 0; blue, 0 }  ][line width=1.5]      (0, 0) circle [x radius= 3.05, y radius= 3.05]   ;
\draw  [dash pattern={on 1.5pt off 5.25pt}]  (295,145.5) -- (328,146.5) ;
\draw [color={rgb, 255:red, 208; green, 2; blue, 27 }  ,draw opacity=1 ][line width=2.25]    (175,136) ;
\draw [shift={(175,136)}, rotate = 0] [color={rgb, 255:red, 208; green, 2; blue, 27 }  ,draw opacity=1 ][fill={rgb, 255:red, 208; green, 2; blue, 27 }  ,fill opacity=1 ][line width=2.25]      (0, 0) circle [x radius= 3.75, y radius= 3.75]   ;
\draw [shift={(175,136)}, rotate = 0] [color={rgb, 255:red, 208; green, 2; blue, 27 }  ,draw opacity=1 ][fill={rgb, 255:red, 208; green, 2; blue, 27 }  ,fill opacity=1 ][line width=2.25]      (0, 0) circle [x radius= 3.75, y radius= 3.75]   ;
\draw [line width=1.5]  [dash pattern={on 1.5pt off 6pt}]  (100,134) -- (132.5,129) ;
\draw [shift={(132.5,129)}, rotate = 351.25] [color={rgb, 255:red, 0; green, 0; blue, 0 }  ][fill={rgb, 255:red, 0; green, 0; blue, 0 }  ][line width=1.5]      (0, 0) circle [x radius= 3.05, y radius= 3.05]   ;
\draw [color={rgb, 255:red, 144; green, 19; blue, 254 }  ,draw opacity=1 ]   (226.5,135.67) .. controls (227.33,120.67) and (247.33,118.33) .. (257.33,110.33) .. controls (267.33,102.33) and (277.67,70.67) .. (321,71.33) .. controls (364.33,72) and (415,163.17) .. (380,184.17) .. controls (345,205.17) and (289.5,215.79) .. (262,208.33) .. controls (234.5,200.88) and (250.67,175) .. (247.33,162.33) .. controls (244,149.67) and (227.33,152.33) .. (226.5,135.67) -- cycle ;
\draw [shift={(226.5,135.67)}, rotate = 273.18] [color={rgb, 255:red, 144; green, 19; blue, 254 }  ,draw opacity=1 ][fill={rgb, 255:red, 144; green, 19; blue, 254 }  ,fill opacity=1 ][line width=0.75]      (0, 0) circle [x radius= 2.68, y radius= 2.68]   ;
\draw [line width=1.5]    (91,136) ;
\draw [shift={(91,136)}, rotate = 0] [color={rgb, 255:red, 0; green, 0; blue, 0 }  ][fill={rgb, 255:red, 0; green, 0; blue, 0 }  ][line width=1.5]      (0, 0) circle [x radius= 3.05, y radius= 3.05]   ;
\draw [shift={(91,136)}, rotate = 0] [color={rgb, 255:red, 0; green, 0; blue, 0 }  ][fill={rgb, 255:red, 0; green, 0; blue, 0 }  ][line width=1.5]      (0, 0) circle [x radius= 3.05, y radius= 3.05]   ;
\draw    (251.5,136) -- (260.33,196.33) ;
\draw    (251.5,136) -- (214.33,155.67) ;
\draw    (251.5,136) -- (213,121) ;
\draw [color={rgb, 255:red, 74; green, 144; blue, 226 }  ,draw opacity=1 ][line width=2.25]    (251.5,136) ;
\draw [shift={(251.5,136)}, rotate = 0] [color={rgb, 255:red, 74; green, 144; blue, 226 }  ,draw opacity=1 ][fill={rgb, 255:red, 74; green, 144; blue, 226 }  ,fill opacity=1 ][line width=2.25]      (0, 0) circle [x radius= 3.75, y radius= 3.75]   ;
\draw [shift={(251.5,136)}, rotate = 0] [color={rgb, 255:red, 74; green, 144; blue, 226 }  ,draw opacity=1 ][fill={rgb, 255:red, 74; green, 144; blue, 226 }  ,fill opacity=1 ][line width=2.25]      (0, 0) circle [x radius= 3.75, y radius= 3.75]   ;

\draw (347.5,120.4) node [anchor=north west][inner sep=0.75pt]    {$v_{0}$};
\draw (244.83,94.9) node [anchor=north west][inner sep=0.75pt]  [font=\footnotesize,color={rgb, 255:red, 128; green, 128; blue, 128 }  ,opacity=1 ]  {$U$};
\draw (78.5,138.4) node [anchor=north west][inner sep=0.75pt]    {$z$};
\draw (228.17,169.23) node [anchor=north west][inner sep=0.75pt]  [font=\footnotesize,color={rgb, 255:red, 128; green, 128; blue, 128 }  ,opacity=1 ]  {$U$};
\draw (193.33,140.4) node [anchor=north west][inner sep=0.75pt]  [font=\scriptsize]  {$e_{1}$};
\draw (218,75.07) node [anchor=north west][inner sep=0.75pt]  [font=\scriptsize]  {$e_{i}$};
\draw (202,173.73) node [anchor=north west][inner sep=0.75pt]  [font=\scriptsize]  {$e_{j}$};
\draw (159.08,136.57) node [anchor=north west][inner sep=0.75pt]  [font=\small,color={rgb, 255:red, 208; green, 2; blue, 27 }  ,opacity=1 ]  {$a$};
\draw (257.42,140.9) node [anchor=north west][inner sep=0.75pt]  [font=\small,color={rgb, 255:red, 74; green, 144; blue, 226 }  ,opacity=1 ]  {$b$};
\draw (301.67,187.4) node [anchor=north west][inner sep=0.75pt]  [color={rgb, 255:red, 144; green, 19; blue, 254 }  ,opacity=1 ]  {$J$};

\end{tikzpicture}
    
    \caption{Jordan curve $J$ separating $v_0$ from $z$. The purple region is disjoint from $\overline V$.}
    \label{fig:sep-curve}
\end{figure}
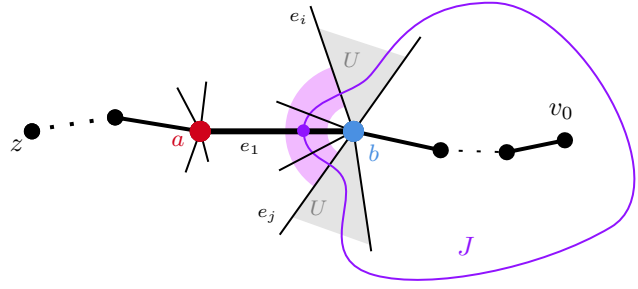

\begin{proposition}\label{prop:2-conn-jordan-curve}
    Let $\Gamma$ be a connected, locally finite, planar graph with fixed good drawing $\vartheta$. Let $x,y \in V(\Gamma)$, $r \geq 0$. Suppose there does \textbf{not} exist $f \in \facepaths(\Gamma)$ such that 
    $
    x, y \in B_\Gamma(f; r)
    $. 
    Then there exists a simple combinatorial loop $\ell$ in $\Gamma$ such that 
    \begin{enumerate}
        \item $\vartheta(x)$, $\vartheta(y)$ lie in distinct components of $\bbS^2 \setminus \vartheta(\ell)$, and

        \item $\{x,y\} \cap B_\Gamma(\ell;r) = \emptyset$.
    \end{enumerate}
\end{proposition}

\begin{proof}
    To ease notation we will identify $\overline \Gamma$ with its image in $\bbS^2$ and suppress mention of $\vartheta$. 
    
    Let $B_x = B_\Gamma(x;r)$, $B_y = B_\Gamma(y;r)$ be the closed $r$-balls about $x$ and $y$. 
    Let  $U_1, \ldots, U_n \in \facedisks(\Gamma)$ be those faces which intersect $B_x$. 
    Let $V$ be the connected component of 
    $
    \bbS^2 \setminus \bigcup_{i} \overline U_i 
    $
    which contains $y$, and thus also contains all of $B_y$. Then by Lemma~\ref{lem:scc-bdry} we have that $\partial V$ contains a simple closed curve $L$\footnote{In fact, since $\bigcup_{i} \overline U_i$ is connected it is not hard to deduce that $\partial V = L$.} which separates $B_x$ from $B_y$. Note that $L \subset \overline \Gamma$. We will modify $L$ and form a combinatorial loop $\ell$ with the same property.

    Let $W_x$, $W_y$ be the connected components of $\bbS^2 \setminus L$ containing $x$ and $y$ respectively. 
    If $L$ is not already a combinatorial loop then we must at least have that $L$ contains two distinct vertices of $\Gamma$, say $u,v \in V(\Gamma)$, since $\Omega (\Gamma)$ is totally disconnected. 
    Let $\gamma_1$, $\gamma_2$ be the connected components of $L \setminus \{v,u\}$. 
    Now, in $\bbS^2$ (or equivalently, in the plane) draw a Jordan curve $J_x$ contained in $W_x$ separating $L$ from $B_x$. Similarly, draw a Jordan curve $J_y$ contained in $W_y$ separating $L$ from $B_y$. Let $P_1$ be an arc in $W_x$ connecting $J_x$ to $u$, otherwise disjoint from $J_x$ and $L$, and $P_2$ an arc in $W_x$ connecting $J_x$ to $v$, also otherwise disjoint from $J_x$. Define $Q_1$, $Q_2 \subset \overline {W_y}$ similarly. Let 
    $$
    Z = J_x \cup J_y \cup P_1 \cup P_2 \cup Q_1 \cup Q_2.
    $$
    Note that we do not require $Z$ to be contained in $\overline \Gamma$. See Figure~\ref{fig:figure-J} for a cartoon. Note that $\bbS^2 \setminus Z$ is a disjoint union of four open disks: each containing exactly one of $B_x$, $B_y$, $\gamma_1$ and $\gamma_2$. Let $U_i$ be the disk containing $\gamma_i$. Let $O_i \subset U_i$ denote the connected component of $\overline \Gamma \cap U_i$ containing $\gamma_i$. Since $\overline \Gamma$ is locally connected, we have that $O_i$ is an open subset of $\overline \Gamma$. Now, using Proposition~\ref{prop:freud-prop} we may replace $\gamma_i$ with a combinatorial path $\ell_i$ contained in $O_i$ with the same endpoints, which is therefore drawn entirely inside of $U_i$. The loop $\ell = \ell_1 \cup \ell_2$ satisfies our requirements. 
\end{proof}

\begin{figure}
    \centering

\tikzset{every picture/.style={line width=0.75pt}} 

\begin{tikzpicture}[x=0.75pt,y=0.75pt,yscale=-1,xscale=1]

\draw [color={rgb, 255:red, 65; green, 117; blue, 5 }  ,draw opacity=1 ] [dash pattern={on 3.75pt off 1.5pt}]  (211.2,78) .. controls (215,67.4) and (223.4,69.8) .. (228.68,64.88) .. controls (233.96,59.97) and (305,58.2) .. (320.6,76.6) ;
\draw [shift={(320.6,76.6)}, rotate = 49.71] [color={rgb, 255:red, 65; green, 117; blue, 5 }  ,draw opacity=1 ][fill={rgb, 255:red, 65; green, 117; blue, 5 }  ,fill opacity=1 ][line width=0.75]      (0, 0) circle [x radius= 2.01, y radius= 2.01]   ;
\draw [shift={(211.2,78)}, rotate = 289.72] [color={rgb, 255:red, 65; green, 117; blue, 5 }  ,draw opacity=1 ][fill={rgb, 255:red, 65; green, 117; blue, 5 }  ,fill opacity=1 ][line width=0.75]      (0, 0) circle [x radius= 2.01, y radius= 2.01]   ;
\draw [color={rgb, 255:red, 65; green, 117; blue, 5 }  ,draw opacity=1 ] [dash pattern={on 3.75pt off 1.5pt}]  (227.8,147) .. controls (236.2,152.6) and (241,155) .. (246.08,159.13) .. controls (251.16,163.26) and (285.4,154.2) .. (298.6,147.8) ;
\draw [shift={(298.6,147.8)}, rotate = 334.13] [color={rgb, 255:red, 65; green, 117; blue, 5 }  ,draw opacity=1 ][fill={rgb, 255:red, 65; green, 117; blue, 5 }  ,fill opacity=1 ][line width=0.75]      (0, 0) circle [x radius= 2.01, y radius= 2.01]   ;
\draw [shift={(227.8,147)}, rotate = 33.69] [color={rgb, 255:red, 65; green, 117; blue, 5 }  ,draw opacity=1 ][fill={rgb, 255:red, 65; green, 117; blue, 5 }  ,fill opacity=1 ][line width=0.75]      (0, 0) circle [x radius= 2.01, y radius= 2.01]   ;
\draw  [color={rgb, 255:red, 155; green, 155; blue, 155 }  ,draw opacity=1 ][fill={rgb, 255:red, 228; green, 228; blue, 228 }  ,fill opacity=1 ] (174.6,117.3) .. controls (174.6,104.87) and (184.67,94.8) .. (197.1,94.8) .. controls (209.53,94.8) and (219.6,104.87) .. (219.6,117.3) .. controls (219.6,129.73) and (209.53,139.8) .. (197.1,139.8) .. controls (184.67,139.8) and (174.6,129.73) .. (174.6,117.3) -- cycle ;
\draw  [color={rgb, 255:red, 155; green, 155; blue, 155 }  ,draw opacity=1 ][fill={rgb, 255:red, 228; green, 228; blue, 228 }  ,fill opacity=1 ] (311.8,115.3) .. controls (311.8,102.87) and (321.87,92.8) .. (334.3,92.8) .. controls (346.73,92.8) and (356.8,102.87) .. (356.8,115.3) .. controls (356.8,127.73) and (346.73,137.8) .. (334.3,137.8) .. controls (321.87,137.8) and (311.8,127.73) .. (311.8,115.3) -- cycle ;
\draw  [color={rgb, 255:red, 208; green, 2; blue, 27 }  ,draw opacity=1 ] (142.17,72.99) .. controls (165.68,60.32) and (211.05,49.68) .. (247.96,72.99) .. controls (284.87,96.3) and (247.72,112.51) .. (247.96,149) .. controls (248.19,185.48) and (165.68,187) .. (142.17,149) .. controls (118.66,110.99) and (118.66,85.66) .. (142.17,72.99) -- cycle ;
\draw [color={rgb, 255:red, 208; green, 2; blue, 27 }  ,draw opacity=1 ]   (228.68,64.88) ;
\draw [shift={(228.68,64.88)}, rotate = 0] [color={rgb, 255:red, 208; green, 2; blue, 27 }  ,draw opacity=1 ][fill={rgb, 255:red, 208; green, 2; blue, 27 }  ,fill opacity=1 ][line width=0.75]      (0, 0) circle [x radius= 3.35, y radius= 3.35]   ;
\draw [color={rgb, 255:red, 208; green, 2; blue, 27 }  ,draw opacity=1 ]   (246.08,159.13) ;
\draw [shift={(246.08,159.13)}, rotate = 0] [color={rgb, 255:red, 208; green, 2; blue, 27 }  ,draw opacity=1 ][fill={rgb, 255:red, 208; green, 2; blue, 27 }  ,fill opacity=1 ][line width=0.75]      (0, 0) circle [x radius= 3.35, y radius= 3.35]   ;
\draw  [color={rgb, 255:red, 65; green, 117; blue, 5 }  ,draw opacity=1 ][dash pattern={on 3.75pt off 1.5pt}] (162.4,78.2) .. controls (182.4,68.2) and (196.4,71.4) .. (221.8,81) .. controls (247.2,90.6) and (240.8,135.4) .. (227.8,147) .. controls (214.8,158.6) and (170.2,158.2) .. (155.4,140.2) .. controls (140.6,122.2) and (142.4,88.2) .. (162.4,78.2) -- cycle ;
\draw  [color={rgb, 255:red, 65; green, 117; blue, 5 }  ,draw opacity=1 ][dash pattern={on 3.75pt off 1.5pt}] (302.4,81.4) .. controls (322.4,71.4) and (334.6,78.6) .. (361,79.4) .. controls (387.4,80.2) and (380.8,138.6) .. (367.8,150.2) .. controls (354.8,161.8) and (301.8,161) .. (295.4,143.4) .. controls (289,125.8) and (282.4,91.4) .. (302.4,81.4) -- cycle ;

\draw (185,107.6) node [anchor=north west][inner sep=0.75pt]  [color={rgb, 255:red, 128; green, 128; blue, 128 }  ,opacity=1 ]  {$B_{x}$};
\draw (325,104.8) node [anchor=north west][inner sep=0.75pt]  [color={rgb, 255:red, 128; green, 128; blue, 128 }  ,opacity=1 ]  {$B_{y}$};
\draw (108.43,88.7) node [anchor=north west][inner sep=0.75pt]  [color={rgb, 255:red, 208; green, 2; blue, 27 }  ,opacity=1 ]  {$L$};
\draw (225.76,45.67) node [anchor=north west][inner sep=0.75pt]  [font=\scriptsize,color={rgb, 255:red, 208; green, 2; blue, 27 }  ,opacity=1 ]  {$u$};
\draw (248.08,162.53) node [anchor=north west][inner sep=0.75pt]  [font=\scriptsize,color={rgb, 255:red, 208; green, 2; blue, 27 }  ,opacity=1 ]  {$v$};

\end{tikzpicture}

    \caption{The figure $Z$ (dashed) drawn in the plane.}
    \label{fig:figure-J}
\end{figure}
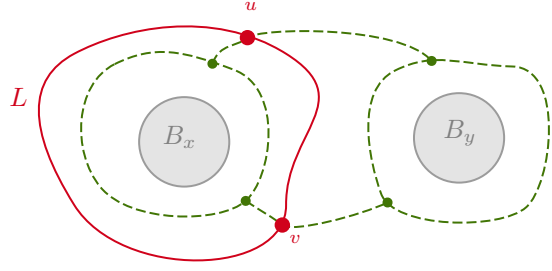







\subsection{Retaining control of our quasi-action}

We begin with some discussion, as what follows will be fairly technical.

It is possible to prove\footnote{To prove this directly, one can adapt and simplify the arguments of this section. It suffices to take $\vs(\Gamma) > 3\lambda^2 + 3$. } that if $\Qp$ is a quasi-planar tuple with 2-connected planar graph $\Gamma$ and $\vs(\Gamma)$ is sufficiently large compared to the constants associated with the induced quasi-action, then $\Qp$ is well-behaved. The pipe-dream we now experience is the hope that we might pass to some `highly connected' torso/part of some canonical tree decomposition, and apply the aforementioned result. Such parts will be quasi-isometric to some highly connected subgraph of $\Gamma$, surely?

Unfortunately, an important subtlety arises here. When we pass to a part of some canonical tree decomposition, simply restricting our quasi-isometries to these parts will not quite produce a well-defined quasi-planar tuple. Instead, we must modify the restricted quasi-inverse slightly in order to ensure that the domains and codomains align. This causes our quasi-isometry constants to increase, and thus moves the goalposts on what we consider to be ``highly connected''. In order to break out of this cycle, some additional bookkeeping is required, which will allow us to refer back to the original quasi-action on the super-graph. Roughly speaking, we will consider a tubular neighbourhood of our subgraph, and approximate the faces of the `outer' subgraph with faces of the `inner' subgraph. The induced quasi-action on this `inner' subgraph will agree with the original quasi-action, and thus we avoid the problem of our quasi-isometry constants inflating. This is the motivation for wanting friendly-faced subgraphs. 

We now set up some notation which will follow us for the rest of \S\ref{sec:quasi-act-planar}. 
In what follows, fix $\lambda \geq 1$ such that the induced quasi-action of $G$ upon $\Gamma$ is a $\lambda$-quasi-action (as in Definition~\ref{def:quasi-action}). We also assume that $\psi$ is a $\lambda$-quasi-inverse to $\varphi$. Applying Theorem~\ref{thm:vertex-subgraph} and Lemma~\ref{lem:qi-increase-cuts}, we obtain a $G$-canonical tree decomposition $(T, \cV)$ of $X$ with finite adhesion and connected parts, where $T/G$ is compact. Writing $\cV = (V_u)_{u \in V(T)}$, we may also assume for each part $X_u  = X[V_u]$ that: 
\begin{itemize}
    \item $X_u$ is quasi-transitively stabilised and connected. 
    \item $\varphi(X_u)$ is highly connected. In particular, we have that
    $$
    \vs(\varphi(X_u)) > 3\lambda^2 + 3. 
    $$
\end{itemize}
By Proposition~\ref{prop:replace-parts-with-nbhd}, we may replace each part of our tree decomposition with a bounded neighbourhood of itself and not affect any of the above properties. In particular, by Theorem~\ref{thm:friendly-faces} we may assume the following without loss of generality:
\begin{itemize}
    \item $\varphi(X_u)$ is a friendly-faced subgraph of $\Gamma$.
\end{itemize}

We now fix $u \in V(T)$ and consider a specific part of this tree decomposition. The notation we fix here will follow us through the rest of this section. Let 
$$
Z = X_u, \ \ \Lambda = \varphi(Z), \ \ H = \Stab(u). 
$$
Note that $H$ acts on $Z$ quasi-transitively. 
We now pick a tubular neighbourhood of $Z$. Let 
$$
\ytwo = B_X(\yone; \lambda), \ \ \ptwo = \varphi(\ytwo). 
$$
Also, let $\mu  : \ytwo \onto \ptwo$ denote the restriction of $\varphi$ to $\ytwo$. By Proposition~\ref{prop:restriction-cts}, $\mu $ is a quasi-isometry (where $Y$ and $\Pi$ are considered with their own intrinsic path-metric).  
We need to be a bit careful about how we choose a quasi-inverse to $\mu$. 

\begin{lemma}\label{lem:qa-agrees}
    There exists a quasi-isometry $\nu : \ptwo \to \ytwo$ such that the following hold:
    \begin{enumerate}
        \item $\nu$ is a quasi-inverse to $\mu $,
        \item $\nu|_{\pone} = \psi|_{\pone}$. 
    \end{enumerate}
    In particular, for all $g \in H$ and $x \in\pone$ we have that $\varphi_g(x) = \mu _g(x)$.
\end{lemma}

\begin{proof}
    Since $\psi$ is a $\lambda$-quasi-inverse to $\varphi$, we have that $\psi(\Lambda) \subset B_X(\yone; \lambda) = \ytwo$. It is immediate from this observation that such a $\nu$ exists.
\end{proof}

Note that since $\Lambda$ and $\Pi$ have uniform coboundary, we have by Remark~\ref{rmk:good-drawing-subgraph} that the restriction of $\vartheta$ to closure of one of these subgraphs in $\overline \Gamma$ is also a good drawing. 

We record the following helpful consequence of $\Lambda$ being friendly-faced. 

\begin{lemma}\label{lem:approx-faces}
    There exists $M > 0$ such that
    for every $f_1 \in \facepaths (\pone)$ there exists $f_2 \in \facepaths(\ptwo)$ with $$\dHaus[\Gamma](f_1, f_2) < M.$$ 
\end{lemma}

\begin{proof}
    By Lemma~\ref{lem:friendly-faces-chain}, we have that $\Lambda$ is a friendly-faced subgraph of $\Pi$. Thus, for every $U \in \facedisks(\Lambda)$ there exists $U' \in \facepaths (\Pi)$ such that $U' \subset U$ and $\facepaths[U]$ is contained in a bounded neighbourhood of $\facepaths[U']$. But since $\Pi$ is contained in a bounded neighbourhood of $\Lambda$, it is clear that actually we must have $\dHaus[\Gamma](\facepaths[U], \facepaths[U'])$ is uniformly bounded. 
\end{proof}

The importance of this is that now we may study the quasi-action of $H$ on the faces of $\ptwo$ by approximating them as faces of $\pone$. Then, we may apply Lemma~\ref{lem:qa-agrees} to inspect their quasi-translates by referring back to the quasi-action of $G$ on $\Gamma$, for which we have `small constants'. 

\medskip
\medskip

{

\centering

\textbf{For the rest of \S\ref{sec:quasi-act-planar}, we will fix all of the above notation.}

}

\medskip
\medskip

We summarise the current sitation below, for ease of reference:
\begin{itemize}
    \item $\Qp = (G, X, \Gamma, \vartheta, \varphi, \psi)$ is a quasi-planar tuple (see Definition~\ref{def:qp}).

    \item $\lambda \geq 1$ is such that the induced quasi-action of $G$ upon $\Gamma$ is a $\lambda$-quasi-action. 
    
    \item $Z \subset X$ is connected, quasi-transitively stabilised, and has uniform coboundary.

    \item $H \leq G$ is the setwise stabiliser of $Z$.

    \item $Y = B_X(Z;\lambda)$.

    \item $\Lambda = \varphi(Z)$, $\Pi = \varphi(Y)$

    \item $\Lambda$ and $\Pi$ have uniform coboundary in $\Gamma$.

    \item $\Lambda$ is a friendly-faced subgraph of both $\Gamma$ and $\Pi$ (see Definition~\ref{def:friendly-faced}).

    \item $\Lambda$ and $\Pi$ are almost 2-connected (see Definition~\ref{def:nearly-2-conn}). 

    \item $\vs(\Lambda) > 3\lambda^2 + 3$.

    \item $\mu  : \ytwo \onto \ptwo$ denotes the restriction of $\varphi$ to $Y$, which is a quasi-isometry. 

    \item $\nu : \Pi \to Y$ is a quasi-inverse to $\mu$, chosen such that the induced quasi-action $\mu_g$ of $H$ upon $\Pi$ satisfies 
    $
    \varphi_g(x) = \mu _g(x)
    $
    for all $x \in \Lambda$, $g \in H$. 
\end{itemize}

We also assume without loss of generality that $Z$ has more than one end. See also Figure~\ref{fig:XYZ} for a small cartoon. 

\begin{figure}
    \centering
    \tikzset{every picture/.style={line width=0.75pt}} 

\begin{tikzpicture}[x=0.75pt,y=0.75pt,yscale=-1,xscale=1]

\draw  [color={rgb, 255:red, 128; green, 128; blue, 128 }  ,draw opacity=1 ][fill={rgb, 255:red, 231; green, 231; blue, 231 }  ,fill opacity=1 ] (139.5,59.75) -- (214.33,59.75) -- (214.33,152) -- (139.5,152) -- cycle ;
\draw  [color={rgb, 255:red, 208; green, 2; blue, 27 }  ,draw opacity=1 ][fill={rgb, 255:red, 250; green, 227; blue, 227 }  ,fill opacity=1 ] (144.5,115.04) .. controls (144.5,96.93) and (159.18,82.25) .. (177.29,82.25) .. controls (195.4,82.25) and (210.08,96.93) .. (210.08,115.04) .. controls (210.08,133.15) and (195.4,147.83) .. (177.29,147.83) .. controls (159.18,147.83) and (144.5,133.15) .. (144.5,115.04) -- cycle ;
\draw  [color={rgb, 255:red, 65; green, 117; blue, 5 }  ,draw opacity=1 ][fill={rgb, 255:red, 223; green, 237; blue, 205 }  ,fill opacity=1 ] (153.88,121.48) .. controls (153.88,108.44) and (164.44,97.88) .. (177.48,97.88) .. controls (190.52,97.88) and (201.08,108.44) .. (201.08,121.48) .. controls (201.08,134.52) and (190.52,145.08) .. (177.48,145.08) .. controls (164.44,145.08) and (153.88,134.52) .. (153.88,121.48) -- cycle ;
\draw  [color={rgb, 255:red, 128; green, 128; blue, 128 }  ,draw opacity=1 ][fill={rgb, 255:red, 231; green, 231; blue, 231 }  ,fill opacity=1 ] (269.25,59.75) -- (344.08,59.75) -- (344.08,152) -- (269.25,152) -- cycle ;
\draw  [color={rgb, 255:red, 208; green, 2; blue, 27 }  ,draw opacity=1 ][fill={rgb, 255:red, 250; green, 227; blue, 227 }  ,fill opacity=1 ] (274.25,115.04) .. controls (274.25,96.93) and (288.93,82.25) .. (307.04,82.25) .. controls (325.15,82.25) and (339.83,96.93) .. (339.83,115.04) .. controls (339.83,133.15) and (325.15,147.83) .. (307.04,147.83) .. controls (288.93,147.83) and (274.25,133.15) .. (274.25,115.04) -- cycle ;
\draw  [color={rgb, 255:red, 65; green, 117; blue, 5 }  ,draw opacity=1 ][fill={rgb, 255:red, 223; green, 237; blue, 205 }  ,fill opacity=1 ] (283.63,121.48) .. controls (283.63,108.44) and (294.19,97.88) .. (307.23,97.88) .. controls (320.27,97.88) and (330.83,108.44) .. (330.83,121.48) .. controls (330.83,134.52) and (320.27,145.08) .. (307.23,145.08) .. controls (294.19,145.08) and (283.63,134.52) .. (283.63,121.48) -- cycle ;
\draw    (221.08,72.25) .. controls (230.33,64.39) and (251.56,64.87) .. (263.4,68.35) ;
\draw [shift={(266.08,69.25)}, rotate = 201.12] [fill={rgb, 255:red, 0; green, 0; blue, 0 }  ][line width=0.08]  [draw opacity=0] (5.36,-2.57) -- (0,0) -- (5.36,2.57) -- cycle    ;
\draw    (223.22,78.02) .. controls (234.24,78.84) and (254.68,79.05) .. (265.08,74.75) ;
\draw [shift={(220.08,77.75)}, rotate = 5.86] [fill={rgb, 255:red, 0; green, 0; blue, 0 }  ][line width=0.08]  [draw opacity=0] (5.36,-2.57) -- (0,0) -- (5.36,2.57) -- cycle    ;
\draw [color={rgb, 255:red, 208; green, 2; blue, 27 }  ,draw opacity=1 ]   (212.93,115) .. controls (225.54,107.01) and (254.77,107.64) .. (270.49,111.27) ;
\draw [shift={(273.33,112)}, rotate = 196.06] [fill={rgb, 255:red, 208; green, 2; blue, 27 }  ,fill opacity=1 ][line width=0.08]  [draw opacity=0] (5.36,-2.57) -- (0,0) -- (5.36,2.57) -- cycle    ;
\draw [color={rgb, 255:red, 208; green, 2; blue, 27 }  ,draw opacity=1 ]   (214.6,121.2) .. controls (229.05,122.08) and (257.63,122.42) .. (271.99,118) ;
\draw [shift={(211.58,121)}, rotate = 4.37] [fill={rgb, 255:red, 208; green, 2; blue, 27 }  ,fill opacity=1 ][line width=0.08]  [draw opacity=0] (5.36,-2.57) -- (0,0) -- (5.36,2.57) -- cycle    ;

\draw (141.5,63.15) node [anchor=north west][inner sep=0.75pt]    {$X$};
\draw (171.75,83.9) node [anchor=north west][inner sep=0.75pt]  [font=\footnotesize,color={rgb, 255:red, 208; green, 2; blue, 27 }  ,opacity=1 ]  {$Y$};
\draw (172,114.9) node [anchor=north west][inner sep=0.75pt]  [font=\footnotesize,color={rgb, 255:red, 65; green, 117; blue, 5 }  ,opacity=1 ]  {$Z$};
\draw (271.25,63.15) node [anchor=north west][inner sep=0.75pt]    {$\Gamma $};
\draw (301.5,83.9) node [anchor=north west][inner sep=0.75pt]  [font=\footnotesize,color={rgb, 255:red, 208; green, 2; blue, 27 }  ,opacity=1 ]  {$\Pi $};
\draw (301.75,114.9) node [anchor=north west][inner sep=0.75pt]  [font=\footnotesize,color={rgb, 255:red, 65; green, 117; blue, 5 }  ,opacity=1 ]  {$\Lambda $};
\draw (235,52.4) node [anchor=north west][inner sep=0.75pt]  [font=\scriptsize]  {$\varphi $};
\draw (238.75,80.4) node [anchor=north west][inner sep=0.75pt]  [font=\scriptsize]  {$\psi $};
\draw (235.91,95.65) node [anchor=north west][inner sep=0.75pt]  [font=\scriptsize,color={rgb, 255:red, 208; green, 2; blue, 27 }  ,opacity=1 ]  {$\mu $};
\draw (235.68,124.4) node [anchor=north west][inner sep=0.75pt]  [font=\scriptsize,color={rgb, 255:red, 208; green, 2; blue, 27 }  ,opacity=1 ]  {$\nu $};

\end{tikzpicture}
    \caption{}
    \label{fig:XYZ}
\end{figure}

\subsection{Quasi-translates of faces}

We now apply the results of the previous section and study the quasi-action of $H$ on the faces of $\ptwo$. 
We will need the following easy lemma, relating to finite planar graphs. We leave the proof as an exercise.

\begin{lemma}\label{lem:jordan-curve-finite}
    Let $\Gamma$ be a finite, connected planar graph. Suppose that $\Gamma$ contains: 
    \begin{enumerate}
        \item Two disjoint connected  subgraphs $\Lambda_1$, $\Lambda_2$.

        \item Three pairwise disjoint paths $\alpha_1$, $\alpha_2$, $\alpha_3$ connecting $\Lambda_1$ to $\Lambda_2$.  

        \item Three paths $\beta_1$, $\beta_2$, $\beta_3$ such that:
        \begin{enumerate}
            \item Each $\beta_i$ connects $\alpha_i$ to $\alpha_{i+1}$, but is disjoint from $\alpha_{i+2}$, where indices are taken modulo 3, and

            \item Each $\beta_i$ is disjoint from both $\Lambda_1$ and $\Lambda_2$. 
        \end{enumerate}
    \end{enumerate}
    Then the image of any embedding $\vartheta$ of $\Gamma$ in $\bbS^2$ contains a Jordan curve separating $\vartheta(\Lambda_1)$ from $\vartheta(\Lambda_2)$. In particular, no $x \in \Lambda_1$ and $y \in \Lambda_2$ can lie on a common facial subgraph of $\Gamma$ in any embedding.  
\end{lemma}

See Figure~\ref{fig:simple-example} for an illustration of Lemma~\ref{lem:jordan-curve-finite}. 

\begin{figure}[h]
    \centering

\tikzset{every picture/.style={line width=0.75pt}} 

\begin{tikzpicture}[x=0.75pt,y=0.75pt,yscale=-1,xscale=1]

\draw  [color={rgb, 255:red, 155; green, 155; blue, 155 }  ,draw opacity=1 ][fill={rgb, 255:red, 228; green, 228; blue, 228 }  ,fill opacity=1 ] (331.8,64.9) .. controls (331.8,52.47) and (341.87,42.4) .. (354.3,42.4) .. controls (366.73,42.4) and (376.8,52.47) .. (376.8,64.9) .. controls (376.8,77.33) and (366.73,87.4) .. (354.3,87.4) .. controls (341.87,87.4) and (331.8,77.33) .. (331.8,64.9) -- cycle ;
\draw  [color={rgb, 255:red, 155; green, 155; blue, 155 }  ,draw opacity=1 ][fill={rgb, 255:red, 228; green, 228; blue, 228 }  ,fill opacity=1 ] (469,62.9) .. controls (469,50.47) and (479.07,40.4) .. (491.5,40.4) .. controls (503.93,40.4) and (514,50.47) .. (514,62.9) .. controls (514,75.33) and (503.93,85.4) .. (491.5,85.4) .. controls (479.07,85.4) and (469,75.33) .. (469,62.9) -- cycle ;
\draw [color={rgb, 255:red, 74; green, 144; blue, 226 }  ,draw opacity=1 ]   (357,45.8) .. controls (397,15.8) and (408.4,48.8) .. (442.6,49) .. controls (476.8,49.2) and (475.4,33.4) .. (490.2,44.6) ;
\draw [shift={(490.2,44.6)}, rotate = 37.12] [color={rgb, 255:red, 74; green, 144; blue, 226 }  ,draw opacity=1 ][fill={rgb, 255:red, 74; green, 144; blue, 226 }  ,fill opacity=1 ][line width=0.75]      (0, 0) circle [x radius= 2.34, y radius= 2.34]   ;
\draw [shift={(357,45.8)}, rotate = 323.13] [color={rgb, 255:red, 74; green, 144; blue, 226 }  ,draw opacity=1 ][fill={rgb, 255:red, 74; green, 144; blue, 226 }  ,fill opacity=1 ][line width=0.75]      (0, 0) circle [x radius= 2.34, y radius= 2.34]   ;
\draw [color={rgb, 255:red, 74; green, 144; blue, 226 }  ,draw opacity=1 ]   (371,61.8) .. controls (411,69) and (398.4,74.4) .. (432.6,74.6) .. controls (466.8,74.8) and (452.2,69.4) .. (474.2,66.6) ;
\draw [shift={(474.2,66.6)}, rotate = 352.75] [color={rgb, 255:red, 74; green, 144; blue, 226 }  ,draw opacity=1 ][fill={rgb, 255:red, 74; green, 144; blue, 226 }  ,fill opacity=1 ][line width=0.75]      (0, 0) circle [x radius= 2.34, y radius= 2.34]   ;
\draw [shift={(371,61.8)}, rotate = 10.2] [color={rgb, 255:red, 74; green, 144; blue, 226 }  ,draw opacity=1 ][fill={rgb, 255:red, 74; green, 144; blue, 226 }  ,fill opacity=1 ][line width=0.75]      (0, 0) circle [x radius= 2.34, y radius= 2.34]   ;
\draw [color={rgb, 255:red, 74; green, 144; blue, 226 }  ,draw opacity=1 ]   (358.6,84.2) .. controls (367.4,116.2) and (391.4,101.8) .. (426.2,95.4) .. controls (461,89) and (478.6,107.4) .. (489.8,83.4) ;
\draw [shift={(489.8,83.4)}, rotate = 295.02] [color={rgb, 255:red, 74; green, 144; blue, 226 }  ,draw opacity=1 ][fill={rgb, 255:red, 74; green, 144; blue, 226 }  ,fill opacity=1 ][line width=0.75]      (0, 0) circle [x radius= 2.34, y radius= 2.34]   ;
\draw [shift={(358.6,84.2)}, rotate = 74.62] [color={rgb, 255:red, 74; green, 144; blue, 226 }  ,draw opacity=1 ][fill={rgb, 255:red, 74; green, 144; blue, 226 }  ,fill opacity=1 ][line width=0.75]      (0, 0) circle [x radius= 2.34, y radius= 2.34]   ;
\draw [color={rgb, 255:red, 208; green, 2; blue, 27 }  ,draw opacity=1 ]   (427.6,47.4) .. controls (435.8,57) and (418.6,66.2) .. (425.8,74.6) ;
\draw [shift={(425.8,74.6)}, rotate = 49.4] [color={rgb, 255:red, 208; green, 2; blue, 27 }  ,draw opacity=1 ][fill={rgb, 255:red, 208; green, 2; blue, 27 }  ,fill opacity=1 ][line width=0.75]      (0, 0) circle [x radius= 1.34, y radius= 1.34]   ;
\draw [shift={(427.6,47.4)}, rotate = 49.5] [color={rgb, 255:red, 208; green, 2; blue, 27 }  ,draw opacity=1 ][fill={rgb, 255:red, 208; green, 2; blue, 27 }  ,fill opacity=1 ][line width=0.75]      (0, 0) circle [x radius= 1.34, y radius= 1.34]   ;
\draw [color={rgb, 255:red, 208; green, 2; blue, 27 }  ,draw opacity=1 ]   (445.2,73.8) .. controls (438.6,82.2) and (450.6,81) .. (447.8,95.4) ;
\draw [shift={(447.8,95.4)}, rotate = 101] [color={rgb, 255:red, 208; green, 2; blue, 27 }  ,draw opacity=1 ][fill={rgb, 255:red, 208; green, 2; blue, 27 }  ,fill opacity=1 ][line width=0.75]      (0, 0) circle [x radius= 1.34, y radius= 1.34]   ;
\draw [shift={(445.2,73.8)}, rotate = 128.16] [color={rgb, 255:red, 208; green, 2; blue, 27 }  ,draw opacity=1 ][fill={rgb, 255:red, 208; green, 2; blue, 27 }  ,fill opacity=1 ][line width=0.75]      (0, 0) circle [x radius= 1.34, y radius= 1.34]   ;
\draw [color={rgb, 255:red, 208; green, 2; blue, 27 }  ,draw opacity=1 ]   (404.6,100.2) .. controls (409,115) and (324.6,110.6) .. (318.6,92.2) .. controls (312.6,73.8) and (321.4,46.2) .. (345.4,25.4) .. controls (369.4,4.6) and (416.1,38.6) .. (415.4,42.2) ;
\draw [shift={(415.4,42.2)}, rotate = 101] [color={rgb, 255:red, 208; green, 2; blue, 27 }  ,draw opacity=1 ][fill={rgb, 255:red, 208; green, 2; blue, 27 }  ,fill opacity=1 ][line width=0.75]      (0, 0) circle [x radius= 1.34, y radius= 1.34]   ;
\draw [shift={(404.6,100.2)}, rotate = 73.44] [color={rgb, 255:red, 208; green, 2; blue, 27 }  ,draw opacity=1 ][fill={rgb, 255:red, 208; green, 2; blue, 27 }  ,fill opacity=1 ][line width=0.75]      (0, 0) circle [x radius= 1.34, y radius= 1.34]   ;

\draw (431.6,53.6) node [anchor=north west][inner sep=0.75pt]  [font=\scriptsize,color={rgb, 255:red, 208; green, 2; blue, 27 }  ,opacity=1 ]  {$\beta _{1}$};
\draw (324.4,82.4) node [anchor=north west][inner sep=0.75pt]  [font=\scriptsize,color={rgb, 255:red, 208; green, 2; blue, 27 }  ,opacity=1 ]  {$\beta _{3}$};
\draw (450.8,78.8) node [anchor=north west][inner sep=0.75pt]  [font=\scriptsize,color={rgb, 255:red, 208; green, 2; blue, 27 }  ,opacity=1 ]  {$\beta _{2}$};
\draw (382.4,37.6) node [anchor=north west][inner sep=0.75pt]  [font=\scriptsize,color={rgb, 255:red, 74; green, 144; blue, 226 }  ,opacity=1 ]  {$\alpha _{1}$};
\draw (385.2,69.2) node [anchor=north west][inner sep=0.75pt]  [font=\scriptsize,color={rgb, 255:red, 74; green, 144; blue, 226 }  ,opacity=1 ]  {$\alpha _{2}$};
\draw (445.2,98.8) node [anchor=north west][inner sep=0.75pt]  [font=\scriptsize,color={rgb, 255:red, 74; green, 144; blue, 226 }  ,opacity=1 ]  {$\alpha _{3}$};
\draw (339.2,55.4) node [anchor=north west][inner sep=0.75pt]  [color={rgb, 255:red, 128; green, 128; blue, 128 }  ,opacity=1 ]  {$\Lambda _{1}$};
\draw (485.6,53.8) node [anchor=north west][inner sep=0.75pt]  [color={rgb, 255:red, 128; green, 128; blue, 128 }  ,opacity=1 ]  {$\Lambda _{2}$};

\end{tikzpicture}

    \caption{}
    \label{fig:simple-example}
\end{figure}

We now have the following application of Menger's theorem, which motivates why we wanted $\vs(\pone)$ to be `large'.

\begin{lemma}\label{lem:disjoint-paths}
    The subgraph $\pone \subset \Gamma$ enjoys the following properties:
    \begin{enumerate}
        \item\label{itm:paths1} For every distinct pair of ends $\omega_1, \omega_2 \in \Omega (\Lambda)$, there exists three distinct bi-infinite paths $\alpha_1, \alpha_2, \alpha_3 : \R \to \pone$ such that each $\alpha_i$ connects $\omega_1$ to $\omega_2$.

        \item\label{itm:paths3}There exists $S > 0$ such that for all $s \geq S$ and all $x,y  \in \pone$ such that $\dist_{\pone}(x,y) > 2s$, there exists three distinct paths $\alpha_1, \alpha_2, \alpha_3 : [0,1] \to \pone$ such that each $\alpha_i$ begins in $B_{\pone}(x, s)$ and ends in $B_{\pone}(y, s)$.
    \end{enumerate}
    Moreover, in each case we have that the $\alpha_i$ satisfy the additional condition  that 
    \begin{equation}\tag{$\ast$}\label{eq:lowerbound}
        \dist_{\Gamma}(z, \alpha_i) > \lambda^2
    \end{equation}
    for every $j \neq i$, $z \in \alpha_j$. 
\end{lemma}

\begin{remark}
    It is very important to note that the lower bound given in (\ref{eq:lowerbound}) above pertains to how far apart these rays are \textbf{in the super-graph $\Gamma$}, not just in $\pone$ or $\ptwo$. 
\end{remark}

\begin{proof}[Proof of Lemma~\ref{lem:disjoint-paths}]
(\ref{itm:paths1}): Let $\omega_1, \omega_2 \in \Omega (\Lambda)$. Since $\vs(\pone) \geq 3\lambda^2 + 3$, we have by Menger's theorem (\ref{thm:menger}) that there exists $ N := 3\lambda^2 + 3$ pairwise disjoint bi-infinite paths paths $\beta_1, \ldots, \beta_N \subset \pone$, between $\omega_1$ and $\omega_2$. 
Since $\pone \subset \Gamma$ is planar with a fixed drawing $\vartheta$, we have that there is a fixed cyclic order in which these paths emerge from $\omega_1$. Assume without loss of generality that $\beta_i$ lies adjacent to $\beta_{i\pm 1}$ in this order, where indices are taken modulo $N$. 
Let 
$$
\alpha_1 = \beta_1, \ \ \alpha_2 = \beta_{\lambda^2 + 2}, \ \ \alpha_3 = \beta_{2\lambda^2 + 3}. 
$$
For every $i \neq j$, there are at least $\lambda^2$ distinct Jordan curves contained in $\overline \Gamma$ which separate $\alpha_i$ from $\alpha_j$. Moreover, these curves only intersect in $\{\omega_1, \omega_2\}$. Thus, we see through an application of the Jordan curve theorem that any path in $\Gamma$ from $\alpha_i$ to $\alpha_j$ must internally intersect at least $\lambda^2$ disjoint rays in $\Gamma$. Thus, we deduce that 
$\dist_{\Gamma}(z, \alpha_i) > \lambda^2$ for every $z \in \alpha_j$. See Figure~\ref{fig:disjoint-rays} for a cartoon. 
\begin{figure}[ht]
    \centering

\tikzset{every picture/.style={line width=0.75pt}} 

\begin{tikzpicture}[x=0.75pt,y=0.75pt,yscale=-1,xscale=1]

\draw [color={rgb, 255:red, 155; green, 155; blue, 155 }  ,draw opacity=0.6 ][line width=1.5]    (102.91,80.17) .. controls (73.34,52.07) and (43.42,23.64) .. (160.59,47.28) .. controls (277.77,70.92) and (270.56,33.23) .. (254.33,85.99) ;
\draw [color={rgb, 255:red, 155; green, 155; blue, 155 }  ,draw opacity=0.6 ][line width=1.5]    (102.91,80.17) .. controls (124.18,63.04) and (127.06,62.01) .. (168.16,70.92) .. controls (209.27,79.82) and (224.05,63.38) .. (254.33,85.99) ;
\draw [color={rgb, 255:red, 155; green, 155; blue, 155 }  ,draw opacity=0.6 ][line width=1.5]    (254.33,85.99) .. controls (249.65,104.83) and (273.08,118.2) .. (195.47,123.85) .. controls (117.86,129.51) and (86.25,132.8) .. (102.91,80.17) ;
\draw [color={rgb, 255:red, 155; green, 155; blue, 155 }  ,draw opacity=0.6 ][line width=1.5]    (254.33,85.99) .. controls (217.1,96.32) and (218.61,106.21) .. (177.58,97) .. controls (136.55,87.79) and (137.16,85.65) .. (102.91,80.17) ;
\draw [color={rgb, 255:red, 155; green, 155; blue, 155 }  ,draw opacity=0.6 ][line width=1.5]    (102.91,80.17) .. controls (58.56,79.14) and (3.07,30.11) .. (73.7,20.55) .. controls (144.33,11) and (176.46,25.35) .. (202.42,32.2) .. controls (228.38,39.05) and (287.51,14.04) .. (287.14,48.65) .. controls (286.78,83.25) and (268.04,81.88) .. (254.33,85.99) ;
\draw [color={rgb, 255:red, 155; green, 155; blue, 155 }  ,draw opacity=0.6 ][line width=1.5]    (254.33,85.99) .. controls (271.28,92.84) and (308.42,89.42) .. (309.14,111) .. controls (309.86,132.58) and (186.55,149.37) .. (132.11,141.83) .. controls (77.67,134.3) and (58.13,125.36) .. (63.97,114.77) .. controls (69.81,104.18) and (66.13,90.79) .. (102.91,80.17) ;
\draw [color={rgb, 255:red, 144; green, 19; blue, 254 }  ,draw opacity=1 ][line width=1.5]    (102.91,80.17) .. controls (156.63,65.43) and (201.7,86.68) .. (254.33,85.99) ;
\draw [shift={(254.33,85.99)}, rotate = 359.25] [color={rgb, 255:red, 144; green, 19; blue, 254 }  ,draw opacity=1 ][fill={rgb, 255:red, 144; green, 19; blue, 254 }  ,fill opacity=1 ][line width=1.5]      (0, 0) circle [x radius= 4.36, y radius= 4.36]   ;
\draw [shift={(102.91,80.17)}, rotate = 344.66] [color={rgb, 255:red, 144; green, 19; blue, 254 }  ,draw opacity=1 ][fill={rgb, 255:red, 144; green, 19; blue, 254 }  ,fill opacity=1 ][line width=1.5]      (0, 0) circle [x radius= 4.36, y radius= 4.36]   ;
\draw [color={rgb, 255:red, 144; green, 19; blue, 254 }  ,draw opacity=1 ][line width=1.5]    (102.91,80.17) .. controls (39.81,51.05) and (33.32,9.25) .. (149.78,37) .. controls (266.23,64.75) and (316.35,19.53) .. (254.33,85.99) ;
\draw [shift={(254.33,85.99)}, rotate = 133.02] [color={rgb, 255:red, 144; green, 19; blue, 254 }  ,draw opacity=1 ][fill={rgb, 255:red, 144; green, 19; blue, 254 }  ,fill opacity=1 ][line width=1.5]      (0, 0) circle [x radius= 4.36, y radius= 4.36]   ;
\draw [shift={(102.91,80.17)}, rotate = 204.78] [color={rgb, 255:red, 144; green, 19; blue, 254 }  ,draw opacity=1 ][fill={rgb, 255:red, 144; green, 19; blue, 254 }  ,fill opacity=1 ][line width=1.5]      (0, 0) circle [x radius= 4.36, y radius= 4.36]   ;
\draw [color={rgb, 255:red, 144; green, 19; blue, 254 }  ,draw opacity=1 ][line width=1.5]    (102.91,80.17) .. controls (70.1,102.44) and (74.42,100.72) .. (93.17,121.62) .. controls (111.92,142.52) and (190.52,131.56) .. (259.02,121.96) .. controls (327.53,112.37) and (247.57,92.41) .. (254.33,85.99) ;
\draw [shift={(254.33,85.99)}, rotate = 316.46] [color={rgb, 255:red, 144; green, 19; blue, 254 }  ,draw opacity=1 ][fill={rgb, 255:red, 144; green, 19; blue, 254 }  ,fill opacity=1 ][line width=1.5]      (0, 0) circle [x radius= 4.36, y radius= 4.36]   ;
\draw [shift={(102.91,80.17)}, rotate = 145.83] [color={rgb, 255:red, 144; green, 19; blue, 254 }  ,draw opacity=1 ][fill={rgb, 255:red, 144; green, 19; blue, 254 }  ,fill opacity=1 ][line width=1.5]      (0, 0) circle [x radius= 4.36, y radius= 4.36]   ;
\draw [line width=1.5]    (102.91,80.17) ;
\draw [shift={(102.91,80.17)}, rotate = 45] [color={rgb, 255:red, 0; green, 0; blue, 0 }  ][line width=1.5]    (-7.99,0) -- (7.99,0)(0,7.99) -- (0,-7.99)   ;
\draw [shift={(102.91,80.17)}, rotate = 0] [color={rgb, 255:red, 0; green, 0; blue, 0 }  ][fill={rgb, 255:red, 0; green, 0; blue, 0 }  ][line width=1.5]      (0, 0) circle [x radius= 4.79, y radius= 4.79]   ;
\draw [line width=1.5]    (254.33,85.99) ;
\draw [shift={(254.33,85.99)}, rotate = 45] [color={rgb, 255:red, 0; green, 0; blue, 0 }  ][line width=1.5]    (-7.99,0) -- (7.99,0)(0,7.99) -- (0,-7.99)   ;
\draw [shift={(254.33,85.99)}, rotate = 0] [color={rgb, 255:red, 0; green, 0; blue, 0 }  ][fill={rgb, 255:red, 0; green, 0; blue, 0 }  ][line width=1.5]      (0, 0) circle [x radius= 4.79, y radius= 4.79]   ;
\draw [color={rgb, 255:red, 155; green, 155; blue, 155 }  ,draw opacity=1 ]   (321.67,81) -- (298.43,81.66) ;
\draw [shift={(296.43,81.72)}, rotate = 358.37] [color={rgb, 255:red, 155; green, 155; blue, 155 }  ,draw opacity=1 ][line width=0.75]    (10.93,-3.29) .. controls (6.95,-1.4) and (3.31,-0.3) .. (0,0) .. controls (3.31,0.3) and (6.95,1.4) .. (10.93,3.29)   ;

\draw (102.54,86.52) node [anchor=north west][inner sep=0.75pt]    {$\omega _{1}$};
\draw (234.19,97.36) node [anchor=north west][inner sep=0.75pt]    {$\omega _{2}$};
\draw (161.57,43.57) node [anchor=north west][inner sep=0.75pt]  [font=\footnotesize,color={rgb, 255:red, 155; green, 155; blue, 155 }  ,opacity=1 ]  {$\vdots $};
\draw (164.1,96.93) node [anchor=north west][inner sep=0.75pt]  [font=\footnotesize,color={rgb, 255:red, 155; green, 155; blue, 155 }  ,opacity=1 ]  {$\vdots $};
\draw (285.04,65.55) node [anchor=north west][inner sep=0.75pt]  [font=\footnotesize,color={rgb, 255:red, 155; green, 155; blue, 155 }  ,opacity=1 ]  {$\vdots $};
\draw (153.2,25.66) node [anchor=north west][inner sep=0.75pt]  [color={rgb, 255:red, 144; green, 19; blue, 254 }  ,opacity=1 ]  {$\alpha _{1}$};
\draw (186.34,85.21) node [anchor=north west][inner sep=0.75pt]  [color={rgb, 255:red, 144; green, 19; blue, 254 }  ,opacity=1 ]  {$\alpha _{2}$};
\draw (257.33,98.42) node [anchor=north west][inner sep=0.75pt]  [color={rgb, 255:red, 144; green, 19; blue, 254 }  ,opacity=1 ]  {$\alpha _{3}$};
\draw (56.39,70.2) node [anchor=north west][inner sep=0.75pt]  [font=\footnotesize,color={rgb, 255:red, 155; green, 155; blue, 155 }  ,opacity=1 ]  {$\vdots $};
\draw (325.32,70.83) node [anchor=north west][inner sep=0.75pt]  [color={rgb, 255:red, 155; green, 155; blue, 155 }  ,opacity=1 ] [align=left] {$\displaystyle \lambda ^{2}$ disjoint paths};

\end{tikzpicture}

    \caption{Any path through $\Gamma$ between $\alpha_i$ and $\alpha_j$ must pass through $\lambda^2$ distinct $\beta_k$ paths different from the $\alpha_i$ and $\alpha_j$. }
    \label{fig:disjoint-rays}
\end{figure}
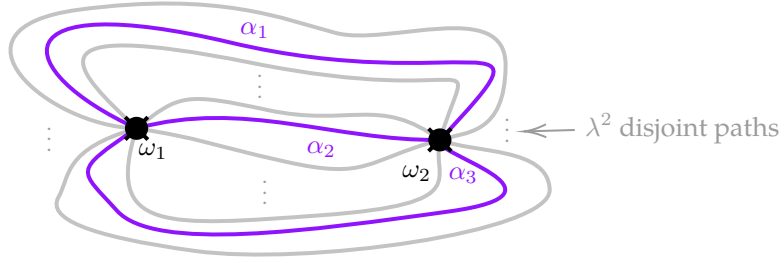

(\ref{itm:paths3}): Let $S \geq 0$ be sufficiently large so that $B_{\pone}(z, s)$ separates ends in $\pone$ for every $s \geq S$, $z \in \pone$. Since $\pone$ admits a cobounded quasi-action, such an $S$ certainly exists. Fix $s \geq S$. Let $x,y \in \Lambda$ such that $\dist_\Lambda (x,y) > 2s$. We have that both $B_1 := B_{\pone}(x, s)$ and $B_2 := B_{\pone}(y, s)$ separate ends in $\Lambda$. We have that $B_1$ and $B_2$ are disjoint. Let $\omega_1, \omega_2 \in \Omega (\Lambda)$ be distinct ends which lie in distinct components of both $\Lambda \setminus B_1$ and $\Lambda \setminus B_2$. By (\ref{itm:paths1}) above we have that there exists $\alpha_1$, $\alpha_2$, $\alpha_3$ in $\Lambda$ connecting $\omega_1$ to $\omega_2$ which sit disjoint from each others $\lambda^2$-neighbourhood in $\Gamma$. All three paths must intersect both $B_1$ and $B_2$, and we are done. 
\end{proof}

\begin{lemma}\label{lem:shared-face-preserved}
    There exists a constant $r > 0$ such that the following holds. Given $f \in \facepaths (\Pi)$, $x, y \in f$, and $g \in H$, there exist $f' \in \facepaths (\Pi)$ such that 
    $
    \mu_g(x), \mu_g(y) \in B_\Pi(f';r)$. 
\end{lemma}

\begin{proof}
Fix $g \in H$. To ease notation, write $x' = \mu_g(x)$, $y' = \mu_g(y)$. We will need some constants. Recall that $\lambda \geq 1$ has been fixed so that the quasi-action of $G$ upon $\Gamma$ is a $\lambda$-quasi-action. Fix $\eta \geq 1$ such that the quasi-isometric embeddings $\Lambda \into \Pi$, $\Pi  \into \Gamma$, and $\Lambda  \into \Gamma$ are all $(\eta, \eta)$-quasi-isometric embeddings, and also assume without loss of generality that $\dHaus[\Pi](\Lambda, \Pi) \leq \eta$. 

Let $a,b  \in \Lambda$ be such that $\dist_\Pi(x',a) \leq \eta$ and $\dist_\Pi(y',b) \leq \eta$. 
Apply Lemma~\ref{lem:disjoint-paths}, and obtain three paths $\alpha_1$, $\alpha_2$, $\alpha_3$ in $\Lambda$ such that
\begin{enumerate}
    \item Each $\alpha_i$ has one endpoint in $B_\Lambda(a;S)$ and one in $B_\Lambda(b;S)$.

    \item\label{itm:distitm} We have that 
    $
    \dist_{\Gamma}(z, \alpha_i) > \lambda^2
    $ 
    for every $j \neq i$, $z \in \alpha_j$, 
\end{enumerate}
where $S \geq 0$ is some fixed constant depending only on the graphs and quasi-isometries at play. 
Note that property (\ref{itm:distitm}) relates to the path metric of $\Gamma$, not $\Lambda$ or $\Pi$. To ease notation, let 
$$
A_a = B_\Lambda(a;S), \  A_b =  B_\Lambda(b;S), 
$$
so each $\alpha_j$ begins in $A_a$ and ends in $A_b$. Choose some large $R > 0$ such that 
$$
B_\Gamma(A_a; \lambda^2 + \eta + 1)\cap \Lambda  \subset B_\Lambda(a; R), \ \ \ \ B_\Gamma(A_b; \lambda^2 + \eta + 1)\cap \Lambda  \subset B_\Lambda(b; R). 
$$
Clearly such a uniform $R$ exists, since $\Lambda$ is quasi-isometrically embedded into $\Gamma$.  We assume without loss of generality that 
$
\dist_\Lambda(a,b) > 2R, 
$
lest $x'$, $y'$ lie a bounded distance apart in $\Pi$, and the statement is vacuously true. 

Suppose there exists some facial subgraph $f_0 \in \facepaths (\Lambda)$ such that 
$$
\{a,b\} \subset B_\Lambda(f_0; R).
$$
Then by Lemma~\ref{lem:approx-faces} there exists a $f_1 \in \facepaths (\Pi)$ such that
$$
\{x', y'\} \subset B_\Pi(f_1 ; R + M + \eta),
$$
where $M > 0$ is some fixed constant. 
With this in mind, let $r = R + M + \eta$, and assume that there is no such $f_0 \in \facepaths (\Lambda)$. We will show that $x$ and $y$ cannot lie on a common facial subgraph, and thus deduce a contradiction.

We have by Proposition~\ref{prop:2-conn-jordan-curve} that there is a simple loop $\ell$ in $\Lambda$ such that 
\begin{enumerate}
        \item $\vartheta(a)$, $\vartheta(b)$ lie in distinct components of $\bbS^2 \setminus \vartheta(\ell)$, and

        \item\label{itm:distss} $\{a,b\} \cap B_\Lambda(\ell;R) = \emptyset$.
\end{enumerate}
Let $\beta_i$ be a segment of $\ell$ connecting $\alpha_i$ to $\alpha_{i+1}$, which is otherwise disjoint from all the $\alpha_j$. Note that $\beta_i$ sits outside of the $\lambda^2$-neighbourhood of $\alpha_{i+2}$ in $\Gamma$, by an easy application of the Jordan curve theorem, since any path between these two curves must intersect either $\alpha_i$ or $\alpha_{i+1}$, or a bounded neighbourhood of $a$, $b$. Since $\ell$ is very far away from $a$ and $b$, and the $\alpha_j$ are pairwise far apart, the claim follows. Let
$$
\Theta = \beta_1 \cup \beta_2 \cup \beta_3 \cup \alpha_1 \cup \alpha_2 \cup \alpha_3. 
$$
See Figure~\ref{fig:delta} for a cartoon of this construction.

\begin{figure}
    \centering

\tikzset{every picture/.style={line width=0.75pt}} 

\begin{tikzpicture}[x=0.75pt,y=0.75pt,yscale=-1,xscale=1]

\draw  [color={rgb, 255:red, 65; green, 117; blue, 5 }  ,draw opacity=1 ][fill={rgb, 255:red, 65; green, 117; blue, 5 }  ,fill opacity=0.23 ] (378.2,149.09) .. controls (378.2,126.45) and (396.56,108.1) .. (419.2,108.1) .. controls (441.84,108.1) and (460.2,126.45) .. (460.2,149.09) .. controls (460.2,171.74) and (441.84,190.09) .. (419.2,190.09) .. controls (396.56,190.09) and (378.2,171.74) .. (378.2,149.09) -- cycle ;
\draw  [color={rgb, 255:red, 65; green, 117; blue, 5 }  ,draw opacity=1 ][fill={rgb, 255:red, 65; green, 117; blue, 5 }  ,fill opacity=0.23 ] (125.4,151.89) .. controls (125.4,129.25) and (143.76,110.9) .. (166.4,110.9) .. controls (189.04,110.9) and (207.4,129.25) .. (207.4,151.89) .. controls (207.4,174.54) and (189.04,192.89) .. (166.4,192.89) .. controls (143.76,192.89) and (125.4,174.54) .. (125.4,151.89) -- cycle ;
\draw [color={rgb, 255:red, 74; green, 144; blue, 226 }  ,draw opacity=1 ]   (171.72,117.09) .. controls (244.61,62.43) and (265.38,122.56) .. (327.7,122.92) .. controls (390.02,123.29) and (387.46,94.5) .. (414.43,114.9) ;
\draw [shift={(414.43,114.9)}, rotate = 37.12] [color={rgb, 255:red, 74; green, 144; blue, 226 }  ,draw opacity=1 ][fill={rgb, 255:red, 74; green, 144; blue, 226 }  ,fill opacity=1 ][line width=0.75]      (0, 0) circle [x radius= 2.34, y radius= 2.34]   ;
\draw [shift={(171.72,117.09)}, rotate = 323.13] [color={rgb, 255:red, 74; green, 144; blue, 226 }  ,draw opacity=1 ][fill={rgb, 255:red, 74; green, 144; blue, 226 }  ,fill opacity=1 ][line width=0.75]      (0, 0) circle [x radius= 2.34, y radius= 2.34]   ;
\draw [color={rgb, 255:red, 74; green, 144; blue, 226 }  ,draw opacity=1 ]   (197.23,146.25) .. controls (270.12,159.37) and (247.16,169.21) .. (309.48,169.57) .. controls (371.79,169.93) and (345.19,160.09) .. (385.28,154.99) ;
\draw [shift={(385.28,154.99)}, rotate = 352.75] [color={rgb, 255:red, 74; green, 144; blue, 226 }  ,draw opacity=1 ][fill={rgb, 255:red, 74; green, 144; blue, 226 }  ,fill opacity=1 ][line width=0.75]      (0, 0) circle [x radius= 2.34, y radius= 2.34]   ;
\draw [shift={(197.23,146.25)}, rotate = 10.2] [color={rgb, 255:red, 74; green, 144; blue, 226 }  ,draw opacity=1 ][fill={rgb, 255:red, 74; green, 144; blue, 226 }  ,fill opacity=1 ][line width=0.75]      (0, 0) circle [x radius= 2.34, y radius= 2.34]   ;
\draw [color={rgb, 255:red, 74; green, 144; blue, 226 }  ,draw opacity=1 ]   (174.64,187.06) .. controls (190.67,245.37) and (234.4,219.13) .. (297.81,207.47) .. controls (361.23,195.81) and (393.3,229.34) .. (413.7,185.6) ;
\draw [shift={(413.7,185.6)}, rotate = 295.02] [color={rgb, 255:red, 74; green, 144; blue, 226 }  ,draw opacity=1 ][fill={rgb, 255:red, 74; green, 144; blue, 226 }  ,fill opacity=1 ][line width=0.75]      (0, 0) circle [x radius= 2.34, y radius= 2.34]   ;
\draw [shift={(174.64,187.06)}, rotate = 74.62] [color={rgb, 255:red, 74; green, 144; blue, 226 }  ,draw opacity=1 ][fill={rgb, 255:red, 74; green, 144; blue, 226 }  ,fill opacity=1 ][line width=0.75]      (0, 0) circle [x radius= 2.34, y radius= 2.34]   ;
\draw [color={rgb, 255:red, 208; green, 2; blue, 27 }  ,draw opacity=1 ]   (300.37,120.01) .. controls (315.31,137.5) and (283.97,154.26) .. (297.09,169.57) ;
\draw [shift={(297.09,169.57)}, rotate = 49.4] [color={rgb, 255:red, 208; green, 2; blue, 27 }  ,draw opacity=1 ][fill={rgb, 255:red, 208; green, 2; blue, 27 }  ,fill opacity=1 ][line width=0.75]      (0, 0) circle [x radius= 1.34, y radius= 1.34]   ;
\draw [shift={(300.37,120.01)}, rotate = 49.5] [color={rgb, 255:red, 208; green, 2; blue, 27 }  ,draw opacity=1 ][fill={rgb, 255:red, 208; green, 2; blue, 27 }  ,fill opacity=1 ][line width=0.75]      (0, 0) circle [x radius= 1.34, y radius= 1.34]   ;
\draw [color={rgb, 255:red, 208; green, 2; blue, 27 }  ,draw opacity=1 ]   (306.2,169.8) .. controls (312.2,185.8) and (295,185.4) .. (301.4,206.2) ;
\draw [shift={(301.4,206.2)}, rotate = 72.9] [color={rgb, 255:red, 208; green, 2; blue, 27 }  ,draw opacity=1 ][fill={rgb, 255:red, 208; green, 2; blue, 27 }  ,fill opacity=1 ][line width=0.75]      (0, 0) circle [x radius= 1.34, y radius= 1.34]   ;
\draw [shift={(306.2,169.8)}, rotate = 69.44] [color={rgb, 255:red, 208; green, 2; blue, 27 }  ,draw opacity=1 ][fill={rgb, 255:red, 208; green, 2; blue, 27 }  ,fill opacity=1 ][line width=0.75]      (0, 0) circle [x radius= 1.34, y radius= 1.34]   ;
\draw [color={rgb, 255:red, 208; green, 2; blue, 27 }  ,draw opacity=1 ]   (258.46,216.22) .. controls (252.93,232.81) and (247,261) .. (167.4,260.2) .. controls (87.8,259.4) and (92.6,233.6) .. (79.8,215) .. controls (67,196.4) and (20.47,139.3) .. (64.2,101.4) .. controls (107.93,63.5) and (79.8,59.4) .. (129,50.2) .. controls (178.2,41) and (270.6,63.8) .. (278.13,110.53) ;
\draw [shift={(278.13,110.53)}, rotate = 80.84] [color={rgb, 255:red, 208; green, 2; blue, 27 }  ,draw opacity=1 ][fill={rgb, 255:red, 208; green, 2; blue, 27 }  ,fill opacity=1 ][line width=0.75]      (0, 0) circle [x radius= 1.34, y radius= 1.34]   ;
\draw [shift={(258.46,216.22)}, rotate = 108.43] [color={rgb, 255:red, 208; green, 2; blue, 27 }  ,draw opacity=1 ][fill={rgb, 255:red, 208; green, 2; blue, 27 }  ,fill opacity=1 ][line width=0.75]      (0, 0) circle [x radius= 1.34, y radius= 1.34]   ;
\draw    (166.4,151.89) ;
\draw [shift={(166.4,151.89)}, rotate = 0] [color={rgb, 255:red, 0; green, 0; blue, 0 }  ][fill={rgb, 255:red, 0; green, 0; blue, 0 }  ][line width=0.75]      (0, 0) circle [x radius= 3.35, y radius= 3.35]   ;
\draw    (419.2,149.09) ;
\draw [shift={(419.2,149.09)}, rotate = 0] [color={rgb, 255:red, 0; green, 0; blue, 0 }  ][fill={rgb, 255:red, 0; green, 0; blue, 0 }  ][line width=0.75]      (0, 0) circle [x radius= 3.35, y radius= 3.35]   ;
\draw  [color={rgb, 255:red, 155; green, 155; blue, 155 }  ,draw opacity=1 ][fill={rgb, 255:red, 155; green, 155; blue, 155 }  ,fill opacity=0.18 ] (82.44,151.89) .. controls (82.44,105.52) and (120.03,67.93) .. (166.4,67.93) .. controls (212.77,67.93) and (250.36,105.52) .. (250.36,151.89) .. controls (250.36,198.27) and (212.77,235.86) .. (166.4,235.86) .. controls (120.03,235.86) and (82.44,198.27) .. (82.44,151.89) -- cycle ;
\draw  [color={rgb, 255:red, 155; green, 155; blue, 155 }  ,draw opacity=1 ][fill={rgb, 255:red, 155; green, 155; blue, 155 }  ,fill opacity=0.18 ] (335.24,149.09) .. controls (335.24,102.72) and (372.83,65.13) .. (419.2,65.13) .. controls (465.57,65.13) and (503.16,102.72) .. (503.16,149.09) .. controls (503.16,195.47) and (465.57,233.06) .. (419.2,233.06) .. controls (372.83,233.06) and (335.24,195.47) .. (335.24,149.09) -- cycle ;

\draw (309.4,133.91) node [anchor=north west][inner sep=0.75pt]  [font=\scriptsize,color={rgb, 255:red, 208; green, 2; blue, 27 }  ,opacity=1 ]  {$\beta _{1}$};
\draw (52.06,153.59) node [anchor=north west][inner sep=0.75pt]  [font=\scriptsize,color={rgb, 255:red, 208; green, 2; blue, 27 }  ,opacity=1 ]  {$\beta _{3}$};
\draw (307.98,185.43) node [anchor=north west][inner sep=0.75pt]  [font=\scriptsize,color={rgb, 255:red, 208; green, 2; blue, 27 }  ,opacity=1 ]  {$\beta _{2}$};
\draw (257.36,106.75) node [anchor=north west][inner sep=0.75pt]  [font=\scriptsize,color={rgb, 255:red, 74; green, 144; blue, 226 }  ,opacity=1 ]  {$\alpha _{1}$};
\draw (272.46,169.93) node [anchor=north west][inner sep=0.75pt]  [font=\scriptsize,color={rgb, 255:red, 74; green, 144; blue, 226 }  ,opacity=1 ]  {$\alpha _{2}$};
\draw (328.99,208.27) node [anchor=north west][inner sep=0.75pt]  [font=\scriptsize,color={rgb, 255:red, 74; green, 144; blue, 226 }  ,opacity=1 ]  {$\alpha _{3}$};
\draw (141.53,167.25) node [anchor=north west][inner sep=0.75pt]  [font=\footnotesize,color={rgb, 255:red, 65; green, 117; blue, 5 }  ,opacity=1 ]  {$B_{\Lambda }( a;S)$};
\draw (399.53,161.25) node [anchor=north west][inner sep=0.75pt]  [font=\footnotesize,color={rgb, 255:red, 65; green, 117; blue, 5 }  ,opacity=1 ]  {$B_{\Lambda }( b;S)$};
\draw (150.8,144.2) node [anchor=north west][inner sep=0.75pt]    {$a$};
\draw (425.6,139.8) node [anchor=north west][inner sep=0.75pt]    {$b$};
\draw (131.13,210.85) node [anchor=north west][inner sep=0.75pt]  [font=\footnotesize,color={rgb, 255:red, 128; green, 128; blue, 128 }  ,opacity=1 ]  {$B_{\Lambda }( a;R)$};
\draw (409.93,205.65) node [anchor=north west][inner sep=0.75pt]  [font=\footnotesize,color={rgb, 255:red, 128; green, 128; blue, 128 }  ,opacity=1 ]  {$B_{\Lambda }( b;R)$};

\end{tikzpicture}

    \caption{Construction of $\Theta$.}
    \label{fig:delta}
\end{figure}
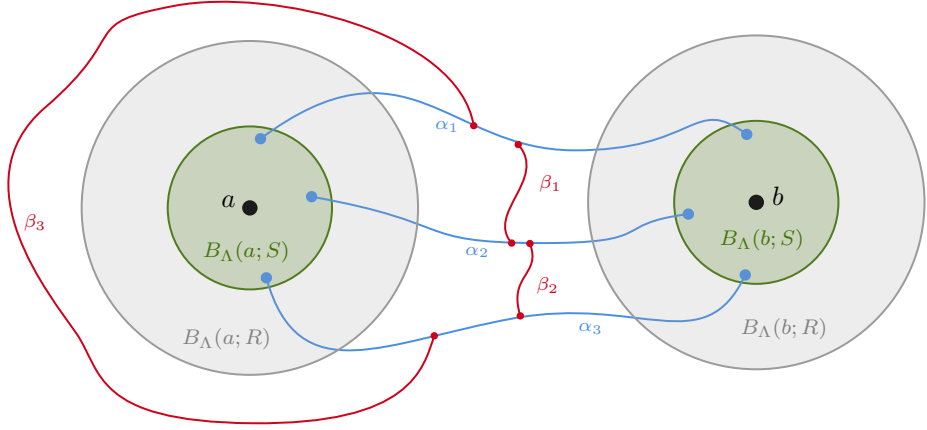

We now apply $\mu_{g^{-1}}$ to $\Theta$. In particular, we claim that $\vartheta(\mu_{g^{-1}}(\Theta))$ contains a Jordan curve which separates $\vartheta(x)$ from $\vartheta(y)$ in $\bbS^2$.  
Since $\Theta \subset \Lambda$, we have by Lemma~\ref{lem:qa-agrees} that 
$$
\mu_{g^{-1}}(\Theta) = \varphi_{g^{-1}}(\Theta). 
$$
For $i = 1,2,3$, write $\alpha_i' = \mu_{g^{-1}}(\alpha_i)$, $\beta_i' = \mu_{g^{-1}}(\beta_i)$. Note that the $\alpha_i'$ are pairwise disjoint, and $\beta_i'$ is disjoint from $\alpha_{i+2}'$. 

It is clear that if $R$ above is chosen to be sufficiently large then in the $\Pi$-metric we have that $x$ and $y$ are strictly closer each of the $\alpha_i'$ than they are to any of the $\beta_j'$.
In particular, if we draw a geodesic in $\Pi$ from $x$ or $y$ to some $\alpha_i'$ then this geodesic will not intersect any of the $\beta_j'$. 
We now apply Lemma~\ref{lem:jordan-curve-finite} and deduce that $x$ and $y$ cannot lie on a common facial subgraph of $\Pi$.
\end{proof}

\begin{lemma}\label{lem:translate-big-faces}
    There exists $C,D \geq 0$ such that the following holds. For every $f \in \facepaths (\Pi)$  such that $\diam_{\Pi}(f) \geq D$ and every $g \in H$, there exists $f' \in \facepaths (\Pi)$ such that
    $$
    \dHaus[\Pi](f', \mu_g(f)) \leq C.
    $$
\end{lemma}

\begin{proof}
    Let $f \in \facepaths (\Pi)$, $g \in H$. We first prove that there exists some $f' \in \facepaths (\Pi)$ such that 
    $$
     \mu_g(f) \subseteq B_\Pi(f';C'),
    $$
    for some uniform constant $C' > 0$. We may assume without loss of generality that $\diam_\Pi(f)$ is bounded below by some constant, lest this claim is vacuously true. We will specify this lower bound later. 
    
    Using Lemma~\ref{lem:two-faces-small-intersection}, let $m = m(r) > 0$ be such that if the $r$-neighbourhoods of any two distinct faces intersects at two points $a$, $b$, then $\dist_\Pi(a,b) < m$. Similarly, let $n = n(r) > 0$ be as in Lemma~\ref{lem:three-faces-small-intersection}, so if the $r$-neighbourhoods of three distinct faces all pairwise intersect, then the union of these intersections has diameter at most $n$. 
    Let $M = \max \{n, m\}$, and fix $\eta \geq 1$ so that the induced quasi-action of $H$ upon $\Pi$ is a $\eta$-quasi-action. 
    
    We assume without loss of generality that $\diam_\Pi(f) \geq \eta(3M + \eta)$ so that 
    $$
    \diam_\Pi(\mu_g(f)) \geq 3M. 
    $$
    Fix $x \in \mu_g(f)$, and let $y \in \mu_g(f)$ be such that $\dist_\Pi(x,y) \geq M$. By Lemma~\ref{lem:shared-face-preserved}, there exists $f_1 \in \facepaths (\Pi)$ such that both $x$ and $y$ are contained in the $r$-neighbourhood of $f_1$. We claim that $\mu_g(f)$ is contained in the $(r+M)$-neighbourhood of $f_1$. Suppose to the contrary that there exists some $z \in \mu_g(f)$ such that $z$ is not contained in the $(r+M)$-neighbourhood of $f_1$. 
    Applying Lemma~\ref{lem:shared-face-preserved} again, we see that there exists $f_2, f_3 \in \facepaths (\Pi)$ such that 
    $$
    x,z \in B_\Pi(f_2;r), \ \ y,z \in B_\Pi(f_3;r). 
    $$
    Clearly, both $f_2$ and $f_3$ must be distinct from $f_1$, since $f_1$ is far away from $z$. Now, either $f_2$ and $f_3$ are equal or they are distinct. If they are equal then we contradict Lemma~\ref{lem:two-faces-small-intersection} since $\dist_\Pi(x,y)$ is large. If they are distinct then we contradict Lemma~\ref{lem:three-faces-small-intersection} since $\diam_\Pi(\{x,y,z\})$ is large. In any case, we deduce that $\mu_g(f)$ is contained in the $(r+M)$-neighbourhood of $f_1$, and so our earlier claim is proven by choosing $C' = r + M$. 

    We now prove that the Hausdorff distance between $f_1$ and $\mu_g(f)$ must be uniformly bounded, provided $\diam_\Pi(f)$ is sufficiently large. By the above claim, there exists $f_0 \in \facepaths (\Pi)$ such that 
    $$
    \mu_h(f_1) \subset B_\Pi(f_0; C'). 
    $$
    In particular, $f$ is contained in the $r'$-neighbourhood of $f_0$, where $r' = C' + \eta C'+ 2\eta$. Let $m' = m(r') > 0$ be as in Lemma~\ref{lem:two-faces-small-intersection}. If $\diam_\Pi(f) > m'$ then this forces $f_0 = f$. It follows quickly from this observation that $f_1$ is contained in a bounded neighbourhood of $\mu_g(f)$, and this proves the Lemma. 
\end{proof}

\subsection{Good behaviour}

By Remark~\ref{rmk:good-drawing-subgraph}, the closure of $\Pi$ in the Freudenthal compactification of $\Gamma$ is homeomorphic to that very compactification of $\Pi$. Thus, $\vartheta$ restricts to a good drawing of $\Pi$. Let $\vartheta' : \overline \Pi \into \bbS^2$ denote this restriction. Let $\Pp$ denote the quasi-planar tuple
$$
    \Pp = (H, \ytwo, \ptwo, \vartheta', \mu, \nu).
$$
Note that Lemma~\ref{lem:translate-big-faces} immediately implies the following.

\begin{lemma}\label{lem:gb1gb2}
    The quasi-planar tuple $\Pp$ defined above satisfies \ref{itm:gb1}.
\end{lemma}

We now address the fact that $\Pi$ may not be 2-connected. This turns out to be an easy fix since $\Pi$ is already almost 2-connected.

\begin{lemma}\label{lem:wbqp-2conn}
    Let $\Qp = (G, X, \Gamma, \vartheta, \varphi, \psi)$ be a quasi-planar tuple which satisfies \ref{itm:gb1}. Suppose that $\Gamma$ is almost 2-connected. Then there exists a \textbf{well-behaved} quasi-planar tuple 
    $$
    \Qp' = (G, X, \Gamma', \vartheta', \varphi', \psi').
    $$
    We may take $\Gamma'$ to be the 2-connected core of $\Gamma$ and $\vartheta'$ to be the restriction of $\vartheta$. 
\end{lemma}

\begin{proof}
 Let $\Gamma_0$ be the 2-connected core of $\Gamma$. 
 Recall (Remark~\ref{rmk:simple-body}) that every facial subgraph of $\Gamma$ decomposes as a simple body (which is a facial subgraph of $\Gamma_0$) plus some adjoined cacti of bounded diameter. 
 
 Recall that the inclusion $\iota : \Gamma_0 \into \Gamma$ is a quasi-isometry. 
 Let $\rho : \Gamma \to \Gamma_0$ be some choice of quasi-inverse to the inclusion $\iota$, obtained by collapsing the small `branches' attached to $\Gamma_0$ to their cut vertices. Let $\varphi'  = \rho \circ \varphi$, and let $\psi' = \psi|_{\Gamma_0}$, which is certainly a quasi-inverse of $\varphi'$. Note that the map $\vartheta$ restricts to a good drawing $\vartheta'$ of $\Gamma_0$. Let
 $
 \Qp' = (G, X, \Gamma_0, \vartheta', \varphi' , \psi')
 $. 
 Trivially, $\Qp'$ satisfies \ref{itm:gb2} since $\Gamma_0$ is 2-connected. 
 It is also easy to verify that $\Qp'$ satisfies \ref{itm:gb1}, since the induced quasi-action of $G$ on $\Gamma_0$ differs from that of $G$ upon on $\Gamma$ only by a bounded amount, and the (large enough) faces of $\Gamma_0$ lie a bounded Hausdorff distance from faces of $\Gamma$. 
\end{proof}

Combining Lemmas~\ref{lem:gb1gb2} and \ref{lem:wbqp-2conn}, we immediately deduce Theorem~\ref{thm:2-conn-to-wb} which we restate below for the convenience of the reader.  

\treedecomp*

This can be interpreted as a coarse version of the statement that every connected, locally finite, quasi-transitive planar graph admits a canonical tree decomposition with bounded adhesion, such that every torso is a 3-connected planar graph, and thus admits a unique drawing in $\bbS^2$ by Whitney's planar embedding theorem (\ref{thm:whitney}).

\section{Piecing it all together}\label{sec:accessibility}

In this section we combine all of the above and deduce the main results of this paper. 

\subsection{Proof of Theorem~\ref{thm:acc-intro}}

Here we will piece together the proof of our headline result. We begin by considering the baby-case where there are no bad loops (see Definition~\ref{def:bad} for the definition of a bad loop).

\begin{lemma}\label{lem:no-bad-loops}
    Let $X$ be a connected, locally finite, quasi-transitive graph, and suppose $X$ is quasi-isometric to a locally finite planar graph with no bad loops. Then $X$ is accessible. 
\end{lemma}

\begin{proof}
    Using Propositions~\ref{prop:qi-wlog},~\ref{prop:cts-inverse}, we may assume without loss of generality that $\Gamma$ has bounded-degree, and that the quasi-isometry $\varphi : X \to \Gamma$ is continuous and onto with continuous quasi-inverse $\psi$. This may involve passing to a subgraph of $\Gamma$, but this subgraph will also contain no bad loops. In particular, we may assume $X$ and $\Gamma$ form part of a quasi-planar tuple $\Qp$. 
    
    Using Theorem~\ref{thm:2-conn-to-wb} and Proposition~\ref{prop:tree-decomp-acc}, we may also assume without loss of generality that $\Qp$ is well-behaved. With this added assumption, by Proposition~\ref{prop:bounded-faces} we have that the finite facial subgraphs of $\Gamma$ are uniformly bounded in length. Finally, by Lemma~\ref{lem:no-bad-loops-small-faces-acc}, we conclude that $X$ is accessible. 
\end{proof}

We now proceed with dealing with the case where bad loops are present. The goal is to reduce to the case where there are no bad loops left. Roughly speaking, we will find `all' of the bad loops (or at least a representative collection), and `cut along them first'. The following lemma, which essentially follows from cut-cycle duality in planar graphs, will be useful in this plan.

\begin{lemma}\label{lem:bad-loop-to-cut}
    Let $\Gamma$ be a connected, locally finite planar graph with a good drawing $\vartheta$, and let $\omega_1, \omega_2 \in \Omega (\Gamma)$. Suppose there exists a bad loop $\ell \subset \Gamma$ such that $\vartheta(\ell)$ separates $\vartheta(\omega_1)$, and $\vartheta(\omega_2)$. 
    Then there is some $b \in \br_{\facepaths}(\Gamma)$ which separates $\omega_1$ from $\omega_2$, where $\facepaths = \inffaces (\Gamma)$. 
\end{lemma}

\begin{proof}
    Assume without loss of generality that $\ell$ is simple. 
    Let $U_1$, $U_2$ be the two connected components of $\bbS^2 \setminus \vartheta(\ell)$. For each $f \in \facepaths$, we have that $\vartheta(f)$ is contained in the closure of exactly one of the $U_i$. Let 
    $$
    b = \vartheta^{-1}(U_1)\cap V(\Gamma).
    $$
    The coboundary $\delta b$ contains only edges which intersect $\ell$, which is finite, and so $b \in \br (\Gamma)$. 
    Then, for every $f \in \facepaths$ we have that either $f\cap V(\Gamma)  \subset b^\ast$, or $f\cap b^\ast \subset \ell$. Since $\ell$ contains finitely many vertices, it follows that $b \in \br_{\facepaths} (\Gamma)$. 
\end{proof}

Thus, bad loops correspond to $\facepaths$-elliptic cuts (in the sense of Definition~\ref{def:elliptic}). We saw in Lemma~\ref{lem:peripheral-system-gives-relacc} that if we can construct a suitable peripheral system $\per$ of $X$ which approximates $\facepaths$, then this would yield control over these elliptic cuts. The good-behaviour of our quasi-actions which we have worked to ensure gives exactly what we need to construct such a peripheral system. This is the content of the next lemma.

\begin{lemma}\label{lem:well-behaved-per-system}
    Let $\Qp = (G, X, \Gamma , \vartheta, \varphi, \psi)$ be a well-behaved quasi-planar tuple. Write $\mathcal F = \inffaces(\Gamma)$. Then there exists a thin, $G$-invariant peripheral system $\per$ of $X$ such that $\psi$ is peripheral preserving with respect to $\mathcal F$ and $\per$.
\end{lemma}

\begin{proof}
    Given $f \in \facepaths$, $g \in G$, define $g(f)$ as the unique $f' \in \facepaths$ such that $\dHaus[\Gamma](\varphi_g(f), f')$ is finite. This induces a well-defined action $G \actson \facepaths$. Given $f \in \facepaths$, denote by $G_f$ the stabiliser of $f$ with respect to this action. Note that since $\facepaths$ is thin (in the sense of Definition~\ref{def:per}), it is immediate that $G$ acts on $\facepaths$ with finitely many orbits. 
Let $f_1, \ldots, f_n \in \facepaths$ be orbit representatives. 
Let $R >0$ be some large constant which we will choose shortly.  
For every $f_i$, let 
$$
Z_{f_i} = G_{f_i} \psi({f_i}). 
$$
For each $i = 1, \ldots,  n$ let $T_i$ be a transversal of $G_{f_i}$ containing the identity. If $f = t (f_i)$ for some non-trivial $t \in t_i$, then define 
$
Z_f = t Z_{f_i}
$. 
It is immediate from \ref{itm:gb1} that the Hausdorff distance between $Z_f$ and $\psi(f)$ is uniformly bounded above for all $f \in \facepaths$, say $\dHaus(Z_f, \psi(f)) < R$ for some fixed $R > 0$. For each $f \in \facepaths$, let 
$$
Y_f = B_X(Z_f;R). 
$$
We have for every $f \in \facepaths$ that $\psi(f) \subset Y_f$ and every component of $Y_f$ intersects $\psi(f)$.
Let 
$$
\per = \{Y_f : f \in \facepaths\}.
$$
We have that $\per$ is a $G$-invariant peripheral system, and $\psi$ is peripheral preserving. Since $\facepaths$ is thin, it is easy to see that $\per$ is thin too. 
\end{proof}

Having now constructed our peripheral system $\per$, we have gained control over the bad loops by Lemma~\ref{lem:peripheral-system-gives-relacc}. This is now enough for us to deduce accessibility in the well-behaved case. 

\begin{lemma}\label{lem:wb-acc}
    Let $\Qp = (G, X, \Gamma , \vartheta, \varphi, \psi)$ be a well-behaved quasi-planar tuple. Then $X$ is accessible. 
\end{lemma}

\begin{proof}
    By Lemma~\ref{lem:well-behaved-per-system}, we have that there exists a thin, $G$-invariant peripheral system $\per$ of $X$ such that $\psi$ is peripheral preserving with respect to $\mathcal F$ and $\per$, where $\facepaths = \inffaces(\Gamma)$. By Lemma~\ref{lem:peripheral-system-gives-relacc}, there exists a nested, $G$-invariant, $G$-finite generating $\mathcal E$ of $\br_\per(\Gamma)$. 
    Let $T = T(\E)$ be the structure tree of $\E$. 
    We now apply Theorem~\ref{thm:vertex-subgraph} and obtain a $G$-canonical tree decomposition $(T, \cV)$ of $X$, say $\cV = (V_u)_{u \in V(T)}$, such that each part $X_u := X[V_u]$ is connected and quasi-transitively stabilised, and $T/G$ is compact. For each $v \in V(T)$, let $\Gamma_v = \varphi(X_v)$. 
    
    \begin{claim}\label{lem:gammav-nobadloops}
        For every $v \in V(T)$, the graph $\Gamma_v$ contains no bad loops. 
    \end{claim}
    
    \begin{proof}
        Suppose $\Gamma_v$ contains a bad loop, say, separating $\vartheta(\omega_1)$ and $\vartheta(\omega_2)$, where $\omega_1, \omega_2 \in \Omega (\Gamma_v)$. By Proposition~\ref{prop:end-injective} we have that the inclusion $X_v \into X$ induces an injection $\Omega (X_v) \into \Omega (X)$. Since quasi-isometries induces bijections on the sets of ends, we have that the inclusion $\Gamma_v \into \Gamma$ also induces an injection $\Omega (\Gamma_v) \into \Omega (\Gamma)$. 
        Then, since the drawing of $\Gamma_v$ is just the restriction of the drawing of $\Gamma$, we find that there is a bad loop in $\Gamma$ which separates this same pair of ends in $\bbS^2$. 
        
        By Lemma~\ref{lem:bad-loop-to-cut} we find some $b \in \br_{\facepaths}(\Gamma)$  separating $\omega_1$ and $\omega_2$. By Lemma~\ref{lem:cut-through-qi} we find some $b' \in \br _{\per}( X)$ which separates $\psi(\omega_1)$ and $\psi(\omega_2)$. Indeed, it is easy to see that this $b'$ is $\per$-elliptic, since the quasi-isometry $\psi$ is peripheral preserving and $b'$ lies boundedly close to $\psi(b)$. By Proposition~\ref{prop:sep-ends} we have that $\psi(\omega_1)$ and $\psi(\omega_2)$ are separated by some $b'' \in \E$. But, by construction, then they cannot both lie in $X_v$. It follows that $\omega_1$ and $\omega_2$ cannot both be ends of $\Gamma_v$, which is a contradiction.
    \end{proof}

    Combining Lemmas~\ref{lem:gammav-nobadloops} and \ref{lem:no-bad-loops}, we immediately see that each $X_v$ is accessible. Finally, using Theorem~\ref{prop:tree-decomp-acc}, we deduce that $X$ is accessible. 
\end{proof}

We now deduce our main theorem. 

\acc*

\begin{proof}
    Assume without loss of generality that $X$ is infinite-ended. 
    Using Propositions~\ref{prop:qi-wlog},~\ref{prop:cts-inverse}, we may assume without loss of generality that $\Gamma$ has bounded-degree, and that the quasi-isometry $\varphi : X \to \Gamma$ is continuous and onto with continuous quasi-inverse $\psi$. In other words, $X$ and $\Gamma$ form part of a quasi-planar tuple. 
    
    By Theorem~\ref{thm:2-conn-to-wb}, we have that $X$ admits a canonical tree decomposition $(T, \cV)$ with bounded adhesion, where each part is a connected, locally finite quasi-transitive graph forming part of a well-behaved quasi-planar tuple. By  Lemma~\ref{lem:wb-acc} we see that every part is accessible. By Theorem~\ref{prop:tree-decomp-acc} we deduce that $X$ itself is accessible. 
\end{proof}

\subsection{Deducing Corollaries~\ref{thm:graph-intro} and \ref{thm:tfae-intro}}

We now deduce our main applications, beginning with the following structural characterisation of transitive graphs which are quasi-isometric to planar graphs. 

\graph*

\begin{proof}
    Suppose that such a canonical tree decomposition $(T, \cV)$ exists, then it follows from Proposition~\ref{prop:finitelymany-cuts} that $T/G$ is finite, and thus the separators of this tree decomposition are of uniformly bounded diameter.
    For each $u \in V(T)$, fix a complete Riemannian plane $R_u$ such that the torso $X_u := X\llbracket V_u\rrbracket$ is quasi-isometric to $R_u$. Fix a triangulation of $R_u$, and let $\Gamma_u$ be the 1-skeleton of this triangulation. Let $f_u : X_u \to \Gamma_u$ be some choice of quasi-isometry. Given an edge $vu \in E(T)$, pick some vertex $x \in V_u \cap V_v$, and glue $f_u(x)$ to $f_v(x)$. Repeat this for every edge, until we have arranged all of the $\Gamma_u$ into a `tree of planar graphs' $\Gamma$. We have that $\Gamma$ is planar. Indeed, this follows from the same argument that planarity of Cayley graphs is preserved by free products \cite[Prop.~1]{arzhantseva2004cayley}. Moreover, since each separator in the given tree decomposition has bounded diameter, it is clear that $X$ is quasi-isometric to $\Gamma$. A detailed construction of this quasi-isometry is presented in \cite{esperet2024minor}.

    Conversely, if $X$ is quasi-isometric to some planar graph then by Theorem~\ref{thm:acc-intro} it is accessible. By Theorems~\ref{thm:access-char-br} and ~\ref{thm:vertex-subgraph} we find that there is a canonical connected tree decomposition $(T, \cV)$ with bounded adhesion where each part has most one end, and $T/G$ is compact (see also \cite[Thm.~6.4]{hamann2022stallings}). Each of the one-ended parts is a quasi-transitive subgraph of $\Gamma$, which is quasi-isometrically embedded. By Theorem~\ref{thm:one-ended-intro} we deduce that they are quasi-isometric to complete Riemannian planes. 
\end{proof}

\begin{remark}
    With some work, it is possible to upgrade the complete Riemannian planes in Theorem~\ref{thm:one-ended-intro} and Corollary~\ref{thm:graph-intro} to be either the Euclidean or hyperbolic planes. In particular, every quasi-transitive graph which is quasi-isometric to a planar graph is actually quasi-isometric to a planar Cayley graph. See \cite{macmanus2024note} for details.
\end{remark}

We now consider the specific case of finitely generated groups. The following deep rigidity result will be needed. 

\begin{theorem}[{\cites{mess1988seifert, tukia1988homeomorphic, gabai1992convergence, casson1994convergence}}]\label{thm:mess}
    Let $G$ be a finitely generated group. Suppose $G$ is quasi-isometric to a complete Riemannian plane. Then $G$ is virtual surface group.
\end{theorem}

In the above theorem, by a \textit{surface group} we mean the fundamental group of a closed orientable surface of positive genus. By a \textit{virtual} surface group, we mean a group which contains a surface group as a subgroup of finite index.
We assume the reader is familiar with the rudiments of graphs of groups. For an introduction, we recommend Serre's monograph \cite{serre2002trees}. 

\begin{proposition}\label{thm:gsd-decom}
    Let $G$ be a finitely generated group. Suppose that $G$ is quasi-isometric to a planar graph. Then $G$ splits as the fundamental group of a finite graph of groups $\mathcal G$ where 
    \begin{enumerate}
        \item The edge groups of $\mathcal G$ are finite and, 

        \item Each vertex group is either finite or a virtual surface group. 
    \end{enumerate}
\end{proposition}

\begin{proof}
    We know that $G$ is accessible by Theorem~\ref{thm:acc-intro} and \cite[Thm.~1.1]{thomassen1993vertex}. Thus, $G$ splits as the fundamental group of a finite graph of groups with finite edge groups and vertex groups with at most one end. Vertex groups of such a splitting are easily seen to be quasi-isometrically embedded. By Proposition~\ref{prop:qi-wlog}, each vertex group is thus quasi-isometric to a subgraph of $\Gamma$, and thus quasi-isometric to a planar graph. By Theorems~\ref{thm:one-ended-intro} and \ref{thm:mess}, we deduce that the one-ended vertex groups are virtual surface groups. 
\end{proof}

\group*

\begin{proof}
    As mentioned in the introduction, we need only show that (1) implies (4). 
    Note that virtual surface groups are residually finite and virtually torsion-free. It is a classical fact that if $\mathcal G$ is a finite graph of groups with finite edge groups and residually finite, virtually torsion-free vertex groups, then $\pi_1(\mathcal G)$ is residually finite and virtually torsion-free.
    
    Now, let $G$ be a finitely generated group which is quasi-isometric to a planar graph. Then by Proposition~\ref{thm:gsd-decom} we have that $G$ splits as a finite graph of groups with finite edge groups and where each vertex group is either finite or a virtual surface group. Thus $G$ is virtually torsion-free and residually finite. So $G$ contains a finite index subgroup which is torsion-free, which is necessarily a free product of free and surface groups. 
\end{proof}


\bibliography{references}

\end{document}